\documentclass[12pt, a4paper]{article}

\usepackage[margin=1.3in]{geometry}
\usepackage{rotating}
\usepackage[svgnames]{xcolor}
\usepackage{amsmath,amsthm}
\usepackage[abbrev,lite,nobysame]{amsrefs}
\usepackage{mathtools}
\usepackage{accents}
\usepackage{csquotes}
\usepackage{tikz}
\usepgflibrary{patterns}
\usetikzlibrary{patterns.meta,decorations.pathmorphing
	,decorations.text} 
\usepackage{enumerate}
\usepackage{subcaption}
\usepackage{nicefrac}

\usepackage[colorlinks=true,linkcolor=NavyBlue,
citecolor=NavyBlue,urlcolor=NavyBlue]{hyperref}
\usepackage{cleveref}
\crefname{equation}{}{}

\usepackage{newtxtext}
\usepackage[varg,nosymbolsc,bigdelims]{newtxmath}
\usepackage{bm}

\usepackage[compact]{titlesec}

\numberwithin{equation}{section}

\theoremstyle{plain}

\newtheorem{thrm}{Theorem}[section]
\newtheorem{lmm}[thrm]{Lemma}
\newtheorem{crllr}[thrm]{Corollary}
\newtheorem{prpstn}[thrm]{Proposition}

\theoremstyle{definition}

\newtheorem{xmpl}[thrm]{Example}
\newtheorem{rmrk}[thrm]{Remark}

\theoremstyle{plain}

\newcommand{\xvec}[1]{\bm{#1}}
\newcommand{\xsym}[1]{\bm{#1}}
\newcommand{\xdop}[1]{\bm{\mathrm{#1}}}
\def\xnab{\xdop{\nabla}}
\newcommand{\xmcal}[1]{\bm{\mathcal{#1}}}
\newcommand{\xcurl}[1]{\xdop{\nabla}\times{#1}}
\newcommand{\xwcurl}[1]{\xdop{\nabla}\wedge{#1}}

\newcommand{\xdiv}[1]{\xdop{\nabla}\cdot{#1}}
\newcommand{\xdx}[1]{{{\rm d}#1}}

\def\cf{\emph{cf.\/}}\def\eg{\emph{e.g.\/}}

\def\apriori{\emph{a priori\/}}

\def\etal{\emph{et al.\/}}

\def\@Rref#1{\hbox{\rm \ref{#1}}}
\def\Rref#1{\@Rref{#1}}

\def\xCzero{{\rm C}^{0}}
\def\xCone{{\rm C}^{1}} 
 
\def\xCinfty{{\rm C}^{\infty}} 
\def\xCn#1{{\rm C}^#1}

\def\xH{{\rm H}}

\def\xHone{{\rm H}^{1}}
\def\xHtwo{{\rm H}^{2}} 

\def\xHn#1{{\rm H}^#1}
\def\xW{{\rm W}}

\def\xWn#1{{\rm W}^#1}

\def\xLone{{\rm L}^{1}}
\def\xLtwo{{\rm L}^{2}}

\def\xLinfty{{\rm L}^{\infty}} 
\def\xLn#1{{\rm L}^#1}

\def\xD{{\rm D}}

\def\xY{{\rm Y}}

\def\xdrv#1#2{\frac{{\rm d}#1}{{\rm d}#2}}

\begin{document}\linespread{1.05}\selectfont
	
	\date{}

	\author{Manuel Rissel \thanks{School of Mathematical Sciences, Shanghai Jiao Tong University, Shanghai, 200240, P. R. China;  manuel.rissel@sjtu.edu.cn} \and Ya-Guang Wang \thanks{School of Mathematical Sciences, Center for Applied Mathematics, MOE-LSC, CMA-Shanghai and SHL-MAC, Shanghai Jiao Tong University, Shanghai, 200240, P. R. China; ygwang@sjtu.edu.cn}}

	\title{Small-time global approximate controllability for incompressible MHD with coupled Navier slip boundary conditions}	
	\maketitle

	\begin{abstract}
		We study the small-time global approximate controllability for incompressible magnetohydrodynamic (MHD) flows in smoothly bounded two- or three-dimensional domains. The controls act on arbitrary nonempty open portions of each connected boundary component, while linearly coupled Navier slip-with-friction conditions are imposed along the uncontrolled parts of the boundary. Some choices for the friction coefficients give rise to interacting velocity and magnetic field boundary layers. We obtain sufficient dissipation properties of these layers by a detailed analysis of the corresponding asymptotic expansions. For certain friction coefficients, or if the obtained controls are not compatible with the induction equation, an additional pressure-like term appears. We show that such a term does not exist for problems defined in planar simply-connected domains and various choices of Navier slip-with-friction boundary conditions.
	\end{abstract}

	${}$
	\\
	
	\section{Introduction}\label{section:introduction} 
	Let $\Omega \subset \mathbb{R}^N$ be a bounded domain of dimension $N \in \{2,3\}$ with $\Gamma \coloneqq \partial \Omega$ being smooth and $\xvec{n}$ denoting the outward unit normal vector to $\Omega$ along $\Gamma$. The goal of this article is to steer incompressible magnetohydrodynamic (MHD) flows from a prescribed initial state approximately towards a desired terminal state, without placing restrictions on the distance between these states or on the control time. This is accomplished by acting on the system via boundary controls supported in a possibly small subset $\Gamma_{\operatorname{c}} \subset \Gamma$, the interior of which is assumed to intersect non-trivially with all connected components of $\Gamma$. 
	
	When it comes to nonlinear evolution equations with controls localized in an arbitrary open subset of the boundary or an interior sub-domain, establishing the global approximate controllability usually constitutes a challenging task and few methods for tackling such questions are available. In the context of fluid dynamics, but not limited to, one successful approach is known as the return method, which has first been introduced by Coron in \cite{Coron1992} for the stabilization of certain mechanical systems and shall also be employed here. A comprehensive introduction to the return method and its applications to nonlinear partial differential equations may  be found in \cite[Part 2, Chapter 6]{Coron2007}. In contrast to the Navier--Stokes equations, for which the global approximate controllability has been actively investigated in the past, nothing seems to be known regarding the global approximate controllability of viscous MHD in the presence of physical boundaries. This article therefore aims to initiate further research in this direction. A broader motivation of this topic is provided by a well-known open problem due to J.-L. Lions, asking for the global approximate controllability of the Navier--Stokes equations with the no-slip condition (\cf~\cite{LionsJL1991} and also \cite{CoronMarbachSueur2020,CoronMarbachSueurZhang2019} for recent progress). 
	
	In this article, we focus on incompressible flows of viscosity $\nu_1 > 0$ and resistivity $\nu_2 > 0$, for which the velocity $\xvec{u}\colon\Omega\times(0,T) \longrightarrow \mathbb{R}^N$, the magnetic field $\xvec{B}\colon\Omega\times(0,T) \longrightarrow \mathbb{R}^N$, and the total pressure $p\colon\Omega\times(0,T) \longrightarrow\mathbb{R}$ are described, until a given terminal time $T_{\operatorname{ctrl}} > 0$, as a  solution to the initial boundary value problem
	\begin{equation}\label{equation:MHD00}
		\begin{cases}
			\partial_t \xvec{u} - \nu_1 \Delta \xvec{u} + (\xvec{u} \cdot \xdop{\nabla}) \xvec{u} - \mu(\xvec{B} \cdot \xdop{\nabla})\xvec{B} + \xdop{\nabla} p =  \xvec{0} & \mbox{ in } \Omega \times (0, T_{\operatorname{ctrl}}),\\
			\partial_t \xvec{B} - \nu_2 \Delta \xvec{B} + (\xvec{u} \cdot \xdop{\nabla}) \xvec{B} - (\xvec{B} \cdot \xdop{\nabla}) \xvec{u} = \xvec{0} & \mbox{ in } \Omega \times (0, T_{\operatorname{ctrl}}),\\
			\xdop{\nabla}\cdot\xvec{u} = \xdop{\nabla}\cdot\xvec{B} = 0  & \mbox{ in } \Omega \times (0, T_{\operatorname{ctrl}}),\\
			\xvec{u} \cdot \xvec{n} = \xvec{B} \cdot \xvec{n} = 0  & \mbox{ on } (\Gamma\setminus\Gamma_{\operatorname{c}}) \times (0, T_{\operatorname{ctrl}}),\\
			\xmcal{N}_1(\xvec{u},\xvec{B}) = \xmcal{N}_2(\xvec{u},\xvec{B}) = \xvec{0}  & \mbox{ on } (\Gamma\setminus\Gamma_{\operatorname{c}}) \times (0, T_{\operatorname{ctrl}}),\\
			\xvec{u}(\cdot, 0)  =  \xvec{u}_0,\, \xvec{B}(\cdot, 0)  =  \xvec{B}_0  & \mbox{ in } \Omega,
		\end{cases}
	\end{equation}
	while the underlying state space, for both the velocity and the magnetic field, is taken as
	\[
	\xLtwo_{\operatorname{c}}(\Omega) \coloneqq \left\{ \xvec{f} \in \xLtwo(\Omega;\mathbb{R}^N) \, \Big| \, \xdop{\nabla}\cdot\xvec{f} = 0 \mbox{ in } \Omega, \xvec{f} \cdot \xvec{n} = 0 \mbox{ on } \Gamma\setminus\Gamma_{\operatorname{c}} \right\}.
	\]
	In \eqref{equation:MHD00}, the vector fields $\xvec{u}_0 \in \xLtwo_{\operatorname{c}}(\Omega)$ and $\xvec{B}_0 \in \xLtwo_{\operatorname{c}}(\Omega)$ represent the given initial data. The parameter~$\mu > 0$ quantifies the magnetic permeability. As explained in \Cref{remark:underdetermined} below, the controls are sought in an implicit form and, therefore, no boundary conditions are prescribed along the controlled boundary $\Gamma_{\operatorname{c}}$. Before introducing the general boundary operators $\xmcal{N}_1$ and $\xmcal{N}_2$ acting at~$\Gamma\setminus\Gamma_{\operatorname{c}}$, it is emphasized that they particularly include all boundary conditions of the form
	\begin{equation}\label{equation:smpl}
		(\xcurl{\xvec{u}}) \times \xvec{n} = [\xvec{M}\xvec{u}]_{\operatorname{tan}}, \quad
		(\xcurl{\xvec{B}})\times\xvec{n} = \xsym{0}, \quad \xvec{u} \cdot \xvec{n} = \xvec{B} \cdot \xvec{n} = 0,
	\end{equation}
	with symmetric $\xvec{M}\in \xCinfty(\Gamma \setminus \Gamma_{\operatorname{c}};\mathbb{R}^{N\times N})$ and $[\cdot]_{\operatorname{tan}}$ denoting the tangential part. 
	
	\begin{rmrk}
		The common three-dimensional notations for the cross product and curl operator are employed whenever $N=3$ is possible. If considering the case $N = 2$, one has to replace~$\xcurl{\xvec{h}}$ by~$\xwcurl{\xvec{h}} \coloneqq \partial_1 h_2 - \partial_2 h_1$, while the curl of a scalar function $h$ refers to $\xnab^{\perp} h \coloneqq [\partial_2 h,-\partial_1h]^{\top}$. Moreover, the cross product $\xvec{h} \times \xvec{g}$ in two-dimensions becomes $\xvec{h} \wedge \xvec{g} \coloneqq h_1g_2 - g_2h_1$. Consequentially, for planar configurations, some objects denoted here as vectors might be scalars and $(\xcurl{\xvec{h}})\times\xvec{n}$ means $(\xwcurl{\xvec{h}}) [n_{2},-n_{1}]^{\top}$.
	\end{rmrk}
	\paragraph{The Navier slip-with-friction boundary conditions.} Let $\mathcal{D} \subset \mathbb{R}^N$ represent any smoothly bounded domain with outward unit normal field $\xvec{n}_{\mathcal{D}}\colon \partial\mathcal{D}\longrightarrow \mathbb{R}^N$. The tangent space of $\partial \mathcal{D}$ at a point $\xvec{x}$ is denoted by ${\rm T}_{\xvec{x}}$.
	Given any $\xvec{x} \in \partial \mathcal{D}$, we introduce the Weingarten map (or shape operator)
	\[
	\xvec{W}_{\mathcal{D}}(\xvec{x})\colon {\rm T}_{\xvec{x}} \longrightarrow {\rm T}_{\xvec{x}}, \quad \xsym{\tau} \longmapsto \xvec{W}_{\mathcal{D}}(\xvec{x})\xsym{\tau} \coloneqq \xdop{\nabla}_{\xsym{\tau}} \xvec{n}_{\mathcal{D}}. 
	\] 
	Then, for $\xvec{h}_1$, $\xvec{h}_2\colon \overline{\Omega} \to \mathbb{R}^N$ and friction coefficient matrices
	\begin{equation}\label{equation:fop}
		\xvec{L}_1, \xvec{L}_2, \xvec{M}_1, \xvec{M}_2 \in \xCinfty(\Gamma\setminus\Gamma_{\operatorname{c}};\mathbb{R}^{N\times N}),
	\end{equation}
	the linearly coupled Navier slip-with-friction operators in \eqref{equation:MHD00} are defined as
	\begin{equation}\label{equation:bc2}
		\begin{gathered}
			\xmcal{N}_i(\xvec{h}_1,\xvec{h}_2) \coloneqq \left[\xdop{D}(\xvec{h}_i)\xvec{n}(\xvec{x}) + \xvec{W}_{\Omega} \xvec{h}_i + \xvec{M}_i(\xvec{x})\xvec{h}_1 + \xvec{L}_i \xvec{h}_2 \right]_{\operatorname{tan}}, \quad i=1,2,
		\end{gathered}
	\end{equation}
	where the tangential part and the symmetrized gradient are respectively denoted by
	\[
	[\xvec{h}]_{\operatorname{tan}} \coloneqq \xvec{h} - \left(\xvec{h} \cdot \xvec{n}_{\mathcal{D}}\right) \xvec{n}_{\mathcal{D}}, \quad \xdop{D}(\xvec{h}) \coloneqq \frac{1}{2}[\xdop{\nabla}\xvec{h} + (\xdop{\nabla}\xvec{h})^{\top}].
	\]
	The Weingarten map $\xvec{W}_{\mathcal{D}}$ is smooth, \cf~\cite[Lemma 1]{CoronMarbachSueur2020} and \cite{ClopeauMikelicRobert1998,GieKelliher2012}, and when $\xvec{h}$ is tangential to $\partial \mathcal{D}$ one has the relation
	\begin{equation}\label{equation:wgtf}
		[\xdop{D}(\xvec{h}(\xvec{x}, t))\xvec{n}_{\mathcal{D}}(\xvec{x}) + \xvec{W}_{\mathcal{D}}(\xvec{x})\, \xvec{h}(\xvec{x}, t)]_{\operatorname{tan}} = - \frac{1}{2}\left(\xdop{\nabla}\times\xvec{h}(\xvec{x}, t)\right)\times\xvec{n}_{\mathcal{D}}(\xvec{x}).
	\end{equation}
	As a consequence of \eqref{equation:wgtf}, the boundary conditions prescribed in \eqref{equation:MHD00} along $(\Gamma\setminus\Gamma_{\operatorname{c}}) \times (0, T_{\operatorname{ctrl}})$ are equivalent to
	\begin{equation}\label{equation:rhobc}
		\begin{aligned}
			(\xcurl{\xvec{u}}) \times \xvec{n} & =  \xsym{\rho}_1(\xvec{u},\xvec{B}) \coloneqq 2\left[\xvec{M}_1(\xvec{x})\xvec{u} + \xvec{L}_1(\xvec{x})\xvec{B}\right]_{\operatorname{tan}}, & \xvec{u} \cdot \xvec{n} = 0,\\
			(\xcurl{\xvec{B}}) \times \xvec{n} & = \xsym{\rho}_2(\xvec{u},\xvec{B}) \coloneqq 2\left[\xvec{M}_2(\xvec{x})\xvec{u} + \xvec{L}_2(\xvec{x})\xvec{B}\right]_{\operatorname{tan}}, & \xvec{B} \cdot \xvec{n} = 0.
		\end{aligned}	
	\end{equation}
	Our main motivation is to treat simply-connected domains and the boundary conditions \eqref{equation:smpl}. This already constitutes a general setup, which has not been studied in terms of controllability but recently attracted increased attention in view of inviscid limit problems. However, the constructions of the controls naturally extend, at least in parts, also to the case of more general domains and allow the prescription of boundary conditions of the form \eqref{equation:rhobc}. Attention shall be paid to situations where $\xvec{M}_2 \neq \xvec{0}$ or $\xvec{L}_2 \neq \xvec{0}$ holds, since such configurations lead to interesting challenges that are not fully resolved here. 
	
	The Navier slip-with-friction boundary conditions, as already proposed by Navier \cite{Navier1823} two centuries ago, are relevant to a range of applications, thus have been studied in the context of the Navier--Stokes equations from various points of view. For instance, in the absence of magnetic fields, inviscid limit problems are treated in \cite{IftimieSueur2011,ClopeauMikelicRobert1998,Kelliher2006,XiaoXin2013}, regularity questions are investigated in \cite{AmrouchePenelSeloula2013,AlBabaAmroucheEscobedo2017,Shibata2007,Shimada2007,AlBaba2019} and controllability problems are tackled in \cite{CoronMarbachSueur2020,Guerrero2006,LionsZuazua1998,Coron1996}. Concerning the situation of incompressible viscous MHD, several singular limit problems involving uncoupled Navier slip-with-friction boundary conditions are addressed in \cite{GuoWang2016,XiaoXinWu2009,MengWang2016}; comparing with these references, the here employed boundary conditions are more general in that the shear stresses of the velocity and the magnetic field at the boundary are linearly coupled with tangential velocity and magnetic field contributions. While~\eqref{equation:rhobc} includes the classical Navier slip condition for the velocity, it can capture also more complex interactions in the presence of magnetic fields. 
	
	From the global approximate controllability point of view, several difficulties appear however when the magnetic shear stress is coupled with the tangential velocity: a magnetic field boundary layer potentially enters the analysis of \Cref{section:approxres2}. This in turn challenges the construction of magnetic field boundary controls without generating a pressure gradient term or additional control forces in the induction equation.

	\subsection{Main results}
	The statements of the main theorems anticipate the notion of weak controlled trajectories, as introduced later in \Cref{subsection:wct}. Briefly speaking, a weak controlled trajectory will be defined as the restriction to $\Omega$ of a Leray--Hopf weak solution to a version of the problem \eqref{equation:MHD00}, posed in an enlarged domain and driven by interior controls.
	
	\begin{thrm}\label{theorem:main1} Assume that $\Omega \subset \mathbb{R}^2$ is simply-connected, that $\xvec{M}_2 = \xvec{L}_2 = \xvec{0}$, and that $\Gamma_{\operatorname{c}}\subset\Gamma$ is connected. Then, for arbitrarily fixed $T_{\operatorname{ctrl}} > 0$, $\delta > 0$, and $\xvec{u}_0, \xvec{B}_0, \xvec{u}_1, \xvec{B}_1 \in \xLtwo_{\operatorname{c}}(\Omega)$, 
		there exists at least one weak controlled trajectory
		\[
		(\xvec{u}, \xvec{B}) \in 	\left[\xCn{0}_{w}([0,T_{\operatorname{ctrl}}];\xLn{2}_{\operatorname{c}}(\Omega)) \cap \xLn{2}((0,T_{\operatorname{ctrl}});\xHn{1}(\Omega))\right]^2
		\]
		to the MHD equations \eqref{equation:MHD00} which obeys the terminal condition
		\begin{equation}\label{equation:endcondition_introapprox00}
			\|\xvec{u}(\cdot, T_{\operatorname{ctrl}}) - \xvec{u}_1\|_{\xLtwo(\Omega)} + 	\|\xvec{B}(\cdot, T_{\operatorname{ctrl}}) - \xvec{B}_1\|_{\xLtwo(\Omega)} < \delta.
		\end{equation}	
	\end{thrm}
	
	\begin{figure}[ht!]
		\centering
		\resizebox{0.45\textwidth}{!}{
			\begin{tikzpicture}
				\clip(0.35,-1.1) rectangle (6.8,4.2);
				
				\draw [line width=0.5mm, color=black, fill=RoyalBlue!5] plot[smooth, tension=1] coordinates {(1,0.5) (0.65,1.9) (2.2,3.8) (6.4,2.18) (5,0.8) (5.8,-0.2) (3.2,-1) (2.35,0) (1.5,0.2) (1,0.5)};

				\draw [line width=1.6mm, color=white]  plot[smooth, tension=0.8] coordinates {(2.58,-0.5)  (2.34,0.035) (1.52,0.2) (0.98,0.5) (0.55,1.3) (0.72,2) (0.8,2.55)};
				
				\draw [dashed,line width=1.6mm, color=black]  plot[smooth, tension=0.8] coordinates {(2.58,-0.5)  (2.34,0.035) (1.52,0.2) (0.98,0.5) (0.55,1.3) (0.72,2) (0.8,2.55)};

				\coordinate[label=left:$\Omega$] (A) at (3.5,1.8);	
				\coordinate[label=left:$\Gamma_{\operatorname{c}}$] (A) at (1,0.4);											
			\end{tikzpicture}
		}
		\caption{Sketch of a smoothly bounded simply-connected domain $\Omega \subset \mathbb{R}^2$ with connected controlled boundary $\Gamma_{\operatorname{c}}$, which is indicated by a dashed line. }
		\label{Figure:MainthrmFig}
	\end{figure}

	\begin{rmrk}
		Under additional assumptions on the normal traces at $\Gamma_{\operatorname{c}}$ of the initial data in \Cref{theorem:main1}, one can choose $\Gamma_{\operatorname{c}}$ as any open subset of $\Gamma$, see \Cref{remark:rca}.
	\end{rmrk}

	The next theorem holds for $\xvec{L}_1, \xvec{L}_2, \xvec{M}_1, \xvec{M}_2 \in \xCinfty(\Gamma\setminus\Gamma_{\operatorname{c}};\mathbb{R}^{N\times N})$, but it involves a pressure-like unknown~$q$ and a control $\xsym{\zeta}$, which however satisfies $\xsym{\zeta} \equiv \xsym{0}$ when~$\xvec{M}_2 = \xsym{0}$. When~$N = 3$, we need additional assumptions, since, to our knowledge, there is currently no literature providing $\xLinfty((0,T);\xHone(\mathcal{E}))$ and $\xLinfty((0,T);\xHtwo(\mathcal{E}))$ strong solutions for MHD under general (non-symmetric) Navier slip-with-friction conditions; this prevents us to prove \Cref{lemma:reg} in the general case. Thus, when $N = 3$, we introduce the following class of the initial data.
	\paragraph{The class $\mathbf{S}$.} When $\xvec{M}_1$, $\xvec{L}_2$ are symmetric, $\xvec{L}_1 = \xvec{M}_2 = \xvec{0}$, and the domain $\Omega \subset \mathbb{R}^3$ is simply-connected, then $\mathbf{S} = \xLtwo_{\operatorname{c}}(\Omega)^2$. Otherwise, the class $\mathbf{S}$ consists of all $(\xvec{u}_0, \xvec{B}_0) \in \xLtwo_{\operatorname{c}}(\Omega)^2$ which are restrictions of divergence-free vector fields $\widetilde{\xvec{u}}_0, \widetilde{\xvec{B}}_0 \in \xHn{3}(\mathcal{E})$ that are tangential to $\partial \mathcal{E}$, where $\mathcal{E} \subset \mathbb{R}^3$ is a smoothly bounded domain extension for $\Omega \subset \mathbb{R}^3$ of the type introduced in \Cref{subsection:domainextensions}.

	\begin{xmpl}
		All states $\xvec{u}_0, \xvec{B}_0 \in \xLtwo_{\operatorname{c}}(\Omega)\cap \xHn{3}(\Omega)$ which vanish at $\Gamma_{\operatorname{c}}$ and have vanishing normal derivatives up to the second order at $\Gamma_{\operatorname{c}}$ belong to $\mathbf{S}$.
	\end{xmpl}
	
	\begin{thrm}\label{theorem:main}
		For any given time $T_{\operatorname{ctrl}} > 0$, fixed initial states $\xvec{u}_0, \xvec{B}_0 \in \xLtwo_{\operatorname{c}}(\Omega)$, belonging to the class $\mathbf{S}$ when $N = 3$, target states $\xvec{u}_1, \xvec{B}_1 \in \xLn{2}_{\operatorname{c}}(\Omega)$, and $\delta > 0$, there exists a smooth function $\xsym{\zeta}\colon \overline{\Omega}\times[0,T_{\operatorname{ctrl}}] \longrightarrow \mathbb{R}^N$, with $\xsym{\zeta} \equiv \xvec{0}$ when $\xvec{M}_2 = \xvec{0}$, such that the MHD system
		\begin{equation}\label{equation:MHD00withq}
			\begin{cases}
				\partial_t \xvec{u} - \nu_1 \Delta \xvec{u} + (\xvec{u} \cdot \xdop{\nabla}) \xvec{u} - \mu(\xvec{B} \cdot \xdop{\nabla})\xvec{B} + \xdop{\nabla} p =  \xvec{0} & \mbox{ in } \Omega \times (0, T_{\operatorname{ctrl}}),\\
				\partial_t \xvec{B} - \nu_2 \Delta \xvec{B} + (\xvec{u} \cdot \xdop{\nabla}) \xvec{B} - (\xvec{B} \cdot \xdop{\nabla}) \xvec{u} = \xdop{\nabla} q + \xsym{\zeta} & \mbox{ in } \Omega \times (0, T_{\operatorname{ctrl}}),\\
				\xdop{\nabla}\cdot\xvec{u} = \xdop{\nabla}\cdot\xvec{B} = 0  & \mbox{ in } \Omega \times (0, T_{\operatorname{ctrl}}),\\
				\xvec{u} \cdot \xvec{n} = \xvec{B} \cdot \xvec{n} = 0  & \mbox{ on } (\Gamma\setminus\Gamma_{\operatorname{c}}) \times (0, T_{\operatorname{ctrl}}),\\
				\xmcal{N}_1(\xvec{u},\xvec{B}) = \xmcal{N}_2(\xvec{u},\xvec{B}) = \xvec{0}  & \mbox{ on } (\Gamma\setminus\Gamma_{\operatorname{c}}) \times (0, T_{\operatorname{ctrl}}),\\
				\xvec{u}(\cdot, 0)  =  \xvec{u}_0,\, \xvec{B}(\cdot, 0)  =  \xvec{B}_0  & \mbox{ in } \Omega
			\end{cases}
		\end{equation}
		admits at least one weak controlled trajectory
		\[
		(\xvec{u}, \xvec{B}) \in \left[\xCn{0}_{w}([0,T_{\operatorname{ctrl}}];\xLn{2}_{\operatorname{c}}(\Omega)) \cap \xLn{2}((0,T_{\operatorname{ctrl}});\xHn{1}(\Omega))\right]^2
		\]
		obeying the terminal condition
		\begin{equation}\label{equation:termcondth}
			\|\xvec{u}(\cdot, T_{\operatorname{ctrl}}) - \xvec{u}_1\|_{\xLtwo(\Omega)} + \|\xvec{B}(\cdot, T_{\operatorname{ctrl}}) - \xvec{B}_1\|_{\xLtwo(\Omega)} < \delta.
		\end{equation}
	\end{thrm}

	\begin{rmrk}\label{remark:cylm2}
		When $\xvec{M}_2 \neq \xvec{0}$, the control $\xsym{\zeta}$ may enter \eqref{equation:MHD00withq} if the magnetic field boundary layer described in \Cref{subsubsection:DefinitionsBl} is not divergence-free. In order to illustrate that this statement is not sharp, we consider, as in \Cref{Figure:DomainExample4}, a cylinder $\Omega \coloneqq (a,b) \times D$, for $-\infty~<~a~<~b~<~+\infty$ and a smoothly bounded connected open set $D \subset \mathbb{R}^2$, with controlled part $\Gamma_{\operatorname{c}} \coloneqq \{a,b\} \times D$. In this case, \Cref{theorem:main} is valid for all $\xvec{L}_1, \xvec{L}_2, \xvec{M}_1, \xvec{M}_2 \in \xCinfty(\Gamma\setminus\Gamma_{\operatorname{c}};\mathbb{R}^{N\times N})$ with $\xsym{\zeta} = \xvec{0}$. This will be illustrated by means of \Cref{example:cylinder} combined with the discussion in \Cref{subsubsection:technicalprofiles}.
	\end{rmrk}

	\begin{figure}[ht!]
		\centering
		\begin{subfigure}[b]{0.48\textwidth}\centering
			\resizebox{0.8\textwidth}{!}{
				\begin{tikzpicture}
					\clip(0.5,-1.1) rectangle (6.8,4.2);
					
					\draw [line width=0.5mm, color=black, fill=RoyalBlue!5] plot[smooth, tension=1] coordinates {(1,0.5) (0.65,1.9) (2.2,3.8) (6.4,2.18) (5,0.8) (5.8,-0.2) (3.2,-1) (2.35,0) (1.5,0.2) (1,0.5)};

					\draw [line width=1mm, color=white]  plot[smooth, tension=0.8] coordinates {(2.58,-0.5)  (2.34,0.035) (1.52,0.2) (0.98,0.5) (0.55,1.3) (0.72,2) (0.8,2.55)};
					
					\draw [dashed,line width=1.4mm, color=black]  plot[smooth, tension=0.8] coordinates {(2.58,-0.5)  (2.34,0.035) (1.52,0.2) (0.98,0.5) (0.55,1.3) (0.72,2) (0.8,2.55)};

					\draw [line width=0.5mm, color=black, fill=white]  plot[smooth, tension=0.8] coordinates {(2,0.5) (1.2,1.6) (2.5,3) (3.2,2.6) (4.5,2.8) (4,1.5) (4.6,0.5) (3.5,0)  (2,0.5)};

					\draw [line width=1mm, color=white]  plot[smooth, tension=0.8] coordinates { (4.03,1.7) (4,1.4) (4.15,1.1) (4.6,0.45) (3.5,0) (2,0.5)};
					\draw [dashed,line width=1.4mm, color=black]  plot[smooth, tension=0.8] coordinates { (4.03,1.7) (4,1.4) (4.15,1.1) (4.6,0.45) (3.5,0) (2,0.5)};				
					
				\end{tikzpicture}
			}
			\caption{A general multiply-connected domain as considered in \Cref{theorem:main}. The sketch is two-dimensional only for simplicity.}
			\label{Figure:DomainExample2}
		\end{subfigure}
		\quad
		\begin{subfigure}[b]{0.48\textwidth}\centering
			\resizebox{0.8\textwidth}{!}{
				\begin{tikzpicture}
					\clip[rotate=0](-0.7,-1) rectangle (8.1,4.6);
					\draw [line width=0.4mm, color=black!200, fill=RoyalBlue!4] 	plot[smooth cycle] (0.2,0) rectangle ++(7,4.5);
					
					\draw [line width=0.4mm, preaction={fill=RoyalBlue!50}] 	(1,2.25) arc (0:360:0.8cm and 2.25cm);
					
					\draw [line width=0.4mm](7,4.5)--(7.2,4.5);
					\draw [line width=0.4mm](7,0)--(7.2,0);
					\draw [line width=0.4mm, preaction={fill=RoyalBlue!50}] 	(8,2.25) arc (0:360:0.8cm and 2.25cm);

					\draw [line width=0.4mm, color=black, fill=white, dashed] plot[smooth cycle] (0.2,1) rectangle ++(7,2.5);
					
					\draw [line width=0.4mm, color=black, fill=white, dashed] (0.6,2.25) arc (0:360:0.4cm and 1.25cm);
					
					\draw [line width=0.4mm, dashed](7,3.5)--(7.2,3.5);
					\draw [line width=0.4mm, dashed](7,1)--(7.2,1);
					\draw [line width=0.4mm, color=black, fill=white, dashed] (7.6,2.25) arc (0:360:0.4cm and 1.25cm);
					\draw [line width=0.4mm] 	(8,2.25) arc (0:360:0.8cm and 2.25cm);
					\draw [line width=0.4mm, preaction={fill=RoyalBlue!50}] 	(8,2.25) arc (0:360:0.8cm and 2.25cm);
					\draw [line width=0.4mm, color=black, fill=white, dashed] (7.6,2.25) arc (0:360:0.4cm and 1.25cm);
					\draw [line width=0.4mm, dashed](6.6,1)--(7,1);
					\draw [line width=0.4mm, dashed](6.6,3.5)--(7,3.5);
				\end{tikzpicture}
			}
			\caption{A multiply-connected cylindrical domain as in \Cref{remark:cylm2}, with controls at the base faces. In this case, one can take $\xsym{\zeta} = \xsym{0}$.}
			\label{Figure:DomainExample4}
		\end{subfigure}
		\caption{Two exemplary domains that are covered by \Cref{theorem:main}.}
		\label{Figure:MainthrmFig2}
	\end{figure}
	
	\begin{rmrk}\label{remark:underdetermined}
		The systems \cref{equation:MHD00,equation:MHD00withq} are under-determined since no boundary condition is prescribed along $\Gamma_{\operatorname{c}}$. Once a weak controlled trajectory is found via Theorem~\Rref{theorem:main1} or \Rref{theorem:main}, one obtains explicit boundary controls by taking traces along $\Gamma_{\operatorname{c}}$, see also \cite{Coron1996,CoronMarbachSueur2020,Fernandez-CaraSantosSouza2016,Glass2000}.
	\end{rmrk}
	
	\begin{rmrk}
		Since the proofs for Theorems~\Rref{theorem:main1} and \Rref{theorem:main} will be carried out in a certain extended domain, one can allow the interior of $\Gamma_{\operatorname{c}}$ to be part of a Lipschitz continuous Jordan curve. Moreover, the controlled boundary $\Gamma_{\operatorname{c}}$ is allowed to meet $\Gamma\setminus\Gamma_{\operatorname{c}}$ in a non-smooth way, as long as one can define domain extensions in the sense of \Cref{subsection:domainextensions}.
	\end{rmrk}

	\subsection{Related literature and organization of the article}
	The global approximate controllability for viscous- and resistive MHD in non-periodic domains has to our knowledge not been studied, neither for incompressible- nor for compressible models. Therefore, the present work constitutes a first step in this direction. As a possible continuation, it would be interesting to generalize \Cref{theorem:main1} for arbitrary~$N~\in~\{2,3\}$ and $\xvec{M}_2, \xvec{L}_2 \in \xCinfty(\Gamma\setminus\Gamma_{\operatorname{c}};\mathbb{R}^{N\times N})$ without additional interior control. Also, the question of global exact controllability to zero or towards trajectories remains open.
	
	Concerning local exact controllability for MHD, where the initial state lies in the vicinity of the target trajectory, there have been some interesting works when the velocity satisfies the no-slip boundary condition. For incompressible viscous MHD, Badra obtained in \cite{Badra2014} the local exact controllability to trajectories, while maintaining truly localized and solenoidal interior controls. However, since the boundary conditions are different from those employed here, one cannot deduce the small-time global exact controllability towards trajectories by combining the approaches given in \cite{Badra2014} with our global approximate results.  A variety of previous local exact controllability results may also be found in \cite{BarbuHavarneanuPopaSritharan2005} by Barbu \etal{} and in \cite{HavarneanuPopaSritharan2006,HavarneanuPopaSritharan2007} by Hav\^{a}rneanu \etal{}, while approximate interior controllability for certain toroidal configurations without boundary has been investigated by Galan in \cite{Galan2013}. Moreover, Anh and Toi studied in \cite{AnhToi2017} the local exact controllability to trajectories for magneto-micropolar fluids, while Tao considered the local exact controllability for planar compressible MHD in the recent work \cite{Tao2018}. 
	Recently, we have studied the global exact controllability for the ideal incompressible MHD in \cite{RisselWang2021}, in which the small-time global exact controllability in rectangular channels is obtained in the presence of a harmonic unknown $q$ as in \eqref{equation:MHD00withq}.
	Subsequently, Kukavica \etal~demonstrated in \cite{KukavicaNovackVicol2022}, likewise restricted to a rectangular domain controlled at two opposing walls, how to find boundary controls such that~$\xnab q$ either vanishes or is explicitly characterized. 
	
	Aside of various MHD specific constructions, this article combines the return method and the well-prepared dissipation method as described by Coron \etal{} in \cite{CoronMarbachSueur2020}, where the small-time global exact controllability to trajectories has been studied for incompressible Navier--Stokes equations in two- and three-dimensional domains with Navier slip-with-friction conditions. Meanwhile, we shall also extend certain asymptotic expansions, obtained by Iftimie and Sueur in \cite{IftimieSueur2011} for the incompressible Navier--Stokes equations, to the present MHD system. Due to the structure of the induction equation, the return method has to be carefully implemented in order to avoid generating pressure-like and additional forcing terms in the induction equation. To this end, under the assumptions of \Cref{theorem:main1}, we modify the return method trajectory from \cite{CoronMarbachSueur2020} to be everywhere divergence-free, but allow a nonzero curl in the control region; this approach seems new and might be useful for further studies on the controllability of the ideal MHD equations. 
	
	Let us also mention other recent works on global controllability problems for fluids that employ the return- and well-prepared dissipation methods. For instance, an incompressible Boussinesq system with Navier slip-with-friction boundary conditions for the velocity is considered by Chaves-Silva \etal~in \cite{ChavesSilva2020SmalltimeGE}. Moreover, the question of smooth controllability for the Navier--Stokes equations with Navier slip-with-friction boundary conditions is investigated in~\cite{LiaoSueurZhang2022}. Further, Coron \etal~obtain in \cite{CoronMarbachSueurZhang2019} global exact controllability results for the Navier--Stokes equations under the no-slip condition in a rectangular domain.

	\paragraph{Organization of this article.} \Cref{section:extentsions} collects several preliminaries and defines notions of weak controlled trajectories. In \Cref{section:approxres2}, the global approximate controllability from sufficiently regular initial data towards smooth states is shown. The main theorems are concluded in \Cref{section:conclth}. In Appendices~\Rref{appendix:higherorderestimates} and \Rref{appendix:proofreg}, boundary layer estimates and a proof of \Cref{lemma:reg} are provided. 
	
	\section{Preliminaries}\label{section:extentsions}
	
	A domain extension~$\mathcal{E}$ for~$\Omega$ is introduced in \Cref{subsection:domainextensions}, several function spaces and norms are defined in \Cref{subsection:FunctionSpaces}, initial data extensions to $\mathcal{E}$ are discussed in \Cref{subsection:initialDataExtensions}, notions of weak controlled trajectories for \eqref{equation:MHD00} and \eqref{equation:MHD00withq} are discussed in~\Cref{subsection:wct}. Finally, \Cref{subsection:Description} briefly outlines the strategy of the paper. Throughout, if not indicated otherwise, constants of the form~$C > 0$ are generic and can change from line to line during the estimates.
	
	\subsection{Domain extensions}\label{subsection:domainextensions}
	In what follows, the sets $\Gamma^1, \dots, \Gamma^{K(\Omega)}$ denote the connected components of $\Gamma$ and $\Gamma_{\operatorname{c}}^1, \dots, \Gamma_{\operatorname{c}}^{{K(\Omega)}}$ stand for the respective intersections $\Gamma^1\cap\Gamma_c,\dots,\Gamma^{K(\Omega)}\cap\Gamma_c$, hence
	\begin{alignat*}{2}
		\Gamma = \mathbin{\dot{\bigcup_{i\in\{1,\dots,K(\Omega)\}}}} \Gamma^i,  \quad  \Gamma_{\operatorname{c}} = \mathbin{\dot{\bigcup_{i\in\{1,\dots,K(\Omega)\}}}} \Gamma_{\operatorname{c}}^i.
	\end{alignat*}
	Let $\mathcal{E} \subset \mathbb{R}^N$ be a smoothly bounded domain, which is an extension of $\Omega$ as shown in Figure \ref{Figure:extension}, satisfying
	\[
	\Omega \subset \mathcal{E}, \quad \Gamma^i_{\operatorname{c}} \subset \overline{\mathcal{E}}, \quad \Gamma\setminus\Gamma_{\operatorname{c}} \subset \partial\mathcal{E}, \quad \Gamma_{\operatorname{c}}^i \cap \mathcal{E} \neq \emptyset, \quad i \in \{1,\dots,K(\Omega)\}.
	\]
	Such an extension exists by the requirements on $\Omega$. Throughout, the outward unit normal at $\partial\mathcal{E}$ is denoted by $\xvec{n}_{\partial\mathcal{E}}$, or simply by $\xvec{n}$ if no confusion can arise. We also make the following assumptions:
	\begin{itemize}
		\item the extension $\mathcal{E}$ is selected such that $\xvec{u}_0$ and $\xvec{B}_0$ are tangential at $\partial\Omega \cap \partial \mathcal{E}$;
		\item for the sake of simplifying the notations, to each connected component of~$\Gamma_{\operatorname{c}}$ at most one connected component of $\mathcal{E}\setminus\Omega$ is attached.
	\end{itemize}
	Moreover, given~$T > 0$, we denote
	\[
		\mathcal{E}_T \coloneqq \mathcal{E} \times (0, T), \quad \Sigma_T \coloneqq \partial\mathcal{E} 	\times (0, T). 
	\]
	
	When $\mathcal{E}$ is a multiply-connected domain, there is a number $L(\mathcal{E}) \in \mathbb{N}$ of smooth $(N-1)$-dimensional and mutually disjoint cuts $\mathcal{C}_1, \dots, \mathcal{C}_{L(\mathcal{E})} \subset \mathcal{E}$, which meet $\partial\mathcal{E}$ transversely, such that one obtains a simply-connected set via $\ring{\mathcal{E}} \coloneqq \mathcal{E} \setminus (\mathcal{C}_1 \cup \dots \cup \mathcal{C}_{L(\mathcal{E})})$; see, \eg, \cite[Appendix I]{Temam2001}. 	Next, for each $i \in \{1,\dots,L(\mathcal{E})\}$, a unit normal field to $\mathcal{C}_i$ is denoted by $\widetilde{\xvec{n}}^i$. When $\mathcal{E}$ is simply-connected, we set $L(\mathcal{E}) \coloneqq 0$.
	
	\begin{figure}[ht!]
		\centering
		\resizebox{0.6\textwidth}{!}{
			\begin{tikzpicture}
				\clip(-2.4,-1.5) rectangle (7.2,4.4);

				\draw [line width=0mm, color=white, fill=red, pattern=dots] plot[smooth, tension=1] coordinates {(-2.2,0.8) (-2,3.2) (0.4,4.2) (0.7,3.5) (2.2,4.2) (7,2.2) (5.8,0.8) (6.8,-0.2) (3.5,-1.5) (2,-0.8) (-1,-1) (-2.2,0.8)};
				
				\draw [line width=0mm, color=white, fill=white] plot[smooth, tension=1] coordinates {(-1.8,0.8) (-1.5,3) (0,3.8) (0.5,1.8) (2,3.8) (6.4,2.2) (5,0.8) (5.8,-0.2) (3.2,-1) (2,0) (-0.4,-0.9) (-1.8,0.8)};
				
				\draw [line width=0mm, color=RoyalBlue!20, fill=RoyalBlue!5] plot[smooth, tension=1] coordinates {(1,0.5) (0.65,1.9) (2.2,3.8) (6.4,2.18) (5,0.8) (5.8,-0.2) (3.2,-1) (2.35,0) (1.5,0.2) (1,0.5)};
				
				\draw [line width=0.3mm, color=black, fill=white, fill opacity=0] plot[smooth, tension=1] coordinates {(-1.8,0.8) (-1.5,3) (0,3.8) (0.47,1.8) (2,3.8) (6.4,2.2) (5,0.8) (5.8,-0.2) (3.2,-1) (2,0) (-0.4,-0.9) (-1.8,0.8)};

				\draw [dashed,line width=0.8mm, color=black]  plot[smooth, tension=0.8] coordinates {(2.58,-0.5)  (2.34,0.035) (1.52,0.2) (0.98,0.5) (0.55,1.3) (0.72,2) (0.8,2.55)};
				

				\draw [line width=0mm, color=white, fill=white]  plot[smooth, tension=0.8] coordinates {(2,0.5) (1.2,1.6) (2.5,3) (3.2,2.6) (4.5,2.8) (4,1.5) (4.6,0.5) (3.5,0)  (2,0.5)};
				\draw [line width=0.3mm, color=black, fill=white,pattern=dots]  plot[smooth, tension=0.8] coordinates {(2,0.5) (1.2,1.6)  (2.5,3) (3.2,2.6) (4.5,2.82) (4.06,1.8) (4,1.2) (3.8,1.2) (3.4,1.6) (2.2,2) (2.6,0.5) (2,0.5)};(
				
				\draw [dashed,line width=0.8mm, color=black]  plot[smooth, tension=0.8] coordinates { (4.03,1.7) (4,1.4) (4.15,1.1) (4.6,0.45) (3.5,0) (2,0.5)};
				
				\draw [line width=0.5mm]  plot[smooth, tension=0.8] coordinates {(3.4,2.6) (3.38,3) (3.51,3.672)};			
				
				
				\coordinate[label=left:\footnotesize$\mathcal{E} \cap \Omega$] (A) at (5.8,1.8);
				\coordinate[label=left:\footnotesize$\mathcal{E} \setminus \Omega$] (A) at (-0.2,2);
				\coordinate[label=left:\footnotesize$\mathcal{E}\setminus \Omega$] (A) at (3.9,0.8);
				\coordinate[label=below:\footnotesize$\Gamma_{\operatorname{c}}^1$] (A) at (-1.6,-0.8);
				\coordinate[label=right:$\quad$] (B) at (0.64,0.84);
				\draw[line width=0.8mm,-stealth] (A) -- (B);
				\coordinate[label=left:\footnotesize$\Gamma_{\operatorname{c}}^2$] (AA) at (2,-1.2);
				\coordinate[label=left:$\quad$] (BB) at (3,0);
				\draw[line width=0.8mm,-stealth] (AA) -- (BB);
				
				\coordinate[label=right:\footnotesize$\mathcal{C}_1$] (AAA) at (5.5,3.5);
				\coordinate[label=left:$\quad$] (BBB) at (3.48,3.2);
				\draw[line width=0.8mm,-stealth] (AAA) -- (BBB);			
			\end{tikzpicture}
		}
		\caption{A multiply-connected domain $\Omega \subset \mathbb{R}^2$ with two controlled boundary components and extension $\mathcal{E}$. The dashed lines mark the controlled boundaries.}
		\label{Figure:extension}
	\end{figure}

	The following Korn and Poincar\'e type inequalities for possibly multiply-connected domains are well-known. 
	\begin{lmm} There exists a constant $C > 0$ such that for any $\xvec{h} \in \xHone(\mathcal{E})$, one has the estimate
		\begin{equation}\label{equation:sKem}
			\begin{aligned}
				\|\xvec{h}\|_{\xHone(\mathcal{E})}
				& \leq 
				C\left(\|\xdiv{\xvec{h}}\|_{\xLtwo(\mathcal{E})} + \|\xcurl{\xvec{h}}\|_{\xLtwo(\mathcal{E})} + \|\xvec{h}\cdot \xvec{n}\|_{\xHn{{1/2}}(\partial\mathcal{E})}\right) \\
				& \quad  + C \sum\limits_{i = 1}^{L(\mathcal{E})}\left| \int_{\mathcal{C}_i} \xvec{h} \cdot \widetilde{\xvec{n}}^i \, \xdx{\mathcal{C}_i} \right|\\
				& \leq 
				C\left(\|\xdiv{\xvec{h}}\|_{\xLtwo(\mathcal{E})} + \|\xcurl{\xvec{h}}\|_{\xLtwo(\mathcal{E})} + \|\xvec{h}\cdot \xvec{n}\|_{\xHn{{1/2}}(\partial\mathcal{E})} + \|\xvec{h}\|_{\xLtwo(\mathcal{E})}\right). \\
			\end{aligned}
		\end{equation}
	\end{lmm}
	\begin{proof}
		It is known (\eg, see \cite[Corollary 3.4]{AmroucheSeloula2013}), that all $\xvec{f} \in \xHone(\mathcal{E})$ with $\xvec{f} \cdot \xvec{n} = 0$ along $\partial \mathcal{E}$ obey
		\begin{equation}\label{equation:sKem_pre}
			\begin{aligned}
				\|\xvec{f}\|_{\xHone(\mathcal{E})} \leq 
				C\left(\|\xdiv{\xvec{f}}\|_{\xLtwo(\mathcal{E})} + \|\xcurl{\xvec{f}}\|_{\xLtwo(\mathcal{E})} + \sum\limits_{i = 1}^{L(\mathcal{E})}\left| \int_{\mathcal{C}_i} \xvec{f} \cdot \widetilde{\xvec{n}}^i \, \xdx{\mathcal{C}_i} \right| \right).
			\end{aligned}
		\end{equation}
		Moreover, as demonstrated in \cite[Theorem III.4.3]{BoyerFabrie2013}, there exists a function $\psi \in \xHtwo(\mathcal{E})$ which solves the Neumann problem
		\[
		\begin{cases}
			\Delta \psi = \xdiv{\xvec{h}} & \mbox{ in } \mathcal{E},\\
			\partial_{\xvec{n}} \psi = \xvec{h} \cdot \xvec{n} & \mbox{ on } \partial \mathcal{E},
		\end{cases}
		\]
		and satisfies
		\begin{equation}\label{equation:sKem_pre2}
			\|\psi\|_{\xHtwo(\mathcal{E})} \leq C \left( \|\xdiv{\xvec{h}}\|_{\xLtwo(\mathcal{E})} +  \|\xvec{h}\cdot \xvec{n}\|_{\xHn{{1/2}}(\partial\mathcal{E})} \right).
		\end{equation}
		By employing trace estimates and the properties of $\psi$, the potential field $\xvec{g} = \xnab \psi$ is seen to satisfy
		\[
		\xdiv{\xvec{g}} = \xdiv{\xvec{h}}, \quad \xcurl{\xvec{g}} = \xvec{0}, \quad \xvec{g} \cdot \xvec{n} = \xvec{h} \cdot \xvec{n}, \quad \sum\limits_{i = 1}^{L(\mathcal{E})}\left| \int_{\mathcal{C}_i} \xvec{g} \cdot \widetilde{\xvec{n}}^i \, \xdx{\mathcal{C}_i} \right| \leq C \|\xvec{g}\|_{\xHone(\mathcal{E})}.
		\]
		Consequently, by means of the estimates \eqref{equation:sKem_pre} and \eqref{equation:sKem_pre2}, the first inequality in \eqref{equation:sKem} follows with $\xvec{f} \coloneqq \xvec{h}-\xvec{g}$ from
		{\allowdisplaybreaks\begin{multline*}
				\|\xvec{h}\|_{\xHone(\mathcal{E})} \leq \|\xvec{f}\|_{\xHone(\mathcal{E})} + \|\xvec{g}\|_{\xHone(\mathcal{E})} \leq C \|\xcurl{\xvec{h}}\|_{\xLtwo(\mathcal{E})} + C \sum\limits_{i = 1}^{L(\mathcal{E})}\left| \int_{\mathcal{C}_i} \xvec{f} \cdot \widetilde{\xvec{n}}^i \, \xdx{\mathcal{C}_i} \right| + \|\xvec{g}\|_{\xHone(\mathcal{E})} \\
				\begin{aligned}
					& \leq C \|\xcurl{\xvec{h}}\|_{\xLtwo(\mathcal{E})} + C \sum\limits_{i = 1}^{L(\mathcal{E})}\left| \int_{\mathcal{C}_i} \xvec{h} \cdot \widetilde{\xvec{n}}^i \, \xdx{\mathcal{C}_i} \right| + C\|\xvec{g}\|_{\xHone(\mathcal{E})} \\
					& \leq C\left(\|\xdiv{\xvec{h}}\|_{\xLtwo(\mathcal{E})} + \|\xcurl{\xvec{h}}\|_{\xLtwo(\mathcal{E})} + \|\xvec{h}\cdot \xvec{n}\|_{\xHn{{1/2}}(\partial\mathcal{E})}\right) + C \sum\limits_{i = 1}^{L(\mathcal{E})}\left| \int_{\mathcal{C}_i} \xvec{h} \cdot \widetilde{\xvec{n}}^i \, \xdx{\mathcal{C}_i} \right|.
				\end{aligned}
			\end{multline*}
			Concerning the second inequality in \eqref{equation:sKem}, let the multi-valued functions $q_1,\dots,q_{L(\mathcal{E})}$ be chosen such that $\{\xnab q_1,\dots,\xnab q_{L(\mathcal{E})}\}$ are smooth and form a basis for the space of curl-free and divergence-free vector fields tangential at $\partial \mathcal{E}$. As in \cite[Appendix I]{Temam2001}, one can select this basis such that $[q_i]_j = \delta_{i,j}$, where $[f]_j$ denotes the jump of $f$ across~$\mathcal{C}_j$ and $\delta_{i,j}$ is the usual Kronecker symbol. Therefore, one has
			\[
				\int_{\mathcal{C}_i} \xvec{h} \cdot \widetilde{\xvec{n}}^i \, \xdx{\mathcal{C}_i} =  \int_{\ring{\mathcal{E}}} \xvec{h} \cdot \xnab q_i \, \xdx{\xvec{x}} + \int_{\ring{\mathcal{E}}} (\xdiv{\xvec{h}}) q_i \, \xdx{\xvec{x}} - \int_{\partial\mathcal{E}} (\xvec{h} \cdot \xvec{n}) q_i \, \xdx{S},
			\]}
			which allows to conclude the proof.
	\end{proof}

	Let $d > 0$ be sufficiently small so that $\mathcal{V} \coloneqq \{\xvec{x} \in \mathbb{R}^N \, | \, \operatorname{dist}(\xvec{x}, \partial\mathcal{E}) < d\}$ represents a thin tubular neighborhood in $\mathbb{R}^N$ of the boundary $\partial\mathcal{E}$. Further, let $\varphi_{\mathcal{E}} \in \xCinfty(\mathbb{R}^N;\mathbb{R})$ satisfy $|\xdop{\nabla}\varphi_{\mathcal{E}}(\xvec{x})| = 1$ for all~$\xvec{x} \in \mathcal{V}$ and
	\begin{alignat*}{4}
		\mathcal{E} \cap \mathcal{V} = \{\varphi_{\mathcal{E}} > 0\}\cap \mathcal{V}, & \quad & (\mathbb{R}^N\setminus\overline{\mathcal{E}}) \cap \mathcal{V} = \{\varphi_{\mathcal{E}} < 0\} \cap \mathcal{V}, & \quad & \partial\mathcal{E} = \{\varphi_{\mathcal{E}} = 0\}.
	\end{alignat*}
	This implies that $\varphi_{\mathcal{E}}(\xvec{x}) = \operatorname{dist}(\xvec{x}, \partial\mathcal{E})$ for all $\xvec{x} \in \mathcal{V}\cap\overline{\mathcal{E}}$, assuming without loss of generality that $\mathcal{V}$ is sufficiently thin.~Now, a smooth extension of $\xvec{n}_{\partial\mathcal{E}}$ to $\overline{\mathcal{E}}$ is provided by
	\begin{equation}\label{equation:definitionnormal}
			\xvec{n}(\xvec{x}) = \xvec{n}_{\mathcal{E}}(\xvec{x}) \coloneqq \begin{cases}
			\xvec{n}_{\partial\mathcal{E}}(\xvec{x}) & \mbox{ if } \xvec{x} \in \partial \mathcal{E},\\
			-\xdop{\nabla} \varphi_{\mathcal{E}}(\xvec{x}) & \mbox{ if } \xvec{x} \in \mathcal{E}.
		\end{cases}
	\end{equation}
	In this sense, the tangential part $[\xvec{h}]_{\operatorname{tan}} = \xvec{h} - \left(\xvec{h} \cdot \xvec{n}\right) \xvec{n}$ of $\xvec{h}\colon \overline{\mathcal{E}} \longrightarrow \mathbb{R}^N$ is then defined everywhere in $\overline{\mathcal{E}}$. Moreover, the Weingarten map $\xvec{W}_{\mathcal{E}}$ and the general friction matrices $\xvec{M}_1,\xvec{M}_2,\xvec{L}_1,\xvec{L}_2$ are smoothly continued to $\overline{\mathcal{E}}$ such that
	\[
		\xvec{W}_{\mathcal{E}}, \xvec{M}_{1}, \xvec{M}_{2}, 	\xvec{L}_{1},\xvec{L}_{2} \in \xCinfty(\overline{\mathcal{E}};\mathbb{R}^{N\times N}),
	\]
	while also extending the assumptions (such as $\xvec{M}_2 = \xvec{0}$ or $\xvec{M}_2 = \rho \xvec{I}$) that might have been made in Theorems~\Rref{theorem:main1}, \Rref{theorem:main}, and \Rref{theorem:annulus}.
	
	For describing boundary layers in the vicinity of $\partial \mathcal{E}$, when a parameter $\epsilon > 0$ is assumed small, some functions will depend on a slow variable $\xvec{x} \in \overline{\mathcal{E}}$, the time $t \geq 0$ and a fast variable
	\[
		z = \varphi_{\mathcal{E}}(\xvec{x})/\sqrt{\epsilon} \in \mathbb{R}_+.
	\]
	In this case, for a map $(\xvec{x},t,z) \longmapsto h(\xvec{x},t,z)$ we denote
	\[
		\left\llbracket h \right\rrbracket_{\epsilon}(\xvec{x},t) \coloneqq h\left(\xvec{x},t,\varphi_{\mathcal{E}}(\xvec{x})/\sqrt{\epsilon}\right).
	\]
	By convention, differential operators are always taken with respect to $\xvec{x} \in \overline{\mathcal{E}}$ only, if not indicated otherwise by the notation. Therefore, as also remarked in \cite{CoronMarbachSueur2020,IftimieSueur2011}, one has the commutation formulas
	\begin{equation}\label{equation:comfo}
		\begin{alignedat}{3}
			& \xdop{\nabla}\cdot\left(\left\llbracket \xvec{h}\right\rrbracket_{\epsilon}\right) && = \left\llbracket \xdop{\nabla}\cdot \xvec{h}\right\rrbracket_{\epsilon} - \xvec{n} \cdot \left\llbracket \partial_z \xvec{h}\right\rrbracket_{\epsilon}/\sqrt{\epsilon},\\
			& \xdop{\nabla}\left(\left\llbracket \xvec{h}\right\rrbracket_{\epsilon}\right) && = \left\llbracket \xdop{\nabla}\xvec{h}\right\rrbracket_{\epsilon} - \left\llbracket \partial_z \xvec{h}\right\rrbracket_{\epsilon}\xvec{n}^{\top}/\sqrt{\epsilon},\\
			& \left[\xdop{D}(\left\llbracket \xvec{h}\right\rrbracket_{\epsilon})\xvec{n}\right]_{\operatorname{tan}} && = \left\llbracket \left[\xdop{D}( \xvec{h})\xvec{n}\right]_{\operatorname{tan}}\right\rrbracket_{\epsilon} - \left\llbracket [\partial_z \xvec{h}]_{\operatorname{tan}}\right\rrbracket_{\epsilon}/\sqrt{4\epsilon},\\
			& \epsilon \Delta \left\llbracket \xvec{h}\right\rrbracket_{\epsilon} && = \epsilon \left\llbracket \Delta \xvec{h}\right\rrbracket_{\epsilon} + \sqrt{\epsilon} \Delta \varphi_{\mathcal{E}} \left\llbracket \partial_z \xvec{h} \right\rrbracket_{\epsilon} - 2\sqrt{\epsilon} \left\llbracket (\xvec{n}\cdot \xdop{\nabla})\partial_z \xvec{h}\right\rrbracket_{\epsilon} \\
			& && \quad  + |\xvec{n}|^2 \left\llbracket \partial_{zz}\xvec{h}\right\rrbracket_{\epsilon},
		\end{alignedat}
	\end{equation}
	and consequently
	\[
	\xdop{\mathcal{N}}_i\left(\left\llbracket \xvec{h}_1\right\rrbracket_{\epsilon},\left\llbracket \xvec{h}_2\right\rrbracket_{\epsilon}\right) = \left\llbracket \xdop{\mathcal{N}}_i(\xvec{h}_1,\xvec{h}_2)\right\rrbracket_{\epsilon} - \left\llbracket [\partial_z \xvec{h}_i]_{\operatorname{tan}}\right\rrbracket_{\epsilon}/\sqrt{4\epsilon}, \quad i \in \{1,2\}.
	\]

	\subsection{Function spaces}\label{subsection:FunctionSpaces}
	The Hilbert spaces $\xH(\mathcal{E})$ and $\xW(\mathcal{E})$ of divergence-free and tangential vector fields are defined by means of
	\begin{equation*}
		\begin{gathered}
			\xH(\mathcal{E}) \coloneqq \operatorname{clos}_{\xLtwo(\mathcal{E};\mathbb{R}^N)} \left(\left\{ \xvec{f} \in \xCone(\overline{\mathcal{E}};\mathbb{R}^N) \, \left| \right. \, \xdop{\nabla}\cdot\xvec{f} = 0 \mbox{ in } \mathcal{E}, \xvec{f} \cdot \xvec{n} = 0 \mbox{ on } \partial \mathcal{E} \right\} \right)
		\end{gathered}
	\end{equation*}
	and
	\begin{equation*}
		\xW(\mathcal{E}) \coloneqq \xH(\mathcal{E};\mathbb{R}^N) \cap \xHone(\mathcal{E};\mathbb{R}^N),
	\end{equation*}
	where  $\operatorname{clos}_{\xLtwo(\mathcal{E})}$ denotes the closure in $\xLtwo(\mathcal{E})$.
	For any $T > 0$, the weakly continuous functions from $[0,T]$ to $\xH(\mathcal{E})$ are denoted by $\xCn{0}_{w}([0,T];\xH(\mathcal{E}))$. The space for weak MHD solutions is 
	\begin{equation*}
		\begin{gathered}
			\mathscr{X}_T \coloneqq \xCn{0}_{w}([0,T];\xH(\mathcal{E}))\cap\xLn{2}((0,T);{\rm W}(\mathcal{E})).
		\end{gathered}
	\end{equation*} 
	Moreover, for $m, p, k, s \in \mathbb{N}_0$, we employ the weighted Sobolev spaces 
	\begin{equation*}
		\begin{gathered}
			\xHn{{k,m,p}}_{\mathcal{E}}  \coloneqq \left\{ f \in \xLtwo(\mathcal{E}\times\mathbb{R}_+) \, \left| \, 	\|f\|_{\xHn{{k,m,p}}_{\mathcal{E}}} \coloneqq \left(\sum\limits_{r = 0}^p |f|_{k,m,r,\mathcal{E}}^2\right)^{\frac{1}{2}} < +\infty  \right. \right\},\\
			\widetilde{\xHn{{k,s}}}(\mathbb{R})  \coloneqq \left\{ f \in \xLtwo(\mathbb{R}) \, \left| \, \|f\|_{\xHn{{k,s}}(\mathbb{R})} \coloneqq \left(\sum\limits_{l=0}^s \int_{\mathbb{R}} (1+z^{2k}) |\partial^l_z f(z)|^2 \, \xdx{z}\right)^{\frac{1}{2}} < +\infty  \right. \right\},
		\end{gathered}
	\end{equation*}
	where $|f|_{k,m,r,\mathcal{E}}$ denotes for functions $(\xvec{x},z) \mapsto f(\xvec{x},z)$ the seminorm
	\[
		|f|_{k,m,r,\mathcal{E}} \coloneqq \left( \sum\limits_{|\beta| \leq m} 		\int_{\mathcal{E}}\int_{\mathbb{R}_+} (1 + z^{2k})|\partial_{\xvec{x}}^{\beta}\partial_z^r f|^2 \, \xdx{z}  \xdx{\xvec{x}} \right)^{\frac{1}{2}}.
	\]

	\subsection{Initial data extensions}\label{subsection:initialDataExtensions}
	We extend the original initial data to $\mathcal{E}$. Whether divergence-free extensions are possible depends on the normal traces of the initial states $(\xvec{u}_0, \xvec{B}_0) \in \xLtwo_{\operatorname{c}}(\Omega) \times \xLtwo_{\operatorname{c}}(\Omega)$ fixed in Theorems~\Rref{theorem:main1} or \Rref{theorem:main}. More specifically, we either choose extensions of the type
	\[
	\widetilde{\xvec{u}}_0, \widetilde{\xvec{B}}_0 \in \xH(\mathcal{E}), \quad \widetilde{\xvec{u}}_0|_{\Omega} = \xvec{u}_0, \quad \widetilde{\xvec{B}}_0|_{\Omega} = \xvec{B}_0,
	\]
	or we select continuations $\widetilde{\xvec{u}}_0, \widetilde{\xvec{B}}_0 \in \xLtwo(\mathcal{E})$ with defined normal trace at $\partial \mathcal{E}$ and which obey
	\[
		\widetilde{\xvec{u}}_0 \cdot \xvec{n} = \widetilde{\xvec{B}}_0 \cdot \xvec{n} = 0 \, \mbox{ at } \partial \mathcal{E}, \quad \widetilde{\xvec{u}}_0|_{\Omega} = \xvec{u}_0, \quad \widetilde{\xvec{B}}_0|_{\Omega} = \xvec{B}_0.
	\]
	These extensions will be made precise below in \Cref{lemma:dataext}, which is a modification of \cite[Proposition 2.1]{ChavesSilva2020SmalltimeGE}; hereto, given any $i \in \{1,\dots, K(\Omega)\}$, the following notations are fixed beforehand.
	\begin{itemize}
		\item The sets $\Gamma_{\operatorname{c},1}^i, \dots \Gamma_{\operatorname{c},m_i}^i$ enumerate the connected components of the $i$-th controlled boundary part~$\Gamma_{\operatorname{c}}^i$. 
		\item The set $\Omega^i \subset \mathcal{E}\setminus\Omega$ is the extension attached to $\Omega$ at $\Gamma_{\operatorname{c}}^i$, namely the maximal union of connected components of $\mathcal{E}\setminus\Omega$ with $(\partial \Omega^i \cap \Gamma) \subset \overline{\Gamma_{\operatorname{c}}^i}$.
		\item For each $j \in \{1,\dots, m_i\}$, the set $\Omega_j^i \subset \Omega^i$ is the connected component of $\mathcal{E}\setminus\Omega$ attached to $\Gamma_{\operatorname{c},j}^i$. If $\Omega_{j}^i$ is attached to $\Gamma_{\operatorname{c},j}^i$ and $\Gamma_{\operatorname{c},l}^i$ with $j\neq l$, then $\Omega_{j}^i = \Omega_{l}^i$. 
	\end{itemize}

	\begin{lmm}\label{lemma:dataext}
		There exists a constant $C > 0$ such that for each $\xvec{h} \in \xLtwo_{\operatorname{c}}(\Omega)$ there are functions $\sigma \in \xCinfty(\mathcal{E}; \mathbb{R})$ with $\operatorname{supp}(\sigma) \subset \mathcal{E}\setminus \overline{\Omega}$ and $\widetilde{\xvec{h}} \in \xLtwo(\mathcal{E})$ satisfying
		\[
			{\widetilde{\xvec{h}}} = \xvec{h} \mbox{ in } \Omega, \quad \xdiv{\widetilde{\xvec{h}}} = \sigma \mbox{ in } \mathcal{E}, \quad \widetilde{\xvec{h}} \cdot \xvec{n} = 0 \mbox{ on } \partial \mathcal{E}, \quad \| \widetilde{\xvec{h}} \|_{\xLtwo(\mathcal{E})} \leq C \| \xvec{h} \|_{\xLtwo(\Omega)}.
		\]
		When the vector field $\xvec{h}$ additionally obeys at $\Gamma_{\operatorname{c}}$ the conditions
		\begin{equation}\label{equation:gencompb}
			\forall i \in \{1,\dots, K(\Omega)\}, \, \forall j \in \{1,\dots, m_i\} \colon \\ \left\langle \xvec{h} \cdot \xvec{n}, 1 \right\rangle_{\xHn{{-1/2}}(\partial \Omega^i_j \cap \Gamma_{\operatorname{c}}^i),\xHn{{1/2}}(\partial \Omega^i_j \cap \Gamma_{\operatorname{c}}^i)} = 0,
		\end{equation}
		then one can choose $\widetilde{\xvec{h}} \in \xH(\mathcal{E})$.
	\end{lmm}
	\begin{proof}
		Let $\xvec{n}$ denote the outward unit normal to $\Omega$ at $\Gamma$, while the outward unit normal to $\mathcal{E}$ at $\partial \mathcal{E}$ is written as $\xvec{n}_{\partial \mathcal{E}}$.
		It is known (\cf~\cite[Chapter IV, Section 3.2]{BoyerFabrie2013}) that there exists a continuous normal trace operator
		\[
			\gamma_{\xvec{n}}\colon \xLtwo_{\operatorname{c}}(\Omega) \longrightarrow \xH^{-1/2}(\Gamma), \quad \forall \xvec{w} \in \xCinfty(\overline{\Omega}; \mathbb{R}^N) \colon \gamma_{\xvec{n}}(\xvec{w}) = \xvec{w} \cdot \xvec{n}.
		\]
		Then, for each $i \in \{1,\dots, K(\Omega)\}$, a smooth function $\sigma^i \in \xCinfty_0(\Omega^i; \mathbb{R})$ is fixed such that
		\[
			\forall l \in \{1,\dots, m_i\} \colon \int_{\Omega^i_l} \sigma^i(\xvec{x}) \, \xdx{\xvec{x}} - \int_{\partial \Omega^i_l \cap \Gamma_{\operatorname{c}}^i}  \gamma_{\xvec{n}}(\xvec{h}) \, \xdx{S} = 0.
		\]
		This guarantees that one can solve for each $i\in\{1,\dots,K(\Omega)\}$ a weak formulation of the respective elliptic problem
		\begin{equation}\label{equation:ellipticprobleminitialdataextension}
			\begin{cases}
				\Delta \varphi^i = \sigma^i & \mbox{ in } \Omega^i,\\
				\xnab \varphi^i \cdot \xvec{n} = \gamma_{\xvec{n}}(\xvec{h}),  & \mbox{ on } \Gamma_{\operatorname{c}}^i \cap \partial\Omega^i,\\
				\xnab \varphi^i \cdot \xvec{n}_{\partial\mathcal{E}} = 0 & \mbox{ on } \partial\Omega^i \setminus \Gamma_{\operatorname{c}}^i.
			\end{cases}
		\end{equation}
		In particular, if \eqref{equation:gencompb} holds for $i \in \{1,\dots, K(\Omega)\}$, then $\sigma^i = 0$ can be chosen. Accordingly, the proof is concluded by taking $\widetilde{\xvec{h}} \coloneqq \xvec{h}$ in $\Omega$ and $\widetilde{\xvec{h}} \coloneqq \xnab \varphi^i$ in $\Omega^i$. The continuity of the extension operator follows from \eqref{equation:ellipticprobleminitialdataextension} and the divergence-free condition encoded in $\xLtwo_{\operatorname{c}}(\Omega)$.
	\end{proof}
	
	\begin{rmrk}\label{remark:ride}
		In the context of \Cref{theorem:main1}, due to the assumption that $\Gamma_{\operatorname{c}}$ is connected, the condition \eqref{equation:gencompb} is automatically satisfied by all $\xvec{u}_0, \xvec{B}_0 \in \xLtwo_{\operatorname{c}}(\Omega)$.
	\end{rmrk}

	\subsection{Weak controlled trajectories}\label{subsection:wct}
	To define notions of weak controlled trajectories for the problems \eqref{equation:MHD00} and \eqref{equation:MHD00withq}, we follow the idea from \cite{CoronMarbachSueur2020} and first introduce Leray--Hopf weak solutions for interior controlled MHD problems posed in the respectively enlarged domain $\mathcal{E}$. Then, by restricting such solutions to $\Omega$, one obtains a notion of boundary controlled weak solutions to \eqref{equation:MHD00} and \eqref{equation:MHD00withq}. The plan is as follows.
	\begin{itemize}
		\item \Cref{subsubsection:case_thm1} defines weak controlled trajectories for initial data in  $\xH(\mathcal{E})$.
		\item \Cref{subsubsection:gq} discusses the appearance of $\xnab q$.
		\item \Cref{subsubsection:changeofvariables} formulates weak controlled trajectories in Elasser variables.
		\item \Cref{subsubsection:secondcase} defines weak controlled trajectories for more general initial data.
	\end{itemize}
	
	Given any time~$T > 0$, let us recall the notation of the space-time cylinder and its mantle
	\[
		\mathcal{E}_T \coloneqq \mathcal{E} \times (0, T), \quad \Sigma_T \coloneqq \partial\mathcal{E} 	\times (0, T).
	\]
	In this subsection, if forces $\xsym{\xi}$ and $\xsym{\eta}$ appear in the right-hand sides of MHD problems, then it is assumed that
	\[
		\xsym{\eta} = \widetilde{\xsym{\eta}} + \xsym{\zeta}, \quad \bigcup_{t \in [0,T]} \left(\operatorname{supp}(\xsym{\xi}(\cdot,t))\cup\operatorname{supp}(\widetilde{\xsym{\eta}}(\cdot,t))\right) \subset \overline{\mathcal{E}}\setminus\overline{\Omega}
	\]
	for some $\xsym{\zeta}, \xsym{\xi}, \widetilde{\xsym{\eta}} \in \xCzero([0,T]; \xLtwo(\mathcal{E};\mathbb{R}^N))$. In particular, $\xsym{\zeta}$ will coincide with that in \Cref{theorem:main}.
	\subsubsection{The case of $\xH(\mathcal{E})$ data (\eg, \texorpdfstring{\Cref{theorem:main1})}{the first main theorem}}\label{subsubsection:case_thm1}
	We focus now on the situation of \Cref{theorem:main1}; however, if the initial data can be extended as $\xH(\mathcal{E})$-functions, the following definitions make also sense for the setting of \Cref{theorem:main}. To streamline the presentation, the three-dimensional cross product and curl notations are employed. 
	
	\paragraph{Definition of weak controlled trajectories.}
	When the initial data $\xvec{u}_0,\xvec{B}_0\in\xLtwo_{\operatorname{c}}(\Omega)$ admit extensions to $\xH(\mathcal{E})$, as emphasized in \Cref{remark:ride} for \Cref{theorem:main1},
	a weak controlled trajectory for \eqref{equation:MHD00} is defined as any pair of vector fields $(\xvec{u}, \xvec{B})$ that are of the form
	\[
	(\xvec{u}, \xvec{B}) \in \left[\xCn{0}_{w}([0,T];\xLn{2}_{\operatorname{c}}(\Omega))\cap\xLn{2}((0,T);\xHn{1}(\Omega))\right]^2, \quad (\xvec{u}, \xvec{B}) = (\widetilde{\xvec{u}}|_{\Omega},\widetilde{\xvec{B}}|_{\Omega}),
	\]
	where $(\widetilde{\xvec{u}}, \widetilde{\xvec{B}})$ denotes a Leray--Hopf weak solution to the viscous and resistive incompressible MHD system
	\begin{equation}\label{equation:MHD_ElsaesserExt}
		\begin{cases}
			\partial_t \xvec{u} - \nu_1 \Delta \xvec{u} + (\xvec{u} \cdot \xdop{\nabla}) \xvec{u} - \mu(\xvec{B} \cdot \xnab) \xvec{B} + \xdop{\nabla} p = \xsym{\xi} & \mbox{ in } \mathcal{E}_T,\\
			\partial_t \xvec{B} - \nu_2 \Delta \xvec{B} + (\xvec{u} \cdot \xdop{\nabla}) \xvec{B} - (\xvec{B} \cdot \xdop{\nabla}) \xvec{u} = \xsym{\eta} & \mbox{ in } \mathcal{E}_T,\\
			\xdop{\nabla}\cdot\xvec{u} = \xdop{\nabla}\cdot\xvec{B} = 0 & \mbox{ in } \mathcal{E}_T,\\
			\xvec{u} \cdot \xvec{n} = \xvec{B} \cdot \xvec{n} = 0 & \mbox{ on } \Sigma_T,\\
			\xmcal{N}_1(\xvec{u},\xvec{B}) = \xmcal{N}_2(\xvec{u},\xvec{B}) = \xvec{0} & \mbox{ on } \Sigma_T,\\
			\xvec{u}(\cdot, 0)  =  \xvec{u}_0,\, \xvec{B}(\cdot, 0)  =  \xvec{B}_0 & \mbox{ in } \mathcal{E}.
		\end{cases}
	\end{equation}
	A pair $(\xvec{u},\xvec{B}) \in \mathscr{X}_T\times\mathscr{X}_T$ is called a Leray--Hopf weak solution to \eqref{equation:MHD_ElsaesserExt}, if it satisfies for all $\xsym{\varphi}, \xsym{\psi} \in \xCinfty_0( \overline{\mathcal{E}}\times[0,T);\mathbb{R}^N) \cap \xCinfty([0,T];\xH(\mathcal{E}))$ and for almost all $t \in [0,T]$, the variational formulation
	\begin{equation}\label{equation:varformwct}
		\begin{gathered}
			\int_{\mathcal{E}} \left( \xvec{u}(\xvec{x},t) \cdot \xsym{\varphi}(\xvec{x},t) + \xvec{B}(\xvec{x},t) \cdot \xsym{\psi}(\xvec{x},t) - \xvec{u}_0(\xvec{x}) \cdot \xsym{\varphi}(\xvec{x},0) - \xvec{B}_0(\xvec{x}) \cdot \xsym{\psi}(\xvec{x},0) \right) \, \xdx{\xvec{x}}\\
			-\int_0^t\int_{\mathcal{E}} \left( \xvec{u}\cdot \partial_t \xsym{\varphi} + \xvec{B} \cdot \partial_t \xsym{\psi} \right) \, \xdx{\xvec{x}} \, \xdx{t} + \nu_1 \int_0^t\int_{\mathcal{E}}  (\xcurl{\xvec{u}}) \cdot (\xcurl{\xsym{\varphi}}) \, \xdx{\xvec{x}} \xdx{t} \\
			+ \nu_2 \int_0^t\int_{\mathcal{E}} (\xcurl{\xvec{B}}) \cdot (\xcurl{\xsym{\psi}})\, \xdx{\xvec{x}} \xdx{t} + \int_0^t\int_{\mathcal{E}} \left((\xvec{u}\cdot\xnab)\xvec{u}-\mu(\xvec{B}\cdot\xnab)\xvec{B}\right) \cdot \xsym{\varphi} \, \xdx{\xvec{x}} \, \xdx{t}\\
			+ \int_0^t\int_{\mathcal{E}} \left((\xvec{u}\cdot\xnab)\xvec{B}-(\xvec{B}\cdot\xnab)\xvec{u}\right) \cdot \xsym{\psi} \, \xdx{\xvec{x}} \, \xdx{t} - \nu_1 \int_0^t\int_{\partial\mathcal{E}}  \xsym{\rho}_1(\xvec{u},\xvec{B}) \cdot \xsym{\varphi} \, \xdx{S} \, \xdx{t}\\
			- \nu_2 \int_0^t\int_{\partial\mathcal{E}} \xsym{\rho}_2(\xvec{u},\xvec{B}) \cdot \xsym{\psi} \, \xdx{S} \, \xdx{t} = \int_0^t \int_{\mathcal{E}} \left(\xsym{\xi} \cdot \xsym{\varphi} + \xsym{\eta} \cdot \xsym{\psi}\right) \, \xdx{\xvec{x}} \, \xdx{t}, 
		\end{gathered}
	\end{equation}
	together with the following energy inequality for almost all $0 \leq s < t \leq T$: 
	\begin{equation}\label{equation:sei}
		\begin{gathered}
			\|\xvec{u}(\cdot, t)\|_{\xLtwo(\mathcal{E})}^2 + \mu\|\xvec{B}(\cdot, t)\|_{\xLtwo(\mathcal{E})}^2 + 2\int_s^t \int_{\mathcal{E}} \left(\nu_1|\xcurl{\xvec{u}}|^2 + \nu_2 \mu |\xcurl{\xvec{B}}|^2\right) \, \xdx{\xvec{x}} \, \xdx{t}  \\
			\leq \|\xvec{u}(\cdot, s)\|_{\xLtwo(\mathcal{E})}^2 + \mu\|\xvec{B}(\cdot, s)\|_{\xLtwo(\mathcal{E})}^2 + 2 \int_s^t \int_{\mathcal{E}} \xsym{\xi} \cdot \xvec{u} \, \xdx{\xvec{x}} \, \xdx{t} + 2 \mu \int_s^t \int_{\mathcal{E}} \xsym{\eta} \cdot \xvec{B} \, \xdx{\xvec{x}} \, \xdx{t}  \\
			\quad + 2\nu_1 \int_s^t \int_{\partial\mathcal{E}}  \xsym{\rho}_1(\xvec{u},\xvec{B}) \cdot \xvec{u} \, \xdx{S} \, \xdx{t} + 2\nu_2\mu\int_s^t \int_{\partial\mathcal{E}}  \xsym{\rho}_2(\xvec{u},\xvec{B}) \cdot \xvec{B} \, \xdx{S} \, \xdx{t}.
		\end{gathered}
	\end{equation}
	The weak formulation \eqref{equation:varformwct} and energy inequality \eqref{equation:sei} are derived by utilizing the identity $\Delta \xvec{h} = - \xcurl{(\xcurl{\xvec{h}})}$ for any sufficiently regular vector field $\xvec{h}$ with $\xdiv{\xvec{h}} = 0$, while also using the integration by parts and vector calculus formulas
	\begin{gather*}
		\int_{\mathcal{E}} \xvec{g} \cdot (\xcurl{\xvec{h}}) \, \xdx{\xvec{x}} = \int_{\mathcal{E}} (\xcurl{\xvec{g}}) \cdot \xvec{h} \, \xdx{\xvec{x}} - \int_{\partial \mathcal{E}} (\xvec{g} \times \xvec{h}) \cdot \xvec{n} \, \xdx{S},\\
		(\xvec{g}\times\xvec{h}) \cdot \xvec{n} = (\xvec{h} \times \xvec{n}) \cdot \xvec{g} = - (\xvec{g} \times \xvec{n}) \cdot \xvec{h},
	\end{gather*}
	where $\xdx{S}$ stands for the surface measure on $\partial \mathcal{E}$.

	\paragraph{Existence of weak solutions.} By analysis similar to the Navier--Stokes equations, for instance via the Galerkin method explained in \cite[Chapter 3]{Temam2001}, one can obtain the existence of Leray--Hopf weak solutions $(\xvec{u}, \xvec{B}) \in \mathscr{X}_T\times\mathscr{X}_T$ satisfying \eqref{equation:varformwct} and \eqref{equation:sei}. Regarding the energy inequality~\eqref{equation:sei}, we refer to \cite[Section 3]{IftimieSueur2011} for a strategy that carries over to the present MHD model. 
	In particular, the boundary integrals in \Cref{equation:varformwct} and \eqref{equation:sei} are not causing additional difficulties in comparison with the references mentioned above. 
	Indeed, when $\xvec{M} \in \xLinfty(\partial\mathcal{E};\mathbb{R}^{N\times N})$ and $\xvec{g},\xvec{h} \in \xLtwo((0,T);\xHone(\mathcal{E}))$, trace inequalities and interpolation imply
	\begin{equation}\label{equation:bdrinttrest}
		\begin{aligned}
			\left| \int_s^t \int_{\partial\mathcal{E}} \left[\xvec{M}(\xvec{x})\xvec{h}\right]_{\operatorname{tan}} \cdot \xvec{g} \, \xdx{S} \, \xdx{r} \right| & \leq \frac{C}{\delta}\!\!\int_s^t \|\xvec{g}(\cdot,r)\|_{\xLtwo(\mathcal{E})}^2 \, \xdx{r} + \delta \int_s^t \|\xvec{g}(\cdot,r)\|_{\xHone(\mathcal{E})}^2  \, \xdx{r}\\
			& \quad + \frac{C}{\delta} \!\! \int_s^t \|\xvec{h}(\cdot,r)\|_{\xLtwo(\mathcal{E})}^2 \, \xdx{r} + \delta \int_s^t \|\xvec{h}(\cdot,r)\|_{\xHone(\mathcal{E})}^2  \, \xdx{r}
		\end{aligned}
	\end{equation}
	for any $\delta \in (0,1)$ and a generic constant $C = C(\mathcal{E},\|\xvec{M}\|_{\xLinfty(\partial\mathcal{E})}) > 0$. The estimates \eqref{equation:sKem} and \eqref{equation:bdrinttrest} facilitate a Galerkin method of the type described in \cite[Chapter 3]{Temam2001}. In this way, one obtains approximate solutions $(\xvec{u}^k,\xvec{B}^k)_{k \in \mathbb{N}}$ to \eqref{equation:varformwct} that are bounded in $\xLinfty((0,T);\xLtwo(\mathcal{E})) \cap \xLtwo((0,T);\xHone(\mathcal{E}))$, satisfy a discrete version of \eqref{equation:sei}, and converge in $\xLtwo((0,T);\xLtwo(\mathcal{E}))$ to a Leray--Hopf weak solution $(\xvec{u},\xvec{B})  \in \mathscr{X}_T\times\mathscr{X}_T$ as $k \to + \infty$. For passing the limit $k \to +\infty$ in the discrete version of the energy inequality \eqref{equation:sei}, one needs to show
	\begin{multline*}\label{equation:convapproxenbdr}
		\nu_1 \int_s^t \int_{\partial\mathcal{E}}  \xsym{\rho}_1(\xvec{u}^k,\xvec{B}^k) \cdot \xvec{u}^k \, \xdx{S} \, \xdx{t} + \nu_2\mu\int_s^t \int_{\partial\mathcal{E}}  \xsym{\rho}_2(\xvec{u}^k,\xvec{B}^k) \cdot \xvec{B}^k \, \xdx{S} \, \xdx{t} \\
		\longrightarrow \nu_1 \int_s^t \int_{\partial\mathcal{E}}  \xsym{\rho}_1(\xvec{u},\xvec{B}) \cdot \xvec{u} \, \xdx{S} \, \xdx{t} + \nu_2\mu\int_s^t \int_{\partial\mathcal{E}}  \xsym{\rho}_2(\xvec{u},\xvec{B}) \cdot \xvec{B} \, \xdx{S} \, \xdx{t}, \, \mbox{ as } k \to +\infty.
	\end{multline*}
	To this end, if $\xvec{h}^k$ denotes either $\xvec{u}^k$ or $\xvec{B}^k$, and $\xvec{h}$ represents either $\xvec{u}$ or $\xvec{B}$, then, by means of trace theorems and interpolation, one has
	\begin{equation}\label{equation:bdrtmtrest}
		\begin{aligned}
			\|\xvec{h}^k - \xvec{h}\|_{\xLtwo((0,T);\xLtwo(\partial\mathcal{E}))}^2 & \leq \|\xvec{h}^k  - \xvec{h}\|_{\xLtwo((0,T);\xLtwo(\mathcal{E}))}\|\xvec{h}^k  - \xvec{h}\|_{\xLtwo((0,T);\xHone(\mathcal{E}))} \\
			& \leq C\|\xvec{h}^k  - \xvec{h}\|_{\xLtwo((0,T);\xLtwo(\mathcal{E}))} \longrightarrow 0, \quad k \longrightarrow +\infty,
		\end{aligned}
	\end{equation}
	which implies for $(\xvec{g}^{1}, \xvec{g}^{2}) = (\xvec{u}, \xvec{B})$ and $(\xvec{g}^{1,k}, \xvec{g}^{2,k} ) = (\xvec{u}^k, \xvec{B}^k)$ that
	\begin{multline*}
		\int_s^t\int_{\partial\mathcal{E}} \left| \left[\xvec{M}_i\xvec{u}^k + \xvec{L}_i\xvec{B}^k \right]_{\operatorname{tan}} \cdot \xvec{g}^{i,k} - \left[\xvec{M}_i\xvec{u} + \xvec{L}_i\xvec{B} \right]_{\operatorname{tan}} \cdot \xvec{g}^i \right|\, \xdx{S} \xdx{r} \\
		\begin{aligned}
			& \leq \int_s^t\int_{\partial\mathcal{E}} \left| \left[(\xvec{M}_i(\xvec{u}^k-\xvec{u}) + \xvec{L}_i(\xvec{B}^k-\xvec{B})) \right]_{\operatorname{tan}}\cdot \xvec{g}^{i,k}\right| \, \xdx{S} \xdx{r} \\
			& \quad + \int_s^t\int_{\partial\mathcal{E}} \left| \left[ \xvec{M}_i\xvec{u} + \xvec{L}_i\xvec{B} \right]_{\operatorname{tan}}\cdot (\xvec{g}^i-\xvec{g}^{i,k})\right| \, \xdx{S} \xdx{r} \,\longrightarrow 0 \, \mbox{ as } k \longrightarrow + \infty, \, i=1,2.
		\end{aligned}
	\end{multline*}
	
	\subsubsection{A pressure-like gradient in the induction equation}\label{subsubsection:gq}
	It is important to verify that sufficiently regular functions which satisfy the variational formulation \eqref{equation:varformwct} are classical solutions to the original problem \eqref{equation:MHD_ElsaesserExt}. On the one hand, if the pair $(\xvec{u},\xvec{B})$ possesses the necessary regularity and satisfies~\eqref{equation:varformwct}, then $(\xvec{u},\xvec{B})$ also classically obeys a version of~\eqref{equation:MHD_ElsaesserExt} where the induction equation is replaced by
	\begin{equation}\label{equation:replindeq}
		\partial_t \xvec{B} - \nu_2 \Delta \xvec{B} + (\xvec{u} \cdot \xdop{\nabla}) \xvec{B} - (\xvec{B} \cdot \xdop{\nabla}) \xvec{u} + \xnab q = \xsym{\eta}.
	\end{equation}
	On the other hand, because $\xsym{\psi}$ in \eqref{equation:varformwct} is divergence-free and tangential at $\partial \mathcal{E}$, one cannot generally conclude that \eqref{equation:MHD_ElsaesserExt} is satisfied. However, let us now suppose that $\xvec{M}_2 = \xvec{L}_2 = \xsym{0}$, and that the control $\xsym{\eta}$ in \eqref{equation:MHD_ElsaesserExt} additionally obeys 
	\begin{equation}\label{2.11}
		\xdiv{\xsym{\eta}} = 0\, {\rm in}\; \mathcal{E}_T, \quad   \xsym{\eta} \cdot \xvec{n} = 0 \, {\rm on}\; \Sigma_T.
	\end{equation}
	In this situation, by acting with the divergence operator on \eqref{equation:replindeq}, while also taking the normal traces of \eqref{equation:replindeq} at $\partial \mathcal{E}$, one finds that $q(\cdot,t)$ solves the Neumann problem
	\begin{equation}\label{equation:npq}
		\begin{cases}
			\Delta q(\cdot,t) = 0 & \mbox{ in } \mathcal{E},\\
			\xnab q(\cdot, t) \cdot \xvec{n} = 0 & \mbox{ on } \partial \mathcal{E},
		\end{cases}
	\end{equation}
	hence $\xnab q(\cdot, t) = 0$ holds in $\mathcal{E}_T$.

	\begin{rmrk}
		If $\xvec{M}_2 \neq \xsym{0}$ or $\xvec{L}_2 \neq \xsym{0}$, even for the uncontrolled model with $\xsym{\xi} = \xsym{\eta} = \xsym{0}$ in $\mathcal{E}_T$, we shall use the modified induction equation \eqref{equation:replindeq}. Indeed, assuming, \eg, for $N = 2$, that there would exist a classical solution $(\xvec{u}, \xvec{B})$ to \eqref{equation:MHD_ElsaesserExt} on a short time interval, then necessarily $\Delta \xvec{B} \cdot \xvec{n} = \xsym{\eta} \cdot \xvec{n}$ on~$\Sigma_T$. Since $\xsym{\eta}$ is supported in $\overline{\mathcal{E}}\setminus\overline{\Omega}$, it is unclear whether this could contradict
		\[
			\Delta \xvec{B} \cdot \xvec{n} = - \xnab^{\perp} \left(\xwcurl{\xvec{B}}\right) \cdot \xvec{n} = - 2\xnab^{\perp} \left( [\xvec{M}_2 \xvec{u} + \xvec{L}_2 \xvec{B}]_{\operatorname{tan}} \right) \cdot \xvec{n} \mbox{ on } \Sigma_T.
		\]
		For instance, if $\xvec{M}_2$ is the identity matrix and $\xvec{L}_2 = \xsym{0}$, then it would follow for $\xsym{\eta} = \xsym{0}$ that $\xvec{u}$ is a constant of time at $\Sigma_T$; thus, $\xvec{u}$ would satisfy a Navier-slip-with-friction and a Dirichlet boundary condition at the same time.
	\end{rmrk}

	Despite that possibly $\xnab q \neq \xsym{0}$ when $\xvec{M}_2 \neq \xsym{0}$ or $\xvec{L}_2 \neq \xsym{0}$, it still remains an interesting question whether magnetic field interior controls with~\eqref{2.11} can be constructed, as this would provide more insights on the nature of the term $\xnab q$ in a control theoretic context. In order to obtain a partial result in that direction (\cf~\Cref{theorem:annulus}), we assume for now that $N = 2$ and let $(\xvec{g},\xvec{h}) \mapsto \xnab \mathcal{Q}(\xvec{g},\xvec{h})$ be the linear operator that assigns to $(\xvec{g},\xvec{h}) \in \xHone(\mathcal{E})\cap\xH(\mathcal{E})$ the gradient of the (unique up to a constant) solution to the Neumann problem
	\[
	\begin{cases}
		\Delta \mathcal{Q}(\xvec{g},\xvec{h}) = 0 & \mbox{ in } \mathcal{E},\\
		\xnab \mathcal{Q}(\xvec{g},\xvec{h}) \cdot \xvec{n} = - 2\nu_2\xnab^{\perp} \left( [\xvec{M}_2 \xvec{g} + \xvec{L}_2 \xvec{h}]_{\operatorname{tan}} \right) \cdot \xvec{n} & \mbox{ on } \partial \mathcal{E}.
	\end{cases}
	\]
	Then, concerning the general case where $\xvec{M}_2 \neq \xsym{0}$ or $\xvec{L}_2 \neq \xsym{0}$, the problem \eqref{equation:MHD_ElsaesserExt} with induction equation replaced by \eqref{equation:replindeq} can be reformulated as
	\begin{equation}\label{equation:MHD_ElsaesserExtReformulated}
		\begin{cases}
			\partial_t \xvec{u} - \nu_1 \Delta \xvec{u} + (\xvec{u} \cdot \xdop{\nabla}) \xvec{u} - \mu(\xvec{B} \cdot \xnab) \xvec{B} + \xdop{\nabla} p = \xsym{\xi} & \mbox{ in } \mathcal{E}_T,\\
			\partial_t \xvec{B} - \nu_2 \Delta \xvec{B} + (\xvec{u} \cdot \xdop{\nabla}) \xvec{B} - (\xvec{B} \cdot \xdop{\nabla}) \xvec{u} + \xnab \mathcal{Q}(\xvec{u},\xvec{B}) = \xsym{\eta} & \mbox{ in } \mathcal{E}_T,\\
			\xdop{\nabla}\cdot\xvec{u} = \xdop{\nabla}\cdot\xvec{B} = 0 & \mbox{ in } \mathcal{E}_T,\\
			\xvec{u} \cdot \xvec{n} = \xvec{B} \cdot \xvec{n} = 0 & \mbox{ on } \Sigma_T,\\
			\xmcal{N}_1(\xvec{u},\xvec{B}) = \xmcal{N}_2(\xvec{u},\xvec{B}) = \xvec{0} & \mbox{ on } \Sigma_T,\\
			\xvec{u}(\cdot, 0)  =  \xvec{u}_0,\, \xvec{B}(\cdot, 0)  =  \xvec{B}_0 & \mbox{ in } \mathcal{E}.
		\end{cases}
	\end{equation}
	By analysis similar to the $2$D incompressible Navier--Stokes system, weak solutions to \eqref{equation:MHD_ElsaesserExtReformulated} are unique and initial data $\xvec{u}_0, \xvec{B}_0 \in \xHone(\mathcal{E})\cap\xH(\mathcal{E})$ give rise to strong solutions when \eqref{2.11} is satisfied. In this article we shall prove in parallel to Theorems~\Rref{theorem:main1} and~\Rref{theorem:main} the following controllability result for \eqref{equation:MHD_ElsaesserExtReformulated} via distributed controls $\xsym{\xi}$ and~$\xsym{\eta}$ satisfying \eqref{2.11}.
	\begin{thrm}\label{theorem:annulus}
		Let $\mathcal{E} \subset \mathbb{R}^2$ be an annulus and $\omega \subset \mathcal{E}$ an open simply-connected control region such that $\mathcal{E}\setminus\omega$ is simply-connected and $\omega$ contains an annulus sector. The friction coefficient matrices are arbitrarily fixed with
		\[
		\xvec{M}_1, \xvec{L}_1, \xvec{L}_2 \in \xCinfty(\mathcal{E};\mathbb{R}^{2\times2}), \quad \xvec{M}_2 \coloneqq \rho \xvec{I}, \quad \rho \in \mathbb{R}, \quad \xvec{I} \coloneqq \begin{bmatrix}
			1 & 0 \\ 0 & 1
		\end{bmatrix}.
		\]  
		Then, for any control time $T_{\operatorname{ctrl}} > 0$, accuracy parameter $\delta > 0$, and
		\begin{itemize}
			\item initial states $\xvec{u}_0, \xvec{B}_0 \in \xH(\mathcal{E}) \cap \xHn{3}(\mathcal{E})$ with $\xvec{B}_0 = \xnab^{\perp} \psi_0$ and $\psi_0$ vanishing on~$\partial\mathcal{E}$,
			\item target states $\xvec{u}_1, \xvec{B}_1 \in \xH(\mathcal{E})$ with $\xvec{B}_1 = \xnab^{\perp} \psi_1$ for a stream function $\psi_1$ that vanishes on $\partial \mathcal{E}$,
		\end{itemize}
		there exist controls $\xsym{\xi}, \xsym{\eta} \in \xCzero([0,T_{\operatorname{ctrl}}];\xHtwo(\mathcal{E};\mathbb{R}^N))$, supported in $\overline{\omega}\times[0,T]$ and obeying \eqref{2.11}, such that the solution $(\xvec{u}, \xvec{B})$ to \eqref{equation:MHD_ElsaesserExtReformulated} satisfies the terminal condition
		\[
			\|\xvec{u}(\cdot,T_{\operatorname{ctrl}}) - \xvec{u}_1\|_{\xLtwo(\mathcal{E})} + \|\xvec{B}(\cdot,T_{\operatorname{ctrl}}) - \xvec{B}_1\|_{\xLtwo(\mathcal{E})} < \delta.
		\]
	\end{thrm}

	\begin{figure}[ht!]
		\centering
		\resizebox{0.28\textwidth}{!}{
			\begin{tikzpicture}
				\clip(-3.1,-3.1) rectangle (3.1,3.1);
				\draw[line width=0.5mm, color=black] (80:1.5) coordinate (A) arc (80:2:1.5) coordinate (C)  (2:2.9) coordinate (D) arc (2:80:2.9) coordinate (B);
				\draw[line width=0.5mm, color=black] (80:1.5) arc (80:365:1.5) (2:2.9) arc (2:365:2.9);

				\draw[dashed,line width=0.5mm, color=black] (A) -- (B);
				\draw[dashed,line width=0.5mm, color=black] (C) -- (D);
				
				\coordinate[label=left:\large{$\mathcal{E}$}] (A) at (0.35,-2.2);
				\coordinate[label=left:\large{$\omega$}] (A) at (2,1.5);									
			\end{tikzpicture}
		}
		\caption{Sketch of an annulus $\mathcal{E}$ with control region $\omega$ as required by \Cref{theorem:annulus}. In order to integrate \Cref{theorem:annulus} into the notational framework of Theorems~\Rref{theorem:main1} and \Rref{theorem:main}, one can identify $\Omega = \mathcal{E}\setminus \omega$. }
		\label{Figure:Annulusinteriorcontrol}
	\end{figure}

	\subsubsection{Change of unknowns}\label{subsubsection:changeofvariables}
	A few exceptions aside, the subsequent analysis can be streamlined by introducing the symmetrized notations
	\begin{equation}\label{equation:chouk}
		\begin{gathered}
			\xvec{z}^{\pm} \coloneqq \xvec{u} \pm \sqrt{\mu}\xvec{B}, \quad  p^{\pm} = p \pm \sqrt{\mu}q, \quad \xsym{\xi}^{\pm} \coloneqq \xsym{\xi} \pm \sqrt{\mu}\xsym{\eta}, \\
			\lambda^{\pm} \coloneqq \frac{\nu_1\pm\nu_2}{2}, \quad  \xvec{z}^{\pm}_0 \coloneqq \xvec{u}_0 \pm \sqrt{\mu}\xvec{B}_0,
		\end{gathered}
	\end{equation}
	as well as
	\begin{equation*}
		\begin{gathered}
			\xmcal{N}^{\pm}(\xvec{h}^+,\xvec{h}^-) \coloneqq \left[\xdop{D}(\xvec{h}^{\pm})\xvec{n}(\xvec{x}) + \xvec{W}_{\mathcal{E}} \xvec{h}^{\pm} + \xvec{M}^{\pm}(\xvec{x})\xvec{h}^+ + \xvec{L}^{\pm} \xvec{h}^{-} \right]_{\operatorname{tan}}
		\end{gathered}
	\end{equation*}
	and
	\[
		\xsym{\rho}^{\pm}(\xvec{h}^+,\xvec{h}^-) \coloneqq 2\left[\xvec{M}^{\pm}(\xvec{x})\xvec{h}^+ + \xvec{L}^{\pm}(\xvec{x})\xvec{h}^-\right]_{\operatorname{tan}},
	\]
	where
	\[
		\xvec{M}^{\pm} \coloneqq \frac{\sqrt{\mu}\xvec{M}_1 \pm \sqrt{\mu} \xvec{M}_2 + L_1 \pm L_2}{2\sqrt{\mu}}, \quad \xvec{L}^{\pm} \coloneqq \frac{\sqrt{\mu}\xvec{M}_1 \pm \sqrt{\mu} \xvec{M}_2 - L_1 \mp L_2}{2\sqrt{\mu}}.
	\]
	By utilizing the inner product structure of $\xLtwo(\mathcal{E})$, one can verify that the energy
	\[
		E(t) \coloneqq \|\xvec{u}(\cdot, t)\|_{\xLtwo(\mathcal{E})}^2 + \mu\|\xvec{B}(\cdot, t)\|_{\xLtwo(\mathcal{E})}^2 + 2\int_s^t \int_{\mathcal{E}} \left(\nu_1|\xcurl{\xvec{u}}|^2 + \nu_2|\xcurl{\xvec{B}}|^2\right) \, \xdx{\xvec{x}} \, \xdx{t}
	\]
	satisfies
	\begin{equation}\label{equation:seizpm}
		\begin{aligned}
			E(t) & =  \frac{1}{2} \sum\limits_{\square\in\{+,-\}} \|\xvec{z}^{\square}(\cdot, t)\|_{\xLtwo(\mathcal{E})}^2 + \lambda^+ \sum\limits_{\square\in\{+,-\}} \int_0^t \int_{\mathcal{E}}\xcurl{\xvec{z}^{\square}}  \cdot \xcurl{\xvec{z}^{\square}} \, \xdx{\xvec{x}} \xdx{s} \\
			& \quad + \lambda^- \sum\limits_{\substack{(\triangle,\circ)\in \\ \{(+,-), (-,+)\}}} \int_0^t \int_{\mathcal{E}} \xcurl{\xvec{z}^{\triangle}} \cdot \xcurl{\xvec{z}^{\circ}} \, \xdx{\xvec{x}} \xdx{s}.
		\end{aligned}
	\end{equation}
	Therefore, if $(\xvec{u}, \xvec{B}) \in \mathscr{X}_T\times\mathscr{X}_T$ is a Leray--Hopf weak solution to \eqref{equation:MHD_ElsaesserExt}, by inserting \eqref{equation:seizpm} to \eqref{equation:sei} and using the transformations from \eqref{equation:chouk}, it follows for almost all $0 \leq s \leq t \leq T$ the inequality
	\begin{equation}\label{equation:seizpm2}
		\begin{gathered}
			\frac{1}{2} \sum\limits_{\square\in\{+,-\}} \|\xvec{z}^{\square}(\cdot, t)\|_{\xLtwo(\mathcal{E})}^2 + \lambda^+ \sum\limits_{\square\in\{+,-\}} \int_s^t \int_{\mathcal{E}}\xcurl{\xvec{z}^{\square}} \cdot \xcurl{\xvec{z}^{\square}} \, \xdx{\xvec{x}} \xdx{r} \\
			+ \lambda^- \sum\limits_{\substack{(\triangle,\circ)\in \\ \{(+,-), (-,+)\}}} \int_s^t \int_{\mathcal{E}} \xcurl{\xvec{z}^{\triangle}} \cdot \xcurl{\xvec{z}^{\circ}} \, \xdx{\xvec{x}} \xdx{r} \\
			\leq  \sum\limits_{\square\in\{+,-\}} \left(\frac{1}{2}\|\xvec{z}^{\square}(\cdot, s)\|_{\xLtwo(\mathcal{E})}^2 + \int_s^t\int_{\mathcal{E}} \xsym{\xi}^{\square}  \cdot \xvec{z}^{\square} \, \xdx{\xvec{x}} \xdx{r} + \lambda^+ \int_s^t\int_{\partial\mathcal{E}} \xsym{\rho}^{\square}(\xvec{z}^+,\xvec{z}^-) \cdot \xvec{z}^{\square} \, \xdx{S} \xdx{r}\right) \\
			+ \lambda^-\sum\limits_{\substack{(\triangle,\circ)\in \\ \{(+,-), (-,+)\}}} \int_s^t\int_{\partial\mathcal{E}} \xsym{\rho}^{\triangle}(\xvec{z}^+,\xvec{z}^-) \cdot \xvec{z}^{\circ} \, \xdx{S} \xdx{r}.
		\end{gathered}
	\end{equation}
	
	\subsubsection{The case of \texorpdfstring{\Cref{theorem:main}}{the second main theorem}}\label{subsubsection:secondcase}
	Given the assumptions of \Cref{theorem:main}, an application of \Cref{lemma:dataext} provides initial data extensions $\xvec{u}_0, \xvec{B}_0 \in \xLtwo(\mathcal{E})$ with $\xvec{u}_0 \cdot \xvec{n} = \xvec{B}_0 \cdot \xvec{n} = 0$ at $\partial\mathcal{E}$. Moreover, the scalar functions  $\xdiv{\xvec{u}_0}$ and $\xdiv{\xvec{B}_0}$ belong to $\xCinfty(\overline{\mathcal{E}};\mathbb{R})$ and are supported in $\mathcal{E}\setminus\overline{\Omega}$.  
	A weak controlled trajectory for \eqref{equation:MHD00withq} is then defined as any pair
	\[
	(\xvec{u}, \xvec{B}) \in 	\left[\xCn{0}_{w}([0,T];\xLn{2}_{\operatorname{c}})\cap\xLn{2}((0,T);\xHn{1}(\Omega))\right]^2,
	\]
	with
	\[
	\xvec{u} = \frac{\xvec{z}^{+} + \xvec{z}^{-}}{2} |_{\Omega}, \quad \xvec{B} = \frac{\xvec{z}^{+} - \xvec{z}^{-}}{\sqrt{\mu}2} |_{\Omega},
	\]
	where $\xvec{z}^{\pm} \in \xCn{0}_{w}([0,T];\xLtwo(\mathcal{E}))\cap\xLn{2}((0,T);\xHone(\mathcal{E}))$ solve in the below specified Leray--Hopf weak sense the Elsasser system\footnote{Systems of the form \eqref{equation:MHD_ElsaesserExt_caseB}, but with different boundary conditions, have been considered by Elsasser in \cite{Elsasser1950}.}
	\begin{equation}\label{equation:MHD_ElsaesserExt_caseB}
		\begin{cases}
			\partial_t \xvec{z}^{\pm} - \Delta (\lambda^{\pm}\xvec{z}^{+} + \lambda^{\mp}\xvec{z}^{-}) + (\xvec{z}^{\mp} \cdot \xdop{\nabla}) \xvec{z}^{\pm} + \xdop{\nabla} p^{\pm} =  \xsym{\xi}^{\pm} & \mbox{ in } \mathcal{E}_T,\\
			\xdop{\nabla}\cdot\xvec{z}^{\pm} = \sigma^{\pm}  & \mbox{ in } \mathcal{E}_T,\\
			\xvec{z}^{\pm} \cdot \xvec{n} = 0 & \mbox{ on }  \Sigma_T,\\
			\xmcal{N}^{\pm}(\xvec{z}^{+},\xvec{z}^{-}) = \xvec{0}  & \mbox{ on }  \Sigma_T,\\
			\xvec{z}^{\pm}(\cdot, 0) = \xvec{z}_0^{\pm} \coloneqq \xvec{u}_0 \pm \sqrt{\mu} \xvec{B}_0  & \mbox{ in } \mathcal{E},
		\end{cases}
	\end{equation}
	where $\xsym{\xi}^{\pm} = \xsym{\xi} \pm \sqrt{\mu} \xsym{\eta}$.
	In \eqref{equation:MHD_ElsaesserExt_caseB}, the boundary operators $\xmcal{N}^{\pm}$ are as defined in \Cref{subsubsection:changeofvariables}. Moreover, the functions $\sigma^{\pm}\colon\overline{\mathcal{E}}\times[0,T]\longrightarrow\mathbb{R}$ are assumed to be smooth, supported in $\overline{\mathcal{E}}\setminus\overline{\Omega}$, and of zero average in $\mathcal{E}$, that is $\sum_{i = 1}^{K(\Omega)} \smallint_{\Omega^i} \sigma^{\pm} \, \xdx{\xvec{x}} = 0$.
	Since $\xsym{M}^{\pm}$ and $\xsym{L}^{\pm}$ are likewise smooth, a notion for weak solutions to \eqref{equation:MHD_ElsaesserExt_caseB} can now be introduced similarly to the Navier--Stokes case in \cite{CoronMarbachSueur2020}. 	
	In order to lift the nonzero divergence constraints, let $\xvec{z}_{\sigma}^{\pm}$ solve the linear Elsasser system
	\begin{equation}\label{equation:MHD_ElsaesserExtStokes}
		\begin{cases}
			\partial_t \xvec{z}_{\sigma}^{\pm} - \Delta(\lambda^{\pm}\xvec{z}_{\sigma}^{+} + \lambda^{\mp}\xvec{z}_{\sigma}^{-})  + \xdop{\nabla} p_{\sigma}^{\pm} =  \xvec{0} & \mbox{ in } \mathcal{E}_T, \\
			\xdop{\nabla}\cdot\xvec{z}_{\sigma}^{\pm} = \sigma^{\pm} & \mbox{ in } \mathcal{E}_T,\\
			\xvec{z}_{\sigma}^{\pm} \cdot \xvec{n} = 0 & \mbox{ on } \Sigma_T,\\
			(\xcurl{\xvec{z}_{\sigma}^{\pm}}) \times \xvec{n} - \xsym{\rho}^{\pm}(\xvec{z}_{\sigma}^+,\xvec{z}_{\sigma}^-) = \xvec{0} & \mbox{ on }  \Sigma_T,\\
			\xvec{z}_{\sigma}^{\pm}(\cdot, 0) = \xvec{z}_0^{\pm}  & \mbox{ in } \mathcal{E}.
		\end{cases}
	\end{equation}
	For defining a weak formulation for \eqref{equation:MHD_ElsaesserExtStokes}, it is of advantage to first eliminate the inhomogeneous divergence data $\sigma^{\pm}$. Hereto, one may decompose $\xvec{z}_{\sigma}^{\pm} \coloneqq \xvec{Z}^{\pm}_{\sigma} + \xnab\theta^{\pm}_{\sigma}$, where $\theta_{\sigma}^{\pm}(\cdot,t)$ are for each $t \in [0,T]$ smooth solutions to the elliptic Neumann problems
	\begin{equation}\label{equation:EE}
		\begin{cases}
			\Delta \theta_{\sigma}^{\pm}(\cdot,t) = \sigma^{\pm}(\cdot,t) & \mbox{ in } \mathcal{E},\\
			\xnab \theta_{\sigma}^{\pm}(\cdot,t) \cdot \xvec{n}(\xvec{x}) = 0 & \mbox{ on } \partial\mathcal{E},\\
		\end{cases}
	\end{equation}
	while $\xvec{Z}^{\pm}_{\sigma}$ obey the inhomogeneous system
	\begin{equation*}
		\begin{cases}
			\partial_t \xvec{Z}^{\pm}_{\sigma} - \Delta (\lambda^{\pm}\xvec{Z}_{\sigma}^{+} + \lambda^{\mp} \xvec{Z}_{\sigma}^-) + \xdop{\nabla} p^{\pm}_{\sigma} = \xsym{\Theta}^{\pm} \coloneqq -\partial_t \xnab \theta_{\sigma}^{\pm} + \Delta (\lambda^{\pm} \xnab\theta_{\sigma}^{+} + \lambda^{\mp} \xnab\theta_{\sigma}^-) \!\!& \mbox{ in } \mathcal{E}_T,\\
			\xdiv{\xvec{Z}^{\pm}_{\sigma}} = 0 \!\!& \mbox{ in } \mathcal{E}_T,\\
			\xvec{Z}^{\pm}_{\sigma} \cdot \xvec{n} = 0 \!\!& \mbox{ on } \Sigma_T, \\
			(\xcurl{\xvec{Z}}_{\sigma}^{\pm}) \times \xvec{n} -  \xsym{\rho}^{\pm}(\xvec{Z}_{\sigma}^+,\xvec{Z}_{\sigma}^-) =  \xsym{\rho}^{\pm}(\xnab\theta_{\sigma}^+,\xnab\theta_{\sigma}^-) \!\!& \mbox{ on } \Sigma_T, \\
			\xvec{Z}^{\pm}_{\sigma}(\cdot, 0) = \xvec{z}^{\pm}_0 - \xnab\theta^{\pm}(\cdot, 0)  \!\!& \mbox{ in } \mathcal{E}.
		\end{cases}
	\end{equation*}
	The regularity of weak solutions to this linear problem can be investigated with the help of estimates that are known for the Navier--Stokes system (\cf~\Cref{lemma:stsreg}). As a result, one finds that $\xvec{Z}^{\pm}_{\sigma}$ are smooth for $t > 0$.
	Finally, the ansatz $\xvec{z}^{\pm} = \widetilde{\xvec{z}}^{\pm} + \xvec{z}_\sigma^{\pm}$ for the solutions of \eqref{equation:MHD_ElsaesserExt_caseB} provides a description of $\widetilde{\xvec{z}}^{\pm}$ by means of the perturbed Elsasser equations
	\begin{equation*}\label{equation:MHD_ElsaesserExtComp}
		\begin{cases}
			\partial_t \widetilde{\xvec{z}}^{\pm} - \Delta (\lambda^{\pm}\widetilde{\xvec{z}}^{+} + \lambda^{\mp}\widetilde{\xvec{z}}^{-}) + (\widetilde{\xvec{z}}^{\mp} + \xvec{z}_\sigma^{\mp}) \cdot \xdop{\nabla} (\widetilde{\xvec{z}}^{\pm} + \xvec{z}_\sigma^{\pm}) + \xdop{\nabla} \widetilde{p}^{\pm} =  \xsym{\xi}^{\pm} & \mbox{ in } \mathcal{E}_T,\\
			\xdop{\nabla}\cdot\widetilde{\xvec{z}}^{\pm} = 0 & \mbox{ in } \mathcal{E}_T\\
			\xmcal{N}^{\pm}(\widetilde{\xvec{z}}^+, \widetilde{\xvec{z}}^-) = \xvec{0} & \mbox{ on }  \Sigma_T,\\
			\widetilde{\xvec{z}}^{\pm}(\cdot, 0) = \xvec{0}  & \mbox{ in } \mathcal{E}.
		\end{cases}
	\end{equation*}
	A Leray--Hopf weak solution to the above system is any pair $(\widetilde{\xvec{z}}^+, \widetilde{\xvec{z}}^-) \in \mathscr{X}_T\times\mathscr{X}_T$ which satisfies for all $\xsym{\varphi}^{\pm} \in \xCinfty_0(\overline{\mathcal{E}}\times[0,T);\mathbb{R}^N) \cap \xCinfty([0,T];\xH(\mathcal{E}))$ and almost all $t \in [0,T]$ the variational formulation
	\begin{equation}\label{equation:varformwctB}
		\begin{gathered}
			\int_{\mathcal{E}} \widetilde{\xvec{z}}^{\pm}(\xvec{x},t) \cdot \xsym{\varphi}^{\pm}(\xvec{x},t) \, \xdx{\xvec{x}} -\int_0^t\int_{\mathcal{E}} \widetilde{\xvec{z}}^{\pm} \cdot \partial_t \xsym{\varphi}^{\pm} \, \xdx{\xvec{x}} \, \xdx{t} \\
			+ \lambda^{\pm} \int_0^t\int_{\mathcal{E}} (\xcurl{\widetilde{\xvec{z}}^{+}}) \cdot (\xcurl{\xsym{\varphi}^{\pm}})  \, \xdx{\xvec{x}} \, \xdx{t} + \lambda^{\mp} \int_0^t\int_{\mathcal{E}} (\xcurl{\widetilde{\xvec{z}}^{-}}) \cdot (\xcurl{\xsym{\varphi}^{\pm}})  \, \xdx{\xvec{x}} \, \xdx{t}  \\
			- \lambda^{\pm} \int_0^t\int_{\partial\mathcal{E}}  \xsym{\rho}^+(\widetilde{\xvec{z}}^{+},\widetilde{\xvec{z}}^{-})\cdot \xsym{\varphi}^{\pm}\, \xdx{S} \, \xdx{t} - \lambda^{\mp} \int_0^t\int_{\partial\mathcal{E}}  \xsym{\rho}^-(\widetilde{\xvec{z}}^{+},\widetilde{\xvec{z}}^{-}) \cdot \xsym{\varphi}^{\pm} \, \xdx{S} \, \xdx{t} \\
			+ \int_0^t \int_{\mathcal{E}} \left((\widetilde{\xvec{z}}^{\mp} \cdot \xdop{\nabla}) \widetilde{\xvec{z}}^{\pm} + (\widetilde{\xvec{z}}^{\mp} \cdot \xdop{\nabla}) \xvec{z}_\sigma^{\pm} + (\xvec{z}_\sigma^{\mp} \cdot \xdop{\nabla}) \widetilde{\xvec{z}}^{\pm} + (\xvec{z}_\sigma^{\mp} \cdot \xdop{\nabla}) \xvec{z}_\sigma^{\pm} \right) \cdot \xsym{\varphi}^{\pm} \, \xdx{\xvec{x}} \, \xdx{t}\\
			= \int_0^t \int_{\mathcal{E}} \xsym{\xi}^{\pm}  \cdot \xsym{\varphi}^{\pm}  \, \xdx{\xvec{x}} \, \xdx{t},
		\end{gathered}
	\end{equation}
	and for almost all $0 \leq s \leq t \leq T$ the energy inequality
	\begin{equation}\label{equation:seiB}
		\begin{gathered}
			\sum\limits_{\square\in\{+,-\}}\|\widetilde{\xvec{z}}^{\square}(\cdot, t)\|_{\xLtwo(\mathcal{E})}^2 + \sum\limits_{\square\in\{+,-\}} (\lambda^+\square \, \lambda^-) \int_s^t \int_{\mathcal{E}} \left|\xcurl{(\widetilde{\xvec{z}}^{+}\square \, \widetilde{\xvec{z}}^{-})} \right|^2  \, \xdx{\xvec{x}} \, \xdx{t}  \\
			\leq \sum\limits_{\square\in\{+,-\}}\left(\|\widetilde{\xvec{z}}^{\square}(\cdot, s)\|_{\xLtwo(\mathcal{E})}^2 + 2\lambda^+ \int_s^t \int_{\partial\mathcal{E}}  \xsym{\rho}^\square(\widetilde{\xvec{z}}^{+},\widetilde{\xvec{z}}^{-}) \cdot \widetilde{\xvec{z}}^{\square} \, \xdx{S} \, \xdx{t}  \right) \\
			+ \sum\limits_{\substack{(\triangle,\circ)\in \\ \{(+,-), (-,+)\}}} \left( 2 \lambda^- \int_s^t \int_{\partial\mathcal{E}}  \xsym{\rho}^{\triangle}(\widetilde{\xvec{z}}^{+},\widetilde{\xvec{z}}^{-}) \cdot \widetilde{\xvec{z}}^{\circ} \, \xdx{S} \, \xdx{t} + \int_s^t \int_{\mathcal{E}} \sigma^{\triangle} |\widetilde{\xvec{z}}^{\circ}|^2 \, \xdx{\xvec{x}} \, \xdx{t}\right)\\
			- 2 \sum\limits_{\substack{(\triangle,\circ)\in \\ \{(+,-), (-,+)\}}} \left(\int_s^t \int_{\mathcal{E}}\left((\widetilde{\xvec{z}}^{\triangle} \cdot \xdop{\nabla}) \xvec{z}_\sigma^{\circ} + (\xvec{z}_\sigma^{\triangle} \cdot \xdop{\nabla}) \xvec{z}_\sigma^{\circ} - \xsym{\xi}^{\circ} \right) \cdot \widetilde{\xvec{z}}^{\circ} \, \xdx{\xvec{x}} \, \xdx{t}\right).
		\end{gathered}
	\end{equation}
	Since the profiles $\widetilde{\xvec{z}}^{\pm}_{\sigma}$ are smooth and $\xvec{L}^{\pm},\xvec{M}^{\pm} \in \xCinfty(\overline{\mathcal{E}};\mathbb{R}^{N\times N})$, the existence of $(\widetilde{\xvec{z}}^+, \widetilde{\xvec{z}}^-) \in \mathscr{X}_T\times\mathscr{X}_T$ satisfying \eqref{equation:varformwctB} and \eqref{equation:seiB} can be obtained through a Galerkin method as explained in \Cref{subsubsection:case_thm1}.

	\begin{rmrk}
		The weak formulation from \Cref{subsubsection:case_thm1} for \eqref{equation:MHD_ElsaesserExt} can be regarded as a special case of \eqref{equation:varformwctB} and \eqref{equation:seiB}. The two formulations are presented separately in order to make a stronger distinction between the case without $\xnab q$ and non-physical situations, where even $\xdiv{\xvec{B}} \neq 0$ is possible in $\overline{\mathcal{E}}\setminus\overline{\Omega}$.
	\end{rmrk}
	
	\subsection{Brief description of the strategy}\label{subsection:Description}
	To prove Theorems~\Rref{theorem:main1}, \Rref{theorem:main}, and \Rref{theorem:annulus}, we develop the approach from \cite{CoronMarbachSueur2020}. Essentially, the time interval $[0,T_{\operatorname{ctrl}}]$ will be divided into two sub-intervals $[0,T_{\operatorname{reg}}]$ and $(T_{\operatorname{reg}}, T_{\operatorname{ctrl}}]$ which correspond to the two main stages of the control strategy.
	\paragraph{Stage 1 (\Cref{section:conclth}).} Leray--Hopf weak solutions to \eqref{equation:MHD_ElsaesserExt} or \eqref{equation:MHD_ElsaesserExt_caseB} with $\xsym{\xi} = \xsym{\eta} = \xvec{0}$ are selected such that their state at $t = T_{\operatorname{reg}}$ belongs to $\xHn{3}(\mathcal{E})\cap\xW(\mathcal{E})$ at a time $T_{\operatorname{reg}} \in [0, T_{\operatorname{ctrl}})$. 
	\paragraph{Stage 2 (\Cref{section:approxres2}).} During $(T_{\operatorname{reg}}, T_{\operatorname{ctrl}}]$, controls~$\xsym{\xi}$ and~$\xsym{\eta}$ are applied in $\overline{\mathcal{E}}\setminus\overline{\Omega}$. More precisely, we determine $\xsym{\xi}$ and $\xsym{\eta}$ such that each Leray--Hopf weak solution to \eqref{equation:MHD_ElsaesserExt} or \eqref{equation:MHD_ElsaesserExt_caseB}, starting from $\xHn{3}(\mathcal{E})\cap\xH(\mathcal{E})$ at $t = T_{\operatorname{reg}}$, approaches the final state in $\xLtwo(\mathcal{E})$ at~$t = T$. As discussed in \Cref{subsection:wct}, for proving Theorems~\Rref{theorem:main1} and \Rref{theorem:annulus}, we have to ensure that~$\xsym{\eta}$ is supported in $\overline{\mathcal{E}}\setminus\overline{\Omega}$ and obeys $\xdiv{\xsym{\eta}} = 0$ in $\mathcal{E}$ and $\xsym{\eta} \cdot \xvec{n} = 0$ at $\partial\mathcal{E}$.

	\section{Approximate controllability between regular states}\label{section:approxres2}
	Let $T > 0$, $\delta > 0$, and $\xvec{u}_0,\xvec{B}_0 \in \xHn{3}(\mathcal{E})\cap \xW(\mathcal{E})$.
	Then, assuming that sufficiently regular $\xsym{\xi}, \xsym{\eta}$, and $\sigma^{\pm}$ are given, the different configurations in Theorems~\Rref{theorem:main1}, \Rref{theorem:main}, and~\Rref{theorem:annulus} are treated simultaneously as follows.
	\begin{itemize}
		\item To show Theorems~\Rref{theorem:main1} and \Rref{theorem:annulus}, the Leray--Hopf weak solution $(\xvec{u}, \xvec{B}) \in \mathscr{X}^T_{\mathcal{E}} \times \mathscr{X}^T_{\mathcal{E}}$ to \eqref{equation:MHD_ElsaesserExt} with the data $(\xvec{u}_0,\xvec{B}_0,\xsym{\xi}, \xsym{\eta})$ is fixed and, by means of \Cref{subsubsection:changeofvariables}, rewritten in the symmetrized variables $\xvec{z}^{\pm}$. 
		\item For proving \Cref{theorem:main}, any Leray--Hopf weak solution $(\xvec{z}^+, \xvec{z}^-) \in \mathscr{X}^T_{\mathcal{E}} \times \mathscr{X}^T_{\mathcal{E}}$ to \eqref{equation:MHD_ElsaesserExt_caseB} with the data $(\xvec{z}^{\pm}_0 = \xvec{u}_0 \pm \sqrt{\mu} \xvec{B}_0, \xsym{\xi}^{\pm} = \xsym{\xi} \pm \sqrt{\mu} \xsym{\eta}, \sigma^{\pm})$ is fixed. 
	\end{itemize}
	It will be shown that, if the \apriori~selected functions $(\xsym{\xi}, \xsym{\eta}, \sigma^{\pm})$ are of a certain form, then $\xvec{z}^{\pm}$ satisfy
	\begin{equation}\label{equation:GoalApproxZerozpm}
		\|\xvec{z}^+(\cdot, T)\|_{\xLtwo(\mathcal{E})} + \|\xvec{z}^-(\cdot, T)\|_{\xLtwo(\mathcal{E})} < \delta\min\{1,\sqrt{\mu}\}.
	\end{equation}
	
	More generally, given arbitrary states $\overline{\xvec{z}}^{\pm}_1 \in \xCinfty(\overline{\mathcal{E}};\mathbb{R}^N)\cap\xH(\mathcal{E})$, it will be demonstrated that for suitable choices $(\xsym{\xi}, \xsym{\eta}, \sigma^{\pm})$ one has
	\begin{equation}\label{equation:GoalApproxZeroGeneral}
		\|\xvec{z}^+(\cdot, T)-\overline{\xvec{z}}^+_1\|_{\xLtwo(\mathcal{E})} + \|\xvec{z}^-(\cdot, T)-\overline{\xvec{z}}^-_1\|_{\xLtwo(\mathcal{E})} < \delta \min\{1,\sqrt{\mu}\},
	\end{equation}
	which implies
	\begin{equation*}\label{equation:GoalApproxZero}
		\|\xvec{u}(\cdot, T)-\overline{\xvec{u}}_1\|_{\xLtwo(\mathcal{E})} + \|\xvec{B}(\cdot, T)-\overline{\xvec{B}}_1\|_{\xLtwo(\mathcal{E})} < \delta,
	\end{equation*}
	where
	\[
	2\xvec{u} \coloneqq (\xvec{z}^++\xvec{z}^-), \quad 2\sqrt{\mu} \xvec{B} \coloneqq (\xvec{z}^+-\xvec{z}^-), \quad  2\overline{\xvec{u}}_1 \coloneqq (\overline{\xvec{z}}^+_1+\overline{\xvec{z}}^-_1), \quad 2\sqrt{\mu} \overline{\xvec{B}}_1 \coloneqq (\overline{\xvec{z}}^+_1-\overline{\xvec{z}}^-_1).
	\]

	\subsection{Asymptotic expansions}\label{seubsection:AsympExp}
	The systems \eqref{equation:MHD_ElsaesserExt} and \eqref{equation:MHD_ElsaesserExt_caseB} are reformulated as a small-dissipation perturbation of an ideal MHD type system in the variables $(\xvec{z}^+,\xvec{z}^-)$. Hereto, for any small $\epsilon \in (0,1)$, the following scaling is performed
	\begin{equation}\label{equation:scaling}
		\begin{gathered}
			\xvec{z}^{\pm,\epsilon} (\xvec{x}, t) \coloneqq \epsilon \xvec{z}^{\pm}(\xvec{x}, \epsilon t), \quad
			p^{\pm,\epsilon} (\xvec{x}, t) \coloneqq \epsilon^2 p^{\pm}(\xvec{x}, \epsilon t), \quad \sigma^{\pm,\varepsilon} (\xvec{x}, t)  \coloneqq \varepsilon \sigma^{\pm}(\xvec{x}, \varepsilon t),
		\end{gathered}
	\end{equation}
	and for the controls
	\begin{equation}\label{equation:scaling2}
		\begin{gathered}
			\xsym{\xi}^{\pm,\epsilon} (\xvec{x}, t) \coloneqq \epsilon^2 \xsym{\xi}^{\pm}(\xvec{x}, \epsilon t).
		\end{gathered}
	\end{equation}
	As a result, the profiles $\xvec{z}^{\pm,\epsilon}$ are seen to satisfy a weak formulation and strong energy inequality for the following problem
	\begin{equation}\label{equation:MHD_ElsaesserExtScaledLongTime}
		\begin{cases}
			\partial_t \xvec{z}^{\pm,\epsilon} - \epsilon\Delta (\lambda^{\pm}\xvec{z}^{+,\epsilon} + \lambda^{\mp} \xvec{z}^{-,\epsilon}) + (\xvec{z}^{\mp,\epsilon} \cdot \xdop{\nabla}) \xvec{z}^{\pm,\epsilon} + \xdop{\nabla} p^{\pm,\epsilon} = \xsym{\xi}^{\pm,\epsilon} & \mbox{ in } \mathcal{E}_{T/\epsilon},\\
			\xdop{\nabla}\cdot\xvec{z}^{\pm,\epsilon} = \sigma^{\pm, \epsilon} & \mbox{ in } \mathcal{E}_{T/\epsilon},\\
			\xvec{z}^{\pm,\epsilon} \cdot \xvec{n} = 0 & \mbox{ on } \Sigma_{T/\epsilon},\\
			\xmcal{N}^{\pm}(\xvec{z}^{+,\epsilon},\xvec{z}^{-,\epsilon}) = \xvec{0} & \mbox{ on } \Sigma_{T/\epsilon},\\
			\xvec{z}^{\pm,\epsilon}(\cdot, 0) = \epsilon\xvec{z}_0^{\pm} & \mbox{ in } \mathcal{E}.
		\end{cases}
	\end{equation}
	In order to achieve the desired estimate \eqref{equation:GoalApproxZeroGeneral}, it shall be verified that, for $\xsym{\xi}^{\pm,\varepsilon}$ and $\sigma^{\pm,\epsilon}$ being of specific forms, all solutions $(\xvec{z}^{+},\xvec{z}^{-}) \in \mathscr{X}^T_{\mathcal{E}}\times\mathscr{X}^T_{\mathcal{E}}$ to \eqref{equation:MHD_ElsaesserExtScaledLongTime} obey
	\begin{equation}\label{equation:asympbeh}
		\|\xvec{z}^{\pm, \epsilon}(\xvec{x}, T/\epsilon) - \overline{\xvec{z}}^{\pm}_1(\cdot, T/\epsilon)\|_{\xLn{2}(\mathcal{E})} = O(\epsilon^{9/8}).
	\end{equation}
	Hence, after choosing $\epsilon = \epsilon(\delta) > 0$ sufficiently small, the asymptotic behavior \eqref{equation:asympbeh} implies \cref{equation:GoalApproxZeroGeneral}.
	To prove \eqref{equation:asympbeh} with $\overline{\xvec{z}}^{\pm}_1 = \xvec{0}$, see \Cref{subsection:approxtowtraj} for the general case, the selected solution to \eqref{equation:MHD_ElsaesserExtScaledLongTime} is expanded according to the ansatz
	\begin{equation}\label{equation:ansatz2}
		\begin{aligned}
			\xvec{z}^{\pm,\epsilon} & = \xvec{z}^{0} + \sqrt{\epsilon}\left\llbracket \xvec{v}^{\pm}\right\rrbracket_{\epsilon} + \epsilon \xvec{z}^{\pm,1} + \epsilon \xdop{\nabla} \theta^{\pm,\epsilon} + \epsilon \left\llbracket \xvec{w}^{\pm}\right\rrbracket_{\epsilon} + \epsilon \xvec{r}^{\pm,\epsilon},	\\ 			p^{\pm,\epsilon} & = p^{0} + \epsilon \left\llbracket q^{\pm}\right\rrbracket_{\epsilon} + \epsilon p^{\pm,1} + \epsilon\vartheta^{\pm,\epsilon} + \epsilon \pi^{\pm,\epsilon},\\
			\sigma^{\pm,\epsilon} & = \sigma^0,
		\end{aligned}
	\end{equation}
	and for the controls
	\begin{equation}\label{3.7}
		\xsym{\xi}^{\pm,\epsilon}  = \xsym{\xi}^{0} + \sqrt{\epsilon}\left\llbracket \xsym{\xvec{\mu}}^{\pm}\right\rrbracket_{\epsilon} + \epsilon \xsym{\xi}^{\pm, 1} + \epsilon \widetilde{\xsym{\zeta}}^{\pm, \epsilon}.
	\end{equation}
	
	For all $t \geq T$, we fix
	\[
	\xvec{z}^{0}(\cdot,t) = \xvec{z}^{\pm,1}(\cdot, t) = \xsym{\xi}^{0}(\cdot, t) = \xsym{\xi}^{\pm, 1}(\cdot, t) = \widetilde{\xsym{\zeta}}^{\pm,\epsilon}(\cdot,t)  = \xsym{\mu}^{\pm}(\cdot,t,\cdot) = \xsym{0}
	\]
	and
	\[
	p^0(\cdot, t) = p^{\pm,1}(\cdot,t)  = \sigma^0(\cdot, t) = 0.
	\]
	On the time interval $[0,T]$, the profiles
	\[
	\xvec{z}^{0}\colon \mathcal{E}_T\longrightarrow \mathbb{R}^N, \quad p^{0}\colon \mathcal{E}_T\longrightarrow \mathbb{R}, \quad \xsym{\xi}^{0}\colon \mathcal{E}_T\longrightarrow \mathbb{R}^N, \quad \sigma^{0}\colon \mathcal{E}_T\longrightarrow \mathbb{R}
	\]
	are chosen in the following way:
	\begin{itemize}
		\item If $\Omega \subset \mathbb{R}^2$ is the simply-connected domain from \Cref{theorem:main1}, then $\xvec{z}^{0}$, $p^{0}$, $\xsym{\xi}^{0}$, and $\sigma^{0}$ are determined by \Cref{lemma:Z^0} below. Regarding the related situation of \Cref{theorem:annulus}, see \eqref{equation:zeordcdannulussect}.
		\item If $\Omega$ is a general smoothly bounded domain as in \Cref{theorem:main}, then $\xvec{z}^{0}$, $p^{0}$, $\xsym{\xi}^{0}$, and $\sigma^{0}$ are determined by \Cref{lemma:original_u^0} below.
		\item If $\Omega$ is the cylinder from \Cref{remark:cylm2}, then $\xvec{z}^{0}$, $p^{0}$, $\xsym{\xi}^{0}$, and $\sigma^{0}$ are given by \Cref{example:cylinder}.
	\end{itemize}
	The profiles $\xvec{z}^{\pm,1}\colon \mathcal{E}_T\longrightarrow \mathbb{R}^N$ are, later on, defined on $[0,T]$ via \Cref{lemma:flushing} as the solutions to \eqref{equation:MHD_ElsaesserExt_Ovareps}, together with associated pressure terms $p^{\pm,1}\colon \mathcal{E}_T\longrightarrow \mathbb{R}$, and interior controls $\xsym{\xi}^{\pm,1}\colon \mathcal{E}_T\longrightarrow \mathbb{R}^N$. 
	The vector field $\xvec{z}^{0}$ fails in general to obey the boundary condition $\xmcal{N}^{\pm}(\xvec{z}^0,\xvec{z}^0) = \xsym{0}$ at $\partial\mathcal{E}$, giving rise to weak amplitude boundary layers in the zero dissipation limit $\epsilon \longrightarrow 0$. These boundary layers are of a similar nature as those studied in \cite{IftimieSueur2011,CoronMarbachSueur2020}. In the particular case $\xmcal{N}^{+}(\xvec{z}^0,\xvec{z}^0) \neq \xmcal{N}^{-}(\xvec{z}^0,\xvec{z}^0)$, there appears not only a velocity boundary layer, but also one for the magnetic field. The profiles
	\[
	\xvec{v}^{\pm}, \xvec{w}^{\pm}\colon \mathcal{E}\times\mathbb{R}_+\times\mathbb{R}_+ \longrightarrow \mathbb{R}, \quad \theta^{\pm}, q^{\pm}, \vartheta^{\pm}\colon \mathcal{E}\times\mathbb{R}_+\times\mathbb{R}_+ \longrightarrow \mathbb{R},
	\]
	which are related to such boundary layers, will be described in \Cref{subsection:blprf}. In the presence of magnetic field boundary layers, the controls $\widetilde{\xsym{\zeta}}^{\pm, \epsilon}\colon \mathcal{E}_T\longrightarrow \mathbb{R}^N$ shall be defined in \Cref{subsubsection:technicalprofiles}. The boundary layer dissipation controls $\xsym{\mu}^{\pm}\colon \mathcal{E}_T\times\mathbb{R}_+\longrightarrow \mathbb{R}^N$ will be obtained in \Cref{subsubsection:vmo}.
	Subsequently, in \Cref{subsection:remestwbctrl}, the remainders~$\xvec{r}^{\pm,\epsilon}$ are estimated. Then, concerning approximate null controllability, the asymptotic behavior \eqref{equation:asympbeh} with $\overline{\xvec{z}}^{\pm}_1 = \xsym{0}$ is shown in \Cref{corollary:approxnlctrsmdt}. The approximate controllability towards arbitrary smooth states is concluded in \Cref{subsection:approxtowtraj}.

	\subsection{A return method trajectory}\label{subsection:return}
	The zero order profiles $\xvec{z}^{0}$, $p^{0}$, $\xsym{\xi}^{0}$, and $\sigma^0$ are chosen for $t \in [0,T]$ as a special solution to the controlled Euler system
	\begin{equation}\label{equation:MHD_Elsaesser_O1_}
		\begin{cases}
			\partial_t \xvec{z}^{0} + (\xvec{z}^{0} \cdot \xdop{\nabla}) \xvec{z}^{0} + \xdop{\nabla} p^{0} =  \xsym{\xi}^{0} & \mbox{ in } \mathcal{E}_T,\\
			\xdop{\nabla}\cdot\xvec{z}^{0} = \sigma^0 & \mbox{ in } \mathcal{E}_T,\\
			\xvec{z}^{0} \cdot \xvec{n} = 0 & \mbox{ on } \Sigma_T,\\
			\xvec{z}^{0}(\cdot, 0) = \xvec{0} & \mbox{ in } \mathcal{E},\\
			\xvec{z}^{0}(\cdot, T) = \xvec{0} & \mbox{ in } \mathcal{E}.
		\end{cases}
	\end{equation}
	Given a smooth vector field $\xvec{z}^0$, let $\xmcal{Z}^{0}(\xvec{x},s,t)$ denote for $(\xvec{x},s,t) \in \overline{\mathcal{E}}\times[0,T]\times[0,T]$ the unique flow which solves the ordinary differential equation
	\begin{equation}\label{equation:flowofy}
		\xdrv{}{t} \xmcal{Z}^{0}(\xvec{x},s,t) = \xvec{z}^{0}(\xmcal{Z}^{0}(\xvec{x},s,t),t), \quad
		\xmcal{Z}^{0}(\xvec{x},s,s) = \xvec{x}. 
	\end{equation}
	
	\begin{rmrk}
		As in \cite{CoronMarbachSueur2020}, the ansatz \eqref{equation:ansatz2}--\eqref{3.7} is based on the idea that the states $(\xvec{z}^{\pm,\epsilon},p^{\pm,\epsilon}, \sigma^{\pm,\varepsilon})$ should be near the return method profile $(\xvec{z}^{0}, p^{0},\sigma^0)$, which starts from $(\xvec{0},0,0)$ at $t = 0$ and returns back to $(\xvec{0},0,0)$ at time $t = T$.
	\end{rmrk}
	
	Under the assumptions of \Cref{theorem:main}, the profiles $\xvec{z}^{0}$, $p^{0}$, $\xsym{\xi}^{0}$, and $\sigma^0$ are chosen by means of \Cref{lemma:original_u^0} below, which is taken from \cite[Lemma 2]{CoronMarbachSueur2020} and has been proved in \cite{Coron1993,Coron1996EulerEq,Glass1997,Glass2000}. 
	\begin{lmm}[{{\cite[Lemma 2]{CoronMarbachSueur2020}}}]\label{lemma:original_u^0}
		There exists a solution $(\xvec{z}^{0}, p^{0}, \xsym{\xi}^{0},\sigma^0) \in \xCinfty(\overline{\mathcal{E}}\times[0,T];\mathbb{R}^{2N+2})$ to the controlled system \eqref{equation:MHD_Elsaesser_O1_}
		such that the flow $\xmcal{Z}^{0}$ obtained via \eqref{equation:flowofy}, satisfies 
		\begin{equation}\label{equation:flushingproperty}
			\forall \xvec{x} \in \overline{\mathcal{E}}, \, \exists t_{\xvec{x}} \in (0,T) \colon \xmcal{Z}^{0}(\xvec{x}, 0, t_{\xvec{x}}) \notin \overline{\Omega}.
		\end{equation}
		Moreover, all functions $(\xvec{z}^{0}, p^{0}, \xsym{\xi}^{0},\sigma^0)$ are compactly supported in $(0,T)$ with respect to time, the control $\xsym{\xi}^0$ obeys 
		\[
		\operatorname{supp}(\xsym{\xi}^0) \subset (\overline{\mathcal{E}}\setminus\overline{\Omega})\times(0,T),
		\]
		while $\xvec{z}^0$ can be chosen with $\xdiv{\xvec{z}^0} = 0$ in $\overline{\Omega}\times[0,T]$ and $\xcurl{\xvec{z}^0} = \xvec{0}$ in $\overline{\mathcal{E}}\times[0,T]$.
	\end{lmm}
	
	\begin{rmrk}\label{remark:rplflp}
		As shown in \cite{Coron1993}, for two-dimensional simply-connected domains one can replace \eqref{equation:flushingproperty} by a uniform flushing property. That is to say, there exists a smoothly bounded open set $\Omega_1 \subset \mathcal{E}$ such that $(\overline{\Omega}\setminus \partial \mathcal{E}) \subset \Omega_1$ and for all $\xvec{x} \in \overline{\Omega}_1$ one has $\xmcal{Z}^{0}(\xvec{x}, 0, T) \notin \overline{\Omega}_1$.
	\end{rmrk}
	
	\begin{xmpl}\label{example:cylinder}
		In order to illustrate \Cref{remark:rplflp} by means of a very specific example, and to provide more details regarding \Cref{remark:cylm2}, let us consider, as in \Cref{Figure:cylar}, for a smoothly bounded connected open set $D \subset \mathbb{R}^2$ the cylindrical setup
		\[
		\Omega \coloneqq (0,1) \times D, \quad \Gamma_{\operatorname{c}} \coloneqq \{0,1\} \times  D.
		\]	
		Let $\widetilde{D} \supseteq D$ be the planar simply-connected extension of $D$ with $\partial \widetilde{D} \subset \partial D$. One can extend $\Omega$ through $\Gamma_{\operatorname{c}}$, in the sense of \Cref{subsection:domainextensions}, to a smoothly bounded domain $\mathcal{E} \subset \mathbb{R}^3$ with
		\[
		(-1,2) \times D \subset \mathcal{E} \subset \mathbb{R} \times \widetilde{D}.
		\]
		Now, let $\chi_1 \in \xCinfty_0((-1,2);[0,1])$ with $\chi_1(s) = 1$ for $s \in [-1/2,3/2]$ and extend $\chi_1$ by zero to $\mathbb{R}$. Moreover, for some large number $M_1 > 0$, take $\gamma_1 \in \xCinfty_0((0,T);[0,1])$ such that $\gamma_1(t) = M_1$ when $t \in [T/8,7T/8]$. Then, for $\xvec{x} = [x_1,x_2,x_3]^{\top} \in \overline{\mathcal{E}}$ and $t \in [0,T]$ choose
		\begin{equation*}\label{equation:zeordcdannulussect0}
			\begin{alignedat}{6}
				& \xvec{z}^{0}(\xvec{x},t) && \coloneqq \xnab\left(\gamma_1(t)\chi_1(x_1) x_1\right),\\
				& p^{0}(\xvec{x},t) && \coloneqq - \partial_t \gamma_1(t)\chi_1(x_1) x_1,\\
				& \xsym{\xi}^{0}(\xvec{x},t) && \coloneqq  (\xvec{z}^{0}(\xvec{x},t) \cdot \xdop{\nabla}) \xvec{z}^{0}(\xvec{x},t),\\
				& \sigma^0(\xvec{x},t) && \coloneqq \xdiv{\xvec{z}^0}(\xvec{x},t).
			\end{alignedat}
		\end{equation*}
		Since $\xvec{z}^0$ does not depend on the spatial variables in $(-1/2,3/2) \times \overline {D}$, the above profiles solve the controllability problem \eqref{equation:MHD_Elsaesser_O1_} with the support of $\sigma^0$ and $\xsym{\xi}^0$ being located away from $\overline{\Omega}$. Also, for~$M_1 > 0$ large enough, one has a uniform flushing property as mentioned in \Cref{remark:rplflp}, see for instance the proof of \cite[Lemma 3.1]{RisselWang2021} which carries over to a three-dimensional pipe.
	\end{xmpl}
	\begin{figure}[ht!]
		\centering 
		\resizebox{0.75\textwidth}{!}{
			\begin{tikzpicture}
				\clip[rotate=0](-2.2,-1) rectangle (9.7,2.6);
				\draw [line width=0.4mm, color=black!200, fill=RoyalBlue!5] plot[smooth cycle] (0.2,0) rectangle ++(7,2.5);
				
				\draw [line width=0.4mm, preaction={fill=RoyalBlue!5}] (0.6,1.25) arc (0:360:0.4cm and 1.25cm);
				
				\draw [line width=0.4mm](7,2.5)--(7.2,2.5);
				\draw [line width=0.4mm](7,0)--(7.2,0);
				\draw [line width=0.4mm, preaction={fill=RoyalBlue!5}] (7.6,1.25) arc (0:360:0.4cm and 1.25cm);
				
				\draw[line width=0.4mm, -stealth, dashed] (3,0.5) -- (4.5,0.5);
				\draw[line width=0.4mm, -stealth, dashed] (3,1.25) -- (4.5,1.25);
				\draw[line width=0.4mm, -stealth, dashed] (3,2) -- (4.5,2);
				\draw[line width=0.4mm] (-0.5,0.5) -- (-0.2,0.5);
				\draw[line width=0.4mm] (-0.5,1.25) -- (-0.2,1.25);
				\draw[line width=0.4mm] (-0.5,2) -- (-0.2,2);
				\draw[line width=0.4mm, -stealth, dashed] (0.2,0.5) -- (1,0.5);
				\draw[line width=0.4mm, -stealth, dashed] (0.2,1.25) -- (1,1.25);
				\draw[line width=0.4mm, -stealth, dashed] (0.2,2) -- (1,2);
				\draw[line width=0.4mm, dashed] (6.5,0.5) -- (6.8,0.5);
				\draw[line width=0.4mm, dashed] (6.5,1.25) -- (6.8,1.25);
				\draw[line width=0.4mm, dashed] (6.5,2) -- (6.8,2);
				\draw[line width=0.4mm, -stealth] (7.2,0.5) -- (8,0.5);
				\draw[line width=0.4mm, -stealth] (7.2,1.25) -- (8,1.25);
				\draw[line width=0.4mm, -stealth] (7.2,2) -- (8,2);
				
				\draw[line width=0.4mm, -stealth] (-1.5,0.5) -- (-0.8,0.5);
				\draw[line width=0.4mm, -stealth] (-1.5,1.25) -- (-0.8,1.25);
				\draw[line width=0.4mm, -stealth] (-1.5,2) -- (-0.8,2);
				\draw[line width=0.4mm, -stealth] (-1.8,0.5) -- (-2.1,0.5);
				\draw[line width=0.4mm, -stealth] (-1.8,1.25) -- (-2.1,1.25);
				\draw[line width=0.4mm, -stealth] (-1.8,2) -- (-2.1,2);
				
				\draw[line width=0.4mm, -stealth] (8.3,0.5) -- (9,0.5);
				\draw[line width=0.4mm, -stealth] (8.3,1.25) -- (9,1.25);
				\draw[line width=0.4mm, -stealth] (8.3,2) -- (9,2);
				\draw[line width=0.4mm, -stealth] (9.6,0.5) -- (9.3,0.5);
				\draw[line width=0.4mm, -stealth] (9.6,1.25) -- (9.3,1.25);
				\draw[line width=0.4mm, -stealth] (9.6,2) -- (9.3,2);
			\end{tikzpicture}
		}
		\caption{A sketch of a cylindrical domain. The arrows indicate the profile $\xvec{z}^0$, which is constant in the original domain, hence curl-free and divergence-free, but ceases to be divergence-free at certain parts in the extended domain.}
		\label{Figure:cylar}
	\end{figure}

	In the context of \Cref{theorem:main1}, the zeroth order profiles $(\xvec{z}^{0}, p^{0}, \xsym{\xi}^{0}, \sigma^0 \coloneqq 0)$ in \eqref{equation:ansatz2}--\eqref{3.7} will be chosen by \Cref{lemma:Z^0} below.
	
	\begin{lmm}\label{lemma:Z^0}
		When $\Omega \subset \mathbb{R}^2$ is simply-connected and $\Gamma_{\operatorname{c}}$ is connected, there exists a number $\ell > 0$, and profiles $\xvec{z}^{0}, \xsym{\xi}^{0} \in \xCinfty(\overline{\mathcal{E}_T};\mathbb{R}^{2})$, and $p^{0} \in \xCinfty(\overline{\mathcal{E}_T};\mathbb{R})$ such that $(\xvec{z}^{0}, p^{0}, \xsym{\xi}^{0}, \sigma^0 = 0)$ solves \eqref{equation:MHD_Elsaesser_O1_} and it holds
		\begin{gather*}
			 \forall t \in [0,T] \colon \operatorname{dist}\left(\operatorname{supp}(\xwcurl{\xvec{z}^0}(\cdot, t)), \overline{\Omega}\right) > \ell, \quad  \forall (\xvec{x}, t) \in \overline{\mathcal{E}}\times[0,T] \colon \xdiv{\xvec{z}_0} = 0 \\
			 \forall x \in \overline{\mathcal{E}}\colon \operatorname{supp}(\xvec{z}^{0}(\xvec{x}, \cdot)) \cup \operatorname{supp}(p^{0}(\xvec{x}, \cdot)) \cup \operatorname{supp}(\xsym{\xi}^{0}(\xvec{x}, \cdot)) \subset (0,T).
		\end{gather*}
		Moreover, the force $\xsym{\xi}^{0}(\cdot,t)$ is supported in $\overline{\mathcal{E}} \setminus \overline{\Omega}$ for all $t \in [0,T]$ and the flow $\xmcal{Z}^{0}$ determined via \eqref{equation:flowofy} obeys the flushing property \eqref{equation:flushingproperty}.
	\end{lmm}
	\begin{proof}
		At first, let $(\xvec{z}^{0}, p^{0}, \xsym{\xi}^{0},\sigma^0)$ be the profiles obtained from \Cref{lemma:original_u^0}; thus, one might have $\sigma^0 = \xdiv{\xvec{z}^0} \neq 0$ at some points. However, as sketched in \Cref{Figure:extensionW}, there exist an open set $B \subset \mathbb{R}^2$ and a number $\ell > 0$ satisfying
		\[
			B \cap \partial \mathcal{E} \neq \emptyset, \quad \operatorname{dist}(\overline{B}, \overline{\Omega}) > \ell
		\]
		such that
		\[
			\operatorname{supp}(\sigma^0) \subset B \times (0,T).
		\]
		It is not restrictive to assume that $V \coloneqq \mathcal{E} \setminus \overline{B}$ is simply-connected; hence, there is a scalar potential $\phi \subset \xCinfty(\overline{V} \times [0,T]; \mathbb{R})$ with
		\[
			\xvec{z}^0(\xvec{x},t) = \xnab^{\perp} \phi(\xvec{x},t), \quad (\xvec{x},t) \in V \times [0,T].
		\]
		Indeed, if $\xvec{R}\colon \mathbb{R}^2 \longrightarrow \mathbb{R}^2$ denotes a rotation by $\pi/2$, then
		\[
			\xdiv{\xvec{z}^0} = 0 \iff \xwcurl{(\xvec{R} \xvec{z}^0)} = 0, \quad \xvec{R}\xvec{z}^0 = \xnab \phi \iff \xvec{z}^0 = \xnab^{\perp}\phi.
		\]
		Let $\widetilde{\phi}\colon\mathbb{R}^2\times[0,T] \longrightarrow \mathbb{R}$ be any extension of $\phi$ to $\mathbb{R}^2 \times [0,T]$ and take a smooth cutoff function $\psi \in \xCinfty(\mathbb{R}^2;[0,1])$ with
		\[
		\psi (\xvec{x}) \coloneqq \begin{cases}
			0 & \mbox{ if } \xvec{x} \in \overline{B},\\
			1 & \mbox{ if } \operatorname{dist}(\xvec{x}, \Omega) < \ell/2.
		\end{cases}
		\]
		Since $\xsym{\tau} = [n_2,-n_1]^{\top}$ is tangential at $\partial \mathcal{E}$, one can observe along $\Sigma_T \cap (\overline{V} \times [0,T])$ the vanishing tangential derivatives
		\[
			\xnab  \phi \cdot \xsym{\tau} = -\xnab^{\perp}\phi \cdot \xvec{n} = -\xvec{z}^0 \cdot \xvec{n} = 0.
		\]
		Therefore, there is a time-dependent constant $c(t)$ such that $\widetilde{\phi}(\cdot,t) = c(t)$ at $\partial \mathcal{E} \setminus B$ for each $t \in [0,T]$. Finally, we define
		\begin{equation*}
			\widetilde{\xvec{z}}^0(\xvec{x},t) \coloneqq
			\xnab^{\perp}\left( \psi(\xvec{x}) \left( \widetilde{\phi}(\xvec{x},t) - c(t) \right)  \right), \quad (\xvec{x}, t) \in \mathcal{E}_T,
		\end{equation*}
		which implies $\xdiv{\widetilde{\xvec{z}}^0} = 0$ in $\overline{\mathcal{E}}\times[0,T]$ and $\widetilde{\xvec{z}}^0 \cdot \xvec{n} = 0$ at $\partial\mathcal{E} \times[0,T]$. Since $\widetilde{\xvec{z}}^0$ can only differ from~$\xvec{z}^0$ if $\operatorname{dist}(\xvec{x}, \Omega) \geq \ell/2$, one has $\operatorname{supp}(\xwcurl{\widetilde{\xvec{z}}^0}) \in \overline{\mathcal{E}} \setminus \overline{\Omega}$, and the flushing property \eqref{equation:flushingproperty} remains valid for $\widetilde{\xvec{z}}^0$. The proof is then concluded by renaming $\widetilde{\xvec{z}}^0$~as~$\xvec{z}^0$, emphasizing that $\xdiv{\widetilde{\xvec{z}}^0} = 0$ in $\mathcal{E}\times(0,T)$ and renaming $\xdiv{\widetilde{\xvec{z}}^0}$ as $\sigma^0$, and by modifying~$\xsym{\xi}^0$ inside $\overline{\mathcal{E}}\setminus\overline{\Omega}$ such that \eqref{equation:MHD_Elsaesser_O1_} holds.
	\end{proof}
	
	\begin{figure}[ht!]
		\centering
		\resizebox{0.65\textwidth}{!}{
			\begin{tikzpicture}
				\clip(-2.4,-1.8) rectangle (7.2,4.5);
				
				\draw [dashed, line width=1.3mm, color=DarkBlue, fill=white, fill opacity=0] plot[smooth, tension=1] coordinates {(-1.8,0.8) (-1.5,3) (0,3.8) (0.48,1.8) (2,3.8) (6.4,2.2) (5,0.8) (5.8,-0.2) (3.2,-1) (2,0) (-0.4,-0.9) (-1.8,0.8)};
				\draw [line width=1mm, color=white, fill=white]  plot (-0.6,3.5) circle(25pt);
				
				\draw [line width=0mm, color=white, fill=RoyalBlue!5] plot[smooth, tension=1] coordinates {(-1.8,0.8) (-1.5,3) (0,3.8) (0.5,1.8) (2,3.8) (6.4,2.2) (5,0.8) (5.8,-0.2) (3.2,-1) (2,0) (-0.4,-0.9) (-1.8,0.8)};

				\draw [line width=0mm, color=RoyalBlue!5, fill=RoyalBlue!5] plot[smooth, tension=1] coordinates {(1,0.5) (0.65,1.9) (2.2,3.8) (6.4,2.18) (5,0.8) (5.8,-0.2) (3.2,-1) (2.35,0) (1.5,0.2) (1,0.5)};

				\draw [line width=0.38mm, color=black, fill=white, fill opacity=0] plot[smooth, tension=1] coordinates {(-1.8,0.8) (-1.5,3) (0,3.8) (0.48,1.8) (2,3.8) (6.4,2.2) (5,0.8) (5.8,-0.2) (3.2,-1) (2,0) (-0.4,-0.9) (-1.8,0.8)};

				\draw [dashed,line width=0.1mm, color=black]  plot[smooth, tension=0.8] coordinates {(2.58,-0.5)  (2.34,0.035) (1.52,0.2) (0.98,0.5) (0.55,1.3) (0.72,2) (0.8,2.55)};

				\draw [line width=0.8mm, color=IndianRed]  plot (-0.6,3.5) circle(25pt);
				
				\draw [line width=0.8mm, color=ForestGreen] plot[smooth, tension=1] coordinates {(1.4,4.4) (0.4,1.2) (2.2,-0.5)  (2.75,-1.3)};
				\coordinate[label=left:{\color{ForestGreen}$\psi = 1$}] (A) at (3,4.2);
				
				\coordinate[label=left:{\color{IndianRed}$B$}] (A) at (-1.4,3.8);
				
				\coordinate[label=left:{$\Omega$}] (A) at (3.5,2);
				\coordinate[label=left:{$\Omega^1$}] (A) at (-0.2,1);

				\draw [line width=0.55mm, color=black, fill=white, fill opacity=0] plot[smooth, tension=1] coordinates {(-1.6,-1.45) (-0.27,-1.45)};
				\draw [dashed, line width=1mm, color=DarkBlue, fill=white, fill opacity=0] plot[smooth, tension=1] coordinates {(-1.6,-1.5) (-0.27,-1.5)};
				\coordinate[label=left:\color{DarkBlue}{$\widetilde{\phi}=c(t)$}] (A) at (1.5,-1.5);
				
				\draw [line width=0.55mm, color=black, fill=white, fill opacity=0] plot[smooth, tension=1] coordinates {(2.6,-1.45) (3.93,-1.45)};
				\coordinate[label=left:\color{black}{$\widetilde{\phi}=\mbox{arbitrary}$}] (A) at (6.5,-1.5);				
			\end{tikzpicture}
		}
		\caption{A sketch of the situation considered in the proof of \Cref{lemma:Z^0}. The (red) circle indicates the set $B$. To the right of the thick (green) line which crosses the domain, one has $\psi = 1$. The thin dashed line stands for $\Gamma_{\operatorname{c}}$. Along the boundary part $\partial(\Omega^1 \cup \Omega)\setminus B$, which is highlighted by thick dashes, it holds $\widetilde{\phi}(\cdot,t) = \phi(\cdot,t) = c(t)$.}
		\label{Figure:extensionW}
	\end{figure}
	
	\begin{rmrk}\label{remark:chooseregionfl}
		In the proof of \Cref{lemma:Z^0}, one can explicitly choose the size and location of $B$, as long as $B$ is open and $B \cap \partial \mathcal{E} \neq \emptyset$. 
	\end{rmrk}
	
	\begin{rmrk}\label{remark:rca}
		When $\Gamma_{\operatorname{c}}$ is not connected, the proof of \Cref{lemma:Z^0} can still be applied to situation where~$\mathcal{E}$ is simply-connected; for instance, by selecting $\xvec{z}^0$ such that all its integral curves cross the same connected component of $\overline{\mathcal{E}}\setminus \Omega$. To avoid creating a gradient term~$\xnab q$ during the regularization stage described in \Cref{section:conclth}, the initial data should then obey relations of the type \eqref{equation:gencompb}.
	\end{rmrk}
	
	\paragraph{The special case of an annulus.}
	Concerning \Cref{theorem:annulus}, where~$\mathcal{E}$ is an annulus (\cf~\Cref{Figure:annulusec}), we introduce an explicit return method trajectory $\xvec{z}^0$, which is curl-free, divergence-free, and tangential at $\partial \mathcal{E}$. This is possible because annuli are doubly-connected. More precisely, we define 
	\[
	\widetilde{\varphi} \colon \overline{\mathcal{E}} \longrightarrow \mathbb{R}_+, \quad \xvec{x} \longmapsto \widetilde{\varphi}(\xvec{x}) \coloneqq \ln|\xvec{x}|
	\]
	and choose for a constant $M_{\mathcal{E}} > 0$ a smooth function $\gamma_M \in \xCinfty_0((0,1);\mathbb{R}_+)$, satisfying 
	\[
	\gamma_{M_{\mathcal{E}}}(t) \geq M_{\mathcal{E}}
	\]
	for all $t \in (T/8, 7T/8)$. Then, for $(\xvec{x},t) \in \overline{\mathcal{E}_T}$, we denote the vector field
	\[
	\xvec{y}^*(\xvec{x}, t) \coloneqq [y^*_1, y^*_2]^{\top}(\xvec{x}, t) \coloneqq \gamma_{M_{\mathcal{E}}}(t) \begin{bmatrix}
		-\partial_2 \widetilde{\varphi} (\xvec{x}) \\ \partial_1 \widetilde{\varphi} (\xvec{x})
	\end{bmatrix},
	\]
	which possesses in $\overline{\mathcal{E}_T}$ the properties
	\[
	\xdiv{\xvec{y}^*} = 0,  \quad \partial_2 y^*_1 - \partial_1 y^*_2 = 0, \quad \xvec{y}^* \cdot \xvec{n} = 0.
	\]
	Due to the symmetry of $\mathcal{E}$, the extended unit normal field $\xvec{n}$ can be chosen everywhere orthogonal to $\xvec{y}^*$.
	Now, for $M_{\mathcal{E}} > 0$ sufficiently large, we define
	\begin{equation}\label{equation:zeordcdannulussect}
		\begin{alignedat}{6}
			& \xvec{z}^{0} && \coloneqq \xvec{y}^* && \mbox{ in } \overline{\mathcal{E}_T},\\
			& p^{0} && \coloneqq - \partial_t \psi^* - \frac{1}{2} | \xvec{y}^* |^2 && \mbox{ in } \overline{\mathcal{E}_T},\\
			& \xsym{\xi}^{0} && \coloneqq  \partial_t \xvec{y}^* + (\xvec{y}^* \cdot \xdop{\nabla}) \xvec{y}^* + \xdop{\nabla}p^{0} \quad && \mbox{ in } \overline{\mathcal{E}_T},
		\end{alignedat}
	\end{equation}
	where $\psi^* \in \xCinfty(\overline{\mathcal{E}_T};\mathbb{R})$ satisfies~$\xvec{y}^*(\xvec{x},\cdot) = \xnab \psi^*(\xvec{x},\cdot)$ whenever~$\operatorname{dist}(\xvec{x}, \overline{\Omega}) < \ell$, for some $\ell > 0$ independent of~$\xvec{x}$. In particular, assuming that $M_{\mathcal{E}} > 0$ is fixed sufficiently large, a flushing property of the type \eqref{equation:flushingproperty} holds. Indeed, the profile~$\xvec{z}^0$ never vanishes in~$\overline{\mathcal{E}}\times(T/8, 7T/8)$ and the associated flow~$\xmcal{Z}^0$ propagates information along circular trajectories around the annulus.
	
	\begin{figure}[ht!]
		\centering
		\resizebox{0.35\textwidth}{!}{
			\begin{tikzpicture}
				\clip(-3.1,-3.1) rectangle (3.1,3.1);
				\draw[line width=0.5mm, color=black] (80:1.5) coordinate (A) arc (80:2:1.5) coordinate (C)  (2:2.9) coordinate (D) arc (2:80:2.9) coordinate (B);
				\draw[line width=0.5mm, color=black] (80:1.5) arc (80:365:1.5) (2:2.9) arc (2:365:2.9);

				\draw[dashed,line width=0.5mm, color=black] (A) -- (B);
				\draw[dashed,line width=0.5mm, color=black] (C) -- (D);

				\draw[line width=0.8mm, color=DarkBlue, ->] (0,2.2) arc (90:380:2.2);

				\coordinate[label=left:{\small \color{DarkBlue}$\xsym{\mathcal{Z}}^0$}] (BB) at (2.77,0.95);
			\end{tikzpicture}	
		}
		\caption{An annulus divided into a physical sector and a control sector. The (blue) arrow indicates the flow $\xsym{\mathcal{Z}}^0$.}
		\label{Figure:annulusec}
	\end{figure}
	
	Finally, we take $\rho \in \mathbb{R}$ and set $\xvec{M}_2 = \rho \xvec{I} \in \xCinfty(\overline{\mathcal{E}};\mathbb{R}^{2\times 2})$, while choosing general friction coefficient matrices $\xvec{M}_1, \xvec{L}_1, \xvec{L}_2 \in \xCinfty(\overline{\mathcal{E}};\mathbb{R}^{2\times 2})$. Then, $\xvec{z}^0$ satisfies the relations
	\begin{equation}\label{equation:condmuwkez_0}
		\begin{aligned}
			&\xdiv{(\xsym{\mathcal{N}}^+(\xvec{z}^0, \xvec{z}^0) - \xsym{\mathcal{N}}^-(\xvec{z}^0,\xvec{z}^0))} = 0 && \mbox{ in } \overline{\mathcal{E}_T}, \\
			& \xvec{z}^0 \cdot \xvec{n} = 0 && \mbox{ in } \overline{\mathcal{E}_T},\\
			&\xdiv{\xvec{z}^0} = 0 && \mbox{ in } \overline{\mathcal{E}_T},\\
			&\xwcurl{\xvec{z}^0} = 0 && \mbox{ in } \overline{\mathcal{E}_T}
		\end{aligned}
	\end{equation}
	because
	\[
	\xdiv{(\xsym{\mathcal{N}}^+(\xvec{z}^0, \xvec{z}^0) - \xsym{\mathcal{N}}^-(\xvec{z}^0,\xvec{z}^0))} = 2\xdiv{[\xvec{M}_2 \xvec{z}^0]_{\operatorname{tan}}} = 2\rho \xdiv{[ \xvec{z}^0]_{\operatorname{tan}}} = 2\rho \xdiv{\xvec{z}^0} = 0.
	\]

	\subsection{Flushing the initial data}\label{subsection:flushing}
	Due to the scaling in \eqref{equation:scaling} and \eqref{equation:scaling2}, the contributions of $\xvec{z}^{\pm}_0$ to~$\xvec{z}^{\pm,\epsilon}$ are at $O(\epsilon)$. In order to avoid that $\xvec{z}^{\pm}_0$ impact the remainder estimates in \Cref{subsection:remestwbctrl} below, the goal is to cancel their influence for $t \geq T$ by using the controls $\xsym{\xi}^{\pm,1}$, which are supported only in $\overline{\mathcal{E}}\setminus\overline{\Omega}$. After inserting \eqref{equation:ansatz2} and \eqref{3.7} into \eqref{equation:MHD_ElsaesserExtScaledLongTime}, motivated by \cite{CoronMarbachSueur2020}, one observes that a good strategy consists of defining $\xvec{z}^{\pm,1}$ as the solutions to the linear problem
	\begin{equation}\label{equation:MHD_ElsaesserExt_Ovareps}
		\begin{cases}
			\partial_t \xvec{z}^{\pm,1} + (\xvec{z}^{\mp,1} \cdot \xdop{\nabla}) \xvec{z}^{0} + (\xvec{z}^{0} \cdot \xdop{\nabla}) \xvec{z}^{\pm,1} + \xdop{\nabla} p^{\pm,1} =  \xsym{\xi}^{\pm,1} + (\lambda^{\pm}+\lambda^{\mp})\Delta \xvec{z}^{0} & \mbox{ in } \mathcal{E}_T,\\
			\xdop{\nabla}\cdot\xvec{z}^{\pm,1} = 0 & \mbox{ in } \mathcal{E}_T,\\
			\xvec{z}^{\pm,1} \cdot \xvec{n} = 0 & \mbox{ on }  \Sigma_T,\\
			\xvec{z}^{\pm,1}(\cdot, 0) = \xvec{z}_0^{\pm} & \mbox{ in } \mathcal{E}.
		\end{cases}
	\end{equation}
	We shall determine the controls $\xsym{\xi}^{\pm,1} \in \xCzero([0,T];\xHtwo(\mathcal{E};\mathbb{R}^N))$ such that the corresponding solution $(\xvec{z}^{+,1}, \xvec{z}^{-,1})$ to \eqref{equation:MHD_ElsaesserExt_Ovareps} satisfies $\xvec{z}^{\pm,1}(\cdot, T) = \xvec{0}$. This is achieved by combining \cite[Lemma 3]{CoronMarbachSueur2020} with new ideas for the cases of Theorems~\Rref{theorem:main1} and \Rref{theorem:annulus}, where we have to maintain the properties 
	\begin{equation}\label{equation:reqctr}
		\begin{aligned}
			\xdiv{(\xvec{\xi}^{+,1} - \xvec{\xi}^{-,1})} = 0  \mbox{ in } \mathcal{E}_T, \quad
			(\xvec{\xi}^{+,1} - \xvec{\xi}^{-,1}) \cdot \xvec{n} = 0  \mbox{ on } \Sigma_T.
		\end{aligned}
	\end{equation}
	\begin{lmm}\label{lemma:flushing}
		There are $\xvec{\xi}^{\pm,1} \in \xCzero([0,T];\xHtwo(\mathcal{E};\mathbb{R}^N))$ 
		such that the solution $(\xvec{z}^{+,1},\xvec{z}^{-,1})$ to \eqref{equation:MHD_ElsaesserExt_Ovareps} obeys $\xvec{z}^{\pm,1}(\xvec{x}, T) = \xvec{0}$ for all $\xvec{x} \in \mathcal{E}$ and is bounded in $\xLn{\infty}((0,T);\xHn{3}(\mathcal{E})^2)$. Moreover, for all $t \in(0,T)$ it holds
		\[
			\operatorname{supp}(\xsym{\xi}^{\pm,1}(\cdot,t)) \subset \overline{\mathcal{E}}\setminus\overline{\Omega}.
		\]
		Given the assumptions of \Cref{theorem:main1}, one can choose the controls $\xvec{\xi}^{\pm,1}$ with \eqref{equation:reqctr}.
	\end{lmm}
	\begin{proof}
		When $\xvec{z}^0$ is determined via \Cref{lemma:original_u^0}, the proof is a direct application of the arguments from \cite[Lemma 3]{CoronMarbachSueur2020} to the uncoupled systems solved by $\xvec{z}^{+,1} \pm \xvec{z}^{-,1}$. Hence, we consider here only the two-dimensional situations of Theorems~\Rref{theorem:main1} and \Rref{theorem:annulus}, where~$\xvec{z}^0$ is obtained either via~\Cref{lemma:Z^0} or by \eqref{equation:zeordcdannulussect}. Our strategy is close that from \cite[Lemma 3]{CoronMarbachSueur2020}, but compared to \cite{CoronMarbachSueur2020} there are two new challenges:
		\begin{itemize}
			\item constructing the controls such that the relations in \eqref{equation:reqctr} are satisfied;
			\item for $\xvec{z}^0$ taken via \Cref{lemma:Z^0}, one might have $\xwcurl{\xvec{z}^0} \neq 0$ in $\overline{\mathcal{E}}\setminus\overline{\Omega}$, which obstructs a direct application of the arguments from \cite{CoronMarbachSueur2020}.
		\end{itemize}

		\paragraph{Step 1. Preliminaries.} Due to \Cref{lemma:Z^0} or \eqref{equation:zeordcdannulussect} one has $\xdiv{\xvec{z}^{0}} = \xvec{0}$ in $\overline{\mathcal{E}}\times[0,T]$. Thus, the smooth vector field
		\[
			(\lambda^{+} + \lambda^{-}) \Delta \xvec{z}^0 = - (\lambda^{+} + \lambda^{-})(\xnab^{\perp}{(\xwcurl{\xvec{z}^0})})
		\]
		is spatially supported in $\overline{\mathcal{E}}\setminus\overline{\Omega}$ and can be absorbed by the control terms
		\[
			\xvec{f}^{\pm} \coloneqq \xsym{\xi}^{\pm,1} + (\lambda^{\pm}+\lambda^{\mp})\Delta \xvec{z}^{0}.
		\] 
		In order to construct suitable functions $\xvec{f}^{\pm}$, we shall denote vector fields proportional to the original MHD unknowns; namely,
		\begin{alignat}{2}
			\xvec{E}^{\pm} \coloneqq \xvec{z}^{+,1} \pm \xvec{z}^{-,1}, & \quad & \xvec{F}^{\pm} \coloneqq \xvec{f}^+ \pm \xvec{f}^-.\label{equation:deffmf}
		\end{alignat}
		
		\paragraph{Step 2. A partition of unity.}
		Due to the regularity of $\xsym{\mathcal{Z}}^0$ and the flushing property \eqref{equation:flushingproperty}, as provided by \Cref{lemma:original_u^0}, \Cref{lemma:Z^0}, or by the definition for $\xvec{z}^0$ in \eqref{equation:zeordcdannulussect}, there exists a small number $a > 0$ such that
		\[
		\forall \xvec{x} \in \overline{\mathcal{E}}, \exists t_{\xvec{x}} \in (0,T) 	\colon \operatorname{dist}(\xsym{\mathcal{Z}}^0(\xvec{x}, 0, t_{\xvec{x}}), \overline{\Omega}) \geq a. 	
		\]
		Hence, one can select a smoothly bounded closed set $\mathcal{S} \subset \overline{\mathcal{E}}$ with $\mathcal{S} \cap \overline{\Omega} = \emptyset$ and
		\[
		\forall \xvec{x} \in \overline{\mathcal{E}}, \exists t_{\xvec{x}} \in (0,T) 	\colon \xsym{\mathcal{Z}}^0(\xvec{x}, 0, t_{\xvec{x}}) \in \mathcal{S}. 
		\]
		Moreover, for some $L \in \mathbb{N}$, we fix a finite covering $c_1, \dots, c_L$ of~$\mathcal{S}$ which consists of interior and boundary squares. The boundary squares are centered in points of $\partial\mathcal{E} \cap \mathcal{S}$, fully included inside $\mathbb{R}^2\setminus\overline{\Omega}$, and one side lies in the interior of $\mathcal{E}$. The interior squares are centered in points of $\mathcal{S} \setminus \partial\mathcal{E}$ and belong to $\mathcal{E} \setminus \overline{\Omega}$.
		Consequently, there exists $b > 0$ and a number $M \in \mathbb{N}$ of balls $B_1, \dots, B_M \subset \mathbb{R}^2$ which cover $\overline{\mathcal{E}}$ such that for each index~$l \in \{1,\dots,M\}$ one has
		\begin{gather}\label{equation:cubesquareflush}
			\exists t_l \in (b, T-b), \exists r_l \in \{1,\dots,L\}, \forall t \in (t_l-b, t_l+b) \colon \xsym{\mathcal{Z}}^0(B_l, 0, t) \in c_{r_l}.
		\end{gather}
		With respect to the balls $B_1, \dots, B_M$, let $(\mu_l)_{l = 1,\dots,M} \subset \xCinfty_0(\mathbb{R}^2;\mathbb{R}^2)$ be any fixed smooth partition of unity in the sense that
		\begin{equation}\label{3.13}
			\forall l \in \{1,\dots,M\}\colon \operatorname{supp}(\mu_l) \subset B_l, \quad \forall \xvec{x} \in \overline{\mathcal{E}}\colon \sum_{l=1}^M \mu_l(\xvec{x}) = 1.
		\end{equation}
		
		\begin{rmrk}
			When $\xvec{z}^0$ is obtained via \Cref{lemma:Z^0}, or as defined in \eqref{equation:zeordcdannulussect}, one can use a simplified partition of unity and only few squares. The notations used here align with the argument from \cite[Lemma 3]{CoronMarbachSueur2020} to which we refer when $\xvec{z}^0$  is the vector field from \Cref{lemma:original_u^0}.
		\end{rmrk}

		\paragraph{Step 3. Flushing the initial magnetic field.}
		Since $\xdiv{\xvec{z}^0} = 0$ holds in $\overline{\mathcal{E}}\times[0,T]$ for the presently case, the initial magnetic field can be flushed without pressure term. To this end, we rely on the existence of a stream function $\widetilde{\psi}_0$ with
		\[
			\xvec{z}_0^{+} - \xvec{z}_0^{-} = \xnab^{\perp} \widetilde{\psi}_0 \mbox{ in } \mathcal{E}_T, \quad \widetilde{\psi}_0 = 0 \, \mbox{ on } \Sigma_T.
		\]
		Indeed, under the assumptions of \Cref{theorem:main1} this follows from $\mathcal{E}$ being simply-connected, while for \Cref{theorem:annulus} it is part of the hypotheses. Then, by the decomposition in \eqref{equation:deffmf}, the vector field $\xvec{E}^-$ satisfies the linear problem
		\begin{equation}\label{equation:MHD_ElsaesserExt_OvarepsBtOwp}
			\begin{cases}
				\partial_t \xvec{E}^{-} + (\xvec{z}^{0} \cdot \xdop{\nabla}) \xvec{E}^{-} - (\xvec{E}^{-} \cdot \xdop{\nabla}) \xvec{z}^{0} + \nabla q^1 = \xvec{F}^- = \xvec{\xi}^{+,1} - \xvec{\xi}^{-,1} & \mbox{ in } \mathcal{E}_T,\\
				\xdop{\nabla}\cdot\xvec{E}^{-} = 0  & \mbox{ in } \mathcal{E}_T,\\
				\xvec{E}^{-} \cdot \xvec{n} = 0  & \mbox{ on }  \Sigma_T,\\
				\xvec{E}^{-}(\cdot, 0) = \xvec{z}_0^{+} - \xvec{z}_0^{-}  & \mbox{ in } \mathcal{E},
			\end{cases}
		\end{equation}
		with $q^1 \coloneqq p^{+,1} - p^{-,1}$.
		By employing the partition of unity $(\mu_l)_{l\in\{1,\dots,M\}}$ given in \eqref{3.13}, we first solve for $l\in\{1,\dots,M\}$ the homogeneous problems
		\begin{equation}\label{equation:MHD_ElsaesserExt_OvarepsBtOwppp}
			\begin{cases}
				\partial_t \xvec{E}^{-}_l + (\xvec{z}^{0} \cdot \xdop{\nabla}) \xvec{E}^{-}_l - (\xvec{E}^{-}_l \cdot \xdop{\nabla}) \xvec{z}^{0} + \xnab q^{1}_l = \xvec{0} & \mbox{ in } \mathcal{E}_T,\\
				\xdop{\nabla}\cdot\xvec{E}^{-}_l = 0 & \mbox{ in } \mathcal{E}_T,\\
				\xvec{E}^{-}_l \cdot \xvec{n} = 0 & \mbox{ on }  \Sigma_T,\\
				\xvec{E}^{-}_l(\cdot, 0) = \xnab^{\perp}(\mu_l \widetilde{\psi}_0) & \mbox{ in } \mathcal{E}.
			\end{cases}
		\end{equation}
		Concerning the pressure gradient, after multiplying in \eqref{equation:MHD_ElsaesserExt_OvarepsBtOwppp} with $\xnab q^{1}_l$ and integrating by parts, one finds
		\[
			\|\xnab q^{1}_l\|_{\xLtwo(\mathcal{E})}^2 = \int_{\partial\mathcal{E}}  (\xvec{z}^0 \wedge \xvec{E}^{-}_l) \xnab q^{1}_l \wedge \xvec{n} \, \xdx{S} = 0.
		\]
		Finally, we take $\xvec{F}^- = \xvec{\xi}^{+,1} - \xvec{\xi}^{-,1} \in \xCzero([0,T];\xHtwo(\mathcal{E};\mathbb{R}^N))$ supported in $\overline{\mathcal{E}}\setminus\overline{\Omega}$, satisfying \eqref{equation:reqctr}, and such that the corresponding solution $\xvec{E}^-$ to \eqref{equation:MHD_ElsaesserExt_OvarepsBtOwp} obeys $\xvec{E}^-(\cdot,T) = \xvec{0}$. For instance, one can choose
		\begin{gather}
			\xvec{E}^{-}(\xvec{x},t) \coloneqq \sum\limits_{l=1}^M \beta(t-t_l) \xvec{E}^{-}_l(\xvec{x},t), \quad \xvec{F}^{-}(\xvec{x},t) \coloneqq \sum\limits_{l=1}^M \xdrv{\beta}{t}(t-t_l) \xvec{E}^{-}_l(\xvec{x},t),\label{equation:solftcEm}
		\end{gather}
		where $\beta\colon\mathbb{R}\longrightarrow[0,1]$ is a smooth cut-off with
		\begin{equation}\label{3.20}
			\beta(t) = \begin{cases}
				1 & \mbox{ if } t \in (-\infty,-b),\\
				0 & \mbox{ if } t \in (b, +\infty)
			\end{cases}
		\end{equation}
		for $b > 0$ from \eqref{equation:cubesquareflush}.

		\paragraph{Step 4. Flushing the initial velocity: the idea.}
		The vector field $\xvec{E}^+$ obeys with $p^1 \coloneqq p^{+,1} + p^{-,1}$ the linear problem
		\begin{equation}\label{equation:MHD_ElsaesserExt_OvarepsBtOwplus}
			\begin{cases}
				\partial_t \xvec{E}^{+} + (\xvec{z}^{0} \cdot \xdop{\nabla}) \xvec{E}^{+} + (\xvec{E}^{+} \cdot \xdop{\nabla}) \xvec{z}^{0} + \nabla p^{1} = \xvec{F}^+ & \mbox{ in } \mathcal{E}_T,\\
				\xdop{\nabla}\cdot\xvec{E}^{+} = 0  & \mbox{ in } \mathcal{E}_T,\\
				\xvec{E}^{+} \cdot \xvec{n} = 0 & \mbox{ on }  \Sigma_T,\\
				\xvec{E}^{+}(\cdot, 0) = \xvec{z}_0^{+} + \xvec{z}_0^{-} & \mbox{ in } \mathcal{E}.
			\end{cases}
		\end{equation}
		The pressure gradient is eliminated by taking the curl in \eqref{equation:MHD_ElsaesserExt_OvarepsBtOwplus}, leading to a transport equation for $\xwcurl{\xvec{E}^+}$ with non-local terms. The goal is to determine $\xvec{F}^{+}$, spatially supported in~$\overline{\mathcal{E}}\setminus\overline{\Omega}$,
		such that the corresponding solution $\xvec{E}^{+}$ to \eqref{equation:MHD_ElsaesserExt_OvarepsBtOwplus} satisfies
		\begin{equation}\label{equation:finaltimegoalellip}
			\begin{cases}
				\xwcurl{\xvec{E}^{+}}(\cdot,T) = \xvec{0} & \mbox{ in } \mathcal{E},\\
				\xdiv{\xvec{E}^{+}}(\cdot,T) = \xvec{0} & \mbox{ in } \mathcal{E},\\
				\xvec{E}^{+}(\cdot,T) \cdot \xvec{n} = \xvec{0} & \mbox{ on } \partial\mathcal{E}.
			\end{cases}
		\end{equation}
		Regarding \Cref{theorem:main1}, where $\mathcal{E}$ is simply-connected, \eqref{equation:finaltimegoalellip} implies $\xvec{E}^+(\cdot,T) = \xvec{0}$ in~$\mathcal{E}$ (\cf~\eqref{equation:sKem}). Concerning \Cref{theorem:annulus}, since $\mathcal{E}$ might in that case be multiply-connected, one can from \eqref{equation:finaltimegoalellip} only conclude that $\xvec{E}^+(\cdot,T) = \lambda_1 \xvec{Q}$, where $\lambda_1 \in \mathbb{R}$ and~$\xvec{Q}$ spans the one-dimensional space of divergence-free, curl-free, and tangential vector fields on the annulus~$\mathcal{E}$. In fact, one can take $\xvec{Q} = \xnab^{\perp} \ln |\xvec{x}|$. Therefore, we need to be able to steer any state of the form~$\lambda_1 \xvec{Q}$ to zero. To this end, note that $\xvec{z}^0 = \gamma_{M_{\mathcal{E}}} \xvec{Q}$. Moreover, since~$\Omega$ is simply-connected, for $\widetilde{\gamma} \in \xCinfty([0,T];[0,1])$ with $\widetilde{\gamma}(t) = 1$ for $t \in [0,T/8]$ and $\widetilde{\gamma}(t) = 0$ when $t \geq T/4$, there exists $\widetilde{\psi} \in \xCinfty(\overline{\Omega}\times[0,T];\mathbb{R})$, with $\widetilde{\psi}(\cdot, t) = 0$ for all $t \in [0,T/8)$, such that  
		\[
		\forall (\xvec{x},t) \in \overline{\Omega} \times[0,T] \colon \lambda_1 \xdrv{\widetilde{\gamma}}{t}(t)  \xvec{Q}(\xvec{x}) = \xnab \widetilde{\psi}(\xvec{x},t).
		\]
		Thus, the function $\xvec{A} \coloneqq \widetilde{\gamma}(t) \lambda_1 \xvec{Q}$ solves together with a smooth pressure $\widetilde{p}$ and a smooth control $\widetilde{\xvec{F}}$, which is spatially supported in $\overline{\mathcal{E}}\setminus\overline{\Omega}$, the controllability problem
		\begin{equation}\label{equation:MHD_ElsaesserExt_OvarepsBtOwplusQ}
			\begin{cases}
				\partial_t \xvec{A} + (\xvec{z}^{0} \cdot \xdop{\nabla}) \xvec{A} + (\xvec{A} \cdot \xdop{\nabla}) \xvec{z}^{0} + \nabla \widetilde{p} = \widetilde{\xvec{F}} & \mbox{ in } \mathcal{E}_T,\\
				\xdop{\nabla}\cdot\xvec{A} = 0  & \mbox{ in } \mathcal{E}_T,\\
				\xvec{A} \cdot \xvec{n} = 0 & \mbox{ on }  \Sigma_T,\\
				\xvec{A}(\cdot, 0) = \lambda_1 \xvec{Q} & \mbox{ in } \mathcal{E},\\
				\xvec{A}(\cdot, T) = \xvec{0} & \mbox{ in } \mathcal{E}.
			\end{cases}
		\end{equation}
		Indeed, one can take $\widetilde{p} = -\widetilde{\psi} - \gamma_{M_{\mathcal{E}}}\widetilde{\gamma}\lambda_1 |\xvec{Q}|^2$ and $\widetilde{\xvec{F}} = \xvec{0}$ in $\overline{\Omega}$. In $\overline{\mathcal{E}}\setminus\overline{\Omega}$, one may choose any smooth extension of $\widetilde{p}$ and fix $\widetilde{\xvec{F}} = \partial_t \xvec{A} + (\xvec{z}^{0} \cdot \xdop{\nabla}) \xvec{A} + (\xvec{A} \cdot \xdop{\nabla}) \xvec{z}^{0} + \nabla \widetilde{p}$. Summarized, first one employs $(\xvec{F}^+,\xvec{F}^-)$ for steering $(\xvec{E}^+,\xvec{E}^-)$ to $(\lambda_1\xvec{Q},\xvec{0})$, followed by utilizing the controls $(\widetilde{\xvec{F}},\xvec{0})$ to connect $(\lambda_1\xvec{Q},\xvec{0})$ with $(\xvec{0}, \xvec{0})$.
		
		\paragraph{Step 5. Flushing the initial velocity: showing \eqref{equation:finaltimegoalellip}.}
		To determine $\xvec{F}^{+}$ such that \eqref{equation:finaltimegoalellip} holds, we rewrite \eqref{equation:MHD_ElsaesserExt_OvarepsBtOwplus} in vorticity form, which describes $\omega^{+} \coloneqq \xwcurl{\xvec{E}^{+}}$. That is, given the partition of unity $(\mu_l)_{l\in\{1,\dots,M\}}$ from \eqref{3.13}, we make an ansatz of the form
		\begin{equation}\label{equation:flushansatz}
			\omega^{+} = \sum_{l=1}^M \omega^{+}_l, \quad \xvec{E}^{+} = \sum_{l=1}^M \xvec{E}^{+}_l, \quad \xvec{F}^{+} = \sum_{l=1}^M \xvec{F}^{+}_l,
		\end{equation}
		where each triple $(\omega^{+}_l, \xvec{E}^+_l, \xvec{F}^+_l)$ is sought to satisfy
		\begin{equation}\label{equation:MHD_ElsaesserExt_OvarepsBtOsole}
			\begin{cases}
				\partial_t \omega^{+}_l + (\xvec{z}^{0} \cdot \xdop{\nabla}) \omega^{+}_l = \xwcurl{\xvec{F}^{+}_l} - ({\xvec{E}^{+}_l} \cdot \xdop{\nabla}) ({\xwcurl{\xvec{z}^{0}}}) & \mbox{ in } \mathcal{E}_T,\\
				\xwcurl{{\xvec{E}^{+}_l}} = \omega^{+}_l, & \mbox{ in } \mathcal{E}_T,\\
				\xdop{\nabla}\cdot{\xvec{E}^{+}_l} = 0 & \mbox{ in } \mathcal{E}_T,\\
				{\xvec{E}^{+}_l} \cdot \xvec{n} = 0 & \mbox{ on }  \Sigma_T,\\
				\omega^{+}_l(\cdot, 0) = \xwcurl{(\mu_l(\xvec{z}_0^{+}+\xvec{z}_0^{-}))} & \mbox{ in } \mathcal{E}.
			\end{cases}
		\end{equation}
		Since $\xwcurl{\xvec{z}^0}$ is supported in $(\overline{\mathcal{E}} \setminus \overline{\Omega})\times(0,T)$, the transport problem decouples in $\Omega$ from the div-curl system. 
		To see this, let $\overline{\omega}^{+}_l$ for each $l \in \{1,\dots,M\}$ be the solution to
		\begin{equation}\label{equation:MHD_ElsaesserExt_OvarepsBtOsoleNorighthandside}
			\begin{cases}
				\partial_t \overline{\omega}^{+}_l + (\xvec{z}^{0} \cdot \xdop{\nabla}) \overline{\omega}^{+}_l = 0 & \mbox{ in } \mathcal{E}_T,\\
				\overline{\omega}^{+}_l(\cdot, 0) = \xwcurl{(\mu_l(\xvec{z}_0^{+}+\xvec{z}_0^{-}))} & \mbox{ in } \mathcal{E}.
			\end{cases}
		\end{equation}
		Then, we take $\widetilde{\omega}^{+}_l \coloneqq \beta(t-t_l) \overline{\omega}^{+}_l$ and define $(\widetilde{\xvec{E}^{+}_l})_{l\in\{1,\dots,M\}}$ via
		\begin{equation}\label{equation:flushingproofelliptic2}
			\begin{cases}
				\xwcurl{\widetilde{\xvec{E}^{+}_l}} = \widetilde{\omega}^{+}_l & \mbox{ in } 	\mathcal{E}_T,\\
				\xdop{\nabla}\cdot\widetilde{\xvec{E}^{+}_l} = 0 & \mbox{ in } 	\mathcal{E}_T,\\
				\widetilde{\xvec{E}^{+}_l} \cdot \xvec{n} = 0 & \mbox{ on }  \Sigma_T,
			\end{cases}
		\end{equation}
		where $\beta$ is the function from \eqref{3.20}. 
		Owing to \eqref{equation:cubesquareflush} and \eqref{3.20}, one finds for all $l \in \{1,\dots,M\}$ and $t \in [0,T]$ the relations
		\[
			\widetilde{\omega}^{+}_l(\cdot,T) = \xvec{0}, \quad  \operatorname{supp}\left(\xdrv{\beta}{t}(t-t_l)\overline{\omega}^{+}_l(\cdot, t)\right) \subset c_{r_l}.
		\]
		Therefore, if it would be possible to choose $\widetilde{\xvec{F}}^{+}_l$ for each $l \in \{1,\dots,M\}$ such that
		\begin{equation}\label{equation:goalconstrF}
			\xwcurl{\widetilde{\xvec{F}}^{+}_l} = 
			\xdrv{\beta}{t}(t-t_l) \overline{\omega}^{+}_l + (\widetilde{\xvec{E}^{+}_l} \cdot \xdop{\nabla}) ({\xwcurl{\xvec{z}^{0}}}),
		\end{equation}
		then $(\omega^+_l \coloneqq \widetilde{\omega}^+_l, \xvec{E}^+_l \coloneqq \widetilde{\xvec{E}}^+_l, \xvec{F}^+_l \coloneqq \widetilde{\xvec{F}}^+_l)$ would satisfy \eqref{equation:MHD_ElsaesserExt_OvarepsBtOsole}. 
		To construct such $\widetilde{\xvec{F}}^+_l$, the following observations are remarked.
		\begin{itemize}
			\item The right-hand side of \eqref{equation:goalconstrF} is supported in $\overline{\mathcal{E}}\setminus \overline{\Omega}$.
			\item When $\xvec{z}^0$ is obtained from \Cref{lemma:Z^0}, one can for simplicity assume that $\operatorname{supp}(\xwcurl{\xvec{z}^0})$ is contained in a single boundary cube (\cf~\Cref{remark:chooseregionfl}).
		\end{itemize}
		Since $\widetilde{\xvec{F}}^{+}_1, \dots, \widetilde{\xvec{F}}^{+}_M$ should be supported in $\overline{\mathcal{E}}\setminus\overline{\Omega}\times[0,T]$ and obey~\eqref{equation:goalconstrF}, the average of the right-hand side in \eqref{equation:goalconstrF} must vanish on each interior cube. This indeed happens because $\overline{\omega}^{+}_l$ solves the homogeneous transport equation \eqref{equation:MHD_ElsaesserExt_OvarepsBtOsoleNorighthandside}; in fact, for each $l \in \{1,\dots,L\}$ with $B_l \subset \mathcal{E}$, the average of $\overline{\omega}^{+}_l$ on $B_l$ vanishes at $t = 0$ and is transported by~$\xmcal{Z}^{0}$. Finally, by undoing the curl in \eqref{equation:goalconstrF} (\cf~\cite[Section A.2]{CoronMarbachSueur2020} for explicit formulas), one obtains the desired controls $\xvec{F}^{+}_1, \dots, \xvec{F}^{+}_M$. Thanks to the assumption $\xvec{z}^{\pm}_0 \in \xHn{3}(\mathcal{E})\cap\xH(\mathcal{E})$, one has $\xwcurl{\xsym{\xi}^{\pm,1}}\in \xCzero([0,T];\xHtwo(\mathcal{E};\mathbb{R}))$ and~$\xvec{z}^{\pm,1}$ are bounded in $\xLn{\infty}((0,T);\xHn{3}(\mathcal{E}))$.
	\end{proof}
	
	\subsection{Boundary layers and technical profiles}\label{subsection:blprf}
	In this subsection, the boundary layers and related technical profiles, appearing in \eqref{equation:ansatz2}--\eqref{3.7}, will be described. In addition to to the neighborhood $\mathcal{V}$, as defined in \Cref{subsection:domainextensions}, another tubular region is denoted by
	\[
		\mathcal{V}_{d^{*}} \coloneqq \{ \xvec{x} \in \overline{\mathcal{E}} \,  | \, \operatorname{dist}(\xvec{x},\partial \mathcal{E}) < d^{*} \} \subset \mathcal{V} \cap \overline{\mathcal{E}},
	\]
	where $d^{*} \in (0, d)$ is a small number to be fixed in \Cref{lemma:chisupp} below. Moreover, given a function~$\psi_{d^{*}}\in\xCinfty(\mathbb{R};[0,1])$ with
	\begin{equation*}\label{equation:psi}
		\psi_{d^{*}}(s) = \begin{cases}
			1 & \mbox{ for } s \leq d^{*}/2,\\
			0 & \mbox{ for } s \geq 2d^{*}/3,
		\end{cases}
	\end{equation*}
	a smooth cutoff $\chi_{\partial\mathcal{E}} \in \xCinfty(\overline{\mathcal{E}};[0,1])$ is defined by
	\begin{equation}\label{equation:def_chi}
		\chi_{\partial\mathcal{E}}(\xvec{x}) \coloneqq \psi_{d^{*}}(\varphi_{\mathcal{E}}(\xvec{x})),
	\end{equation}
	where $\varphi_{\mathcal{E}}(\xvec{x})=\operatorname{dist}(\xvec{x},\partial\mathcal{E})$ for $\xvec{x} \in \mathcal{V}$ as described in \Cref{subsection:domainextensions}.
	By construction, one observes that $\chi_{\partial\mathcal{E}} = 1$ in the vicinity of $\partial\mathcal{E}$ and that $\operatorname{supp}(\chi_{\partial\mathcal{E}}) \subset \mathcal{V}_{d^{*}}$. Furthermore, in view of~\eqref{equation:definitionnormal}, the gradient of~$\chi_{\partial\mathcal{E}}$ in $\mathcal{E}$ can be calculated as
	\begin{equation}\label{equation:propcutoff}
		\left(\xnab\chi_{\partial\mathcal{E}}\right)(\xvec{x}) = \left(\xdrv{}{s}\psi_{d^{*}}\right)(\varphi_{\mathcal{E}}(\xvec{x})) \xnab \varphi_{\mathcal{E}}(\xvec{x}) = -\left(\xdrv{}{s}\psi_{d^{*}}\right)(\varphi_{\mathcal{E}}(\xvec{x})) \xvec{n}(\xvec{x}).
	\end{equation}
	
	\begin{rmrk}
		The property \eqref{equation:propcutoff} of $\chi_{\partial\mathcal{E}}$ is employed later on for stating a condition under which the profiles $\xsym{\mu}^{\pm}$ in \eqref{3.7} can be chosen with $\xdiv{\xsym{\mu}^+} = \xdiv{\xsym{\mu}^-}$. 
	\end{rmrk}
	
	\subsubsection{Boundary layer equations}\label{subsubsection:DefinitionsBl}
	Our definition of the boundary layer correctors is motivated by \cite{CoronMarbachSueur2020,IftimieSueur2011}. After plugging the relations \eqref{equation:ansatz2} and \eqref{3.7} into \eqref{equation:MHD_ElsaesserExtScaledLongTime}, there appears a term at order $O(1)$ that is not absorbed by \eqref{equation:MHD_Elsaesser_O1_}. However, resorting to the idea from \cite{IftimieSueur2011} of writing
	\[
	\left\llbracket (\xvec{z}^{0} \cdot \xvec{n})\partial_z \xvec{v}^{\pm}\right\rrbracket_{\epsilon} = \sqrt{\epsilon}\left\llbracket \frac{\xvec{z}^{0} \cdot \xvec{n}}{\varphi_{\mathcal{E}}}z \partial_z \xvec{v}^{\pm}\right\rrbracket_{\epsilon},
	\]
	this contribution is seen to behave as $O(\sqrt{\epsilon})$. In order to also offset the mismatching boundary values $\xmcal{N}^{\pm}(\xvec{z}^0,\xvec{z}^0) \neq \xvec{0}$, the boundary layer profiles $(\xvec{v}^{+},\xvec{v}^{-})$ in \eqref{equation:ansatz2} are introduced in $\overline{\mathcal{E}} \times \mathbb{R}_+ \times \mathbb{R}_+$ as the solution to the coupled linear problem
	\begin{equation}\label{equation:MHD_ElsaesserBlProfilevpm_contr}
		\partial_t \xvec{v}^{\pm} - \partial_{zz}(\lambda^{\pm}\xvec{v}^{+} + \lambda^{\mp}\xvec{v}^{-}) + \left[ (\xvec{z}^{0} \cdot \xdop{\nabla}) \xvec{v}^{\pm} + (\xvec{v}^{\mp} \cdot \xdop{\nabla})\xvec{z}^{0} \right]_{\operatorname{tan}} + \mathfrak{f}z\partial_z\xvec{v}^{\pm} = \xsym{\xvec{\mu}}^{\pm}
	\end{equation}
	with boundary and initial conditions
	\begin{equation}\label{equation:MHD_ElsaesserBlProfilevpm_contribc}
		\begin{cases}
			\partial_z \xvec{v}^{\pm}(\xvec{x},t,0) = \xsym{\mathfrak{g}}^{\pm}(\xvec{x},t),  & \xvec{x} \in \overline{\mathcal{E}}, t \in \mathbb{R}_+ ,\\
			\xvec{v}^{\pm}(\xvec{x},t,z) \longrightarrow \xvec{0}, \mbox{ as } z \longrightarrow +\infty, & \xvec{x} \in \overline{\mathcal{E}}, t \in \mathbb{R}_+ ,\\
			\xvec{v}^{\pm}(\xvec{x},0,z) = \xvec{0}, & \xvec{x} \in \overline{\mathcal{E}}, z \in \mathbb{R}_+.
		\end{cases}
	\end{equation}
	Above, the functions $\mathfrak{f}$ and $\xsym{\mathfrak{g}}^{\pm}$ are for all $(\xvec{x},t) \in \overline{\mathcal{E}} \times \mathbb{R}_+$ given by
	\begin{equation}\label{equation:fktbdrleq_contr}
		\begin{gathered}
			\mathfrak{f}(\xvec{x},t) \coloneqq -\frac{\xvec{z}^{0}(\xvec{x},t)\cdot\xvec{n}(\xvec{x})}{\varphi_{\mathcal{E}}(\xvec{x})}, \quad
			\xsym{\mathfrak{g}}^{\pm}(\xvec{x},t) \coloneqq 2\chi_{\partial\mathcal{E}}(\xvec{x})\xmcal{N}^{\pm}(\xvec{z}^{0},\xvec{z}^{0})(\xvec{x},t).
		\end{gathered}
	\end{equation}
	Since $\xvec{z}^0 \cdot \xvec{n} = 0$ on $\partial \mathcal{E}$ and $\xnab \varphi_{\mathcal{E}} = - \xvec{n}$ in $\mathcal{E}$, the function $\mathfrak{f}$ is smooth (\cf~\cite[Lemma 4]{IftimieSueur2011}). Like $\xvec{z}^{0}$ and $\varphi_{\mathcal{E}}$, the functions $\xsym{\mathfrak{g}}^{\pm}$ are smooth as well.

	\begin{rmrk}
		It would be sufficient for our purpose to define $\xvec{v}^{\pm}$ only on the time interval $[0,T/\epsilon]$. By stating \eqref{equation:MHD_ElsaesserBlProfilevpm_contr} and \eqref{equation:MHD_ElsaesserBlProfilevpm_contribc} for all $t \in \mathbb{R}_+$, it is emphasized that $\xvec{v}^{\pm}$ are independent of $\epsilon$.
	\end{rmrk}
	
	Now, several properties of the solutions to \Cref{equation:MHD_ElsaesserBlProfilevpm_contr} and \eqref{equation:MHD_ElsaesserBlProfilevpm_contribc} are summarized; recall that $\xHn{{k,m,p}}_{\mathcal{E}}$ is the weighted space defined in \Cref{subsection:FunctionSpaces} via
	\[
		\xHn{{k,m,p}}_{\mathcal{E}}  = \left\{ f \in \xLtwo(\mathcal{E}\times\mathbb{R}_+) \, \left| \,  \left(\sum\limits_{r = 0}^p \sum\limits_{|\beta| \leq m} \int_{\mathcal{E}}\int_{\mathbb{R}_+} (1 + z^{2k})|\partial_{\xvec{x}}^{\beta}\partial_z^r f|^2 \, \xdx{z}  \xdx{\xvec{x}} \right)^{\frac{1}{2}} < +\infty  \right. \right\}.
	\]
	
	\begin{lmm}\label{lemma:wellpvpmctrl}
		Assume that $\xsym{\mu}^{\pm}\colon \mathcal{E}\times \mathbb{R}_+ \times \mathbb{R}_+ \longrightarrow \mathbb{R}^N$ are smooth and satisfy $\xsym{\mu}^{\pm} \cdot \xvec{n} = 0$. There exists a unique solution $(\xvec{v}^{+},\xvec{v}^{-})$ to \eqref{equation:MHD_ElsaesserBlProfilevpm_contr} and \eqref{equation:MHD_ElsaesserBlProfilevpm_contribc} which possesses for all $k,l,r,s \in \mathbb{N}_0$, and any~$T^* > 0$, the regularity
		\begin{alignat}{1}
			\partial_t^s\xvec{v}^{\pm} \in \xLinfty((0,T^*); \xHn{{k,l,r}}_{\mathcal{E}}) \cap \xLtwo((0,T^*);\xHn{{k,l,r+1}}_{\mathcal{E}}).\label{equation:lmmblreg}
		\end{alignat}
		In addition, for each $(\xvec{x},t,z) \in \overline{\mathcal{E}}\times\mathbb{R}_+\times\mathbb{R}_+$ the profiles $\xvec{v}^{\pm}$ verify the orthogonality relation
		\begin{equation}\label{equation:lmmblorth}
			\xvec{v}^{\pm}(\xvec{x},t,z) \cdot \xvec{n}(\xvec{x}) = 0.
		\end{equation} 
	\end{lmm}
	\begin{proof}
		The well-posedness of the linear problem \eqref{equation:MHD_ElsaesserBlProfilevpm_contr}, \eqref{equation:MHD_ElsaesserBlProfilevpm_contribc} is analogous to that of the Navier slip-with-friction boundary layers for the Navier--Stokes equations (\cf~\cite{IftimieSueur2011,CoronMarbachSueur2020}).
		Since $\xsym{\xvec{\mu}}^{\pm}$ are assumed smooth, the regularity stated in \eqref{equation:lmmblreg} is obtained from \Cref{lemma:higherorder}. The relation \eqref{equation:lmmblorth} follows by multiplying in \eqref{equation:MHD_ElsaesserBlProfilevpm_contr} with $\xvec{n}$, which leads to \apriori~estimates for $(\xvec{v}^+ \pm \xvec{v}^-)\cdot \xvec{n}$ similar to that given in \cite[Section 5]{IftimieSueur2011}.
	\end{proof}
	
	\subsubsection{Technical profiles}\label{subsubsection:technicalprofiles}
	
	For the sake of having $\xvec{v}^{\pm}\cdot\xvec{n} = 0$ in $\overline{\mathcal{E}}\times\mathbb{R}_+\times\mathbb{R}_+$, the normal contributions of $(\xvec{z}^{0} \cdot \xdop{\nabla}) \xvec{v}^{\pm} + (\xvec{v}^{\mp} \cdot \xdop{\nabla})\xvec{z}^{0}$, which appear at $O(\sqrt{\epsilon})$ when inserting \eqref{equation:ansatz2} into \eqref{equation:MHD_ElsaesserExtScaledLongTime}, have been omitted in \eqref{equation:MHD_ElsaesserBlProfilevpm_contr}. This and the commutation formula
	\[
	\xnab \llbracket q^{\pm}\rrbracket_{\epsilon} = \llbracket \xnab q^{\pm}\rrbracket_{\epsilon} - \frac{1}{\sqrt{\epsilon}}\llbracket\partial_z q^{\pm}\rrbracket_{\epsilon} \xvec{n}
	\]
	motivate introducing the profiles $q^{\pm}$ in \eqref{equation:ansatz2} as the solutions to
	\begin{equation}\label{equation:MHD_ElsaesserBlProfilepressurepm_control}
		\begin{cases}
			\partial_z q^{\pm} = \left[ (\xvec{z}^{0} \cdot \xdop{\nabla} \xvec{v}^{\pm}) + (\xvec{v}^{\mp} \cdot \xdop{\nabla})\xvec{z}^{0} \right] \cdot \xvec{n} & \mbox{in } \overline{\mathcal{E}} \times \mathbb{R}_+ \times \mathbb{R}_+,\\
			\lim\limits_{z \longrightarrow + \infty} q^{\pm}(\xvec{x},t,z) = 0,  & \xvec{x} \in \mathcal{E}, t \in (0, T).
		\end{cases}
	\end{equation}
	Next, let us define the second boundary layer correctors $\xvec{w}^{\pm}$. The normal parts of $\xvec{w}^{\pm}$ will compensate for the non-vanishing divergence of $\xvec{v}^{\pm}$, while their tangential parts constitute a lifting for~$\xmcal{N}^{\pm}(\xvec{v}^{+},\xvec{v}^-)(\xvec{x},t,0)$ and later on enable sufficient remainder estimates. Namely,
	\begin{equation}\label{equation:MHD_ElsaesserBlProfilesecondvelpm_contrl}
		\begin{aligned}
			\xvec{w}^{\pm}(\xvec{x},t,z) & \coloneqq \overline{w}^{\pm}(\xvec{x},t,z) \xvec{n}(\xvec{x}) -2\operatorname{e}^{-z}\xmcal{N}^{\pm}(\xvec{v}^{+},\xvec{v}^-)(\xvec{x},t,0), && \xvec{x} \in \overline{\mathcal{E}}, t \in \mathbb{R}_+, z \in \mathbb{R}_+,\\
			\overline{w}^{\pm}(\xvec{x},t,z) & \coloneqq - \int_z^{+\infty} \xdop{\nabla}\cdot\xvec{v}^{\pm} (\xvec{x},t,s) \, \xdx{s},  && \xvec{x} \in \overline{\mathcal{E}}, t \in \mathbb{R}_+, z \in \mathbb{R}_+,
		\end{aligned}
	\end{equation}
	noting that $\xvec{w}^{\pm}$ satisfy under the assumption $\operatorname{supp}(\xvec{v}^{\pm}(\cdot,t,z)) \subset \mathcal{V}$ the relations
	\begin{equation}\label{equation:MHD_ElsaesserBlProfilesecondvelprop_control}
		\begin{aligned}
			\xdop{\nabla}\cdot\xvec{v}^{\pm} & = \xvec{n}\cdot \partial_z \xvec{w}^{\pm} && \mbox{in } \overline{\mathcal{E}} \times \mathbb{R}_+ \times \mathbb{R}_+,\\
			\xmcal{N}^{\pm}(\xvec{v}^{+},\xvec{v}^-)(\xvec{x},t,0) & = \frac{1}{2}\left[\partial_z \xvec{w}^{\pm}\right]_{\operatorname{tan}}(\xvec{x},t,0),  && \xvec{x} \in \mathcal{E}, t \in (0, T).
		\end{aligned}
	\end{equation}
	
	\begin{rmrk}\label{remark:nbhd}
		The constructions in \cref{equation:MHD_ElsaesserBlProfilepressurepm_control,equation:MHD_ElsaesserBlProfilesecondvelpm_contrl} only serve their purpose, if the extended normal vector $\xvec{n}$ is nonzero in the $\xvec{x}$-support of $\xvec{v}^{\pm}$. Therefore, by means of a sufficiently small choice for $d^* > 0$ in the definition of $\chi_{\partial\mathcal{E}}$, it will be ensured that $\operatorname{supp}(\xvec{v}^{\pm}(\cdot,t,z)) \subset \mathcal{V}$ (\cf~\Cref{lemma:chisupp}). 
	\end{rmrk}
	In order to balance the nonzero divergence contributions and normal fluxes generated by $\xvec{w}^{\pm}$, the  correctors $\theta^{\pm,\epsilon}$ are introduced as solutions to
	\begin{equation}\label{equation:divcorrNeum}
		\begin{cases}
			\Delta \theta^{\pm,\epsilon} = - \left\llbracket \xdop{\nabla}\cdot\xvec{w}^{\pm}\right\rrbracket_{\epsilon} & \mbox{in } \mathcal{E} \times \mathbb{R}_+,\\
			\partial_{\xvec{n}} \theta^{\pm,\epsilon} = - \xvec{w}^{\pm}(\xvec{x},t,0) \cdot \xvec{n}(\xvec{x}), & \xvec{x} \in \partial \mathcal{E} \times \mathbb{R}_+.
		\end{cases}
	\end{equation}
	For each $t \in \mathbb{R}_+$, the corresponding Neumann problem in \eqref{equation:divcorrNeum} is well-posed; see \Cref{lemma:wpth} below.
	
	It remains to specify the profiles $\vartheta^{\pm,\epsilon}$ and the forces $\widetilde{\xsym{\zeta}}^{\pm,\epsilon}$. Inserting the ansatz \eqref{equation:ansatz2} into \eqref{equation:MHD_ElsaesserExtScaledLongTime} gives at the order $O(\epsilon)$ rise to the terms
	\begin{equation}\label{equation:cotwauw}
		\partial_t\xdop{\nabla}\theta^{\pm,\epsilon} + (\xvec{z}^0 \cdot \xnab)\xnab\theta^{\pm,\epsilon} + (\xnab\theta^{\mp,\epsilon} \cdot \xnab)\xvec{z}^0,
	\end{equation}
	which are not behaving well regarding the remainder estimates in \Cref{subsection:remestwbctrl}. In particular, the $\xLtwo(\partial \mathcal{E})$ norms of the boundary data in \Cref{equation:divcorrNeum} are of order $O(1)$. For this reason, the pressure correctors $\vartheta^{\pm,\epsilon}$ are defined by
	\[
	\vartheta^{\pm,\epsilon} \coloneqq -\partial_t \theta^{\pm,\epsilon} - \xvec{z}^0\cdot \xnab \theta^{\pm,\epsilon}
	\]
	such that one has the representations
	\begin{equation}\label{equation:ttcp}
		\partial_t\xdop{\nabla}\theta^{\pm,\epsilon} + (\xvec{z}^0 \cdot \xnab)\xnab\theta^{\pm,\epsilon} + (\xnab\theta^{\mp,\epsilon} \cdot \xnab)\xvec{z}^0 + \xnab \vartheta^{\pm,\epsilon} = (\xnab (\theta^{\mp,\epsilon} - \theta^{\pm,\epsilon})\cdot \xnab) \xvec{z}^0
	\end{equation}
	at all points where $\xcurl{\xvec{z}^0} = \xvec{0}$. Consequentially, in view of Lemmas~\Rref{lemma:original_u^0} and \Rref{lemma:Z^0}, or by \Cref{example:cylinder}, the relations in \eqref{equation:ttcp} are always true in $\overline{\Omega} \times \mathbb{R}_+$. However, in $(\overline{\mathcal{E}}\setminus\overline{\Omega}) \times (0,T)$ this might not be the case; but $\widetilde{\xsym{\zeta}}^{\pm,\epsilon}$ can be utilized to improve the remainder estimates. More precisely, motivated by the desired estimate \eqref{equation:aae} in \Cref{subsubsection:remaindereq} below, we shall define $\widetilde{\xsym{\zeta}}^{\pm,\epsilon}$ either via \eqref{equation:defzetpmeps} or by means of \eqref{equation:defzetpmeps2}, as explained next. 
	
	\paragraph{The first case of the definition for $\widetilde{\xsym{\zeta}}^{\pm,\epsilon}$.} When $(\xvec{w}^{+}-\xvec{w}^{-})_{|_{z=0}} \cdot \xvec{n} = 0$ is satisfied at~$\Sigma_T$, \Cref{lemma:wpth} will provide good estimates for $\theta^{+,\epsilon}- \theta^{-,\epsilon}$, allowing us to choose
	\begin{equation}\label{equation:defzetpmeps}
		\widetilde{\xsym{\zeta}}^{\pm,\epsilon} \coloneqq 
		(\xvec{z}^0 \cdot \xnab)\xnab\theta^{\pm,\epsilon} + (\xnab\theta^{\mp,\epsilon} \cdot \xnab)\xvec{z}^0 - \xnab(\xvec{z}^0\cdot \xnab \theta^{\pm,\epsilon}) - (\xnab (\theta^{\mp,\epsilon} - \theta^{\pm,\epsilon})\cdot \xnab) \xvec{z}^0.
	\end{equation}
	There are at least two situations with $(\xvec{w}^{+}-\xvec{w}^{-})_{|_{z=0}} \cdot \xvec{n} = 0$ at $\Sigma_T = \partial \mathcal{E} \times (0,T)$.
	\begin{itemize}
		\item If $\xvec{v}^+-\xvec{v}^- = \xvec{0}$ in $\mathcal{E}\times\mathbb{R}_+\times\mathbb{R}_+$, then also $(\xvec{w}^{+}-\xvec{w}^{-})\cdot \xvec{n} = 0$. This happens, for instance, when $\xvec{M}_2 = \xvec{0}$ and $\xvec{\mu}^{\pm}$ are determined such that $\xvec{\mu}^+ - \xvec{\mu}^- = \xvec{0}$ (\cf~\Cref{lemma:M2} and \Cref{lemma:blcxi2}).
		\item In the case of \Cref{theorem:annulus}, where the definition of $\xvec{z}^0$ from \eqref{equation:zeordcdannulussect} ensures $(\xvec{w}^{+}-\xvec{w}^{-})\cdot \xvec{n} = 0$.
		To see this, the proof of \Cref{lemma:blcxi2} below provides $\xdiv{(\xvec{v}^+-\xvec{v}^-)} = 0$ and $\xdiv{(\xsym{\mu}^+ - \xsym{\mu}^-)} = 0$.
	\end{itemize}
	Given the assumptions of \Cref{theorem:main1} or \Cref{theorem:annulus}, one can conclude also the additional properties
	\[
	\xdiv{(\widetilde{\xsym{\zeta}}^{+,\epsilon}-\widetilde{\xsym{\zeta}}^{-,\epsilon})} = 0 \, \mbox{ in } \mathcal{E}_T, \quad (\widetilde{\xsym{\zeta}}^{+,\epsilon}-\widetilde{\xsym{\zeta}}^{-,\epsilon}) \cdot \xvec{n} = 0 \, \mbox{ on } \Sigma_T.
	\]
	Indeed, either $\xvec{M}_2 = \xvec{0}$ implies that $\xnab(\theta^{+,\epsilon} - \theta^{-,\epsilon}) = \xvec{0}$, hence $\widetilde{\xsym{\zeta}}^{+,\epsilon} - \widetilde{\xsym{\zeta}}^{-,\epsilon} = \xvec{0}$, or, if $\xvec{M}_2 = \rho \xvec{I}$ in the case of \Cref{theorem:annulus}, one has $\xwcurl{\xvec{z}}^0 = 0$ by means of \eqref{equation:zeordcdannulussect}, which even provides $\widetilde{\xsym{\zeta}}^{\pm,\epsilon} = \xvec{0}$.
	
	\paragraph{The second case of the definition for $\widetilde{\xsym{\zeta}}^{\pm,\epsilon}$.} When $\xcurl{\xvec{z}^0} = \xvec{0}$ in $\overline{\mathcal{E}_T}$ and $(\xvec{w}^{+}-\xvec{w}^{-})_{|_{z=0}} \cdot \xvec{n} \neq 0$ at some points of $\Sigma_T$, then we define
	\begin{equation}\label{equation:defzetpmeps2}
		\widetilde{\xsym{\zeta}}^{\pm,\epsilon} \coloneqq (\xnab (\theta^{\mp,\epsilon} - \theta^{\pm,\epsilon})\cdot \xnab) \xvec{z}^0.
	\end{equation}
	This applies to the general situation of \Cref{theorem:main}, where $\xvec{z}^0$ is obtained from \Cref{lemma:original_u^0}. 
	
	\begin{rmrk}
		In the special case of \Cref{remark:cylm2}, see also \Cref{example:cylinder}, one has $\widetilde{\xsym{\zeta}}^{\pm,\epsilon} = \xvec{0}$ in $\overline{\Omega}\times[0,T]$, since in $\overline{\Omega}\times[0,T]$ the vector field $\xvec{z}^0$ is constant with respect to the spatial variables.
	\end{rmrk}

	When $\xvec{M}_2 = \xvec{0}$, magnetic field boundary layers cannot arise, as emphasized by the following lemma.
	
	\begin{lmm}\label{lemma:M2}
		$\xvec{M}_2 = \xvec{0}$ and $\xvec{\mu}^+ - \xvec{\mu}^- = \xvec{0}$ imply $\xvec{v}^+-\xvec{v}^- = \xvec{0}$.
	\end{lmm}
	\begin{proof}
		$\xvec{M}_2 = \xvec{0}$ implies $\xmcal{N}^{+}(\xvec{z}^0,\xvec{z}^0)-\xmcal{N}^{-}(\xvec{z}^0,\xvec{z}^0) = [2\xvec{M}_2 \xvec{z}^0]_{\operatorname{tan}} = \xvec{0}$. Since $\xvec{\mu}^+ - \xvec{\mu}^- = \xsym{0}$, if $\xvec{v}^{\pm}$ are obtained from \eqref{equation:MHD_ElsaesserBlProfilevpm_contr} and \eqref{equation:MHD_ElsaesserBlProfilevpm_contribc}, then their difference~$\xvec{W}$ obeys
		\begin{equation}\label{equation:ds1}
			\partial_t \xvec{W} - (\lambda^+-\lambda^-)\partial_{zz}\xvec{W} + \left[ (\xvec{z}^{0} \cdot \xdop{\nabla}) \xvec{W} - (\xvec{W} \cdot \xdop{\nabla})\xvec{z}^{0} \right]_{\operatorname{tan}} + \mathfrak{f}z\partial_z\xvec{W} = \xvec{0},
		\end{equation}
		with vanishing boundary and initial conditions
		\begin{equation}\label{equation:ds2}
			\begin{cases}
				\partial_z \xvec{W}(\xvec{x},t,0) = \xvec{0},  & \xvec{x} \in \overline{\mathcal{E}}, t \in \mathbb{R}_+ ,\\
				\xvec{W}(\xvec{x},t,z) \longrightarrow \xvec{0}, \mbox{ as } z \longrightarrow +\infty, & \xvec{x} \in \overline{\mathcal{E}}, t \in \mathbb{R}_+ ,\\
				\xvec{W}(\xvec{x},0,z) = \xvec{0}, & \xvec{x} \in \overline{\mathcal{E}}, z \in \mathbb{R}_+.
			\end{cases}
		\end{equation}
		Due to $\xvec{W} \cdot \xvec{n} = 0$ in $\overline{\mathcal{E}}\times\mathbb{R}_+\times\mathbb{R}_+$ (see \eqref{equation:lmmblorth}), one can show by means of direct energy estimates that $\xvec{W}$ must be the unique solution to \eqref{equation:ds1} and \eqref{equation:ds2}. Thus $\xvec{W} = \xvec{0}$.
	\end{proof}
	
	\subsubsection{Boundary layer dissipation via vanishing moment conditions}\label{subsubsection:vmo}
	The boundary layer controls $\xsym{\mu}^{\pm}$ appearing in \eqref{equation:MHD_ElsaesserBlProfilevpm_contr} are now determined. For large times $t \geq T$, in \Cref{seubsection:AsympExp} we already fixed
	\[
	\xsym{\mu}^{\pm}(\xvec{x},t,z) = \xvec{0}, \quad (\xvec{x},t,z) \in \overline{\mathcal{E}} \times [T,+\infty)\times\mathbb{R}_+.
	\]
	For all $t \in [0,T)$, the controls $\xsym{\mu}^{\pm}(\cdot,t)$ will be chosen such that ~$\xvec{v}^{\pm}$ admit improved decay rates as $t\longrightarrow +\infty$. To this end, we implement the well-prepared dissipation method, described in \cite{CoronMarbachSueur2020} for the Navier--Stokes equations and previously in \cite{Marbach2014} for a viscous Burgers' equation. Here, two different constructions for $\xsym{\mu}^{\pm}$ will be given:~the first one, namely \Cref{lemma:blcxi}, follows closely the known results and is suitable for showing \Cref{theorem:main} with nonzero $\xsym{\zeta}$; the second one, namely \Cref{lemma:blcxi2}, allows concluding Theorems~\Rref{theorem:main1}, \Rref{theorem:annulus}, and the assertion for $\xvec{M}_2 = \xvec{0}$ of \Cref{theorem:main}. To begin with, we state a direct modification of \cite[Lemma 6]{CoronMarbachSueur2020}, which involves the space (\cf~\Cref{subsection:FunctionSpaces})
	\[
		\widetilde{\xHn{{k,s}}}(\mathbb{R}) = \left\{ f \in \xLtwo(\mathbb{R}) \, \left| \, \left(\sum\limits_{l=0}^s \int_{\mathbb{R}} (1+z^{2k}) |\partial^l_z f(z)|^2 \, \xdx{z}\right)^{\frac{1}{2}} < +\infty  \right. \right\}.
	\]
	\begin{lmm}\label{lemma:dissiplmqt}
		Let $s,r \in \mathbb{N}$ and suppose that $f^{\pm}_0 \in \widetilde{\xHn{{r+1,s}}}(\mathbb{R})$ satisfy for all integers $0 \leq k < r$ the vanishing moment conditions
		\begin{equation}\label{equation:lmmvanishmomentcond}
			\int_{\mathbb{R}} z^k f^{\pm}_0 \, \xdx{z} = 0.
		\end{equation}
		Furthermore, assume that $(z,t) \longmapsto f^{\pm}(z,t)$ solve the coupled parabolic system
		\begin{equation*}\label{equation:lmmvanishmomheat}
			\begin{cases}
				\partial_t f^{\pm} - \partial_{zz}(\lambda^{\pm} f^+ + \lambda^{\mp} f^-) = 0 & \mbox{ in } \mathbb{R}\times\mathbb{R}_+,\\
				f^{\pm}(\cdot,0) = f^{\pm}_0 & \mbox{ in } \mathbb{R}.
			\end{cases}
		\end{equation*}
		Then, for all $k \in \{0,\dots,r\}$ one has the decay estimate
		\begin{equation}\label{equation:lmmdcest}
			\|f^{\pm}(\cdot,t)\|_{\widetilde{\xHn{{k,s}}}(\mathbb{R})} \leq C\max_{\triangle, \square \in\{+,-\}} \|f^{\triangle}_0\|_{\widetilde{\xHn{{r+1,s}}}(\mathbb{R})}^2 \left| \frac{\ln(2+(\lambda^+ \square \, \lambda^-)t)}{2+(\lambda^+ \square \, \lambda^-)t} \right|^{\frac{1}{4}+\frac{r}{2}-\frac{k}{2}},
		\end{equation}
		where $C = C(s,r)$ is a constant independent of $t \geq 0$ and the initial data $f^{\pm}_0$.
	\end{lmm}
	\begin{proof}
		We define the functions $F^{\pm} \coloneqq f^{+} \pm f^{-}$ and introduce for $\lambda_1, \lambda_2$ from \eqref{equation:chouk} the scaled versions $F^{\pm}_{\lambda_1,\lambda_2}(z,t) \coloneqq F^{\pm}(z,(\lambda^+\pm\lambda^-)^{-1}t)$, which obey the uncoupled heat equations		
		\begin{equation}\label{equation:lmmvanishmomheatcoupscal}
			\begin{cases}
				\partial_t F^{\pm}_{\lambda_1,\lambda_2} - \partial_{zz}F^{\pm}_{\lambda_1,\lambda_2} = 0 & \mbox{ in } \mathbb{R}\times\mathbb{R}_+,\\
				F^{\pm}_{\lambda_1,\lambda_2}(\cdot,0) = f^{+}_0 \pm f^{-}_0 & \mbox{ in } \mathbb{R}.
			\end{cases}
		\end{equation}
		By applying \cite[Lemma 6]{CoronMarbachSueur2020} to \eqref{equation:lmmvanishmomheatcoupscal}, the bound \eqref{equation:lmmdcest} follows first for $F^{\pm}$ and then by the triangle inequality also for $f^{\pm}$.
	\end{proof}
	
	The following result, which is a consequence of \Cref{lemma:dissiplmqt} and \cite[Lemma 7]{CoronMarbachSueur2020}, provides the controls $\xsym{\mu}^{\pm}$ on the time interval $[0,T]$. It will be applied in the general situation of \Cref{theorem:main}. 
	\begin{lmm}\label{lemma:blcxi}
		For any $r \in \mathbb{N}$, there exist $\xsym{\xvec{\mu}}^{\pm} \in \xCinfty(\overline{\mathcal{E}}\times[0,T]\times\mathbb{R}_+;\mathbb{R}^N)$ satisfying 
		\begin{gather*}
			\forall (\xvec{x},t,z) \in \mathcal{E}\times\mathbb{R}_+\times\mathbb{R}_+ \colon \xsym{\xvec{\mu}}^{\pm}(\xvec{x},t,z) \cdot \xvec{n}(\xvec{x}) = 0, \\
			\forall z \in \mathbb{R}_+ \colon 
			\operatorname{supp}(\xsym{\xvec{\mu}}^{\pm}(\cdot,\cdot,z)) \subset \left(\overline{\mathcal{E}}\setminus\overline{\Omega}\right) \times (0,T)
		\end{gather*}
		such that $\xvec{v}^{\pm}$ obey the decay rate
		\begin{equation}\label{equation:lmmdecestbl}
			\|\xvec{v}^{\pm}(\cdot,t,\cdot)\|_{\xHn{{k,p,s}}_{\mathcal{E}}}^2 \leq C \max_{\square \in \{+,-\}} \left| \frac{\ln(2+(\lambda^+ \square \, \lambda^-)t)}{2+(\lambda^+ \square \, \lambda^-)t} \right|^{\frac{1}{4}+\frac{r}{2}-\frac{k}{2}},
		\end{equation}
		for all $0 \leq k \leq r$ and $s,p,k \in \mathbb{N}_0$, with a constant $C = C(s,r,p,k) > 0$ not depending on the time $t$. 
	\end{lmm}
	\begin{proof}
		Consider the even extensions of $\xvec{v}^{\pm}$ to $\overline{\mathcal{E}}\times\mathbb{R}_+\times\mathbb{R}$ plus lifted boundary data, defined via
		\begin{equation}\label{equation:lf}
			\xvec{V}^{\pm}(\xvec{x},t,z) \coloneqq \xvec{v}^{\pm}(\xvec{x},t,|z|) + \xsym{\mathfrak{g}}^{\pm}(\xvec{x},t) \operatorname{e}^{-|z|}, \quad (\xvec{x},t,z) \in \overline{\mathcal{E}}\times\mathbb{R}_+\times\mathbb{R}.
		\end{equation}
		For $t \geq T$, one has by construction
		\[
		\xvec{z}^{0}(\cdot,t) = \xsym{0}, \quad \xsym{\mathfrak{g}}^{\pm}(\cdot,t) = \xsym{0}, \quad \xsym{\mu}^{\pm}(\cdot,t,\cdot) = \xsym{0}.
		\]
		Thus, $\xvec{V}^{\pm}$ are governed by the parabolic system
		\begin{equation}\label{equation:MHD_ElsaesserBlProfilevpm_contr_largeTheat}
			\partial_t \xvec{V}^{\pm} - 	\partial_{zz}(\lambda^{\pm}\xvec{V}^{+}+\lambda^{\mp}\xvec{V}^{-}) = \xvec{0} \quad \mbox{in } \overline{\mathcal{E}} \times [T, +\infty) \times \mathbb{R},
		\end{equation} 
		where $\xvec{x} \in \overline{\mathcal{E}}$ is a parameter. Therefore, in view of \Cref{lemma:dissiplmqt} and \eqref{equation:lf}, the decay estimate~\eqref{equation:lmmdecestbl} follows if enough vanishing moment conditions of the type \eqref{equation:lmmvanishmomentcond} are satisfied. To see this, the proof of \cite[Lemma 7]{CoronMarbachSueur2020} can be applied individually to the equations satisfied by $\xvec{V}^+ + \xvec{V}^-$ and $\xvec{V}^+ - \xvec{V}^-$. This provides controls $\xsym{\mu}^+ \pm \xsym{\mu}^-$ such that for each $k \in \{1,\dots,r-1\}$ the Fourier transformed functions
		\[
		\widehat{\xvec{V}}^{\pm}(\cdot,\cdot,\zeta) \coloneqq \int_{\mathbb{R}}  \operatorname{e}^{-i\zeta z} \xvec{V}^{\pm}(\cdot,\cdot,z) \, \xdx{z}
		\]
		obey the relations
		\begin{equation}\label{equation:auxprfcnstrmu}
			\partial_{\zeta}^k \widehat{\xvec{V}}^{+}(\cdot,T,\zeta)_{|_{\zeta=0}} \pm \partial_{\zeta}^k \widehat{\xvec{V}}^{-}(\cdot,T,\zeta)_{|_{\zeta=0}} = \xvec{0}.
		\end{equation}
		Since \eqref{equation:auxprfcnstrmu} implies sufficient vanishing moment conditions and \Cref{lemma:wellpvpmctrl} provides uniform bounds for $\xvec{v}^{\pm}$ on the time interval $[0,T]$, the proof is complete.
	\end{proof}

	The next lemma provides magnetic field boundary layer controls that are not only tangential at $\partial \mathcal{E}$, but also divergence-free in $\mathcal{E}\times(0,T)\times\mathbb{R}_+$; this is required for showing \Cref{theorem:annulus}. Also regarding \Cref{theorem:main1}, and \Cref{theorem:main} with $\xvec{M}_2 = \xvec{0}$, the proof below will explain how~$\xsym{\mu}^{\pm}$ can be selected with $\xsym{\mu}^+ - \xsym{\mu}^- = \xsym{0}$. The approach is based on ideas from \cite[Lemma 7]{CoronMarbachSueur2020}. 
	\begin{lmm}\label{lemma:blcxi2}
		Given any $r \in \mathbb{N}$, the controls $\xsym{\mu}^{\pm}$ with the properties from \Cref{lemma:blcxi} can under additional assumptions be chosen as follows.
		\begin{enumerate}[1)]
			\item When $\xmcal{\mathcal{N}}^+(\xvec{z}^0,\xvec{z}^0)(\xvec{x},t) = \xmcal{\mathcal{N}}^-(\xvec{z}^0,\xvec{z}^0)(\xvec{x},t)$ is valid for all $(\xvec{x},t) \in \operatorname{supp}(\chi_{\partial\mathcal{E}}) \times (0,T)$, then $\xsym{\mu}^{\pm}$ can be fixed with
			\[
			\xsym{\mu}^{+}-\xsym{\mu}^{-} = \xvec{0}.
			\]
			\item When the profile $\xvec{z}^0$ selected in \Cref{seubsection:AsympExp} satisfies
			\begin{equation}\label{equation:condmuwkez}
				\begin{aligned}
					&\xdiv{(\xsym{\mathcal{N}}^+(\xvec{z}^0, \xvec{z}^0) - \xsym{\mathcal{N}}^-(\xvec{z}^0,\xvec{z}^0))} = 0 && \mbox{ in } (\mathcal{V}\cap\overline{\mathcal{E}}) \times (0,T), \\
					& \xvec{z}^0 \cdot \xvec{n} = 0 && \mbox{ in } (\mathcal{V}\cap\overline{\mathcal{E}}) \times (0,T),\\
					&\xdiv{\xvec{z}^0} = 0 && \mbox{ in } \mathcal{E}_T,
				\end{aligned}
			\end{equation}
			then after fixing the number $d^* \in (0,d)$ from the definition of $\chi_{\partial\mathcal{E}}$ sufficiently small, one can construct $\xsym{\mu}^{\pm}$ with
			\begin{equation}\label{equation:divcond}
				\xdiv{(\xsym{\mu}^+-\xsym{\mu}^-)} = 0, \quad (\xvec{x},t,z) \in \mathcal{E}\times\mathbb{R}_+\times\mathbb{R}_+.
			\end{equation}
			\item If in addition to the conditions in \eqref{equation:condmuwkez} one has $\xvec{z}^0 \cdot \xvec{n} = 0$ in $\mathcal{E}_T$, then $\xsym{\mu}^{\pm}$ can be chosen with \eqref{equation:divcond} for any choice $d^* \in (0,d)$.
		\end{enumerate}
	\end{lmm}
	\begin{proof} 
		If $\xmcal{\mathcal{N}}^+(\xvec{z}^0,\xvec{z}^0) = \xmcal{\mathcal{N}}^-(\xvec{z}^0,\xvec{z}^0)$ in $\operatorname{supp}(\chi_{\partial\mathcal{E}}) \times (0,T)$, one may take $\xsym{\mu}^{+}-\xsym{\mu}^{-} = \xvec{0}$. Indeed, the magnetic field boundary layer $\xvec{v}^{+}-\xvec{v}^{-}$ solves in that case a well-posed linear problem with zero data, which yields $\xvec{v}^{+}-\xvec{v}^{-} = \xvec{0}$ in $\overline{\mathcal{E}}\times\mathbb{R}_+\times\mathbb{R}_+$ as shown by \Cref{lemma:M2}. In any case, we apply the arguments from the proof of \Cref{lemma:blcxi} for determining $\xsym{\mu}^+ + \xsym{\mu}^-$ such that
		\begin{equation*}
			\partial_{\zeta}^k \widehat{\xvec{V}}^{+}(\cdot,T,\zeta)_{|_{\zeta=0}} + \partial_{\zeta}^k \widehat{\xvec{V}}^{-}(\cdot,T,\zeta)_{|_{\zeta=0}} = \xvec{0}, \quad k \in \{1,\dots,r-1\}.
		\end{equation*}
		Thus, to show 2) and 3), it remains to identify a suitable control $\xsym{\mu}^+ - \xsym{\mu}^-$ (having the desired properties), which acts in the equation satisfied by $\xvec{v}^+ - \xvec{v}^-$ in a way that
		\begin{equation*}
			\partial_{\zeta}^k \widehat{\xvec{V}}^{+}(\cdot,T,\zeta)_{|_{\zeta=0}} - \partial_{\zeta}^k \widehat{\xvec{V}}^{-}(\cdot,T,\zeta)_{|_{\zeta=0}} = \xvec{0}, \quad k \in \{1,\dots,r-1\}.
		\end{equation*}
		\paragraph{Step 1. Preliminaries.}
		In view of \eqref{equation:lf}, the vector field $\xvec{W} \coloneqq \xvec{V}^+ - \xvec{V}^-$ can with $\xsym{\mathfrak{g}} \coloneqq \xsym{\mathfrak{g}}^{+} - \xsym{\mathfrak{g}}^{-}$ be written as
		\begin{equation}\label{equation:repW}
			\xvec{W}(\xvec{x},t,z) = \xvec{v}^{+}(\xvec{x},t,|z|)-\xvec{v}^{-}(\xvec{x},t,|z|) +  \xsym{\mathfrak{g}}(\xvec{x},t)\operatorname{e}^{-|z|}.
		\end{equation}
		Under the assumption $\xsym{\mu}^{\pm} \cdot \xvec{n} = 0$, noting that $\xsym{\mathfrak{g}}^{\pm} \cdot \xvec{n} = 0$ is true by construction, it follows from \eqref{equation:lmmblorth} that $\xvec{W} \cdot \xvec{n} = 0$ holds as well. Also, \cref{equation:propcutoff,equation:condmuwkez} imply $\xdiv{\xsym{\mathfrak{g}}} = 0$ by means of
		\begin{equation}\label{equation:divgpmgmv}
			\begin{aligned}
				\xdiv{(\xsym{\mathfrak{g}}^{+}-\xsym{\mathfrak{g}}^{-})} & =  	2\chi_{\partial\mathcal{E}} \xdiv{(\xsym{\mathcal{N}}^+(\xvec{z}^0, \xvec{z}^0) - \xsym{\mathcal{N}}^-(\xvec{z}^0,\xvec{z}^0))}\\
				& \quad - 2 \left[\left(\xdrv{}{s}\psi_{d^{*}}\right)\circ\varphi_{\mathcal{E}}\right] 	\xvec{n} \cdot (\xsym{\mathcal{N}}^+(\xvec{z}^0, \xvec{z}^0) - \xsym{\mathcal{N}}^-(\xvec{z}^0,\xvec{z}^0)) \\
				& = 2\chi_{\partial\mathcal{E}} \xdiv{(\xsym{\mathcal{N}}^+(\xvec{z}^0, \xvec{z}^0) - \xsym{\mathcal{N}}^-(\xvec{z}^0,\xvec{z}^0))} \\ 
				& = 0.
			\end{aligned}
		\end{equation}
		Even more, $\xvec{n} = - \xnab \varphi_{\mathcal{E}}$ being a gradient, hence~$\xcurl{\xvec{n}} = \xvec{0}$, ensures for two vector fields~$\xvec{h}_1$ and~$\xvec{h}_2$ defined in~$\mathcal{E}$ that
		\[
		\left(\xvec{h}_1 \cdot \xvec{n} = 0 \mbox{ and } \xvec{h}_2 \cdot \xvec{n} = 0\right) \, \Longrightarrow \, \left((\xvec{h}_1 \cdot \xnab)\xvec{h}_2 - (\xvec{h}_2 \cdot \xnab)\xvec{h}_1\right)  \cdot \xvec{n} = 0.
		\]
		Therefore, one can infer at all points where $\xvec{z}^0 \cdot \xvec{n} = 0$ holds the relation
		\begin{gather*}
			\left[ (\xvec{z}^{0} \cdot \xdop{\nabla}) (\xvec{v}^{+}-\xvec{v}^{-}) - ((\xvec{v}^{+}-\xvec{v}^{-}) \cdot \xdop{\nabla})\xvec{z}^{0} \right]_{\operatorname{tan}} = (\xvec{z}^{0} \cdot \xdop{\nabla}) (\xvec{v}^{+}-\xvec{v}^{-}) - ((\xvec{v}^{+}-\xvec{v}^{-}) \cdot \xdop{\nabla})\xvec{z}^{0}.
		\end{gather*}
		Next, while from \eqref{equation:condmuwkez} it only follows that $\mathfrak{f} = 0$ in $(\mathcal{V}\cap \overline{\mathcal{E}})\times\mathbb{R}_+$, we temporarily guarantee that $\mathfrak{f} = 0$ in all of $\overline{\mathcal{E}}\times\mathbb{R}_+$ by means of the artificially strong assumption, which will be removed in the last step by choosing $d^* \in (0,d)$ small\footnote{In this article, \eqref{equation:artificialassump} is always satisfied when we employ \Cref{lemma:blcxi2} with the hypotheses in \eqref{equation:condmuwkez}; see \eqref{equation:condmuwkez_0}.}:
		\begin{equation}\label{equation:artificialassump}
			\xvec{z}^0(\xvec{x},t) \cdot \xvec{n}(\xvec{x}) = 0 \mbox{ in } \overline{\mathcal{E}}\times\mathbb{R}_+.
		\end{equation}
		As a result of the foregoing considerations, and by understanding the yet unspecified controls $\xsym{\mu}^{\pm}$ as extended evenly to all $z \in \mathbb{R}$, the function $\xvec{W}(x,t,z)$ satisfies the following problem:
		\begin{equation}\label{equation:MHD_ElsaesserBlProfilevpm_contr_hom}
			\begin{cases}
				\partial_t \xvec{W} - \partial_{zz} (\lambda^{+} - \lambda^{-})\xvec{W} + (\xvec{z}^{0} \cdot \xdop{\nabla})\xvec{W} - (\xvec{W} \cdot \xdop{\nabla})\xvec{z}^{0} = \xsym{\mathfrak{G}}\operatorname{e}^{-|z|} + \xsym{{{\mu}}} & \mbox{ in } \overline{\mathcal{E}} \times \mathbb{R}_+ \times \mathbb{R},\\
				\xvec{W}|_{t=0} = \xvec{0} & \mbox{ in } \overline{\mathcal{E}} \times \mathbb{R},
			\end{cases}
		\end{equation}
		where
		\begin{equation*}
			\xsym{\mathfrak{G}} \coloneqq \partial_t \xsym{\mathfrak{g}} - (\lambda^+ - \lambda^-) \xsym{\mathfrak{g}} + (\xvec{z}^0 \cdot \xnab)\xsym{\mathfrak{g}} - (\xsym{\mathfrak{g}} \cdot \xnab) \xvec{z}^0, \quad \xsym{{{\mu}}} \coloneqq \xsym{\mu}^+ - \xsym{\mu}^-.
		\end{equation*}
		Consequently, the partial Fourier transform
		\[
		\widehat{\xvec{W}}(\xvec{x},t,\zeta) \coloneqq \int_{\mathbb{R}} \xvec{W}(\xvec{x},t,z) \operatorname{e}^{-i\zeta z} \, \xdx{z}
		\]
		satisfies the problem
		\begin{equation}\label{equation:pfev}
			\begin{cases}
				\partial_t \widehat{\xvec{W}} + \zeta^2 (\lambda^{+}-\lambda^-)\widehat{\xvec{W}} +  (\xvec{z}^0 \cdot \xdop{\nabla})\widehat{\xvec{W}} - (\widehat{\xvec{W}} \cdot \xdop{\nabla})\xvec{z}^0 = \frac{2}{1+\zeta^2}\xsym{\mathfrak{G}} + \widehat{\xsym{{{\mu}}}} & \mbox{ in } \mathcal{E} \times \mathbb{R}_+ \times \mathbb{R},\\
				\widehat{\xvec{W}}|_{t=0}= \xvec{0} & \mbox{ in } \mathcal{E} \times \mathbb{R}.
			\end{cases}		
		\end{equation}
		Therefore, during the time interval $[0,T]$, for each $k \in \mathbb{N}_0$ the evolution of the evaluated derivatives
		\[
		\xvec{Q}^{k}(\xvec{x},t) \coloneqq \partial_{\zeta}^k \widehat{\xvec{W}}(\xvec{x},t,\zeta)_{|_{\zeta=0}}
		\]
		is governed by the transport equation
		\begin{equation}\label{equation:pfevQ}
			\begin{cases}
				\partial_t \xvec{Q}^{k} + (\xvec{z}^0 \cdot \xdop{\nabla}) \xvec{Q}^{k} - (\xvec{Q}^{k}  \cdot \xdop{\nabla})\xvec{z}^0 = \xvec{P}^{k} & \mbox{ in } \mathcal{E}_T,\\
				\xvec{Q}^{k}(\cdot,0) = \xvec{0} & \mbox{ in } \overline{\mathcal{E}},
			\end{cases}	
		\end{equation}
		which contains the source term
		\begin{alignat}{3}\label{equation:pollterms}
			\xvec{P}^{k} \coloneqq \partial_{\zeta}^k \left(\widehat{\xsym{{{\mu}}}} + \frac{2\xsym{\mathfrak{G}}}{1+\zeta^2}\right)\Big|_{\zeta=0} -k(k-1)(\lambda^{+}-\lambda^-)\xvec{Q}^{\max\{0,k-2\}}.
		\end{alignat}
		
		\paragraph{Step 2. Determining $\xsym{\xvec{\mu}}^{+}-\xsym{\xvec{\mu}}^{-}$.}
		Let $\widetilde{r} \in \mathbb{N}$ denote the integer part of $(r-1)/2$. Since~$\xvec{W}$ is symmetric about the $z=0$ axis, it remains to steer for $l \in \{0,\dots,\widetilde{r}\}$ the even moments $\xvec{Q}^{2l}$ to zero. Hereto, we make for $\xsym{{{\mu}}}$ the ansatz
		\begin{equation}\label{equation:formblc}
			\xsym{{{\mu}}}(\xvec{x},t,z) = \sum\limits_{i = 0}^{\widetilde{r}} \xsym{{{\mu}}}^{i}(\xvec{x},t) \phi_i(z),
		\end{equation}
		where the even functions $(\phi_j)_{j \in \{0,\dots,\widetilde{r}\}} \subset \xCinfty(\mathbb{R})\cap\xLtwo(\mathbb{R})$ are chosen such that
		\[
			\widehat{\phi}_j(\zeta) \coloneqq \int_{\mathbb{R}} \phi_j(z) \operatorname{e}^{-i\zeta z} \, \xdx{z}, \quad l \in \{0,\dots,\widetilde{r}\}
		\]
		obey the relations
		\[
		\partial_{\zeta}^{2l} \widehat{\phi}_j(0) \coloneqq \begin{cases}
			1 & \mbox{ if } j = l,\\
			0 & \mbox{ otherwise. }
		\end{cases}
		\]
		For instance, for any even smooth cutoff $\widetilde{\varphi} \in \xCinfty_0(\mathbb{R})$ with $\widetilde{\varphi} = 1$ in a neighborhood of the origin one may take as $\phi_0, \dots, \phi_{\widetilde{r}}$ the inverse 
		Fourier transforms of the functions
		\[
			\widetilde{\varphi},\frac{1}{2}\zeta^2\widetilde{\varphi},\dots,\frac{1}{(2\widetilde{r})!}\zeta^{2\widetilde{r}}\widetilde{\varphi}.
		\] 
		Subsequently, inserting the ansatz \eqref{equation:formblc} for each even choice $k \in \{1,\dots,r-1\}$ into \eqref{equation:pfevQ} provides the cascade system of transport equations
		\begin{equation}\label{equation:pfevQ2casc}
			\begin{cases}
				\partial_t \xvec{Q}^{0} + (\xvec{z}^0 \cdot \xdop{\nabla}) \xvec{Q}^{0} - (\xvec{Q}^{0}  \cdot \xdop{\nabla})\xvec{z}^0 & = \xsym{{{\mu}}}^0 + 2 \xsym{\mathfrak{G}},\\
				\partial_t \xvec{Q}^{2} + (\xvec{z}^0 \cdot \xdop{\nabla}) \xvec{Q}^{2} - (\xvec{Q}^{2} \cdot \xdop{\nabla})\xvec{z}^0 & = \xsym{{{\mu}}}^1 -4\xsym{\mathfrak{G}} - 2(\lambda^+-\lambda^-)\xvec{Q}^{0},\\
				& \,\,\, \vdots \\
				\partial_t \xvec{Q}^{2\widetilde{r}} + (\xvec{z}^0 \cdot \xdop{\nabla}) \xvec{Q}^{2\widetilde{r}} - (\xvec{Q}^{2\widetilde{r}}  \cdot \xdop{\nabla})\xvec{z}^0 & = \xsym{{{\mu}}}^{\widetilde{r}} +\frac{2(2\widetilde{r})!}{(-1)^{\widetilde{r}}}\xsym{\mathfrak{G}} - (4\widetilde{r}^2-2\widetilde{r})(\lambda^+-\lambda^-)\xvec{Q}^{2\widetilde{r}-2}
			\end{cases}
		\end{equation}
		with zero initial conditions
		\[
		\xvec{Q}^0(\cdot,0) = \xvec{Q}^{2}(\cdot,0) = \dots = \xvec{Q}^{2\widetilde{r}}(\cdot, 0) = \xvec{0}.
		\]
		
		Let us begin with determining the control $\xsym{{{\mu}}}^0$ in \eqref{equation:pfevQ2casc}. Hereto, the auxiliary function $\overline{\xvec{Q}}^0$ is taken as the unique solution to
		\begin{equation}\label{equation:pfevQ2hom}
			\begin{cases}
				\partial_t \overline{\xvec{Q}}^0 + (\xvec{z}^0 \cdot \xdop{\nabla}) \overline{\xvec{Q}}^0 - (\overline{\xvec{Q}}^0  \cdot \xdop{\nabla})\xvec{z}^0 = 2 \xsym{\mathfrak{G}} & \mbox{ in } \mathcal{E}_T,\\
				\overline{\xvec{Q}}^0(\cdot,0) = \xvec{0} & \mbox{ in } \mathcal{E}.
			\end{cases}	
		\end{equation}
		Since $\xdiv{\xsym{\mathfrak{G}}} = 0$ holds in $\mathcal{E}_T$ by \eqref{equation:divgpmgmv}, and $\xsym{\mathfrak{G}} \cdot \xvec{n} = 0$ is true in $\mathcal{E}_T$ as well, for all $(\xvec{x},t) \in \mathcal{E}_T$ one may observe that
		\[
		\xdiv{\overline{\xvec{Q}}^0}(\xvec{x},t) = 0, \quad \overline{\xvec{Q}}^0(\xvec{x},t) \cdot \xvec{n}(\xvec{x}) = 0.
		\]
		Furthermore, the profile $\widetilde{z}^0(t) \coloneqq -\xvec{z}^0(T-t)$ is without loss of generality assumed to satisfy a flushing property of the type \eqref{equation:flushingproperty}. Indeed, when~$\xvec{z}^0$ is defined via \eqref{equation:zeordcdannulussect}, one notices that the flow associated with~$\widetilde{z}^0$ simply moves particles around the annulus in the opposite direction. More generally, the flushing property of $\widetilde{z}^0$ can be ensured by constructing~$\xvec{z}^0$ with~$T$ replaced by~$T/2$, followed by gluing the resulting profile at the time $t = T/2$ to a time-reversed version. Now, we take $\widetilde{\xvec{Q}}^0$ as the unique solution to
		\begin{equation}\label{equation:pfevQ2cont}
			\begin{cases}
				\partial_t \widetilde{\xvec{Q}}^0 + (\widetilde{\xvec{z}}^0 \cdot \xdop{\nabla}) \widetilde{\xvec{Q}}^0 - (\widetilde{\xvec{Q}}^0  \cdot \xdop{\nabla})\widetilde{\xvec{z}}^0 = \widetilde{\xsym{{{\mu}}}}^0 & \mbox{ in } \mathcal{E}_T,\\
				\widetilde{\xvec{Q}}^0(\cdot,0) = \overline{\xvec{Q}}^0(\cdot,T) & \mbox{ in } \mathcal{E},
			\end{cases}	
		\end{equation}
		where the control $\widetilde{\xsym{{{\mu}}}}^0 \in \xCinfty(\overline{\mathcal{E}}\times[0,T];\mathbb{R}^N)$ is chosen such that
		\begin{gather*}
			\widetilde{\xvec{Q}}^0(\cdot,T) = \xvec{0}, \quad \forall (\xvec{x},t) \in \mathcal{E}_T \colon \xdiv{\widetilde{\xvec{Q}}^0}(\xvec{x},t) = 0, \quad \forall (\xvec{x},t) \in \Sigma_T \colon \widetilde{\xvec{Q}}^0(\xvec{x},t) \cdot \xvec{n}(\xvec{x}) = 0,\\
			\forall (\xvec{x},t) \in \mathcal{E}_T \colon \xdiv{\widetilde{\xsym{{{\mu}}}}^0}(\xvec{x},t) = 0, \quad \forall (\xvec{x},t) \in \Sigma_T \colon \widetilde{\xsym{{{\mu}}}}^0(\xvec{x},t) \cdot \xvec{n}(\xvec{x}) = 0, \\ \operatorname{supp}(\widetilde{\xsym{{{\mu}}}}^0) \subset \left(\overline{\mathcal{E}}\setminus\overline{\Omega}\right)\times(0,T).
		\end{gather*}
		In order to find~$\widetilde{\xsym{{{\mu}}}}^0$, we proceed analogously to the constructions of controlled solutions to \eqref{equation:MHD_ElsaesserExt_OvarepsBtOwp} in the proof of \Cref{lemma:flushing}, noting that the problems \eqref{equation:MHD_ElsaesserExt_OvarepsBtOwp} and \eqref{equation:pfevQ2cont} are of the same type. As a result, we can define in $\mathcal{E}_T$ the vector field
		\[
			\xvec{Q}^{0}(\xvec{x},t) \coloneqq  \overline{\xvec{Q}}^0(\xvec{x},t) - \widetilde{\xvec{Q}}^0(\xvec{x},T-t),
		\]
		which solves the first equation in \eqref{equation:pfevQ2casc} with $\xsym{{{\mu}}}^0(\cdot, t) \coloneqq \widetilde{\xsym{{{\mu}}}}^0(T-t)$ and obeys the desired initial and terminal conditions
		\[
			\xvec{Q}^{0}(\cdot,0) = \xvec{0}, \quad \xvec{Q}^{0}(\cdot,T) = \xvec{0}.
		\]
		To determine $\xsym{{{\mu}}}^1$, the same arguments as for finding $\xsym{{{\mu}}}^0$ can be repeated, but now with the known source term $-4\xsym{\mathfrak{G}}-2(\lambda^+-\lambda^-)\xvec{Q}^{0}$.
		Due to the cascade structure of \eqref{equation:pfevQ2casc}, all controls $(\xsym{{{\mu}}}^j)_{j\in\{1,\dots,\widetilde{r}\}}$ are obtained in this way. Finally, $\xsym{{{\mu}}} = \xsym{\mu}^+ - \xsym{\mu}^-$ is constructed via \eqref{equation:formblc} and obeys \eqref{equation:divcond}.
		
		\paragraph{Step 3. Removing the assumption \eqref{equation:artificialassump}.} Without assuming \eqref{equation:artificialassump}, the equation \eqref{equation:MHD_ElsaesserBlProfilevpm_contr_hom} for $\xvec{W}$ might not be correct in $(\overline{\mathcal{E}}\setminus\mathcal{V})\times\mathbb{R}_+\times\mathbb{R}$, since \eqref{equation:MHD_ElsaesserBlProfilevpm_contr} contains the terms $\mathfrak{f}z\partial_z\xvec{v}^{\pm}$.  However, for small $d^* \in (0,d)$ it will be shown below that the control~$\xsym{{{\mu}}}$ obtained in the previous step already ensures
		\begin{equation}\label{equation:vpmsupp}
			\widetilde{\xvec{v}} = \xvec{0}  \mbox{ in } (\overline{\mathcal{E}}\setminus\mathcal{V})\times\mathbb{R}_+\times\mathbb{R}_+
		\end{equation}
		for the solution to
		\begin{equation}\label{equation:MHD_ElsaesserBlProfilevpm_contrALT}
			\begin{cases}
				\partial_t \widetilde{\xvec{v}} - \nu_2\partial_{zz}\widetilde{\xvec{v}} + (\xvec{z}^{0} \cdot \xdop{\nabla}) \widetilde{\xvec{v}} - (\widetilde{\xvec{v}}\cdot \xdop{\nabla})\xvec{z}^{0} = \xsym{{{\mu}}} & \mbox{in } \overline{\mathcal{E}} \times \mathbb{R}_+ \times \mathbb{R}_+,\\
				\partial_z \widetilde{\xvec{v}}(\xvec{x},t,0) = \xsym{\mathfrak{g}}(\xvec{x},t),  & \xvec{x} \in \overline{\mathcal{E}}, t \in \mathbb{R}_+ ,\\
				\widetilde{\xvec{v}}(\xvec{x},t,z) \longrightarrow \xvec{0}, \mbox{ as } z \longrightarrow +\infty, & \xvec{x} \in \overline{\mathcal{E}}, t \in \mathbb{R}_+ ,\\
				\widetilde{\xvec{v}}(\xvec{x},0,z) = \xvec{0}, & \xvec{x} \in \overline{\mathcal{E}}, z \in \mathbb{R}_+.
			\end{cases}
		\end{equation}
		Then, because the assumptions in \eqref{equation:condmuwkez} together with \eqref{equation:vpmsupp} imply
		\[
		\left[(\xvec{z}^{0} \cdot \xdop{\nabla}) \widetilde{\xvec{v}} - (\widetilde{\xvec{v}}\cdot \xdop{\nabla})\xvec{z}^{0}\right]_{\operatorname{tan}} = (\xvec{z}^{0} \cdot \xdop{\nabla}) \widetilde{\xvec{v}} - (\widetilde{\xvec{v}}\cdot \xdop{\nabla})\xvec{z}^{0}, \quad \mathfrak{f}z\partial_z\widetilde{\xvec{v}} = \xvec{0}
		\]
		in all of $\overline{\mathcal{E}}\times\mathbb{R}_+\times\mathbb{R}_+$, it follows that $\widetilde{\xvec{v}}$ satisfies a version of \eqref{equation:MHD_ElsaesserBlProfilevpm_contrALT} where the first line is replaced by
		\[
		\partial_t \widetilde{\xvec{v}} - \nu_2\partial_{zz}\widetilde{\xvec{v}} + \left[ (\xvec{z}^{0} \cdot \xdop{\nabla}) \widetilde{\xvec{v}} - (\widetilde{\xvec{v}}\cdot \xdop{\nabla})\xvec{z}^{0} \right]_{\operatorname{tan}} + \mathfrak{f}z\partial_z\widetilde{\xvec{v}} = \xsym{{{\mu}}} \quad \mbox{in } \overline{\mathcal{E}} \times \mathbb{R}_+ \times \mathbb{R}_+,
		\]
		which corresponds to the equation for $\xvec{v}^+ - \xvec{v}^-$ derived from \eqref{equation:MHD_ElsaesserBlProfilevpm_contr}. 
		Thus, after verifying \eqref{equation:vpmsupp} for some $d^* \in (0,d)$, it follows retrospectively that the equation \eqref{equation:MHD_ElsaesserBlProfilevpm_contr_hom} for $\xvec{W}$ is correct even without assuming \eqref{equation:artificialassump}. 
		
		To show \eqref{equation:vpmsupp}, we apply the ideas from \cite[Section 3.4]{CoronMarbachSueur2020}. Hereto, for any~$\varrho \in (0,d)$, consider the tube $\overline{\mathcal{V}_{\varrho}} = \{ 0 \leq \varphi_{\mathcal{E}} \leq \varrho\}$ and define its maximal distance of influence during the time interval $[0,T]$ under the flow $\xmcal{Z}^{0}$ via
		\[
		\mathcal{J}(\varrho) \coloneqq \max\limits_{\stackrel{s,t \in [0,T]}{\xvec{x}\in \overline{\mathcal{V}_{\varrho}}}} \varphi_{\mathcal{E}}\left(\xmcal{Z}^{0}(\xvec{x},s,t)\right).
		\]
		The controls $\xsym{{{\mu}}}^k$ in \eqref{equation:pfevQ2casc} only act where pollution, caused by $\xsym{\mathfrak{G}}$, $\xsym{{{\mu}}}^l$, or $\xvec{Q}^{2l-2}$, arrives via the flow $\xmcal{Z}^{0}$ associated with $\xvec{z}^0$. Thus,~$\xsym{{{\mu}}}^0$ is supported in~$\mathcal{V}_{\mathcal{J}(\mathcal{J}(d^*))}$ and its action travels at most into~$\mathcal{V}_{\mathcal{J}(\mathcal{J}(\mathcal{J}(d^*)))}$. Consequently, the effects of~$\xsym{{{\mu}}}^{\widetilde{r}}$ are propagated into~$\mathcal{V}_{\mathcal{J}^{3\widetilde{r}}(r_1)}$. Since the $\xvec{x}$-support of $\xsym{\mathfrak{G}}$ is contained in~$\mathcal{V}_{d^*}$ by the definition of~$\chi_{\partial\mathcal{E}}$, 
		and in view of \eqref{equation:repW}, maintaining the $\xvec{x}$-support of $\widetilde{\xvec{v}}$ within~$\mathcal{V}$ is achieved by adjusting the support of $\chi_{\partial\mathcal{E}}$. Indeed, $\mathcal{J}$ is continuous and one has~$\mathcal{J}(0)~=~0$ because~$\xvec{z}^{0}$ is tangential to $\partial \mathcal{E}$, which yields the existence of $d^{*}_0 \in (0,d)$ with~$\mathcal{J}^{3\widetilde{r}}(d^{*}_0) < d$. Then, every choice $d^* \in (0, d^*_0)$ is suitable.
	\end{proof}
	
	The last part of the proof of \Cref{lemma:blcxi2} shows why one can take $d^* \in (0,d)$ in the definition of $\chi_{\partial\mathcal{E}}$ such that $|\xvec{n}| = 1$ in the $\xvec{x}$-support of $\xvec{v}^{+}-\xvec{v}^{-}$. A similar argument ensures such a property for $\xvec{v}^{+}+\xvec{v}^{-}$ and is valid for the general situation of \Cref{lemma:blcxi} (\cf~\cite[Section 3.4]{CoronMarbachSueur2020}).
	\begin{lmm}\label{lemma:chisupp}
		Let $\xsym{\mu}^{\pm}$ be obtained by \Cref{lemma:blcxi} or \Cref{lemma:blcxi2}. There exists $d_0 \in (0,d)$ for which any choice $d^{*} \in (0, d_0)$ guarantees that the $\xvec{x}$-supports of $\xvec{v}^{\pm}$ are included in $\mathcal{V}$.
	\end{lmm}
	
	Now, on the unbounded time interval $t \geq T$ either \Cref{lemma:blcxi} or \Cref{lemma:blcxi2} is employed with $k = 4$, $p=5$, $s=3$ and $r = 6$ in order to fix $\xsym{\mu}^{\pm}$. Subsequently, in order to fix $\chi_{\partial\mathcal{E}}$ in \eqref{equation:def_chi}, any $d^* \in (0,d_0)$ is selected, where $d_0$ is the one determined in \Cref{lemma:chisupp}.
	
	\begin{rmrk}
		When $\xvec{M}_2 = \xvec{0}$ in \eqref{equation:bc2}, then, as seen in \Cref{lemma:M2}, the first case of \Cref{lemma:blcxi2} can be applied. In the case of \Cref{theorem:annulus} we have \eqref{equation:condmuwkez_0}, which allows to employ the third part of \Cref{lemma:blcxi2}.
	\end{rmrk}

	\subsubsection{Properties of the boundary layers and technical profiles}\label{subsubsection:firstprop}
	Due to the fast variable scaling for the boundary layer profiles $\xvec{v}^{\pm}$, $\xvec{w}^{\pm}$, and~$q^{\pm}$, several estimates will profit from a gain of order $O(\epsilon^{1/4})$ as stated below.
	
	\begin{lmm}[{\cite[Lemma 3]{IftimieSueur2011}}]\label{lemma:epsgain}
		There exists a constant $C > 0$ such that, for all $\epsilon > 0$ and functions $h = h(\xvec{x},z)$ in $\xLtwo_z(\mathbb{R}_+;\xHone_{\xvec{x}}(\mathcal{E}))$ with $\cup_{z\in\mathbb{R}_+}\operatorname{supp}(h(\cdot,z)) \subset \mathcal{V}$, it holds
		\[
		\|h \left(\cdot, \frac{\varphi_{\mathcal{E}}(\cdot)}{\sqrt{\epsilon}}\right)\|_{\xLtwo(\mathcal{E})} \leq C \epsilon^{\frac{1}{4}} \|h\|_{\xLtwo_z(\mathbb{R}_+;\xHone_{\xvec{x}}(\mathcal{E}))}.
		\]
	\end{lmm}
	
	\begin{lmm}\label{lemma:propqpm}
		The functions $q^{\pm}$ from \eqref{equation:MHD_ElsaesserBlProfilepressurepm_control} satisfy $\operatorname{supp}(q^{\pm}(\cdot,\cdot,z)) \subset \operatorname{supp}(\xvec{z}^0)$ for all~$z \in \mathbb{R}_+$. In addition, there exists a constant $C > 0$ independent of $\epsilon > 0$ such that
		\begin{equation}\label{equation:hestqrdq}
			\|\left\llbracket \xdop{\nabla} q^{\pm}\right\rrbracket_{\epsilon}(\cdot,t)\|_{\xLtwo(\mathcal{E})} \leq \epsilon^{\frac{1}{4}}C \sum\limits_{\square\in\{+,-\}} \|\xvec{v}^{\square}(\cdot,t,\cdot)\|_{\xHn{{1,3,0}}_{\mathcal{E}}}.
		\end{equation}
	\end{lmm}
	\begin{proof}
		One utilizes the properties of $q^{\pm}$ observed from \eqref{equation:MHD_ElsaesserBlProfilepressurepm_control}, the $\epsilon^{1/4}$ gain due to the fast variable scaling (\cf~\Cref{lemma:epsgain}), and integration by parts. Indeed, for $C > 0$ and small $\ell \in (0,1)$, both independent of $\epsilon > 0$, one has
		\begin{equation*}\label{equation:prflmgradqpmctrl}
			\begin{aligned}
				\|\left\llbracket \xdop{\nabla} q^{\pm}\right\rrbracket_{\epsilon}(\cdot,t)\|_{\xLtwo(\mathcal{E})}^2 & \leq \epsilon^{\frac{1}{2}} C \| \xdop{\nabla} q^{\pm}(\cdot,t) \|_{\xLtwo_z(\xvec{R}_+;\xHone(\mathcal{E}))}^2\\
				& \leq \epsilon^{\frac{1}{2}} C \sum\limits_{|\xsym{\alpha}| \leq 1} \int_{\mathbb{R}_+} \int_{\mathcal{E}} |\partial_{\xvec{x}}^{\xsym{\alpha}} \xdop{\nabla} q^{\pm}|^2 \partial_z z \, \xdx{\xvec{x}} \xdx{z} \\
				& \leq \epsilon^{\frac{1}{2}} C(\ell) \sum\limits_{|\xsym{\alpha}| \leq 1} \int_{\mathbb{R}_+}\int_{\mathcal{E}} (1+z^2)|\partial_{\xvec{x}}^{\xsym{\alpha}} \xdop{\nabla} \partial_zq^{\pm}|^2  \, \xdx{\xvec{x}} \xdx{z} \\
				& \quad + \epsilon^{\frac{1}{2}}\ell \sum\limits_{|\xsym{\alpha}| \leq 1} \int_{\mathbb{R}_+}\int_{\mathcal{E}} |\partial_{\xvec{x}}^{\xsym{\alpha}} \xdop{\nabla}q^{\pm}|^2  \, \xdx{\xvec{x}} \xdx{z},
			\end{aligned}
		\end{equation*}
		which implies that
		\begin{equation*}
			\begin{aligned}
				\|\left\llbracket \xdop{\nabla} q^{\pm}\right\rrbracket_{\epsilon}(\cdot,t)\|_{\xLtwo(\mathcal{E})}^2 & \leq \epsilon^{\frac{1}{2}} C_1
				\sum\limits_{|\xsym{\alpha}| \leq 1} \int_{\mathbb{R}_+}\int_{\mathcal{E}} (1+z^2)|\partial_{\xvec{x}}^{\xsym{\alpha}} \xdop{\nabla} \partial_zq^{\pm}|^2  \, \xdx{\xvec{x}} \xdx{z} \\
				& \leq \epsilon^{\frac{1}{2}} C_2 \sum\limits_{\square\in\{+,-\}} \|\xvec{v}^{\square}(\cdot,t,\cdot)\|_{\xHn{{1,3,0}}_{\mathcal{E}}}^2
			\end{aligned}
		\end{equation*}
		for two constants $C_1, C_2$, by using 
		\eqref{equation:MHD_ElsaesserBlProfilepressurepm_control}.
	\end{proof}
	
	\begin{lmm}\label{lemma:propwpm}
		For all $k,m,s \in \mathbb{N}$, the profiles $\xvec{w}^{\pm}$ determined in \eqref{equation:MHD_ElsaesserBlProfilesecondvelpm_contrl} satisfy
		\begin{align}
			\|\xvec{w}^{\pm}(\cdot,t,\cdot)\|_{\xHn{{k,m,s}}_{\mathcal{E}}} & \leq C \sum\limits_{\square\in\{+,-\}} \|\xvec{v}^{\square}(\cdot,t,\cdot)\|_{\xHn{{k+1,m+1,\max\{1,s-1\}}}_{\mathcal{E}}},\label{equation:lmmblregwlinfzebv}\\
			\| \left\llbracket \partial_{z}^2 \xvec{w}^{\pm}\right\rrbracket_{\epsilon}(\cdot,t)\|_{\xLtwo(\mathcal{E})} & \leq \epsilon^{\frac{1}{4}} C \sum\limits_{\square\in\{+,-\}} \| \xvec{v}^{\square}(\cdot,t,\cdot) \|_{\xHn{{1,2,1}}_{\mathcal{E}}}, \label{equation:lmmblregwlinfzbl1}\\
			\| \left\llbracket \partial_t \xvec{w}^{\pm}\right\rrbracket_{\epsilon}(\cdot,t) \|_{\xLtwo(\mathcal{E})} & \leq \epsilon^{\frac{1}{4}} C  \sum\limits_{\square\in\{+,-\}} \| \xvec{v}^{\square}(\cdot,t,\cdot) \|_{\xHn{{2,3,3}}_{\mathcal{E}}} + C \| \xsym{\xvec{\mu}}^{\pm}(\cdot,t,\cdot) \|_{\xHn{{1,2,1}}_{\mathcal{E}}}. \label{equation:lmmblregwlinfzbl2}
		\end{align}
	\end{lmm}
	\begin{proof}
		One can show \eqref{equation:lmmblregwlinfzebv} by separately estimating the tangential and normal parts
		\begin{align*}
			\xvec{w}^{\pm}_T(\xvec{x},t,z) & \coloneqq -2\operatorname{e}^{-z}\xmcal{N}^{\pm}(\xvec{v}^{+},\xvec{v}^{-})(\xvec{x},t,0), \\
			\xvec{w}^{\pm}_N(\xvec{x},t,z) & \coloneqq - \xvec{n}(\xvec{x})\int_z^{+\infty} \xdop{\nabla}\cdot\xvec{v}^{\pm}(\xvec{x},t,s) \, \xdx{s}.
		\end{align*}
		For instance, integration by parts yields
		\begin{multline*}
			\|\xvec{w}^{\pm}_N(\cdot,t,z)\|_{\xHn{{k,m,0}}_{\mathcal{E}}}^2  \leq C\sum\limits_{|\beta| \leq m} \int_{\mathcal{E}} \int_{\mathbb{R}_+} \partial_z(z+z^{2k+1}) \left| \int_z^{+\infty} \partial_{\xvec{x}}^{\beta} (\xdop{\nabla}\cdot\xvec{v}^{\pm})(\xvec{x},t,s) \, \xdx{s} \right|^2 \, \xdx{z}  \xdx{\xvec{x}} \\
			\leq C\sum\limits_{|\beta| \leq m} \int_{\mathcal{E}} \int_{\mathbb{R}_+} \Bigg| (z+z^{2k+1}) \partial_{\xvec{x}}^{\beta} (\xdop{\nabla}\cdot\xvec{v}^{\pm})(\xvec{x},t,z) \left(\int_z^{+\infty} \partial_{\xvec{x}}^{\beta}  (\xdop{\nabla}\cdot\xvec{v}^{\pm})(\xvec{x},t,s) \, \xdx{s} \right) \Bigg| \, \xdx{z}  \xdx{\xvec{x}}.
		\end{multline*}
		Hence, for arbitrary $\ell > 0$ one has
		\begin{equation*}
			\begin{multlined}
				\|\xvec{w}^{\pm}_N(\cdot,t,z)\|_{\xHn{{k,m,0}}_{\mathcal{E}}}^2 \leq C(\ell)\sum\limits_{|\beta| \leq m+1} \int_{\mathcal{E}} \int_{\mathbb{R}_+} (1+z^{2k+2}) |\partial_{\xvec{x}}^{\beta}\xvec{v}^{\pm}(\xvec{x},t,z)|^2 \, \xdx{z}  \xdx{\xvec{x}}\\
				+ \ell \sum\limits_{|\beta| \leq m} \int_{\mathcal{E}} \int_{\mathbb{R}_+} \partial_z(z+z^{2k+1}) \left| \int_z^{+\infty} \partial_{\xvec{x}}^{\beta}  (\xdop{\nabla}\cdot\xvec{v}^{\pm})(\xvec{x},t,s) \, \xdx{s} \right|^2 \, \xdx{z}  \xdx{\xvec{x}}.
			\end{multlined}
		\end{equation*}
		Thus, by choosing small $\ell > 0$ and a new constant $C = C(\ell) > 0$, one obtains
		\[
			\|\xvec{w}^{\pm}_N(\cdot,t\cdot)\|_{\xHn{{k,m,0}}_{\mathcal{E}}} \leq C \|\xvec{v}^{\pm}(\cdot,t,\cdot)\|_{\xHn{{k+1,m+1,0}}_{\mathcal{E}}}.
		\]
		The other aspects of \eqref{equation:lmmblregwlinfzebv} are along the same lines, noting that estimating $\xvec{w}^{\pm}_T$ costs one regularity level in $z$ due the application of a trace theorem. The estimate \eqref{equation:lmmblregwlinfzbl1} follows from a combination of \Cref{lemma:epsgain}, the above idea for showing \eqref{equation:lmmblregwlinfzebv} and the identity
		\[
			\partial_{z}^2 \xvec{w}^{\pm} = (\xdiv{\partial_z \xvec{v}^{\pm}}) \xvec{n} - 2\operatorname{e}^{-z} \left[\xsym{\mathcal{N}}^{\pm}(\xvec{v}^+, \xvec{v}^-)\right]_{|_{z = 0}}.
		\]
		Regarding \eqref{equation:lmmblregwlinfzbl2}, the starting point is to derive from \eqref{equation:MHD_ElsaesserBlProfilesecondvelpm_contrl} the representation
		\[
			\partial_t \xvec{w}^{\pm} = \partial_t \overline{w}^{\pm} \xvec{n} -2\operatorname{e}^{-z} \left[\xsym{\mathcal{N}}^{\pm}(\partial_t \xvec{v}^+,\partial_t \xvec{v}^-)\right]_{|_{z = 0}},
		\]
		into which one can subsequently insert the equation \eqref{equation:MHD_ElsaesserBlProfilevpm_contr} and proceed as before.
	\end{proof}
	
	\begin{lmm}\label{lemma:wpth}
		The Neumann problems \eqref{equation:divcorrNeum} are well-posed, with uniqueness of solutions up to a constant, and all solutions $\theta^{\pm,\epsilon}$ obey for $l \in \{0,1,2\}$ the estimates
		\begin{gather}
			\|\theta^{\pm,\epsilon}(\cdot, t)\|_{\xHn{{2+l}}(\mathcal{E})}  \leq \epsilon^{\frac{1}{4}-\frac{l}{2}} C\| \xvec{w}^{\pm} (\cdot,t,\cdot)\|_{\xHn{{0,2+l,l}}_{\mathcal{E}}} + C\sum\limits_{\square\in\{+,-\}}\| \xvec{v}^{\square} (\cdot,t,\cdot)\|_{\xHn{{1,1+l,0}}_{\mathcal{E}}}.\label{equation:potflowthpmlap}
		\end{gather}
		If $(\xvec{w}^{+}(\xvec{x},t,\cdot)-\xvec{w}^{-}(\xvec{x},t,\cdot))\cdot \xvec{n} = 0$ for all $\xvec{x} \in \partial \mathcal{E}$, it additionally holds
		\begin{gather}
			\|\theta^{+,\epsilon}(\cdot, t)-\theta^{-,\epsilon}(\cdot, t)\|_{\xHn{{2}}(\mathcal{E})}  \leq \epsilon^{\frac{1}{4}} C \sum\limits_{\square\in\{+,-\}} \|\xvec{v}^{\square} (\cdot,t,\cdot)\|_{\xHn{{1,3,1}}_{\mathcal{E}}}.\label{equation:potflowthpmlap2}
		\end{gather}
		Furthermore, for all $t \in [0, T/\epsilon]$, one has
		\begin{align}
			\|\Delta \xdop{\nabla} \theta^{\pm,\epsilon}(\cdot,t)\|_{\xLtwo(\mathcal{E})} \leq  \epsilon^{-\frac{1}{4}} C \sum\limits_{\square\in\{+,-\}} \| \xvec{v}^{\square}(\cdot,t,\cdot) \|_{\xHn{{2,4,2}}_{\mathcal{E}}}. \label{equation:lmmblregthetalinfzl2}
		\end{align}
	\end{lmm}
	\begin{proof}
		In \eqref{equation:divcorrNeum}, there is no coupling between $\pm$ superscribed functions. Thus, the well-posedness of \eqref{equation:divcorrNeum} together with \eqref{equation:potflowthpmlap} and \eqref{equation:lmmblregthetalinfzl2} can be established by analysis similar to \cite[Equations (4.29), (4.31)--(4.33) and (4.58)]{CoronMarbachSueur2020}. In particular, by employing \eqref{equation:comfo}, \eqref{equation:lmmblorth}, \eqref{equation:MHD_ElsaesserBlProfilesecondvelpm_contrl}, and \eqref{equation:MHD_ElsaesserBlProfilesecondvelprop_control}, one can verify the necessary compatibility conditions for \eqref{equation:divcorrNeum} via
		\begin{multline*}
			\int_{\partial\mathcal{E}} \xvec{w}^{\pm}(\xvec{x},t,0) \cdot \xvec{n}(\xvec{x}) \, \xdx{S(\xvec{x})} = \int_{\partial\mathcal{E}} \left\llbracket \xvec{w}^{\pm}\right\rrbracket_{\epsilon}(\xvec{x},t) \cdot \xvec{n}(\xvec{x}) \, \xdx{S(\xvec{x})} = \int_{\mathcal{E}} \xdiv{\left\llbracket \xvec{w}^{\pm}\right\rrbracket_{\epsilon}}(\xvec{x},t) \, \xdx{\xvec{x}}\\
			\begin{aligned}
				& = \int_{\mathcal{E}} \left(\left\llbracket \xdiv{\xvec{w}^{\pm}}\right\rrbracket_{\epsilon} - \frac{1}{\sqrt{\epsilon}} \xvec{n} \cdot \left\llbracket \partial_z\xvec{w}^{\pm}\right\rrbracket_{\epsilon}\right) (\xvec{x},t) \, \xdx{\xvec{x}}\\
				& = \int_{\mathcal{E}} \left(\left\llbracket \xdiv{\xvec{w}^{\pm}}\right\rrbracket_{\epsilon} - \frac{1}{\sqrt{\epsilon}} \xdiv{\left\llbracket \xvec{v}^{\pm}\right\rrbracket_{\epsilon}} - \frac{1}{\epsilon} \xvec{n} \cdot \left\llbracket \partial_z\xvec{v}^{\pm}\right\rrbracket_{\epsilon}\right) (\xvec{x},t) \, \xdx{\xvec{x}}\\
				& = \int_{\mathcal{E}} \left\llbracket \xdiv{\xvec{w}^{\pm}}\right\rrbracket_{\epsilon} (\xvec{x},t) \, \xdx{\xvec{x}}.
			\end{aligned}
		\end{multline*}
		It remains to show \eqref{equation:potflowthpmlap2} when $(\xvec{w}^{+}-\xvec{w}^{-})\cdot \xvec{n} = 0$ at $\partial \mathcal{E}$. By elliptic regularity for weak solutions to the Laplace equation \eqref{equation:divcorrNeum}, \Cref{lemma:epsgain}, and \eqref{equation:lmmblregwlinfzebv} one obtains
		\begin{multline*}
			\|(\theta^{+,\epsilon}-\theta^{-,\epsilon})(\cdot, t)\|_{\xHn{{2}}(\mathcal{E})} \\ \leq C\|\left\llbracket \xdiv{(\xvec{w}^{+}-\xvec{w}^{-})}(\cdot,t,\cdot)\right\rrbracket_{\epsilon}\|_{\xLtwo(\mathcal{E})} \leq \epsilon^{\frac{1}{4}}C \sum\limits_{\square\in\{+,-\}} \|\xvec{v}^{\square}\|_{\xHn{{1,3,1}}_{\mathcal{E}}}. 
		\end{multline*}
	\end{proof}
	
	Finally, several properties of the remainder terms $\xvec{r}^{\pm,\epsilon}$ in the ansatz \eqref{equation:ansatz2} are summarized.
	\begin{lmm}\label{lemma:prpbcicrmpm}
		The remainder terms $\xvec{r}^{\pm,\epsilon}$ given in \eqref{equation:ansatz2} satisfy the conditions
		\begin{equation}\label{equation:bcdivfrem}
			\begin{aligned}
				& \xvec{r}^{\pm,\epsilon}(\cdot,0)  = \xvec{0} && \mbox{ in } \mathcal{E},\\
				& \xdop{\nabla}\cdot\xvec{r}^{\pm,\epsilon} = 0 && \mbox{ in } \mathcal{E}\times\mathbb{R}_+,\\
				& \xvec{r}^{\pm,\epsilon} \cdot \xvec{n} = 0 && \mbox{ on } \partial \mathcal{E}\times\mathbb{R}_+,\\
				&\xmcal{N}^{\pm}(\xvec{r}^{+,\epsilon},\xvec{r}^{-,\epsilon}) = \xsym{\mathfrak{g}}^{\pm,\epsilon} && \mbox{ on } \partial \mathcal{E}\times\mathbb{R}_+,
			\end{aligned}
		\end{equation}
		where
		\[
			\xsym{\mathfrak{g}}^{\pm,\epsilon} \coloneqq - \xmcal{N}^{\pm}(\xvec{z}^{+,1},\xvec{z}^{-,1}) - \xmcal{N}^{\pm}(\xdop{\nabla}\theta^{+,\epsilon},\xdop{\nabla}\theta^{-,\epsilon}) - \xmcal{N}^{\pm}(\xvec{w}^{+},\xvec{w}^{-})_{|_{z=0}}.
		\]
		Moreover, for a constant $C > 0$ independent of $\epsilon > 0$, the boundary data $\xsym{\mathfrak{g}}^{\pm,\epsilon}$ can be estimated by
		\begin{equation}\label{equation:wblresbdtest}
			\| \xsym{\mathfrak{g}}^{\pm,\epsilon}\|_{\xLtwo((0,T/\epsilon); \xHone(\mathcal{E}))}^2 \leq
			C \sum\limits_{\square\in\{+,-\}} \left(\| \xvec{z}^{\square,1} \|_{\xLtwo((0,T); \xHtwo(\mathcal{E}))}^2 + \epsilon^{-\frac{1}{2}} \| \xvec{v}^{\square} \|_{\xLtwo((0,T); \xHn{{3,4,2}}_{\mathcal{E}})}^2\right).
		\end{equation}
	\end{lmm}
	\begin{proof}
		The assertions in \eqref{equation:bcdivfrem} are a consequence of the definitions for the functions in \eqref{equation:ansatz2}--\eqref{3.7} studied in Subsections~ \Rref{subsection:return}, \Rref{subsection:flushing} and \Rref{subsection:blprf}. Particularly, in view  of \Cref{remark:nbhd}, \Cref{lemma:chisupp}, the relations in \eqref{equation:MHD_ElsaesserBlProfilesecondvelprop_control}, and the boundary conditions \eqref{equation:MHD_ElsaesserBlProfilevpm_contribc}, one has
		\begin{alignat*}{3}
			2\xdop{\mathcal{N}^{\pm}}(\xvec{z}^{0},\xvec{z}^{0}) = \left[\partial_z \xvec{v}^{\pm}\right]_{\operatorname{tan}} |_{z = 0}, & \quad &
			2\xmcal{N}^{\pm}(\xvec{v}^{+},\xvec{v}^{-}) |_{z = 0} = \left[\partial_z \xvec{w}^{\pm}\right]_{\operatorname{tan}} |_{z = 0}.
		\end{alignat*}
	\end{proof}

	\subsection{Remainder estimates}\label{subsection:remestwbctrl}
	The goal is now to show that $\|\xvec{z}^{\pm,\epsilon}(\cdot,T/\epsilon)\|_{\xLtwo(\mathcal{E})} = O(\epsilon^{9/8})$ as $\epsilon \longrightarrow 0$. As various estimates are similar to those for the Navier--Stokes problem in \cite{CoronMarbachSueur2020}, we place emphasis on the arguments that are specific for the current MHD problem.
	
	\subsubsection{Equations satisfied by the remainders}\label{subsubsection:remaindereq}
	In order to derive the equations satisfied by the remainders $\xvec{r}^{\pm,\epsilon}$, the definitions of $q^{\pm}$ given in \eqref{equation:MHD_ElsaesserBlProfilepressurepm_control} are employed for rewriting \eqref{equation:MHD_ElsaesserBlProfilevpm_contr} in the form
	\begin{equation}\label{equation:blvaf}
		\partial_t \xvec{v}^{\pm} - \partial_{zz}(\lambda^{\pm}\xvec{v}^{+} + \lambda^{\mp}\xvec{v}^{-}) + (\xvec{z}^{0} \cdot \xdop{\nabla}) \xvec{v}^{\pm} + (\xvec{v}^{\mp} \cdot \xdop{\nabla})\xvec{z}^{0} + \mathfrak{f}z\partial_z\xvec{v}^{\pm} - \xvec{n} \partial_z q^{\pm} = \xsym{\xvec{\mu}}^{\pm} .
	\end{equation}
	Then, \eqref{equation:ansatz2}--\eqref{3.7} are inserted into \eqref{equation:MHD_ElsaesserExtScaledLongTime} while using \eqref{equation:comfo}, \eqref{equation:blvaf} and \Cref{lemma:prpbcicrmpm}. Since the terms $\llbracket (\xvec{v}^{\mp}\cdot\xvec{n})\partial_z\xvec{v}^{\pm}\rrbracket_{\epsilon} $ and $\llbracket (\xvec{v}^{\mp}\cdot\xvec{n})\partial_z\xvec{w}^{\pm}\rrbracket_{\epsilon} $ vanish, the remainders $\xvec{r}^{\pm,\epsilon}$ satisfy the following problem
	\begin{equation}\label{equation:MHD_ElsaesserExtLtScaled_raminder_blctrl}
		\begin{cases}
			\partial_t \xvec{r}^{\pm,\epsilon} - \epsilon\Delta (\lambda^{\pm}\xvec{r}^{+,\epsilon} + \lambda^{\mp}\xvec{r}^{-,\epsilon}) + (\xvec{z}^{\mp,\epsilon} \cdot \xdop{\nabla}) \xvec{r}^{\pm,\epsilon} + \xdop{\nabla} \pi^{\pm,\epsilon} \! = \! \left\llbracket \xvec{h}^{\pm,\epsilon} - A^{\epsilon,\pm} \xvec{r}^{\mp,\epsilon}\right\rrbracket_{\epsilon} & \!\!\!\!\!\! \mbox{ in } \mathcal{E}_{T/\epsilon},\\
			\xdop{\nabla}\cdot\xvec{r}^{\pm,\epsilon} = 0  & \!\!\!\!\!\! \mbox{ in } \mathcal{E}_{T/\epsilon},\\
			\xvec{r}^{\pm,\epsilon} \cdot \xvec{n} = 0 & \!\!\!\!\!\! \mbox{ on }  \Sigma_{T/\epsilon}, \\
			\xmcal{N}^{\pm}(\xvec{r}^{+,\epsilon},\xvec{r}^{-,\epsilon}) = \xsym{\mathfrak{g}}^{\pm,\epsilon} & \!\!\!\!\!\! \mbox{ on }  \Sigma_{T/\epsilon},\\
			\xvec{r}^{\pm,\epsilon}(\cdot, 0) = \xvec{0} & \!\!\!\!\!\! \mbox{ in } \mathcal{E},
		\end{cases}
	\end{equation}
	where the amplification terms $A^{\epsilon,\pm} \xvec{r}^{\mp,\epsilon}$ are given by
	\begin{equation*}
		A^{\epsilon,\pm} \xvec{r}^{\mp,\epsilon} \!
		\coloneqq \! (\xvec{r}^{\mp,\epsilon} \cdot \xdop{\nabla}) \! \left( \xvec{z}^{0} + \sqrt{\epsilon} \xvec{v}^{\pm} + \epsilon \xvec{z}^{\pm,1} + \epsilon \xdop{\nabla}\theta^{\pm,\epsilon} + \epsilon \xvec{w}^{\pm} \right) - (\xvec{r}^{\mp,\epsilon} \cdot \xvec{n}) \! \left(\partial_z \xvec{v}^{\pm} + \sqrt{\epsilon} \partial_z \xvec{w}^{\pm} \right),
	\end{equation*}
	and the remaining terms
	\[
	\xvec{h}^{\pm,\epsilon} \coloneqq \lambda^{\pm}\xvec{h}^{+,\epsilon,1} + 	\lambda^{\mp}\xvec{h}^{-,\epsilon,1} + \xvec{h}^{\pm,\epsilon,2} -\partial_t \xvec{w}^{\pm} - \xdop{\nabla}q^{\pm},
	\]
	contain
	\begin{equation*}\label{equation:remtrm_blctrl}
		\begin{aligned}
			\xvec{h}^{\pm,\epsilon,1} & \coloneqq  \left(\partial_{z}^2 \xvec{w}^{\pm} +  \Delta \varphi_{\mathcal{E}} \partial_z \xvec{v}^{\pm} - 2 (\xvec{n}\cdot\xdop{\nabla})\partial_z \xvec{v}^{\pm}\right) + \epsilon \left( \Delta \xvec{w}^{\pm} + \Delta \xvec{z}^{\pm,1} + \Delta \xdop{\nabla}\theta^{\pm,\epsilon} \right)\\
			& \quad \, + \sqrt{\epsilon} \left( \Delta \xvec{v}^{\pm} + \Delta \varphi_{\mathcal{E}}  \partial_z \xvec{w}^{\pm} -2(\xvec{n} \cdot \xdop{\nabla})\partial_z \xvec{w}^{\pm} \right),
		\end{aligned}
	\end{equation*}
	and
	\begin{equation*}
		\begin{aligned}
			\xvec{h}^{\pm,\epsilon,2} & \coloneqq -\left\{ \left[ \xvec{v}^{\mp} + \sqrt{\epsilon}\left(\xvec{w}^{\mp} + \xvec{z}^{\mp,1} + \xdop{\nabla}\theta^{\mp,\epsilon} \right) \right] \cdot \xdop{\nabla} \right\} \left[ \xvec{v}^{\pm} + \sqrt{\epsilon} \left( \xvec{w}^{\pm} + \xvec{z}^{\pm,1} + \xdop{\nabla}\theta^{\pm,\epsilon} \right) \right] \\
			& \quad \, - (\xvec{z}^{0} \cdot \xdop{\nabla})\xvec{w}^{\pm} - (\xvec{w}^{\mp} \cdot \xdop{\nabla}) \xvec{z}^{0} + \left[\left(\xvec{w}^{\mp} + \xvec{z}^{\mp,1} + \xdop{\nabla}\theta^{\mp,\epsilon}\right) \cdot \xvec{n}\right]\partial_z\left( \xvec{v}^{\pm} + \sqrt{\epsilon} \xvec{w}^{\pm} \right)\\
			& \quad \, + \widetilde{\xsym{\zeta}}^{\pm,\epsilon} - \left((\xvec{z}^0 \cdot \xnab)\xnab\theta^{\pm,\epsilon} + (\xnab\theta^{\mp,\epsilon} \cdot \xnab)\xvec{z}^0 - \xnab(\xvec{z}^0\cdot \xnab \theta^{\pm,\epsilon})\right) + \frac{1}{\sqrt{\epsilon}}\xvec{z}^0 \cdot \xvec{n} \partial_z \xvec{w}^{\pm}.
		\end{aligned}
	\end{equation*}
	
	Before deriving energy estimates for $\xvec{r}^{\pm,\epsilon}$, several asymptotic properties of the right-hand side terms in \cref{equation:MHD_ElsaesserExtLtScaled_raminder_blctrl} are summarized.
	\begin{lmm}\label{lemma:remaindest} 
		For all $t \in [0, T/\epsilon]$, one has
		\begin{gather}\label{equation:remesblAgfl1}
			\|\left\llbracket A^{\epsilon,\pm} \xvec{r}^{\pm,\epsilon} \cdot \xvec{r}^{\mp,\epsilon} \right\rrbracket_{\epsilon}\|_{\xLone(\mathcal{E}\times(0,t))} \leq C \left(\|\xvec{r}^{+,\epsilon}\|^2_{\xLtwo(\mathcal{E}\times(0,t))} + \|\xvec{r}^{-,\epsilon}\|^2_{\xLtwo(\mathcal{E}\times(0,t))}\right)
		\end{gather}
		with a constant $C > 0$  independent of $t$, $\epsilon$ and $\xvec{r}^{\pm,\epsilon}$. Furthermore, as $\epsilon \longrightarrow 0$, it holds
		\begin{align}
			\|\xsym{\mathfrak{g}}^{\pm,\epsilon}\|_{\xLtwo((0,T/\epsilon); \xHone(\mathcal{E}))}^2 = O(\epsilon^{-\frac{1}{2}}), \label{equation:remesblAgfl2} \quad
			\|\left\llbracket \xvec{h}^{\pm,\epsilon}\right\rrbracket_{\epsilon}\|_{\xLone((0,T/\epsilon); \xLtwo(\mathcal{E}))}^2 =  O(\epsilon^{\frac{1}{2}}). 
		\end{align}
	\end{lmm}
	\begin{proof}
		Throughout, it will be used that $\xsym{\mu}^{\pm}$ have been fixed in \Cref{subsubsection:vmo} either by \Cref{lemma:blcxi} (or \Cref{lemma:blcxi2}) applied with $k = 4$, $p=5$, $s=3$ and $r = 6$. This provides bounds for $\xvec{v}^{\pm}$ in $\xLone((0,T/\epsilon);\xHn{{k,p,s}}_{\mathcal{E}})$ which are uniform in $\epsilon$, since for $\ell > 0$ and $a \in (1,+\infty)$ one has the convergence of the integral
		\[
			\int_0^{\infty} \left(\frac{\ln(2+ \ell s)}{2+\ell s}\right)^{a} \, \xdx{s} \leq C = C(\lambda,a) < + \infty.
		\]
		
		In order to show \eqref{equation:remesblAgfl1}, let $\xvec{A}^{\pm,\epsilon} $ denote the functions
		\[
		\xvec{A}^{\pm,\epsilon} \coloneqq 
		\xdop{\nabla}\left( \xvec{z}^{0} + \sqrt{\epsilon} \xvec{v}^{\pm} + \epsilon \xvec{z}^{\pm,1} + \epsilon \xdop{\nabla}\theta^{\pm,\epsilon} + \epsilon \xvec{w}^{\pm} \right) - \left( \partial_z \xvec{v}^{\pm} + \sqrt{\epsilon} \partial_z \xvec{w}^{\pm} \right) \xvec{n}^{\top}.
		\]
		From \Cref{lemma:flushing} one knows that $\xvec{z}^{\pm,1}$ are bounded in $\xLinfty((0,T);\xHn{3}(\mathcal{E}))$ as long as the initial data satisfy $\xvec{z}^{\pm}_0 \in \xHn{3}(\mathcal{E})\cap \xH(\mathcal{E})$. Hence, combining Sobolev embeddings with \Cref{lemma:wpth} allows to infer
		\begin{multline*}
			\|\left\llbracket \xvec{A}^{\pm,\epsilon}\right\rrbracket_{\epsilon}\|_{\xLinfty(\mathcal{E}\times(0,T/\epsilon))}
			\\	\leq  C \sum\limits_{\square\in\{+,-\}} \left( \|\xvec{z}^{0}\|_{\xLinfty((0,T);\xHn{3}(\mathcal{E}))} + \epsilon \|\xvec{z}^{\square,1}\|_{\xLinfty((0,T);\xHn{3}(\mathcal{E}))} + \|\xvec{v}^{\square}\|_{\xLinfty((0,T);\xHn{{2,5,3}}_{\mathcal{E}})}\right).
		\end{multline*}
		Moreover, by \Cref{lemma:blcxi} or \Cref{lemma:blcxi2} one finds a constant $C > 0$ independent of~$\epsilon$ with
		\[
		\|\xvec{v}^{\pm}\|_{\xLinfty((0,T);\xHn{{2,5,3}}_{\mathcal{E}})} \leq C,
		\]
		which eventually implies \eqref{equation:remesblAgfl1}.
		
		The bounds for $\xsym{\mathfrak{g}}^{\pm,\epsilon}$ in \eqref{equation:remesblAgfl2} follow from \eqref{equation:wblresbdtest} and \Cref{lemma:blcxi} (or \Cref{lemma:blcxi}) such that it remains to establish the estimates for $\xvec{h}^{\pm,\epsilon}$ in \eqref{equation:remesblAgfl2}. We begin with the terms
		\[
		\xvec{a}^{\pm} \coloneqq \widetilde{\xsym{\zeta}}^{\pm,\epsilon} - \left((\xvec{z}^0 \cdot \xnab)\xnab\theta^{\pm,\epsilon} + (\xnab\theta^{\mp,\epsilon} \cdot \xnab)\xvec{z}^0 - \xnab(\xvec{z}^0\cdot \xnab \theta^{\pm,\epsilon})\right).
		\]
		According to the definition of $\widetilde{\xsym{\zeta}}^{\pm,\epsilon}$ in \Cref{subsubsection:technicalprofiles}, the terms $\xvec{a}^{\pm}$ vanish when $\xcurl{\xvec{z}^0} = \xvec{0}$ in~${\mathcal{E}_T}$ and $(\xvec{w}^{+}-\xvec{w}^{-})_{|_{z=0}} \cdot \xvec{n} \neq 0$ at some points of~$\Sigma_T$. When $(\xvec{w}^{+}-\xvec{w}^{-})_{|_{z=0}} \cdot \xvec{n} = 0$ is satisfied at~$\Sigma_T$, then
		\[
		\xvec{a}^{\pm} = -(\xnab (\theta^{\mp,\epsilon} - \theta^{\pm,\epsilon})\cdot \xnab) \xvec{z}^0.
		\]
		Concerning the latter case, the estimate \eqref{equation:potflowthpmlap2} implies
		\begin{equation*}
			\begin{aligned}
				\|(\xnab (\theta^{\mp,\epsilon} - \theta^{\pm,\epsilon})\cdot \xnab) \xvec{z}^0 \|_{\xLone((0,T/\epsilon); \xLtwo(\mathcal{E}))} \leq \epsilon^{\frac{1}{4}} C \sum\limits_{\square\in\{+,-\}} \| \xvec{v}^{\square} \|_{\xLone((0,T);\xHn{{1,3,1}}_{\mathcal{E}})}
			\end{aligned}
		\end{equation*}
		and an invocation of \Cref{lemma:blcxi} or \Cref{lemma:blcxi2} yields
		\begin{equation}\label{equation:aae}
			\| \xvec{a}^{\pm} \|_{\xLone((0,T/\epsilon); \xLtwo(\mathcal{E}))} = O(\epsilon^{\frac{1}{4}}).
		\end{equation}
		In order to treat the terms which appear with a factor $\epsilon^{-1/2}$ in $\xvec{h}^{\pm,\epsilon,2}$, we resort to a trick similar to \cite[Equation (69)]{IftimieSueur2011}. Indeed, the definitions for $\xvec{w}^{\pm}$ in \eqref{equation:MHD_ElsaesserBlProfilesecondvelpm_contrl} provide
		\begin{equation*}
			\begin{aligned}
				\left\llbracket \frac{(\xvec{z}^0 \cdot \xvec{n})}{\sqrt{\epsilon}} \partial_z \xvec{w}^{\pm} \right\rrbracket_{\epsilon} =  -\left\llbracket z \mathfrak{f}\left((\xdop{\nabla}\cdot\xvec{v}^{\pm})\xvec{n} + 2 \operatorname{e}^{-z}\left(\xmcal{N}^{\pm}(\xvec{v}^{+},\xvec{v}^{-})_{|_{z = 0}}\right) \right) \right\rrbracket_{\epsilon},
			\end{aligned}
		\end{equation*}
		which due to $\mathfrak{f}$ being bounded leads to
		\begin{multline*}
			\|\left\llbracket \frac{(\xvec{z}^0 \cdot \xvec{n})}{\sqrt{\epsilon}} \partial_z \xvec{w}^{\pm} \right\rrbracket_{\epsilon}\| _{\xLone((0,T/\epsilon); \xLtwo(\mathcal{E}))}\\ \leq  C\|\left\llbracket z\left[(\xdop{\nabla}\cdot\xvec{v}^{\pm})\xvec{n} + 2 \operatorname{e}^{-z}\left(\xmcal{N}^{\pm}(\xvec{v}^{+},\xvec{v}^{-})_{|_{z = 0}}\right) \right] \right\rrbracket_{\epsilon}\|_{\xLone((0,T/\epsilon); \xLtwo(\mathcal{E}))} 
			\leq  I_1^{\pm} + 2 I_2^{\pm},
		\end{multline*}
		where
		\begin{equation*}
			\begin{aligned}
				I_1^{\pm} & \coloneqq C\|\left\llbracket z (\xdop{\nabla}\cdot\xvec{v}^{\pm})\xvec{n} \right\rrbracket_{\epsilon}\|_{\xLone((0,T/\epsilon); \xLtwo(\mathcal{E}))},\\ 
				I_2^{\pm} & \coloneqq C\|\left\llbracket z \operatorname{e}^{-z}\left(\xmcal{N}^{\pm}(\xvec{v}^{+},\xvec{v}^{-})_{|_{z = 0}}\right) \right\rrbracket_{\epsilon}\|_{\xLone((0,T/\epsilon); \xLtwo(\mathcal{E}))}.
			\end{aligned}
		\end{equation*}
		Since $z (\xdop{\nabla}\cdot\xvec{v}^{\pm})\xvec{n} \in \xLtwo_z(\mathbb{R};\xHone_{\xvec{x}}(\mathcal{E}))$, \Cref{lemma:epsgain} can be applied to $I_1^{\pm}$ and, by similar analysis for $I_2^{\pm}$, there exists a constant $C > 0$ independent of $\epsilon$ such that
		\begin{equation*}
			\begin{aligned}
				I_1^{\pm} + I_2^{\pm} \leq \epsilon^{\frac{1}{4}}C \sum\limits_{\square\in\{+,-\}} \| \xvec{v}^{\square} \|_{\xLone((0,T);\xHn{{1,2,0}}_{\mathcal{E}})}.
			\end{aligned}
		\end{equation*}
		As a result, \Cref{lemma:blcxi} (respectively \Cref{lemma:blcxi2}) allows to infer
		\[
		\|\left\llbracket \frac{(\xvec{z}^0 \cdot \xvec{n})}{\sqrt{\epsilon}} \partial_z \xvec{w}^{\pm} \right\rrbracket_{\epsilon}\| _{\xLone((0,T/\epsilon); \xLtwo(\mathcal{E}))} = O(\epsilon^{\frac{1}{4}}).
		\]
		
		The remaining terms contained in $\xvec{h}^{+,\epsilon}$ and $\xvec{h}^{-,\epsilon}$ behave as in the Navier--Stokes case treated in \cite[Section 4.4]{CoronMarbachSueur2020}. Carrying out these details involves the estimates \eqref{equation:hestqrdq}, \eqref{equation:lmmblregwlinfzbl1}, \eqref{equation:lmmblregwlinfzbl2}, and \eqref{equation:lmmblregthetalinfzl2}. In particular, for estimating $\partial_t\xvec{w}^{\pm}$, several norms of~$\xsym{\mu}^{\pm}$ enter through \eqref{equation:lmmblregwlinfzbl2} with $O(\epsilon^{1/4})$ coefficients. Since $\xsym{\mu}^{\pm}$ are supported in $[0,T]$ and smooth, there are $O(\epsilon^{1/4})$ bounds for these terms.
	\end{proof}

	\subsubsection{Energy estimates}\label{subsubsection:energyestimates_remainder}
	The desired asymptotic behavior in \eqref{equation:asympbeh} is now obtained as a consequence of the next proposition.
	\begin{prpstn}\label{proposition:aprxxtrlest} The functions $\xvec{r}^{\pm,\epsilon}$ determined from \eqref{equation:ansatz2} by means of Subsections~\Rref{subsection:return}, \Rref{subsection:flushing} and \Rref{subsection:blprf} satisfy
		\begin{equation}\label{equation:remesblAgf5}
			\begin{aligned}
				\|\xvec{r}^{\pm,\epsilon}\|_{\xLinfty((0,T/\epsilon);\xLtwo(\mathcal{E}))}^2 + \epsilon\|\xvec{r}^{\pm,\epsilon}\|_{\xLtwo((0,T/\epsilon);\xHone(\mathcal{E}))}^2 = O(\epsilon^{\frac{1}{4}}).
			\end{aligned}
		\end{equation}
	\end{prpstn}
	\begin{proof}
		The idea is to multiply the equations in the first line of \eqref{equation:MHD_ElsaesserExtLtScaled_raminder_blctrl} by $\xvec{r}^{\pm,\epsilon}$ respectively, followed by integrating over $\mathcal{E}\times(0,t)$ for $t \in (0,T/\epsilon)$, which however is not justified. Indeed, the regularity of $\xvec{z}^{\pm,\epsilon},  \xvec{r}^{\pm,\epsilon} \in \mathscr{X}_{T/\epsilon}$ does not guarantee the convergence of the integrals
		\[
		\int_0^t\int_{\mathcal{E}} (\xvec{z}^{\mp,\epsilon}(\xvec{x},s) \cdot \xdop{\nabla}) \xvec{r}^{\pm,\epsilon}(\xvec{x},s) \cdot \xvec{r}^{\pm,\epsilon}(\xvec{x},s) \, \xdx{\xvec{x}} \xdx{s}.
		\]
		Since \eqref{equation:ansatz2}--\eqref{3.7} imply that the term $\epsilon\xvec{r}^{\pm,\epsilon} = \xvec{z}^{\pm,\epsilon} - \xvec{s}^{\pm,\epsilon}$, where $\xvec{s}^{\pm,\epsilon}$ is bounded in $\xLinfty((0,T/\epsilon);\xHn{3}(\mathcal{E}))$ and the temporal derivatives obey $\partial_t \xvec{s}^{\pm,\epsilon} \in \xLtwo((0,T/\epsilon);\xLtwo(\mathcal{E}))$, the above mentioned convergence issue can be avoided by using the strong energy inequality as explained in \cite[Page 167-168]{IftimieSueur2011}. 
		
		\paragraph{Step 1. Employing the energy inequality.} Let us sketch the aforementioned approach of utilizing the energy inequality. For the sake of simplicity, we carry out the steps for the case from \Cref{subsubsection:case_thm1} where $\sigma^0 = 0$.
		First, from \eqref{equation:MHD_ElsaesserExtLtScaled_raminder_blctrl} and \eqref{equation:MHD_ElsaesserExtScaledLongTime}, one observes that $\xvec{s}^{\pm,\epsilon}$ satisfy a weak formulation for the problem
		\begin{equation}\label{equation:MHD_spmeps}
			\begin{cases}
				\partial_t \xvec{s}^{\pm,\epsilon} - \epsilon\Delta (\lambda^{\pm}\xvec{s}^{+,\epsilon} + \lambda^{\mp}\xvec{s}^{-,\epsilon}) + (\xvec{z}^{\mp,\epsilon} \cdot \xdop{\nabla}) \xvec{s}^{\pm,\epsilon} + \xdop{\nabla} o^{\pm,\epsilon} = \xsym{\Xi}^{\pm,\epsilon} & \mbox{ in } \mathcal{E}_{T/\epsilon},\\
				\xdop{\nabla}\cdot\xvec{s}^{\pm,\epsilon} = 0  & \mbox{ in } \mathcal{E}_{T/\epsilon},\\
				\xvec{s}^{\pm,\epsilon} \cdot \xvec{n} = 0 & \mbox{ on }  \Sigma_{T/\epsilon}, \\
				(\xcurl{\xvec{s}^{\pm,\epsilon}}) \times \xvec{n} = \xsym{\rho}^{\pm}(\xvec{s}^{+,\epsilon},\xvec{s}^{-,\epsilon}) - \varepsilon\xsym{\mathfrak{g}}^{\pm,\epsilon} & \mbox{ on }  \Sigma_{T/\epsilon},\\
				\xvec{s}^{\pm,\epsilon}(\cdot, 0) = \epsilon \xvec{z}^{\pm}_0 & \mbox{ in } \mathcal{E},
			\end{cases}
		\end{equation}
		where
		\[
		o^{\pm,\epsilon} \coloneqq p^{\pm,\epsilon}-\epsilon\pi^{\pm,\epsilon}, \quad \xsym{\Xi}^{\pm,\epsilon} \coloneqq \xsym{\xi}^{\pm} - \epsilon\left\llbracket \xvec{h}^{\pm,\epsilon} - A^{\epsilon,\pm} \xvec{r}^{\mp,\epsilon}\right\rrbracket_{\epsilon}.
		\]
		Multiplying the equations \eqref{equation:MHD_spmeps} with $\xvec{z}^{\pm,\epsilon} - \xvec{s}^{\pm,\epsilon} \in \mathcal{X}_{T/\epsilon}$, integrating over $\mathcal{E}\times(0,t)$, and summing up the results, leads to
		{\allowdisplaybreaks 
			\begin{gather*}
				\sum\limits_{\square\in\{+,-\}} \int_0^t \int_{\mathcal{E}} \partial_t\xvec{s}^{\square,\epsilon} \cdot (\xvec{z}^{\square,\epsilon} - \xvec{s}^{\square,\epsilon}) \, \xdx{\xvec{x}} \xdx{r} \\
				+ \sum\limits_{\square\in\{+,-\}}\epsilon \lambda^+ \int_0^t \int_{\mathcal{E}}\xcurl{\xvec{s}^{\square,\epsilon}}  \cdot \xcurl{(\xvec{z}^{\square,\epsilon} - \xvec{s}^{\square,\epsilon})} \, \xdx{\xvec{x}} \xdx{r}\\
				+ \epsilon \lambda^- \sum\limits_{\substack{(\triangle,\circ)\in \\ \{(+,-), (-,+)\}}} \int_0^t \int_{\mathcal{E}} \xcurl{\xvec{s}^{\circ,\epsilon}} \cdot \xcurl{(\xvec{z}^{\triangle,\epsilon} - \xvec{s}^{\triangle,\epsilon})} \, \xdx{\xvec{x}} \xdx{r} \\
				+ \sum\limits_{\substack{(\triangle,\circ)\in \\ \{(+,-), (-,+)\}}} \int_0^t \int_{\mathcal{E}} (\xvec{z}^{\circ,\epsilon} \cdot \xnab)\xvec{s}^{\triangle,\epsilon} \cdot \xvec{z}^{\triangle,\epsilon}  \, \xdx{\xvec{x}} \xdx{r} \\
				= \sum\limits_{\substack{(\triangle,\circ)\in \\ \{(+,-), (-,+)\}}} \int_0^t\int_{\mathcal{E}} \left(\xsym{\xi}^{\triangle,\epsilon} -\epsilon\llbracket \xvec{h}^{\triangle,\epsilon} \rrbracket_{\epsilon} + \epsilon\llbracket A^{\triangle,\epsilon} \xvec{r}^{\circ,\epsilon} \rrbracket_{\epsilon}\right) \cdot (\xvec{z}^{\triangle,\epsilon} - \xvec{s}^{\triangle,\epsilon}) \, \xdx{\xvec{x}} \xdx{r} \stepcounter{equation} \tag{\theequation} \label{equation:dee1}\\
				+ \epsilon \lambda^+ \sum\limits_{\square\in\{+,-\}} \int_0^t\int_{\partial\mathcal{E}} \left( \xsym{\rho}^{\square}(\xvec{s}^{+,\epsilon
				}, \xvec{s}^{-,\epsilon}) - \epsilon \xsym{\mathfrak{g}}^{\square,\epsilon} \right) \cdot (\xvec{z}^{\square,\epsilon} - \xvec{s}^{\square,\epsilon}) \, \xdx{S} \xdx{r} \\
				+ \epsilon \lambda^- \sum\limits_{\substack{(\triangle,\circ)\in \\ \{(+,-), (-,+)\}}} \int_0^t\int_{\partial\mathcal{E}} \left( \xsym{\rho}^{\circ}(\xvec{s}^{+,\epsilon
				}, \xvec{s}^{-,\epsilon}) - \epsilon \xsym{\mathfrak{g}}^{\circ,\epsilon} \right) \cdot (\xvec{z}^{\triangle,\epsilon} - \xvec{s}^{\triangle,\epsilon}) \, \xdx{S} \xdx{r}.
			\end{gather*}
			In \eqref{equation:dee1}, the following cancellations, which are justified via integration by parts and by using the regularity of $\xvec{s}^{\pm,\epsilon}$, have been taken into account:
			\[
				\sum\limits_{\substack{(\triangle,\circ)\in \\ \{(+,-), (-,+)\}}} \int_0^t \int_{\mathcal{E}} (\xvec{z}^{\circ,\epsilon} \cdot \xnab)\xvec{s}^{\triangle,\epsilon} \cdot \xvec{s}^{\triangle,\epsilon}  \, \xdx{\xvec{x}} \xdx{r} = 0.
			\]
			Second, taking the test function $\xvec{s}^{\pm,\epsilon}$ in the weak formulation for $\xvec{z}^{\pm,\epsilon}$, which is justified thanks to the regularity of $\xvec{s}^{\pm,\epsilon}$, yields
			\begin{gather*}
				\sum\limits_{\square\in\{+,-\}} \int_{\mathcal{E}} \xvec{z}^{\square,\epsilon}(\xvec{x},t) \cdot \xvec{s}^{\square,\epsilon}(\xvec{x},t) \, \xdx{\xvec{x}} -\sum\limits_{\square\in\{+,-\}} \int_{\mathcal{E}} \xvec{z}^{\square,\epsilon}(\xvec{x},0) \cdot \xvec{s}^{\square,\epsilon}(\xvec{x},0) \, \xdx{\xvec{x}}\\
				- \sum\limits_{\square\in\{+,-\}} \int_0^t \int_{\mathcal{E}} \xvec{z}^{\square,\epsilon} \cdot \partial_t \xvec{s}^{\square,\epsilon} \, \xdx{\xvec{x}} \xdx{r} + \epsilon \lambda^+ \sum\limits_{\square\in\{+,-\}} \int_0^t \int_{\mathcal{E}}\xcurl{\xvec{z}^{\square,\epsilon}}  \cdot \xcurl{\xvec{s}^{\square,\epsilon}} \, \xdx{\xvec{x}} \xdx{r} \\
				+ \sum\limits_{\substack{(\triangle,\circ)\in \\ \{(+,-), (-,+)\}}}  \int_0^t \int_{\mathcal{E}} \left(\epsilon \lambda^- \xcurl{\xvec{z}^{\circ,\epsilon}} \cdot \xcurl{ \xvec{s}^{\triangle,\epsilon}} + (\xvec{z}^{\circ,\epsilon} \cdot \xnab)\xvec{z}^{\triangle,\epsilon} \cdot \xvec{s}^{\triangle,\epsilon}\right)  \, \xdx{\xvec{x}} \xdx{r} \stepcounter{equation} \tag{\theequation} \label{equation:dee2}\\
				= \sum\limits_{\square \in \{+,-\}} \left(\int_0^t\int_{\mathcal{E}}  \xsym{\xi}^{\square,\epsilon} \cdot \xvec{s}^{\square,\epsilon} \, \xdx{\xvec{x}} \xdx{r} + \epsilon \lambda^+ \int_0^t\int_{\partial\mathcal{E}} \xsym{\rho}^{\square,\epsilon}(\xvec{z}^{+,\epsilon},\xvec{z}^{-,\epsilon}) \cdot \xvec{s}^{\square,\epsilon} \, \xdx{S} \xdx{r}\right)\\
				+ \epsilon \lambda^- \sum\limits_{\substack{(\triangle,\circ)\in \\ \{(+,-), (-,+)\}}} \int_0^t\int_{\partial\mathcal{E}} \xsym{\rho}^{\circ,\epsilon}(\xvec{z}^{+,\epsilon},\xvec{z}^{-,\epsilon}) \cdot \xvec{s}^{\triangle,\epsilon} \, \xdx{S} \xdx{r}.
			\end{gather*}
			Third, multiplying the energy inequality \eqref{equation:seizpm2}  with $\epsilon^2$, followed by evaluation at $s = 0$ and $\epsilon t$ for $t \in [0, T/\epsilon]$, while performing the change of variables $r \to \epsilon r$, one has
			\begin{gather*}
				\frac{1}{2} \sum\limits_{\square\in\{+,-\}} \|\xvec{z}^{\square,\epsilon}(\cdot, t)\|_{\xLtwo(\mathcal{E})}^2 + \epsilon\lambda^+ \sum\limits_{\square\in\{+,-\}} \int_0^t \int_{\mathcal{E}}\xcurl{\xvec{z}^{\square,\epsilon}} \cdot \xcurl{\xvec{z}^{\square,\epsilon}} \, \xdx{\xvec{x}} \xdx{r} \\
				+ \epsilon\lambda^- \sum\limits_{\substack{(\triangle,\circ)\in \\ \{(+,-), (-,+)\}}} \int_0^t \int_{\mathcal{E}} \xcurl{\xvec{z}^{\triangle,\epsilon}} \cdot \xcurl{\xvec{z}^{\circ,\epsilon}} \, \xdx{\xvec{x}} \xdx{r} \\
				\leq \frac{1}{2} \sum\limits_{\square\in\{+,-\}} \|\xvec{z}^{\square,\epsilon}(\cdot, 0)\|_{\xLtwo(\mathcal{E})}^2 + \epsilon\lambda^+ \sum\limits_{\square\in\{+,-\}} \int_0^t\int_{\partial\mathcal{E}} \xsym{\rho}^{\square}(\xvec{z}^{+,\epsilon},\xvec{z}^{-,\epsilon}) \cdot \xvec{z}^{\square,\epsilon} \, \xdx{S} \xdx{r} \stepcounter{equation} \tag{\theequation} \label{equation:dee3}\\
				+ \epsilon\lambda^-\sum\limits_{\substack{(\triangle,\circ)\in \\ \{(+,-), (-,+)\}}} \int_0^t\int_{\partial\mathcal{E}} \xsym{\rho}^{\triangle}(\xvec{z}^{+,\epsilon},\xvec{z}^{-,\epsilon}) \cdot \xvec{z}^{\circ,\epsilon} \, \xdx{S} \xdx{r} \\
				+ \sum\limits_{\square\in\{+,-\}} \int_0^t\int_{\mathcal{E}} \xsym{\xi}^{\square,\epsilon}  \cdot \xvec{z}^{\square,\epsilon} \, \xdx{\xvec{x}} \xdx{r}.
			\end{gather*}
		}
		\paragraph{Step 2. Conclusion.} By subtracting from \eqref{equation:dee3} the equations \cref{equation:dee1} and \cref{equation:dee2}, while considering for the general case $\sigma^0 \neq 0$ the identity
		\begin{multline*}
			\int_0^t \int_{\mathcal{E}} (\xvec{z}^{\circ,\epsilon} \cdot \xnab)\xvec{z}^{\triangle,\epsilon} \cdot \xvec{s}^{\triangle,\epsilon} \, \xdx{\xvec{x}} \xdx{r} \\
			= -\int_0^t \int_{\mathcal{E}} (\xvec{z}^{\circ,\epsilon} \cdot \xnab)\xvec{s}^{\triangle,\epsilon} \cdot \xvec{z}^{\triangle,\epsilon}  \, \xdx{\xvec{x}} \xdx{r} -\int_0^t \int_{\mathcal{E}} \sigma^0 \xvec{s}^{\triangle,\epsilon} \cdot \xvec{z}^{\triangle,\epsilon}  \, \xdx{\xvec{x}} \xdx{r},
		\end{multline*}
		where $(\triangle,\circ) \in \{(+,-), (-,+)\}$, one arrives at the inequality \allowdisplaybreaks
		\begin{gather*}
			\frac{1}{2} \sum\limits_{\square\in\{+,-\}} \|\xvec{r}^{\square,\epsilon}(\cdot, t)\|_{\xLtwo(\mathcal{E})}^2 + \epsilon \lambda^+ \sum\limits_{\square\in\{+,-\}} \int_0^t \int_{\mathcal{E}}\xcurl{\xvec{r}^{\square,\epsilon}}  \cdot \xcurl{\xvec{r}^{\square,\epsilon}} \, \xdx{\xvec{x}} \xdx{r} \\
			+ \epsilon \lambda^- \sum\limits_{\substack{(\triangle,\circ)\in \\ \{(+,-), (-,+)\}}} \int_0^t \int_{\mathcal{E}} \xcurl{\xvec{r}^{\triangle,\epsilon}} \cdot \xcurl{\xvec{r}^{\circ,\epsilon}}\, \xdx{\xvec{x}} \xdx{r} \\
			\leq \sum\limits_{\substack{(\triangle,\circ)\in \\ \{(+,-), (-,+)\}}} \int_0^t\int_{\mathcal{E}} \left(\llbracket \xvec{h}^{\triangle,\epsilon} \rrbracket_{\epsilon} - \llbracket A^{\triangle,\epsilon} \xvec{r}^{\circ,\epsilon} \rrbracket_{\epsilon}\right) \cdot \xvec{r}^{\triangle,\epsilon} \, \xdx{\xvec{x}} \xdx{r} \\
			+ \epsilon \lambda^+ \sum\limits_{\square\in\{+,-\}} \int_0^t\int_{\partial\mathcal{E}} J^{\square,\square} \, \xdx{S} \xdx{r} + \epsilon \lambda^- \sum\limits_{\substack{(\triangle,\circ)\in \\ \{(+,-), (-,+)\}}} \int_0^t\int_{\partial\mathcal{E}} J^{\triangle,\circ} \, \xdx{S} \xdx{r}\\
			+ \frac{1}{2}\sum\limits_{\square\in\{+,-\}} \int_0^t\int_{\mathcal{E}} \sigma^0 |\xvec{r}^{\square,\epsilon}|^2 \, \xdx{\xvec{x}} \xdx{r},
		\end{gather*}
		with
		\[
		J^{\square,\triangle} \coloneqq \left(\xsym{\rho}^{\triangle}(\xvec{r}^{+,\epsilon}, \xvec{r}^{-,\epsilon}) + \xsym{\mathfrak{g}}^{\triangle,\epsilon}\right) \cdot \xvec{r}^{\square,\epsilon}.
		\]
		Since $\xdiv{\xvec{r}^{\pm,\epsilon}} = 0$ in $\mathcal{E}$ and $\xvec{r}^{\pm,\epsilon} \cdot \xvec{n} = 0$ on $\partial \mathcal{E}$, the inequality \eqref{equation:sKem} provides
		\[
		\|\xvec{r}^{\pm,\epsilon}\|_{\xHone(\mathcal{E})} \leq C\left(\|\xcurl{\xvec{r}^{\pm,\epsilon}}\|_{\xLtwo(\mathcal{E})} + \|\xvec{r}^{\pm,\epsilon}\|_{\xLtwo(\mathcal{E})}\right),
		\]
		which in turn yields
		\begin{equation}\label{equation:mltplremeqwrpm}
			\begin{gathered}
				\frac{1}{2} \sum\limits_{\square\in\{+,-\}} \|\xvec{r}^{\square,\epsilon}(\cdot, t)\|_{\xLtwo(\mathcal{E})}^2 + \frac{\epsilon}{2} \sum\limits_{\square\in\{+,-\}}(\lambda^+ \square \, \lambda^-) \int_0^t \|(\xvec{r}^{+,\epsilon}\square \, \xvec{r}^{-,\epsilon})(\cdot,r)\|_{\xHone(\mathcal{E})}^2 \xdx{r} \\
				\leq \sum\limits_{\substack{(\triangle,\circ)\in \\ \{(+,-), (-,+)\}}} \int_0^t\int_{\mathcal{E}} \left(\llbracket \xvec{h}^{\triangle,\epsilon} \rrbracket_{\epsilon} - \llbracket A^{\triangle,\epsilon} \xvec{r}^{\circ,\epsilon} \rrbracket_{\epsilon}\right) \cdot \xvec{r}^{\triangle,\epsilon}\, \xdx{\xvec{x}} \xdx{r} \\
				+ \epsilon  \lambda^+ \sum\limits_{\square\in\{+,-\}}  \int_0^t\int_{\partial\mathcal{E}} J^{\square,\square} \, \xdx{S} \xdx{r} + \epsilon \lambda^- \sum\limits_{\substack{(\triangle,\circ)\in \\ \{(+,-), (-,+)\}}}  \int_0^t\int_{\partial\mathcal{E}} J^{\triangle,\circ}  \, \xdx{S} \xdx{r}\\
				+  \epsilon C \sum\limits_{\square\in\{+,-\}}\frac{\lambda^+ \square \lambda^-}{2} \int_0^t \|\xvec{r}^{\square,\epsilon}(\cdot, r)\|_{\xLtwo(\mathcal{E})}^2 \, \xdx{r},
			\end{gathered}
		\end{equation} 
		where $C > 0$ depends on $\lambda^{\pm}$ and the fixed quantity $\max_{(\xvec{x,s})\in\overline{\mathcal{E}}\times[0,T]} |\sigma^0(\xvec{x},s)|$.
		The boundary integrals containing $J^{\square,\triangle}$ are treated by applying for $\xvec{f},\xvec{h} \in \xHone(\mathcal{E})$ and $\ell > 0$ the estimate
		\[
			\left|\int_{\partial\mathcal{E}} \xvec{f} \cdot \xvec{h} \, \xdx{S} \right| \leq C(\ell) (\|\xvec{f}\|_{\xLtwo(\mathcal{E})}^2 + \|\xvec{h}\|_{\xLtwo(\mathcal{E})}^2 ) + \ell (\|\xvec{f}\|_{\xHone(\mathcal{E})}^2+\|\xvec{h}\|_{\xHone(\mathcal{E})}^2).
		\]
		Thus, for $t \in [0, T/\epsilon]$ and arbitrary $\ell > 0$ one has 
		\begin{multline*}
			\left| \int_{\partial\mathcal{E}} J^{\square,\triangle}(\xvec{x},t) \, \xdx{S(\xvec{x})} \right| \\
			\begin{aligned}
				& \leq \ell \left(\|\xsym{\mathfrak{g}}^{\triangle,\epsilon}(\cdot,t)\|_{\xHone(\mathcal{E})}^2 + \|\xvec{r}^{+,\epsilon}(\cdot,t)\|_{\xHone(\mathcal{E})}^2 + 	\|\xvec{r}^{-,\epsilon}(\cdot,t)\|_{\xHone(\mathcal{E})}^2\right) \\
				& \quad + C(\ell) \left(\|\xsym{\mathfrak{g}}^{\triangle,\epsilon}(\cdot,t)\|_{\xLtwo(\mathcal{E})}^2 + \|\xvec{r}^{+,\epsilon}(\cdot,t)\|_{\xLtwo(\mathcal{E})}^2 + 	\|\xvec{r}^{-,\epsilon}(\cdot,t)\|_{\xLtwo(\mathcal{E})}^2 \right).
			\end{aligned}
		\end{multline*}
		Consequently, by selecting $\ell > 0$ small, employing \Cref{lemma:remaindest}, and applying Gr\"onwall's inequality in \eqref{equation:mltplremeqwrpm}, one arrives at \eqref{equation:remesblAgf5}.
	\end{proof}
	
	\begin{crllr}\label{corollary:approxnlctrsmdt}
		The functions $\xvec{z}^{\pm,\epsilon}$ fixed in the beginning of \Cref{section:approxres2} satisfy
		\begin{equation*}\label{equation:remesblAgf7}
			\begin{aligned}
				\|\xvec{z}^{\pm,\epsilon}(\cdot,T/\epsilon)\|_{\xLtwo(\mathcal{E})}  = O(\epsilon^{\frac{9}{8}}).
			\end{aligned}
		\end{equation*}
	\end{crllr}
	\begin{proof}
		We remind that $\xsym{\mu}^{\pm}$ has been fixed via \Cref{lemma:blcxi} (or \Cref{lemma:blcxi2}) with $r = 6$ and $k = 4$, while noting that $\lim_{a\longrightarrow + \infty} a^{-1/2}\ln(a) = 0$. Therefore, by combining \eqref{equation:remesblAgf5} with \eqref{equation:ansatz2}, \Cref{lemma:epsgain} and \Cref{lemma:propwpm}, one arrives at 
		\begin{equation*}
			\begin{aligned}
				\|\xvec{z}^{\pm,\epsilon}(\cdot,T/\epsilon)\|_{\xLtwo(\mathcal{E})} & \leq \sqrt{\epsilon}\|\left\llbracket\xvec{v}^{\pm}\right\rrbracket_{\epsilon}(\cdot,T/\epsilon)\|_{\xLtwo(\mathcal{E})} + \epsilon\|\left\llbracket\xvec{w}^{\pm}\right\rrbracket_{\epsilon}(\cdot,T/\epsilon)\|_{\xLtwo(\mathcal{E})} \\
				& \quad \epsilon\|\xnab\theta^{\pm,\epsilon}(\cdot,T/\epsilon)\|_{\xLtwo(\mathcal{E})} + \epsilon\|\xvec{r}^{\pm,\epsilon}(\cdot,T/\epsilon)\|_{\xLtwo(\mathcal{E})} \\
				& \leq \epsilon^{\frac{3}{4}}\epsilon^{\frac{1}{2}}\epsilon^{-\frac{1}{2}}\|\xvec{v}^{\pm}(\cdot,T/\epsilon)\|_{\xHn{{1,1,0}}_{\mathcal{E}}} + \epsilon^{\frac{5}{4}}\|\xvec{w}^{\pm}(\cdot,T/\epsilon)\|_{\xHn{{0,2,0}}_{\mathcal{E}}} +  O(\epsilon^{\frac{9}{8}})\\
				& = O(\epsilon^{\frac{9}{8}}).
			\end{aligned}
		\end{equation*}
	\end{proof}

	\subsection{Controlling towards arbitrary smooth states}\label{subsection:approxtowtraj}
	Let $\overline{\xvec{z}}^{\pm}_1 \in \xCinfty(\overline{\mathcal{E}};\mathbb{R}^N)\cap\xH(\mathcal{E})$ be arbitrarily fixed. The previous arguments for approximate null controllability can be modified for the target $\overline{\xvec{z}}^{\pm}_1$. The idea is similar to that described in \cite[Section 5]{CoronMarbachSueur2020} for a Navier--Stokes problem.  
	First, the ansatz \eqref{equation:ansatz2} is modified such that for $\xvec{z}^{\pm,\epsilon}$ on the time interval $[0,T]$ one chooses an expansion of the form
	\begin{equation}\label{equation:ansatz3trajshortt}
		\begin{gathered}
			\xvec{z}^{\pm,\epsilon} = \xvec{z}^{0} + \sqrt{\epsilon}\left\llbracket \xvec{v}^{\pm}\right\rrbracket_{\epsilon} + \epsilon \overline{\xvec{z}}^{\pm,1} + \epsilon \xdop{\nabla} \theta^{\pm,\epsilon} + \epsilon \left\llbracket \xvec{w}^{\pm}\right\rrbracket_{\epsilon} + \epsilon \xvec{r}^{\pm,\epsilon},\\
		\end{gathered}
	\end{equation}
	while on $[T, T/\epsilon]$ it is assumed that
	\begin{equation}\label{equation:ansatz3trajlarget}
		\begin{aligned}
			\xvec{z}^{\pm,\epsilon} & =  \sqrt{\epsilon}\left\llbracket \xvec{v}^{\pm}\right\rrbracket_{\epsilon} + \epsilon \overline{\xvec{z}}^{\pm}_1 + \epsilon \xdop{\nabla} \theta^{\pm,\epsilon} + \epsilon \left\llbracket \xvec{w}^{\pm}\right\rrbracket_{\epsilon} + \epsilon \xvec{r}^{\pm,\epsilon}.
		\end{aligned}
	\end{equation}
	The profiles $\overline{\xvec{z}}^{\pm,1}$ in \eqref{equation:ansatz3trajshortt} belong to $\xLn{\infty}((0,T);\xHn{3}(\mathcal{E}))$ and, by adopting the constructions from \Cref{lemma:flushing}'s proof, there exist~$\xsym{\xi}^{\pm,1}\in\xCzero([0,T];\xHtwo(\mathcal{E};\mathbb{R}^N))$ with $\operatorname{supp}(\xsym{\xi}^{\pm,1}) \subset (\overline{\mathcal{E}}\setminus\overline{\Omega}) \times [0,T]$
	such that $(\overline{\xvec{z}}^{\pm,1},\xsym{\xi}^{\pm,1})$ solve the controllability problem
	\begin{equation}\label{equation:MHD_ElsaesserExt_Ovarepstraj}
		\begin{cases}
			\partial_t \overline{\xvec{z}}^{\pm,1} + (\overline{\xvec{z}}^{\mp,1} \cdot \xdop{\nabla}) \xvec{z}^{0} + (\xvec{z}^{0} \cdot \xdop{\nabla}) \overline{\xvec{z}}^{\pm,1} + \xdop{\nabla} p^{\pm,1} = \xsym{\xi}^{\pm,1} + (\lambda^{\pm}+\lambda^{\mp})\Delta \xvec{z}^{0} & \mbox{ in } \mathcal{E}_T,\\
			\xdop{\nabla}\cdot\overline{\xvec{z}}^{\pm,1} = 0  & \mbox{ in } \mathcal{E}_T,\\
			\overline{\xvec{z}}^{\pm,1} \cdot \xvec{n} = 0 & \mbox{ on }  \Sigma_T,\\
			\overline{\xvec{z}}^{\pm,1}(\cdot, 0) = \xvec{z}_0^{\pm} & \mbox{ in } \mathcal{E},\\
			\overline{\xvec{z}}^{\pm,1}(\cdot, T) = \overline{\xvec{z}}^{\pm}_1 & \mbox{ in } \mathcal{E}.
		\end{cases}
	\end{equation}
	In particular, all bounds for $\overline{\xvec{z}}^{\pm,1}$ and $\xsym{\xi}^{\pm,1}$ are independent of $\epsilon > 0$, as this parameter does not appear in \eqref{equation:MHD_ElsaesserExt_Ovarepstraj}.
	Because $\overline{\xvec{z}}^{\pm}_1$ are smooth and independent of time, by analysis similar to \Cref{subsection:remestwbctrl}, one can infer $\| \xvec{r}^{\pm,\epsilon}(\cdot, T/\epsilon)\| = O(\epsilon^{\frac{1}{8}})$. As a result, the rescaled functions $\xvec{z}^{\pm, (\epsilon)}(\xvec{x}, t) \coloneqq \epsilon^{-1}\xvec{z}^{\pm,\epsilon}\left(\xvec{x}, \epsilon^{-1}t \right)$ 
	satisfy
	\begin{equation*}
		\begin{aligned}
			\| \xvec{z}^{\pm, (\epsilon)}(\xvec{x}, T) - \overline{\xvec{z}}^{\pm}_1\|_{\xLtwo(\mathcal{E})} = O(\epsilon^{\frac{1}{8}}).
		\end{aligned}
	\end{equation*}

	\section{Conclusion of the main results}\label{section:conclth}
	In order to relax the assumption $\xvec{u}_0, \xvec{B}_0 \in \xHn{3}(\mathcal{E}) \cap \xW(\mathcal{E})$ employed in \Cref{section:approxres2}, we connect initial data from $\xH(\mathcal{E})$ by a weak controlled trajectory to a state which belongs to $\xHn{3}\cap \xW(\mathcal{E})$.  This is done via \Cref{lemma:reg} below, and a proof of this argument, which is a modification of \cite[Lemma 9]{CoronMarbachSueur2020}, will be outlined in \Cref{appendix:proofreg}.
	
	\begin{lmm}\label{lemma:reg}
		When $N = 3$, assume that $\xvec{M}_1$ and $\xvec{L}_2$ are symmetric, $\xvec{L}_1 = \xvec{M}_2 = \xvec{0}$, and that $\Omega$ is simply-connected. 
		For any given $T^{*} > 0$ and $\xvec{u}_0, \xvec{B}_0 \in \xH(\mathcal{E})$, there exists a smooth function $C_{T^{*}} > 0$ with $C_{T^{*}}(0) = 0$ such that a Leray--Hopf weak solution $(\xvec{u},\xvec{B}) \in \mathscr{X}_{T^{*}}^2$ to \eqref{equation:MHD_ElsaesserExt}
		obeys for some $t_{\operatorname{reg}} \in [0,T^{*}]$ the estimate
		\[
		\| \xvec{u} (\cdot, t_{\operatorname{reg}}) \|_{\xHn{3}(\mathcal{E})} + \| \xvec{B} (\cdot, t_{\operatorname{reg}}) \|_{\xHn{3}(\mathcal{E})} \leq C_{T^*}\left(\|\xvec{u}_0\|_{\xLtwo(\mathcal{E})} + \|\xvec{B}_0\|_{\xLtwo(\mathcal{E})}\right).
		\]
	\end{lmm}
	
	Let the control time $T_{\operatorname{ctrl}} > 0$, the states $\xvec{u}_0, \xvec{B}_0,\xvec{u}_1, \xvec{B}_1  \in \xLn{2}_{\operatorname{c}}(\Omega)$, and any $\delta > 0$ be arbitrarily fixed. Then, the proof of \Cref{theorem:main1} is completed by means of the ensuing steps.
	
	\begin{enumerate}[1)]
		\item\label{item:conclusionStep1} The physical domain $\Omega$ is extended to $\mathcal{E}$, as explained in \Cref{section:extentsions}, and the weak formulation given in \Cref{subsubsection:case_thm1} is chosen.
		\item\label{item:conclusionStep2} By \Cref{lemma:reg}, there exists $T_1 \in [0, T_{\operatorname{ctrl}}/4)$ such that a Leray--Hopf weak solution $(\xvec{u}, \xvec{B})$ to \eqref{equation:MHD_ElsaesserExt} with initial data $(\xvec{u}_0,\xvec{B}_0)$ and zero forces $\xsym{\xi} = \xsym{\eta} = \xvec{0}$ obeys $\xvec{u}(\cdot,T_1), \xvec{B}(\cdot,T_1) \in \xHn{3}(\mathcal{E})\cap\xH(\mathcal{E})$. 
		\item\label{item:conclusionStep3} By a density argument, one can select states $\overline{\xvec{u}}_1, \overline{\xvec{B}}_1 \in \xCinfty(\overline{\mathcal{E}};\mathbb{R}^N)\cap\xW(\mathcal{E})$ with
		\[
			\|\overline{{\xvec{u}}}_1-\xvec{u}_1\|_{\xLtwo(\Omega)} + \|\overline{\xvec{B}}_1-\xvec{B}_1\|_{\xLtwo(\Omega)} < \delta/2.
		\]
		\item\label{item:conclusionStep4} The arguments in \Cref{section:approxres2} are carried out with $T = T_{\operatorname{ctrl}}-T_1$, initial data $\xvec{u}(\cdot,T_1)$, $\xvec{B}(\cdot,T_1)$ and target states $\overline{\xvec{u}}_1, \overline{\xvec{B}}_1$.
		This provides controls $\xsym{\xi},\xsym{\eta}$ such that all Leray--Hopf weak solutions $(\overline{\xvec{u}}, \overline{\xvec{B}})$ to \eqref{equation:MHD_ElsaesserExt} with initial data $\xvec{u}(\cdot,T_1)$ and $\xvec{B}(\cdot,T_1)$ satisfy
		\[
			\|\overline{\xvec{u}}(\cdot,T_{\operatorname{ctrl}}-T_1) - \overline{\xvec{u}}_1\|_{\xLtwo(\mathcal{E})} + \|\overline{\xvec{B}}(\cdot,T_{\operatorname{ctrl}}-T_1) - \overline{\xvec{B}}_1\|_{\xLtwo(\mathcal{E})} < \delta/2.
		\]
		\item\label{item:conclusionStep5} At $t = T_1$, a Leray--Hopf weak solution $(\overline{\xvec{u}}, \overline{\xvec{B}})$ chosen via Step~\Rref{item:conclusionStep4} is glued to a Leray--Hopf weak solution $(\xvec{u}, \xvec{B})$ from Step~\Rref{item:conclusionStep2}. After renaming, one obtains a Leray--Hopf weak solution $(\xvec{u}, \xvec{B})$ to \eqref{equation:MHD_ElsaesserExt}, defined on the whole time interval $[0,T_{\operatorname{ctrl}}]$, which starts from the initial data $(\xvec{u}_0,\xvec{B}_0)$ and satisfies
		\[
			\|{\xvec{u}}(\cdot,T_{\operatorname{ctrl}}) - \xvec{u}_1\|_{\xLtwo(\Omega)} + \|{\xvec{B}}(\cdot,T_{\operatorname{ctrl}}) - \xvec{B}_1\|_{\xLtwo(\Omega)} < \delta.
		\]
	\end{enumerate}
	
	The previous arguments, while skipping the initial data extension and regularization steps, also yield \Cref{theorem:annulus}. In order to conclude \Cref{theorem:main}, we proceed as follows.
	
	\begin{enumerate}[1)]
		\item The physical domain $\Omega$ is extended to $\mathcal{E}$ as described in \Cref{section:extentsions}, but now the weak formulation in \Cref{subsubsection:secondcase} is considered. When \Cref{lemma:reg} cannot be applied, the extended initial data are chosen with $\xvec{u}_0, \xvec{B}_0 \in \xW(\mathcal{E})\cap\xHn{3}(\mathcal{E})$. Otherwise, in order to reach a divergence-free state, one defines $\sigma^{\pm}(\xvec{x},t) \coloneqq \beta(t) (\xdop{\nabla}\cdot\xvec{z}^{\pm}_0)(\xvec{x})$, with $\beta \in \xCinfty(\mathbb{R};\mathbb{R})$ obeying $\beta(0) = 1$ and $\beta(t) = 0$ for all $t \geq \widehat{T} \coloneqq T_{\operatorname{ctrl}}/8$. A corresponding weak solution to \eqref{equation:MHD_ElsaesserExt_caseB} on $[0,\widehat{T}]$ with zero forces $\xsym{\xi} = \xsym{\eta} = \xvec{0}$ is denoted by $(\widehat{\xvec{z}}^+,\widehat{\xvec{z}}^+)$ and it follows that $\widehat{\xvec{z}}^{\pm}(\cdot, \widehat{T}) \in \xH(\mathcal{E})$. 
		\item Any Leray--Hopf weak solution $(\xvec{z}^{+},\xvec{z}^{-})$ to \eqref{equation:MHD_ElsaesserExt_caseB} with $\sigma^{\pm} = 0$ and zero forces $\xsym{\xi} = \xsym{\eta} = \xvec{0}$ also obeys, by means of the transformation 
		\[
			(\xvec{u}, \xvec{B}) = \frac{1}{2}\left(\xvec{z}^++\xvec{z}^-, \frac{1}{\sqrt{\mu}}(\xvec{z}^+-\xvec{z}^-)\right),
		\]
		the weak form introduced for \eqref{equation:MHD_ElsaesserExt}, and vice versa.
		Therefore, either one can take $T_1 = 0$, or \Cref{lemma:reg} provides a time $T_1 \in [\widehat{T}, T_{\operatorname{ctrl}}/4)$ such that a Leray--Hopf weak solution $(\xvec{z}^{+}, \xvec{z}^-)$ to \eqref{equation:MHD_ElsaesserExt_caseB} with initial data $\widehat{\xvec{z}}^{\pm}(\widehat{T})$ and zero forces $\xsym{\xi} = \xsym{\eta} = \xvec{0}$ obeys $\xvec{z}^{\pm}(\cdot,T_1) \in \xHn{3}(\mathcal{E})\cap\xH(\mathcal{E})$. 
		\item As before, by density, one can choose regular states $\overline{\xvec{z}}^{\pm}_1 \in \xCinfty_0(\overline{\mathcal{E}};\mathbb{R}^N)\cap\xH(\mathcal{E})$ with
		\[
			\|{\overline{\xvec{z}}^{\pm}_1} - (\xvec{u}_1 \pm \sqrt{\mu}\xvec{B}_1)\|_{\xLtwo(\Omega)} < \delta/4.
		\]
		\item Now, \Cref{section:approxres2} is applied with $T = T_{\operatorname{ctrl}}-T_1$, initial data $\xvec{z}^{\pm}(\cdot,T_1)$, and target states~$\overline{\xvec{z}}^{\pm}_1$.
		As a result, there are controls $\xsym{\xi},\xsym{\eta}$ such that all corresponding Leray--Hopf weak solutions $(\overline{\xvec{z}}^+, \overline{\xvec{z}}^-)$ to \eqref{equation:MHD_ElsaesserExt_caseB} satisfy
		\[
			\|\overline{\xvec{z}}^{\pm}(\cdot,T_{\operatorname{ctrl}}-T_1) - \overline{\xvec{z}}^{\pm}_1\|_{\xLtwo(\mathcal{E})} < \delta/4.
		\]
		\item By a gluing argument, one obtains a Leray--Hopf weak solution $(\xvec{z}^+, \xvec{z}^-)$ to \eqref{equation:MHD_ElsaesserExt_caseB} on $[0,T_{\operatorname{ctrl}}]$, starting from the initial data $\xvec{z}^{\pm}_0$ and satisfying
		\[
		\|{\xvec{z}^{\pm}}(\cdot,T_{\operatorname{ctrl}}) - (\xvec{u}_1 \pm \sqrt{\mu}\xvec{B}_1)\|_{\xLtwo(\Omega)} < \delta/2.
		\]
	\end{enumerate}

	\appendix
	\gdef\thesection{\Alph{section}} 
	\makeatletter
	\renewcommand\@seccntformat[1]{\appendixname\ \csname the#1\endcsname.\hspace{0.5em}}
	\makeatother
	
	\section{Boundary layer estimates}\label{appendix:higherorderestimates}
	The estimates used in the proof of \Cref{lemma:wellpvpmctrl} are outlined; more general than there, we take now $\xvec{z}^{\pm,0} \in \xCinfty(\overline{\mathcal{E}}\times[0,T];\mathbb{R}^N)$ with $\xvec{z}^{\pm,0} \cdot \xvec{n} = 0$ along $\partial \mathcal{E}$ and $\operatorname{supp}(\xvec{z}^{\pm,0}(\xvec{x},\cdot)) \subset (0,T]$ for all $\xvec{x} \in \overline{\mathcal{E}}$. In \Cref{lemma:wellpvpmctrl}, it is $\xvec{z}^0 = \xvec{z}^{+,0} = \xvec{z}^{-,0}$. We now consider in $ \mathcal{E} \times (0,T) \times \mathbb{R}_+$ the equations (\cf~\eqref{equation:MHD_ElsaesserBlProfilevpm_contr})
	\begin{equation}\label{equation:MHD_ElsaesserBlProfilevpm}
		\partial_t \xvec{v}^{\pm} - \partial_{zz}(\lambda^{\pm}\xvec{v}^{+} + \lambda^{\mp}\xvec{v}^{-}) + \left[ (\xvec{z}^{\mp,0} \cdot \xdop{\nabla}) \xvec{v}^{\pm} + (\xvec{v}^{\mp} \cdot \xdop{\nabla})\xvec{z}^{\pm,0} \right]_{\operatorname{tan}} + \mathfrak{f}^{\pm}z\partial_z\xvec{v}^{\pm}  =  \xvec{0}
	\end{equation}
	together with the initial and boundary conditions (\cf~\eqref{equation:MHD_ElsaesserBlProfilevpm_contribc})
	\begin{equation}\label{equation:MHD_ElsaesserBlProfilevpm2}
		\begin{cases}
			\partial_z \xvec{v}^{\pm}(\xvec{x},t,0) - \left[ \partial_z\xvec{v}^{\pm}(\xvec{x},t,0) \cdot \xvec{n}(\xvec{x}) \right]\xvec{n}(\xvec{x}) = \xsym{\mathfrak{g}}^{\pm}(\xvec{x},t),  & \xvec{x} \in \overline{\mathcal{E}}, t \in (0, T),\\
			\xvec{v}^{\pm}(\xvec{x},t,0) \cdot \xvec{n}(\xvec{x}) = 0,  & \xvec{x} \in \overline{\mathcal{E}}, t \in (0, T),\\
			\xvec{v}^{\pm}(\xvec{x},t,z) \longrightarrow \xvec{0}, \mbox{ as } z \longrightarrow +\infty, & \xvec{x} \in \overline{\mathcal{E}}, t \in \mathbb{R}_+ ,\\
			\xvec{v}^{\pm}(\xvec{x},0,z) = \xvec{0}, & \xvec{x} \in \overline{\mathcal{E}}, z > 0.
		\end{cases}
	\end{equation}
	While also more general data could be chosen, here we take the cutoff $\chi_{\partial\mathcal{E}}$ defined in \eqref{equation:def_chi} and consider
	\begin{equation*}
		\mathfrak{f}^{\pm}(\xvec{x},t) = -\frac{\xvec{z}^{\mp,0}(\xvec{x},t)\cdot\xvec{n}(\xvec{x})}{\varphi_{\mathcal{E}}(\xvec{x})}, \quad \xsym{\mathfrak{g}}^{\pm}(\xvec{x},t) = 2\chi_{\partial\mathcal{E}}\xmcal{N}^{\pm}(\xvec{z}^{+,0},\xvec{z}^{-,0})(\xvec{x},t).
	\end{equation*}
	Due to the support of $\xvec{z}^{\pm,0}$, compatibility conditions up to all orders are satisfied by the initial and boundary data in \eqref{equation:MHD_ElsaesserBlProfilevpm2}. Multiplying in \eqref{equation:MHD_ElsaesserBlProfilevpm} with $\xvec{n}$, one may similarly to \cite[Section 5]{IftimieSueur2011} establish energy estimates which imply for all $(\xvec{x},t,z) \in \overline{\mathcal{E}} \times[0,T]\times\mathbb{R}_+$ that
	\begin{gather*}
		\left[\xvec{v}^{+}(\xvec{x},t,z) \pm \xvec{v}^{-}(\xvec{x},t,z)\right] \cdot \xvec{n}(\xvec{x}) =  0.\label{equation:opropg}
	\end{gather*}

	The goal consists now of showing the following lemma.
	\begin{lmm}\label{lemma:higherorder}
		For any choice of $k,m_1,m_2,m_3 \in \mathbb{N}_0$, there exists a constant
		\begin{equation}\label{equation:constant_type}
			C = C(\mathcal{E},\lambda^{\pm},k,m_1,m_2,m_3,T,\xvec{z}^{\pm,0}, \xvec{M}^{\pm}, \xvec{L}^{\pm}) > 0
		\end{equation}
		such that every smooth solution $(\xvec{v}^+, \xvec{v}^-)$ to \eqref{equation:MHD_ElsaesserBlProfilevpm} and \eqref{equation:MHD_ElsaesserBlProfilevpm2} obeys
		\begin{equation}\label{equation:hibl}
			\| \xvec{v}^{\pm}  \|_{\xWn{{m_2,\infty}}((0,T);\xHn{{k,m_1,m_3}}_{\mathcal{E}})} + \| \xvec{v}^{\pm}  \|_{\xHn{{m_2}}((0,T); \xHn{{k,m_1,m_3+1}}_{\mathcal{E}})} \leq C,
		\end{equation}
		with $C = 0$ when $\xsym{\mathfrak{g}}^{\pm} = \xvec{0}$.
	\end{lmm}
	\begin{proof}
		The ideas and arguments for proving \Cref{lemma:higherorder} are based on \cite{IftimieSueur2011}. All constants $C > 0$ which appear during the estimates can depend on $\mathcal{E}$, $\lambda^{\pm}$, $k$, $m_1$, $m_2$, $m_3$, $T$, $\xvec{z}^{\pm,0}$, and $\xsym{\rho}^{\pm}$.
		\paragraph{Step 1. Estimates for $\partial_{\xvec{x}}^{\xsym{\alpha}}\partial_t^{\gamma}\xvec{v}^{\pm}$.}
		We take in \eqref{equation:MHD_ElsaesserBlProfilevpm} the partial derivatives $\partial_{\xvec{x}}^{\xsym{\alpha}}\partial_t^{\gamma}$ for $\gamma \in \mathbb{N}_0$ and $\xsym{\alpha} \in \mathbb{N}_0^N$. As a result, 
		\begin{equation}\label{equation:MHD_ElsaesserBlProfilevpmhigherxt}
			\begin{aligned}
				\partial_t (\partial_{\xvec{x}}^{\xsym{\alpha}}\partial_t^{\gamma}\xvec{v}^{\pm}) & = \partial_{zz} \partial_{\xvec{x}}^{\xsym{\alpha}}\partial_t^{\gamma}(\lambda^{\pm}\xvec{v}^{+} + \lambda^{\mp}\xvec{v}^{-}) - z\partial_{\xvec{x}}^{\xsym{\alpha}}\sum\limits_{l=0}^{\gamma} \binom{\gamma}{l} \partial_t^l \mathfrak{f}^{\pm}\partial_t^{\gamma-l}\partial_z\xvec{v}^{\pm}\\
				& \quad - \partial_{\xvec{x}}^{\xsym{\alpha}} \sum\limits_{l=0}^{\gamma}\binom{\gamma}{l} \left[ (\partial_t^{l}\xvec{z}^{\mp,0} \cdot \xdop{\nabla}) \partial_t^{\gamma-l}\xvec{v}^{\pm} + (\partial_t^{l}\xvec{v}^{\mp} \cdot \xdop{\nabla})\partial_t^{\gamma-l}\xvec{z}^{\pm,0} \right]_{\operatorname{tan}}.
			\end{aligned}
		\end{equation}
		Furthermore, multiplying \eqref{equation:MHD_ElsaesserBlProfilevpmhigherxt} for arbitrary $k \in \mathbb{N}$ with $(1+z^{2k})\partial_{\xvec{x}}^{\xsym{\alpha}}\partial_t^{\gamma}\xvec{v}^{\pm}$ and integrating in $(\xvec{x},z)$ over $\mathcal{E}\times\mathbb{R}_+$ yields
		\begin{equation*}\label{equation:hm12estblhigherxt}
			\begin{gathered}
				\frac{1}{2}\frac{d}{dt} \int_{\mathcal{E}}\int_{\mathbb{R}_+} 	(1+z^{2k})|\partial_{\xvec{x}}^{\xsym{\alpha}}\partial_t^{\gamma} \xvec{v}^{\pm}(\xvec{x},t,z)|^2 \, \xdx{z}  \xdx{\xvec{x}} = I_1^{\pm}(t) - I_2^{\pm}(t) - I_3^{\pm,a}(t) - I_3^{\pm,b}(t)
			\end{gathered}
		\end{equation*}
		with the right-hand side being given by
		\begin{equation}\label{equation:hm12estbldthigherxt}
			\begin{aligned}
				I_1^{\pm} & \coloneqq \int_{\mathcal{E}}\int_{\mathbb{R}_+} (1+z^{2k})\partial_{zz}	\partial_{\xvec{x}}^{\xsym{\alpha}}\partial_t^{\gamma} (\lambda^{\pm}\xvec{v}^{+} + \lambda^{\mp}\xvec{v}^{-}) \cdot \partial_{\xvec{x}}^{\xsym{\alpha}}\partial_t^{\gamma} \xvec{v}^{\pm} \, \xdx{z}  \xdx{\xvec{x}},\\
				I_2^{\pm} & \coloneqq \int_{\mathcal{E}}\int_{\mathbb{R}_+} (z+z^{2k+1}) 	\partial_{\xvec{x}}^{\xsym{\alpha}}\sum\limits_{l=0}^{\gamma} \binom{\gamma}{l} \partial_t^l \mathfrak{f}^{\pm}\partial_t^{\gamma-l}\partial_z\xvec{v}^{\pm} \cdot \partial_{\xvec{x}}^{\xsym{\alpha}}\partial_t^{\gamma} \xvec{v}^{\pm} \, \xdx{z}  \xdx{\xvec{x}},\\
				I_3^{\pm,a} & \coloneqq \int_{\mathcal{E}}\int_{\mathbb{R}_+} (1+z^{2k}) \partial_{\xvec{x}}^{\xsym{\alpha}} \sum\limits_{l=0}^{\gamma}\binom{\gamma}{l} \left[ (\partial_t^{l}\xvec{z}^{\mp,0} \cdot \xdop{\nabla}) \partial_t^{\gamma-l}\xvec{v}^{\pm}\right]_{\operatorname{tan}} \cdot \partial_{\xvec{x}}^{\xsym{\alpha}} \partial_t^{\gamma} \xvec{v}^{\pm} \, \xdx{z}  \xdx{\xvec{x}},\\
				I_3^{\pm,b} & \coloneqq \int_{\mathcal{E}}\int_{\mathbb{R}_+} (1+z^{2k}) \partial_{\xvec{x}}^{\xsym{\alpha}} \sum\limits_{l=0}^{\gamma}\binom{\gamma}{l} \left[ (\partial_t^{l}\xvec{v}^{\mp} \cdot \xdop{\nabla})\partial_t^{\gamma-l}\xvec{z}^{\pm,0} \right]_{\operatorname{tan}} \cdot \partial_{\xvec{x}}^{\xsym{\alpha}}\partial_t^{\gamma} \xvec{v}^{\pm} \, \xdx{z}  \xdx{\xvec{x}}.
			\end{aligned}
		\end{equation}
		
		We focus now on the situations where $\gamma > 0$ and $|\xsym{\alpha}| > 0$. For the terms $I^{\pm}_1$, integration by parts in $z$ leads to
		\begin{equation}\label{equation:hm12estbldthigherxtI1}
			\begin{aligned}
				I_1^{\pm} & = -\int_{\mathcal{E}}\int_{\mathbb{R}_+} (1+z^{2k}) \left(\lambda^{\pm}\partial_z \partial_{\xvec{x}}^{\xsym{\alpha}}\partial_t^{\gamma} \xvec{v}^{+} + \lambda^{\mp}\partial_z \partial_{\xvec{x}}^{\xsym{\alpha}}\partial_t^{\gamma} \xvec{v}^{-} \right) \cdot \partial_z \partial_{\xvec{x}}^{\xsym{\alpha}}\partial_t^{\gamma} \xvec{v}^{\pm} \, \xdx{z} \xdx{\xvec{x}} \\
				& \quad - 2k\int_{\mathcal{E}}\int_{\mathbb{R}_+} z^{2k-1}\left(\lambda^{\pm}\partial_z \partial_{\xvec{x}}^{\xsym{\alpha}}\partial_t^{\gamma} \xvec{v}^{+} + \lambda^{\mp}\partial_z \partial_{\xvec{x}}^{\xsym{\alpha}}\partial_t^{\gamma} \xvec{v}^{-} \right) \cdot \partial_{\xvec{x}}^{\xsym{\alpha}}\partial_t^{\gamma} \xvec{v}^{\pm} \, \xdx{z}  \xdx{\xvec{x}} \\
				& \quad -\int_{\mathcal{E}} \left(\lambda^{\pm}\partial_z \partial_{\xvec{x}}^{\xsym{\alpha}}\partial_t^{\gamma} \xvec{v}^{+} + \lambda^{\mp}\partial_z \partial_{\xvec{x}}^{\xsym{\alpha}}\partial_t^{\gamma} \xvec{v}^{-} \right)(\xvec{x},t,0) \cdot \partial_{\xvec{x}}^{\xsym{\alpha}}\partial_t^{\gamma} \xvec{v}^{\pm}(\xvec{x},t,0) \, \xdx{\xvec{x}}\\
				& = -\int_{\mathcal{E}}\int_{\mathbb{R}_+} (1+z^{2k}) \left(\lambda^{\pm}\partial_z \partial_{\xvec{x}}^{\xsym{\alpha}}\partial_t^{\gamma} \xvec{v}^{+} + \lambda^{\mp}\partial_z \partial_{\xvec{x}}^{\xsym{\alpha}}\partial_t^{\gamma} \xvec{v}^{-} \right) \cdot \partial_z \partial_{\xvec{x}}^{\xsym{\alpha}}\partial_t^{\gamma} \xvec{v}^{\pm} \, \xdx{z}  \xdx{\xvec{x}} \\
				& \quad - I_{11}^{\pm} - I_{12}^{\pm}.
			\end{aligned}
		\end{equation}
		By means of Young's inequality and the identities $2\xvec{v}^{\pm} = (\xvec{v}^+ + \xvec{v}^-) \pm (\xvec{v}^+ - \xvec{v}^-)$, one obtains
		\begin{equation}\label{equation:hm12estbldthigherxtI11}
			\begin{aligned}
				|I_{11}^{\pm}| & \leq \sum\limits_{\square\in\{+,-\}}\frac{\lambda^+ \square \lambda^-}{8} \int_{\mathcal{E}}\int_{\mathbb{R}_+} (1 + z^{2k})|\partial_z \partial_{\xvec{x}}^{\xsym{\alpha}}\partial_t^{\gamma} (\xvec{v}^{+} \square \, \xvec{v}^-)|^2 \, \xdx{z}  \xdx{\xvec{x}} 
				\\
				& \quad + C \int_{\mathcal{E}}\int_{\mathbb{R}_+} (1 + z^{2k})|\partial_{\xvec{x}}^{\xsym{\alpha}}\partial_t^{\gamma} \xvec{v}^{\pm}|^2 \, \xdx{z}  \xdx{\xvec{x}}.
			\end{aligned}
		\end{equation}
		For a constant $C_0 > 0$ which vanishes when $\xsym{\mathfrak{g}}^{\pm} = \xvec{0}$, one can infer
		\begin{multline}\label{equation:hm12estbldthigherxtI12}
			|I_{12}^{\pm}| \leq \left|\int_{\mathcal{E}} \partial_{\xvec{x}}^{\xsym{\alpha}} \partial_t^{\gamma} (\lambda^{\pm}\xsym{\mathfrak{g}}^{+} + \lambda^{\mp}\xsym{\mathfrak{g}}^{-})(\xvec{x},t) \cdot \left[\int_{\mathbb{R}_+} \partial_z \partial_{\xvec{x}}^{\xsym{\alpha}} \partial_t^{\gamma} \xvec{v}^{\pm}(\xvec{x},t,z)  \, \xdx{z}\right] \, \xdx{\xvec{x}} \right|\\
			\begin{aligned}
				& \leq \int_{\mathcal{E}} \int_{\mathbb{R}_+} \left| (1 + z^{2k})^{-\frac{1}{2}} \partial_{\xvec{x}}^{\xsym{\alpha}} \partial_t^{\gamma} (\lambda^{\pm}\xsym{\mathfrak{g}}^{+} + \lambda^{\mp}\xsym{\mathfrak{g}}^{-})(\xvec{x},t) \cdot (1 + z^{2k})^{\frac{1}{2}} \partial_z \partial_{\xvec{x}}^{\xsym{\alpha}} \partial_t^{\gamma} \xvec{v}^{\pm}(\xvec{x},t,z) \right|  \, \xdx{z}  \xdx{\xvec{x}} \\
				& \leq \sum\limits_{\square\in\{+,-\}}\frac{\lambda^+ \square \lambda^-}{8} \int_{\mathcal{E}}\int_{\mathbb{R}_+} (1 + z^{2k})|\partial_z \partial_{\xvec{x}}^{\xsym{\alpha}}\partial_t^{\gamma} (\xvec{v}^{+} \square \, \xvec{v}^-)|^2 \, \xdx{z}  \xdx{\xvec{x}} + C_0.
			\end{aligned}
		\end{multline}
		Thus, after collecting \cref{equation:hm12estbldthigherxtI1,equation:hm12estbldthigherxtI11,equation:hm12estbldthigherxtI12}, one obtains the bound
		\begin{multline}\label{equation:hm12estbldthigherxtI1c}
			\sum\limits_{\square\in\{+,-\}} \left(|I_{1}^{\square}|  + \frac{\lambda^+ \square \, \lambda^-}{4}\int_{\mathcal{E}}\int_{\mathbb{R}_+} (1+z^{2k})|\partial_z \partial_{\xvec{x}}^{\xsym{\alpha}}\partial_t^{\gamma} (\xvec{v}^{+} \square \, \xvec{v}^{-})|^2 \, \xdx{z} \xdx{\xvec{x}}\right) \\
			\leq C\sum\limits_{\square\in\{+,-\}}\|\partial_t^{\gamma}\xvec{v}^{\pm}\|_{k,|\xsym{\alpha}|,0, \mathcal{E}}^2 + C_0.
		\end{multline}
		
		Concerning $I_{2}^{\pm}$, expanding the derivatives leads to
		\begin{equation*}
			\begin{aligned}
				|I_{2}^{\pm}| & \leq \left| \sum\limits_{l=0}^{\gamma} \sum\limits_{\xvec{0} < \xsym{\kappa} \leq \xsym{\alpha}} \binom{\gamma}{l}\binom{\xsym{\alpha}}{\xsym{\kappa}} \int_{\mathcal{E}}  \int_{\mathbb{R}_+} (z+z^{2k+1}) 	  \partial_{\xvec{x}}^{\xsym{\kappa}}\partial_t^l \mathfrak{f}^{\pm} \partial_{\xvec{x}}^{\xsym{\alpha}-\xsym{\kappa}}\partial_t^{\gamma-l}\partial_z\xvec{v}^{\pm} \cdot \partial_{\xvec{x}}^{\xsym{\alpha}}\partial_t^{\gamma} \xvec{v}^{\pm}  \, \xdx{z}  \xdx{\xvec{x}}  \right| \\
				& \quad + \left| \sum\limits_{l=0}^{\gamma} \binom{\gamma}{l} \int_{\mathcal{E}}  \int_{\mathbb{R}_+} (z+z^{2k+1})  \partial_t^l \mathfrak{f}^{\pm}  \partial_{\xvec{x}}^{\xsym{\alpha}}\partial_t^{\gamma-l}\partial_z\xvec{v}^{\pm} \cdot \partial_{\xvec{x}}^{\xsym{\alpha}}\partial_t^{\gamma} \xvec{v}^{\pm}  \, \xdx{z}  \xdx{\xvec{x}} \right|,
			\end{aligned}
		\end{equation*}
		which implies that $|I_{2}^{\pm}|$ is less than or equal to
		\begin{multline}\label{equation:hm12estbldthigherxtI2}
			C \sum\limits_{l=0}^{\gamma} \sum\limits_{\xvec{0} < \xsym{\kappa} \leq \xsym{\alpha}} \int_{\mathbb{R}_+} (z+z^{2k+1}) 	   \|\partial_{\xvec{x}}^{\xsym{\kappa}}\partial_t^l \mathfrak{f}^{\pm}\|_{\xLn{6}(\mathcal{E})} \|\partial_{\xvec{x}}^{\xsym{\alpha}-\xsym{\kappa}}\partial_t^{\gamma-l}\partial_z\xvec{v}^{\pm}\|_{\xLn{3}(\mathcal{E})} \|\partial_{\xvec{x}}^{\xsym{\alpha}}\partial_t^{\gamma} \xvec{v}^{\pm}\|_{\xLn{2}(\mathcal{E})} \, \xdx{z}\\
			\begin{aligned}
				& \quad + C \sum\limits_{l=1}^{\gamma} \int_{\mathbb{R}_+} (z+z^{2k+1}) \| \partial_{\xvec{x}}^{\xsym{\alpha}}\partial_t^{\gamma-l}\partial_z\xvec{v}^{\pm}\|_{\xLtwo(\mathcal{E})}\| \partial_{\xvec{x}}^{\xsym{\alpha}}\partial_t^{\gamma} \xvec{v}^{\pm}\|_{\xLtwo(\mathcal{E})} \, \xdx{z}\\
				& \quad + C \int_{\mathbb{R}_+} (1+(2k+1)z^{2k}) \| \partial_{\xvec{x}}^{\xsym{\alpha}}\partial_t^{\gamma}\xvec{v}^{\pm}\|_{\xLn{2}(\mathcal{E})}^2 \, \xdx{z}.
			\end{aligned}
		\end{multline}
		Hence, the interpolation inequality $\|\cdot\|_{\xLn{3}(\mathcal{E})} \leq C \|\cdot\|_{\xLn{2}(\mathcal{E})}^{1/2}\|\cdot\|_{\xHone(\mathcal{E})}^{1/2}$ and \eqref{equation:hm12estbldthigherxtI2} yield for arbitrary $\ell > 0$ that
		\begin{equation}\label{equation:hm12estbldthigherxtI2b}
			\begin{aligned}
				|I_{2}^{\pm}| & \leq C \sum\limits_{l=0}^{\gamma} \sum\limits_{\xvec{0} < \xsym{\kappa} \leq \xsym{\alpha}} \int_{\mathbb{R}_+} \Big\{ \|(1+z^{2k})^{\frac{1}{2}} \partial_{\xvec{x}}^{\xsym{\alpha}}\partial_t^{\gamma} \xvec{v}^{\pm}\|_{\xLn{2}(\mathcal{E})} \|(1+z^{2k+4})^{\frac{1}{2}} \partial_{\xvec{x}}^{\xsym{\alpha}-\xsym{\kappa}}\partial_t^{\gamma-l}\partial_z\xvec{v}^{\pm}\|_{\xLn{2}(\mathcal{E})}^{\frac{1}{2}}  \\
				& \quad \times \|(1+z^{2k})^{\frac{1}{2}} \partial_{\xvec{x}}^{\xsym{\alpha}-\xsym{\kappa}}\partial_t^{\gamma-l}\partial_z\xvec{v}^{\pm}\|_{\xHone(\mathcal{E})}^{\frac{1}{2}} \Big\} \, \xdx{z} \\
				& \quad + \ell \sum\limits_{l=1}^{\gamma} \|\partial_t^{\gamma-l}\xvec{v}^{\pm}\|_{k+1,|\xsym{\alpha}|,1,\mathcal{E}}^2 +  C(\ell) \|\partial_t^{\gamma}\xvec{v}^{\pm}\|_{k,|\xsym{\alpha}|,0,\mathcal{E}}^2 + C \|\partial_t^{\gamma}\xvec{v}^{\pm}\|_{k,|\xsym{\alpha}|,0,\mathcal{E}}^2\\
				& \leq \sum\limits_{l=0}^{\gamma} \sum\limits_{\xvec{0} < \xsym{\kappa} \leq \xsym{\alpha}} C \|\partial_t^{\gamma-l} \xvec{v}^{\pm}\|_{k+2,|\xsym{\alpha}-\xsym{\kappa}|,1,\mathcal{E}}^{\frac{1}{2}}\|\partial_t^{\gamma-l} \xvec{v}^{\pm}\|_{k,|\xsym{\alpha}-\xsym{\kappa}|+1,1,\mathcal{E}}^{\frac{1}{2}}\|\partial_t^{\gamma} \xvec{v}^{\pm}\|_{k,|\xsym{\alpha}|,0,\mathcal{E}}\\
				& \quad + \ell \sum\limits_{l=1}^{\gamma}  \|\partial_t^{\gamma-l}\xvec{v}^{\pm}\|_{k+1,|\xsym{\alpha}|,1,\mathcal{E}}^2 + C(\ell) \|\partial_t^{\gamma}\xvec{v}^{\pm}\|_{k,|\xsym{\alpha}|,0,\mathcal{E}}^2.
			\end{aligned}
		\end{equation}
		In \eqref{equation:hm12estbldthigherxtI2b}, we used the elementary inequality $z + z^{2k+1} \leq (1+z^k)(1+z^{k+1})$. Then, the number $|I_{2}^{\pm}|$ is less than or equal to
		\begin{equation}\label{equation:hm12estbldthigherxtI2c}
			\begin{multlined}
				\sum\limits_{l=0}^{\gamma} \sum\limits_{\xvec{0} < \xsym{\kappa} \leq \xsym{\alpha}}\left( C(\ell)\left( \|\partial_t^{\gamma-l} \xvec{v}^{\pm}\|_{k+2,|\xsym{\alpha}-\xsym{\kappa}|,1,\mathcal{E}}^2 + \|\partial_t^{\gamma} \xvec{v}^{\pm}\|_{k,\xsym{\alpha},0,\mathcal{E}}^2\right) + \ell\|\partial_t^{\gamma-l} \xvec{v}^{\pm}\|_{k,|\xsym{\alpha}-\xsym{\kappa}|+1,1,\mathcal{E}}^2\right) \\
				+ \ell \sum\limits_{l=1}^{\gamma} \|\partial_t^{\gamma-l}\xvec{v}^{\pm}\|_{k+1,|\xsym{\alpha}|,1,\mathcal{E}}^2 +  C(\ell) \|\partial_t^{\gamma}\xvec{v}^{\pm}\|_{k,|\xsym{\alpha}|,0,\mathcal{E}}^2.
			\end{multlined}
		\end{equation}
		
		It remains treating $I_3^{\pm,a}$ and $I_3^{\pm,b}$. For a vector field  $\widetilde{\xvec{z}}$ on $\mathcal{E}$, we denote by $\xD^m(\widetilde{\xvec{z}})$ arbitrary linear combinations of components of $\widetilde{\xvec{z}}$ and derivatives of such, which are taken in $\xvec{x}$ and are of order~$\leq~m$, while the coefficients can depend on $\xvec{n}$. Given a multi-index $\widetilde{\xsym{\alpha}}$ with $|\widetilde{\xsym{\alpha}}| = \widetilde{m} \in \mathbb{N}$, the relations $\xvec{v}^{\pm} \cdot \xvec{n} = 0$ imply
		\begin{equation}\label{equation:tanexflouoc}
			\xvec{n} \cdot \partial_t^{l} \partial_{\xvec{x}}^{\widetilde{\xsym{\alpha}}} \xvec{v}^{\pm} = \xD^{\widetilde{m}-1}(\partial_t^{l}\xvec{v}^{\pm}) = \partial_t^{l} \xD^{\widetilde{m}-1}(\xvec{v}^{\pm}).
		\end{equation}
		Therefore, in view of \eqref{equation:tanexflouoc} and the definition of the tangential part $[\cdot]_{\operatorname{tan}}$, one may write
		\begin{equation*}\label{equation:hm12estbldthigherxtI3a}
			\begin{aligned}
				I_3^{\pm,a} & = \sum\limits_{l=0}^{\gamma}\binom{\gamma}{l} \int_{\mathcal{E}}\int_{\mathbb{R}_+} (1+z^{2k}) \xD^{|\xsym{\alpha}|}(\partial_t^{l}\xvec{z}^{\mp,0}) \xD^{|\xsym{\alpha}|}(\partial_t^{\gamma-l}\xvec{v}^{\pm})\xD^{|\xsym{\alpha}|}(\partial_t^{\gamma}\xvec{v}^{\pm}) \, \xdx{z}  \xdx{\xvec{x}}\\
				& \quad + \sum\limits_{l=0}^{\gamma}\binom{\gamma}{l} \int_{\mathcal{E}}\int_{\mathbb{R}_+} (1+z^{2k})\left( (\partial_t^{l}\xvec{z}^{\mp,0} \cdot \xdop{\nabla}) \partial_{\xvec{x}}^{\xsym{\alpha}} \partial_t^{\gamma-l}\xvec{v}^{\pm} \cdot \xvec{n}\right) \xD^{|\xsym{\alpha}|-1} \left(\partial_t^{\gamma} \xvec{v}^{\pm}\right) \, \xdx{z}  \xdx{\xvec{x}}\\
				& \quad + \sum\limits_{l=0}^{\gamma}\binom{\gamma}{l} \int_{\mathcal{E}}\int_{\mathbb{R}_+} (1+z^{2k}) (\partial_t^{l}\xvec{z}^{\mp,0} \cdot \xdop{\nabla}) \partial_{\xvec{x}}^{\xsym{\alpha}} \partial_t^{\gamma-l}\xvec{v}^{\pm} \cdot \partial_{\xvec{x}}^{\xsym{\alpha}} \partial_t^{\gamma} \xvec{v}^{\pm} \, \xdx{z}  \xdx{\xvec{x}},
			\end{aligned}
		\end{equation*}
		and integration by parts implies that the second line is of the same type as the first one. Thus, 
		\begin{equation}\label{equation:hm12estbldthigherxtI3aa}
			\begin{aligned}
				I_3^{\pm,a} & = \sum\limits_{l=0}^{\gamma}\binom{\gamma}{l} 	\int_{\mathcal{E}}\int_{\mathbb{R}_+} (1+z^{2k}) \xD^{|\xsym{\alpha}|}(\partial_t^{l}\xvec{z}^{\mp,0}) \xD^{|\xsym{\alpha}|}(\partial_t^{\gamma-l}\xvec{v}^{\pm})\xD^{|\xsym{\alpha}|}(\partial_t^{\gamma}\xvec{v}^{\pm}) \, \xdx{z}  \xdx{\xvec{x}}\\
				& \quad + \sum\limits_{l=1}^{\gamma}\binom{\gamma}{l} 	\int_{\mathcal{E}}\int_{\mathbb{R}_+} (1+z^{2k}) (\partial_t^{l}\xvec{z}^{\mp,0} \cdot \xdop{\nabla}) \partial_{\xvec{x}}^{\xsym{\alpha}} \partial_t^{\gamma-l}\xvec{v}^{\pm} \cdot \partial_{\xvec{x}}^{\xsym{\alpha}} \partial_t^{\gamma} \xvec{v}^{\pm} \, \xdx{z}  \xdx{\xvec{x}}\\
				& \quad + \int_{\mathcal{E}}\int_{\mathbb{R}_+} (1+z^{2k}) (\xvec{z}^{\mp,0} \cdot \xdop{\nabla}) \partial_{\xvec{x}}^{\xsym{\alpha}} \partial_t^{\gamma}\xvec{v}^{\pm} \cdot \partial_{\xvec{x}}^{\xsym{\alpha}} \partial_t^{\gamma} \xvec{v}^{\pm} \, \xdx{z}  \xdx{\xvec{x}}.
			\end{aligned}
		\end{equation}	
		Due to $\xvec{z}^{\mp,0} \cdot \xvec{n} = 0$ on $\partial \mathcal{E}$, the last integral in \eqref{equation:hm12estbldthigherxtI3aa} reads
		\begin{multline*}
			\int_{\mathcal{E}}\int_{\mathbb{R}_+} (1+z^{2k}) (\xvec{z}^{\mp,0} \cdot \xdop{\nabla}) \partial_{\xvec{x}}^{\xsym{\alpha}} \partial_t^{\gamma}\xvec{v}^{\pm} \cdot \partial_{\xvec{x}}^{\xsym{\alpha}} \partial_t^{\gamma} \xvec{v}^{\pm} \, \xdx{z}  \xdx{\xvec{x}} \\
			= -\frac{1}{2}\int_{\mathcal{E}}\int_{\mathbb{R}_+} (1+z^{2k}) |\partial_{\xvec{x}}^{\xsym{\alpha}} \partial_t^{\gamma} \xvec{v}^{\pm}|^2 (\xdiv{\xvec{z}^{\mp,0}}) \, \xdx{z} \xdx{\xvec{x}}
		\end{multline*}
		such that \eqref{equation:hm12estbldthigherxtI3aa} eventually implies the bound
		\begin{equation}\label{equation:hm12estbldthigherxtI3a2}
			\begin{gathered}
				|I_3^{\pm,a}| \leq C \sum\limits_{l=1}^{\gamma}  \|\partial_t^{\gamma-l}\xvec{v}^{\pm}\|_{k,|\xsym{\alpha}|+1,0,\mathcal{E}}^2 + C\|\partial_t^{\gamma}\xvec{v}^{\pm}\|_{k,|\xsym{\alpha}|,0,\mathcal{E}}^2.
			\end{gathered}
		\end{equation}
		When it comes to $I_3^{\pm,b}$, the corresponding estimates are less demanding compared to those for $I_3^{\pm,a}$ and one finds
		\begin{equation}\label{equation:hm12estbldthigherxtI3b1}
			\begin{aligned}
				|I_3^{\pm,a}| + |I_3^{\pm,b}| & \leq C \left( \|\partial_t^{\gamma}\xvec{v}^{+}\|_{k,|\xsym{\alpha}|,0,\mathcal{E}}^2 + \|\partial_t^{\gamma}\xvec{v}^{-}\|_{k,|\xsym{\alpha}|,0,\mathcal{E}}^2\right) \\
				& \quad + C\sum\limits_{l=0}^{\gamma-1} \left( \|\partial_t^{l}\xvec{v}^{+}\|_{k,|\xsym{\alpha}|+1,0,\mathcal{E}}^2 + \|\partial_t^{l}\xvec{v}^{-}\|_{k,|\xsym{\alpha}|+1,0,\mathcal{E}}^2\right).
			\end{aligned}
		\end{equation}
		
		In order to collect the previous estimates, for fixed $m_1, m_2 \in \mathbb{N}_0$, we sum in \cref{equation:hm12estbldthigherxtI1c}, \cref{equation:hm12estbldthigherxtI2c,equation:hm12estbldthigherxtI3a2,equation:hm12estbldthigherxtI3b1} over all $|\xsym{\alpha}| \leq m_1$ and $\gamma \leq m_2$. Moreover, we denote
		\begin{equation*}
			\begin{aligned}
				\varPhi^{\pm} &\coloneqq C(\ell)\sum\limits_{\gamma=0}^{m_2}\left( \delta_{0,m_1}^{\dagger}\|\partial_t^{\gamma} \xvec{v}^{\pm}\|_{\xHn{{k+2,\max\{m_1-1,0\},1}}_{\mathcal{E}}}^2 + \|\partial_t^{\gamma} \xvec{v}^{\pm}\|_{\xHn{{k,m_1,0}}_{\mathcal{E}}}^2\right) + \ell\sum\limits_{\gamma=0}^{m_2}\|\partial_t^{\gamma} \xvec{v}^{\pm}\|_{\xHn{{k,m_1,1}}_{\mathcal{E}}}^2 \\
				& \quad + C(\ell)\sum\limits_{\gamma=0}^{m_2}  \delta_{0,\gamma}^{\dagger}\left(\|\partial_t^{\max\{\gamma-1,0\}}\xvec{v}^{\pm}\|_{\xHn{{k+1,m_1,1}}_{\mathcal{E}}}^2 + \|\partial_t^{\max\{\gamma-1,0\}} \xvec{v}^{\pm}\|_{\xHn{{k,m_1+1,0}}_{\mathcal{E}}}^2 \right)  + C_0,
			\end{aligned}
		\end{equation*}
		with
		\[
		\delta_{a,b}^{\dagger} \coloneqq \begin{cases}
			0 & \mbox{ if } a = b,\\ 
			1 & \mbox{ otherwise}.
		\end{cases}
		\]
		As a result, one obtains the estimate
		\begin{equation}\label{equation:hm12estbldthigherxtsum}
			\sum\limits_{\square\in\{+,-\}}\left(\partial_t \sum\limits_{\gamma=0}^{m_2} \|\partial_t^{\gamma}\xvec{v}^{\square}\|_{\xHn{{k,m_1,0}}_{\mathcal{E}}}^2 + \sum\limits_{\gamma=0}^{m_2}\frac{\lambda^+ \square \, \lambda^-}{2}\|\partial_t^{\gamma}(\xvec{v}^{+}\square \, \xvec{v}^{-})\|_{\xHn{{k,m_1,1}}_{\mathcal{E}}}^2\right)
			\leq 2\left(\varPhi^+ + \varPhi^-\right).
		\end{equation}
		On the right-hand side of \eqref{equation:hm12estbldthigherxtsum}, all terms containing norms of the spaces
		\[
		\xHn{{k+2,\max\{m_1-1,0\},1}}_{\mathcal{E}}, \quad \xHn{{k+1,m_1,1}}_{\mathcal{E}}, \quad \xHn{{k,m_1+1,1}}_{\mathcal{E}}
		\]
		disappear in the respective base cases when $m_1 = 0$ or $m_2 = 0$.
		Thus, inductively with respect to~$m_1$~and~$m_2$, and by using a Gr\"onwall argument for \eqref{equation:hm12estbldthigherxtsum} with $\ell > 0$ sufficiently small, one can obtain
		\begin{equation*}\label{equation:hiblvxt}
			\partial_t^{\gamma}\xvec{v}^{\pm} \in \xLinfty((0,T);H^{{k,m_1,0}}_{\mathcal{E}}) \cap \xLtwo((0,T); \xHn{{k,m_1,1}}_{\mathcal{E}}), \quad \gamma \in \{0,\dots,m_2\}
		\end{equation*}
		and the estimate \eqref{equation:hibl} when $m_3 = 0$.
		
		\paragraph{Step 2. Estimates for $\partial_{\xvec{x}}^{\xsym{\alpha}}\partial_t^{\gamma}\partial_z^{\beta}\xvec{v}^{\pm}$.}
		From \eqref{equation:MHD_ElsaesserBlProfilevpm2} and the regularity obtained in Step~1, one can estimate the boundary values of $\partial_{\xvec{x}}^{\xsym{\alpha}}\partial_t^{\gamma}\partial_z^2 \xvec{v}^{\pm}$ at $z = 0$ in $\xLinfty((0,T);\xHn{{m}}(\mathcal{E}))$ for any $m \in \mathbb{N}$ by a constant $C > 0$ of the type stated in \eqref{equation:constant_type}. Therefore, by acting on \eqref{equation:MHD_ElsaesserBlProfilevpm2} with $\partial_z$, one obtains the $\xWn{{m_2,\infty}}((0,T);\xHn{{k,m_1,1}}_{\mathcal{E}}) \cap \xHn{{m_2}}((0,T);\xHn{{k,m_1,2}}_{\mathcal{E}})$ estimates for $\xvec{v}^{\pm}$ by analysis similar to Step~1. The boundary values of $\partial_{\xvec{x}}^{\xsym{\alpha}}\partial_t^{\gamma}\partial_z^3 \xvec{v}^{\pm}$ at $z = 0$ are then again bounded via \eqref{equation:MHD_ElsaesserBlProfilevpm2} and Step~1.  After acting with $\partial_z^2$ on \eqref{equation:MHD_ElsaesserBlProfilevpm2}, one can derive the $\xWn{{m_2,\infty}}((0,T);\xHn{{k,m_1,2}}_{\mathcal{E}}) \cap \xHn{{m_2}}((0,T);\xHn{{k,m_1,3}}_{\mathcal{E}})$ estimates for $\xvec{v}^{\pm}$. By induction over~$m_3$, the proof of \Cref{lemma:higherorder} is concluded.
	\end{proof}

	\section{Proof of \texorpdfstring{\Cref{lemma:reg}}{the regularization lemma}}\label{appendix:proofreg}
	To prove \Cref{lemma:reg}, we proceed essentially along the lines of \cite[Lemma 9]{CoronMarbachSueur2020} and \cite[Lemma 2.1]{ChavesSilva2020SmalltimeGE}, where Navier--Stokes and Boussinesq systems have been considered. That is, in \Cref{subsection:estsuffregsol} below, assuming~$N \in \{2,3\}$ and~$\xvec{M}_{1}, \xvec{M}_{2}, \xvec{L}_{1},\xvec{L}_{2} \in \xCinfty(\overline{\mathcal{E}};\mathbb{R}^{N\times N})$, we obtain \apriori~estimates that are valid for smooth solutions. To conclude \Cref{lemma:reg} from there, it seems the existing literature does not provide a suitable theory of strong solutions under general coupled Navier slip-with-friction boundary conditions for MHD. In particular, since the boundary conditions in \eqref{equation:MHD_ElsaesserExt} are generally non-symmetric, we are unaware of eigenvector bases for respectively coupled Stokes type problems, preventing us to employ usual Galerkin method arguments, \eg, as in \cite{Temam2001}, to obtain the $\xLinfty((0,T);\xHone(\mathcal{E}))$ and $\xLinfty((0,T);\xHtwo(\mathcal{E}))$ strong solutions to \eqref{equation:MHD_ElsaesserExt}. When the domain is simply-connected, $\xvec{M}_1 = \xvec{L}_2 = \xvec{M}$ with $\xvec{M}$ being positive symmetric, and $\xvec{M}_2 = \xvec{L_1} = \xvec{0}$, such a basis of eigenvectors is available in \cite{GuoWang2016}. If $\xvec{L}_1 = \xvec{M_2} = \xvec{0}$, the argument from \cite{XiaoXin2013} also allow taking symmetric matrices $\xvec{M}_1$ and $\xvec{L}_2$ with $\xvec{M}_1 \neq \xvec{L}_2$; these references apparently do not provide eigenvector bases under general symmetric boundary conditions, where $\xvec{M}_1$, $\xvec{L}_2$ are symmetric and $\nu_1\xvec{L}_1 = \nu_2 \xvec{M}_2^{\top}$. 
	
	In $2$D, \Cref{lemma:reg} follows from the estimates in \Cref{subsection:estsuffregsol} by a different approach. First, the initial data are approximated, similarly to \cite[Appendix]{Kelliher2006}, in $\xLtwo(\mathcal{E})$ by $\xW(\mathcal{E})\cap\xHtwo(\mathcal{E})$ functions satisfying the Navier slip-with-friction boundary conditions. Then, a Leray--Hopf weak solution is constructed as the limit of a sequence of $\xLinfty((0,T);\xHn{2}(\mathcal{E})\cap\xW(\mathcal{E}))$ solutions. This idea relies on the assumption $N = 2$, which guarantees that the sequence of strong solutions is defined on a fixed time interval $[0,T]$. 
	
	\subsection{Estimates for sufficiently regular solutions}\label{subsection:estsuffregsol}
	To display the estimates for $\xvec{u}$ and $\xvec{B}$ simultaneously, the symmetric notations from \Cref{subsubsection:changeofvariables} are employed. If $(\xvec{u}, \xvec{B}) \in \mathscr{X}_{T^*} \times \mathscr{X}_{T^*}$ is a Leray--Hopf weak solution to \eqref{equation:MHD_ElsaesserExt}, then the functions
	$\xvec{z}^{\pm} = \xvec{u} \pm \sqrt{\mu}\xvec{B}$ obey the energy inequality \eqref{equation:seizpm2} and a corresponding weak formulation for the Elsasser system
	\begin{equation}\label{equation:MHD_ElsaesserExtGenLem}
		\begin{cases}
			\partial_t \xvec{z}^{\pm} - \Delta (\lambda^{\pm}\xvec{z}^{+} + \lambda^{\mp} \xvec{z}^{-}) + (\xvec{z}^{\mp} \cdot \xdop{\nabla}) \xvec{z}^{\pm} + \xdop{\nabla} p^{\pm} = \xvec{0} & \mbox{ in } \mathcal{E}_{T^*},\\
			\xdop{\nabla}\cdot\xvec{z}^{\pm} = 0 & \mbox{ in } \mathcal{E}_{T^*},\\
			\xvec{z}^{\pm} \cdot \xvec{n} = 0  & \mbox{ on }  \Sigma_{T^*},\\
			(\xcurl{\xvec{z}^{\pm}})\times\xvec{n} = \xsym{\rho}^{\pm}(\xvec{z}^{+},\xvec{z}^{-}) &\mbox{ on } \Sigma_{T^*},\\
			\xvec{z}^{\pm}(\cdot, 0) = \xvec{z}_0^{\pm}  & \mbox{ in } \mathcal{E}.
		\end{cases}
	\end{equation}
	
	To begin with, \apriori~estimates for a related stationary problem are shown based on known results for the Navier--Stokes equations.
	\begin{lmm}\label{lemma:stsreg}
		Let $k \in \mathbb{N}_0$, forces $\xvec{f}^{\pm} \in \xHn{k}(\mathcal{E})$, a vector $\xvec{b}^{\pm} \in \xHn{{k+1/2}}(\partial\mathcal{E})$ tangential to $\partial\mathcal{E}$, and friction operators $\xvec{M}^{\pm}, \xvec{L}^{\pm} \in \xCinfty(\partial\mathcal{E};\mathbb{R}^{N\times N})$ be arbitrary. Then, every solution $(\xvec{Z}^+, \xvec{Z}^-, P^+, P^-)$ with $\xvec{Z}^{\pm} \in \xHn{{k+1}}(\mathcal{E})$ to the coupled Stokes type system
		\begin{equation}\label{equation:cnbStokes}
			\begin{cases}
				- \Delta (\lambda^{\pm}\xvec{Z}^{+} + \lambda^{\mp} \xvec{Z}^-) + \xdop{\nabla} P^{\pm} = \xvec{f}^{\pm} & \mbox{ in } \mathcal{E}, \\
				\, \xdiv{\xvec{Z}^{\pm}} = 0 & \mbox{ in } \mathcal{E},\\
				\xvec{Z}^{\pm} \cdot \xvec{n} = 0 & \mbox{ on } \partial \mathcal{E},\\
				\, (\xcurl{\xvec{Z}}^{\pm}) \times \xvec{n} +  \left[ \xvec{M}^{\pm}\xvec{Z}^+ + \xvec{L}^{\pm} \xvec{Z}^-\right]_{\operatorname{tan}} = \xvec{b}^{\pm} & \mbox{ on } \partial \mathcal{E}
			\end{cases}
		\end{equation}
		obeys the estimate
		\begin{multline}\label{equation:aprioristokestype}
			\sum\limits_{\square\in\{+,-\}} \left(\| \xvec{Z}^{\square} \|_{\xHn{{k+2}}(\mathcal{E})} + \| P^{\square} \|_{\xHn{{k+1}}(\mathcal{E})}\right) 
			\\ 
			\leq C \sum\limits_{\square\in\{+,-\}} \left(\| \xvec{f}^{\square} \|_{\xHn{{k}}(\mathcal{E})} + \| \xvec{b}^{\square} \|_{\xHn{{k+1/2}}(\partial\mathcal{E})} + \|\xvec{Z}^{\square}\|_{\xHn{{k+1}}(\mathcal{E})}\right).
		\end{multline}
	\end{lmm}
	\begin{proof}
		In the case of uncoupled boundary conditions, where $\xvec{M}^{\pm} = \xvec{L}^{\pm} = \xvec{0}$, the functions $\xvec{U} \coloneqq \xvec{Z}^+ + \xvec{Z}^-$ and $\xvec{V} \coloneqq \xvec{Z}^+ - \xvec{Z}^-$ both obey independent Stokes problems under Navier slip-with-friction boundary conditions. Thus, from \cite[Pages 90-94]{Guerrero2006}, one has for each $k \geq 0$ the estimates
		\begin{multline*}
			\| \xvec{U} \|_{\xHn{{k+2}}(\mathcal{E})} + \| P^+ + P^- \|_{\xHn{{k+1}}(\mathcal{E})}
			\\
			\leq C \left(\| \xvec{f}^+ + \xvec{f}^- \|_{\xHn{{k}}(\mathcal{E})} + \| \xvec{b}^+ + \xvec{b}^- \|_{\xHn{{k+1/2}}(\partial\mathcal{E})} + \|\xvec{U}\|_{\xHn{{k+1}}(\mathcal{E})}\right)
		\end{multline*}
		and
		\begin{multline*}
			\| \xvec{V} \|_{\xHn{{k+2}}(\mathcal{E})} + \| P^+ - P^- \|_{\xHn{{k+1}}(\mathcal{E})}
			\\
			\leq C \left(\| \xvec{f}^+ - \xvec{f}^- \|_{\xHn{{k}}(\mathcal{E})} + \| \xvec{b}^+ - \xvec{b}^- \|_{\xHn{{k+1/2}}(\partial\mathcal{E})} + \|\xvec{V}\|_{\xHn{{k+1}}(\mathcal{E})}\right),
		\end{multline*}
		which imply \eqref{equation:aprioristokestype} by means of the triangle inequality.
		For the general case, we start with $k = 0$ and observe that every solution $(\xvec{Z}^+, \xvec{Z}^-, P^+, P^-)$ to \eqref{equation:cnbStokes} satisfies
		\begin{equation}\label{equation:cnbStokes2}
			\begin{cases}
				- \Delta (\lambda^{\pm}\xvec{Z}^{+} + \lambda^{\mp} \xvec{Z}^-) + \xdop{\nabla} P^{\pm} = \xvec{f}^{\pm} & \mbox{ in } \mathcal{E},\\
				\xdiv{\xvec{Z}^{\pm}} = 0 & \mbox{ in } \mathcal{E},\\
				\xvec{Z}^{\pm} \cdot \xvec{n} = 0 & \mbox{ on } \partial \mathcal{E},\\
				(\xcurl{\xvec{Z}}^{\pm}) \times \xvec{n} = \widetilde{\xvec{b}}^{\pm} & \mbox{ on } \partial \mathcal{E},
			\end{cases}
		\end{equation}
		with
		\begin{equation*}
			\begin{aligned}
				\widetilde{\xvec{b}}^+ \coloneqq \xvec{b}^+ - \left[\xvec{M}^+ \xvec{Z}^+ + \xvec{L}^+ \xvec{Z}^-\right]_{\operatorname{tan}}, \quad \widetilde{\xvec{b}}^- \coloneqq \xvec{b}^- - \left[\xvec{M}^- \xvec{Z}^+ + \xvec{L}^- \xvec{Z}^-\right]_{\operatorname{tan}}.
			\end{aligned}
		\end{equation*}
		Since $\widetilde{\xvec{b}}^{\pm} \in \xHn{{1/2}}(\partial\mathcal{E})$, after applying to \eqref{equation:cnbStokes2} the result for uncoupled boundary conditions explained above, one finds
		\begin{multline*}
			\sum\limits_{\square\in\{+,-\}} \left(\| \xvec{Z}^{\square} \|_{\xHn{{2}}(\mathcal{E})} + \| P^{\square} \|_{\xHn{{1}}(\mathcal{E})}\right)
			\\
			\leq C \sum\limits_{\square\in\{+,-\}} \left(\| \xvec{f}^{\square} \|_{\xLtwo(\mathcal{E})} + \| \xvec{b}^{\square} \|_{\xHn{{1/2}}(\partial\mathcal{E})}  + \|\xvec{Z}^{\square}\|_{\xHone(\mathcal{E})}\right).
		\end{multline*}
		Inductively, if \eqref{equation:aprioristokestype} is true for a fixed $k \in \mathbb{N}$, one has $\widetilde{\xvec{b}}^{\pm} \in \xHn{{k+3/2}}(\partial\mathcal{E})$ and the known estimates for \eqref{equation:cnbStokes2} lead to \eqref{equation:aprioristokestype} with $k$ being replaced by $k + 1$.
	\end{proof}
	
	\begin{rmrk}\label{remark:Wkp}
		The statement of \Cref{lemma:stsreg} remains valid in $\xWn{{k,p}}(\mathcal{E})$ spaces. For instance, if $\xvec{Z}^{\pm} \in \xWn{{1,p}}(\mathcal{E})$ and $\xvec{b}^{\pm} \in \xWn{{1-\frac{1}{p},p}}(\partial \mathcal{E})$, then $\widetilde{\xvec{b}}^{\pm} \in \xWn{{1-\frac{1}{p},p}}(\partial \mathcal{E})$ in \eqref{equation:cnbStokes2} and one can apply \cite[Theorem 4.1]{AmroucheRejaiba2014} instead of \cite[Pages 90-94]{Guerrero2006}.
	\end{rmrk}
	
	In what follows, the operator $\mathbb{P}$ denotes the Leray projector in $\xLtwo(\mathcal{E})$ onto $\xH(\mathcal{E})$ and thus, for any selection $\xvec{h}^+, \xvec{h}^- \in \xH(\mathcal{E}) \cap \xHtwo(\mathcal{E})$ with $\xmcal{N}^{\pm}(\xvec{h}^+, \xvec{h}^-) = \xvec{0}$, \Cref{lemma:stsreg} provides
	\begin{equation}\label{equation:estlerayproj}
		\begin{aligned}
			\|\xvec{h}^{\pm} \|_{\xHtwo(\mathcal{E})}^2
			\leq C \left(\sum\limits_{\square\in\{+,-\}}\|\mathbb{P}\Delta(\lambda^{\pm}\xvec{h}^+ + \lambda^{\mp}\xvec{h}^-)\|_{\xLtwo(\mathcal{E})}^2 + \|\xvec{h}^{\square}\|_{\xHn{{1}}(\mathcal{E})}^2\right).
		\end{aligned}
	\end{equation}
	
	The proof of \Cref{lemma:reg} is now completed by means of the following steps.
	\paragraph{Step 1. Basic energy estimates when $t \in (0,T^*/3)$.}
	We introduce for a positive parameter $\varkappa > 0$ the quantity
	\[
	F^{\varkappa}(t) \coloneqq \frac{1}{2} \sum\limits_{\square\in\{+,-\}} \left(\frac{d}{dt} \| \xvec{z}^{\square}(\cdot, t) \|_{\xLtwo(\mathcal{E})}^2 + \varkappa(\lambda^+ \square \, \lambda^-)\| \xcurl{(\xvec{z}^{+} \square \, \xvec{z}^{-})}(\cdot, t) \|_{\xLtwo(\mathcal{E})}^2\right).
	\]
	Then, the strong energy inequality \eqref{equation:seizpm2} with $\xsym{\xi}^{\pm} = 0$ provides
	\begin{equation*}
		\begin{aligned}
			\int_0^t F^1(s) \, \xdx{s} \leq \sum\limits_{\substack{(\triangle,\circ)\in \\ \{(+,-), (-,+)\}}} \int_0^t\int_{\partial\mathcal{E}} \left(\lambda^+\xsym{\rho}^{\circ}(\xvec{z}^+,\xvec{z}^-) + \lambda^-\xsym{\rho}^{\triangle}(\xvec{z}^+,\xvec{z}^-)\right) \cdot \xvec{z}^{\circ} \, \xdx{S} \xdx{s}.
		\end{aligned}
	\end{equation*}
	Furthermore, by \eqref{equation:sKem} and trace estimates, together with $\xsym{\rho}^{\pm} \in \xCinfty(\partial\mathcal{E};\mathbb{R}^{N\times 2N})$ being fixed, one has for small $\ell > 0$ and $\square,\triangle \in \{+,-\}$ the bound
	\begin{equation*}
		\begin{aligned}
			\int_{\partial\mathcal{E}} \xsym{\rho}^{\square}(\xvec{z}^+,\xvec{z}^-) \cdot \xvec{z}^{\triangle} \, \xdx{S} & \leq C \sum\limits_{\diamond,\circ \in\{+,-\}} \|\xvec{z}^{\diamond}\|_{\xHone(\mathcal{E})} \| \xvec{z}^{\circ}\|_{\xLtwo(\mathcal{E})} \\
			& \leq \sum\limits_{\diamond,\circ \in\{+,-\}}\left(\ell \|\xcurl{\xvec{z}^{\diamond}}\|_{\xLtwo(\mathcal{E})}^2 + C(\ell) \| \xvec{z}^{\circ}\|_{\xLtwo(\mathcal{E})}^2\right).
		\end{aligned}
	\end{equation*} 
	Thus, for $\ell > 0$ sufficiently small it follows
	\begin{equation*}
		\begin{aligned}
			\int_0^t F^{1/2}(s) \, \xdx{s} \leq C(\ell) \int_0^t \left(\|\xvec{z}^{+}(\cdot,s)\|_{\xLtwo(\mathcal{E})}^2 + \|\xvec{z}^{-}(\cdot,s)\|_{\xLtwo(\mathcal{E})}^2\right) \, \xdx{s} .
		\end{aligned}
	\end{equation*}
	
	As a result, by employing \eqref{equation:sKem} similarly as in \Cref{subsubsection:energyestimates_remainder} and further utilizing Gr\"onwall's inequality, one obtains for $t \in (0,T^*)$ the energy estimate
	\begin{multline}\label{equation:beepr}
		\sum\limits_{\square\in\{+,-\}} \left(\| \xvec{z}^{\square}(\cdot, t) \|_{\xLtwo(\mathcal{E})}^2  + \frac{(\lambda^+ \square\, \lambda^-)}{2} \int_0^t \| (\xvec{z}^{+} \square \, \xvec{z}^{-})(\cdot, s) \|_{\xHone(\mathcal{E})}^2 \, \xdx{s}\right) \\
		\leq C \sum\limits_{\square\in\{+,-\}} \| \xvec{z}^{\square}_0 \|_{\xLtwo(\mathcal{E})}^2.
	\end{multline}
	Therefore, by a contradiction argument, there exists $C_1 > 0$ and a possibly small time $t_1 \in [0, T^*/3]$ for which
	\begin{equation*}
		\begin{gathered}
			\| \xvec{z}^{\pm}(\cdot, t_1) \|_{\xHone(\mathcal{E})} \leq \sqrt{\frac{3C_1}{T^*} \left(\| \xvec{z}^{+}_0 \|_{\xLtwo(\mathcal{E})}^2 + \| \xvec{z}^{-}_0 \|_{\xLtwo(\mathcal{E})}^2\right)}.
		\end{gathered}
	\end{equation*}
	
	\paragraph{Step 2. Higher order \apriori~estimates when $t \in (t_1,2T^*/3)$.} 
	We apply the Leray projector $\mathbb{P}$ in \eqref{equation:MHD_ElsaesserExtGenLem} and multiply with $\mathbb{P} \Delta \xvec{z}^{\pm}$. Subsequently, the results are added up and integrated over $\mathcal{E}$. Hereto, we denote for $\delta > 0$ the auxiliary function
	\begin{equation*}
		\begin{aligned}
			G^{\delta}(t) & \coloneqq  \frac{1}{2} \sum\limits_{\square\in\{+,-\}}  \left(\frac{d}{dt} \| \xcurl{\xvec{z}^{\square}}(\cdot, t) \|_{\xLtwo(\mathcal{E})}^2 + \delta  (\lambda^+ \square \, \lambda^-) \| \Delta (\xvec{z}^{+} \square \, \xvec{z}^{-})(\cdot, t) \|_{\xLtwo(\mathcal{E})}^2\right),
		\end{aligned}
	\end{equation*}
	and the nonlinear terms
	\[
	J^{\pm}(s) \coloneqq \int_{\mathcal{E}} \left| \mathbb{P}(\xvec{z}^{\mp}(\xvec{x},s) \cdot \xnab) \xvec{z}^{\pm}(\xvec{x},s)  \cdot \mathbb{P} \Delta\xvec{z}^{\pm}(\xvec{x},s) \right| \, \xdx{\xvec{x}}.
	\]
	\begin{lmm}\label{lemma:aux1212}
		Assume that $\xvec{M}_1$, $\xvec{L}_2$ are symmetric and $\nu_1\xvec{L}_1 = \nu_2\xvec{M}_2^{\top}$. For any small $\ell > 0$, and a constant $C = C(\ell)$ which is reciprocal to $\ell$, it holds
		\begin{equation}\label{equation:prfstrgslest2}
			\begin{multlined}
				\int_{t_1}^{t} G^{1/2}(s) \, \xdx{s} \leq \int_{t_1}^{t} J^{+}(s)  \, \xdx{s} + \int_{t_1}^{t} J^{-}(s) \, \xdx{s} \\ + \sum\limits_{\square\in\{+,-\}} \left(\ell \| \xcurl{\xvec{z}}^{\square}(\cdot,t)\|_{\xLtwo(\mathcal{E})}^2 +  C(C_1,\ell)\|\xvec{z}^{\square}_0\|_{\xLtwo(\mathcal{E})}^2\right).
			\end{multlined}
		\end{equation}
	\end{lmm}
	\begin{proof}
		One has to estimate several boundary integrals of the form
		\begin{equation*}
			\int_{t_1}^t \int_{\partial \mathcal{E}} \partial_t \xvec{z}^{\square} \cdot \widetilde{\xvec{M}}^{\square,\triangle} \xvec{z}^{\triangle} \, \xdx{S} \xdx{s},
		\end{equation*}
		with $\widetilde{\xvec{M}}^{\square,\triangle} \in \xCinfty(\overline{\mathcal{E}};\mathbb{R}^{N \times N})$ and $\square, \triangle \in \{+,-\}$. Thanks to the symmetry assumptions, it follows that
		\begin{multline*}\label{equation:binsa}
			\int_{t_1}^t \int_{\partial \mathcal{E}} \left(\nu_1 \xvec{M}_1 \xvec{u} \cdot \partial_t \xvec{u} + \nu_2 \xvec{L}_2 \xvec{B} \cdot \partial_t \xvec{B}\right) \, \xdx{S} \xdx{s} \\
			\begin{aligned}
				& \leq \frac{1}{2} \int_{t_1}^t \partial_t \int_{\partial \mathcal{E}} \left(\nu_1 \xvec{M}_1 \xvec{u} \cdot \xvec{u} + \nu_2 \xvec{L}_2 \xvec{B} \cdot  \xvec{B}  \right) \, \xdx{S} \xdx{s} \\
				& \leq \ell \left( \|\xvec{u}(\cdot,t)\|_{\xHone(\mathcal{E})}^2 + \|\xvec{B}(\cdot,t)\|_{\xHone(\mathcal{E})}^2\right) + C(\ell) \left( \|\xvec{u}(\cdot,t)\|_{\xLtwo(\mathcal{E})}^2 + \|\xvec{B}(\cdot,t)\|_{\xLtwo(\mathcal{E})}^2 \right) \\
				& \quad +  C\left(\|\xvec{u}(\cdot,t_1)\|_{\xHone(\mathcal{E})}^2 + \|\xvec{B}(\cdot,t_1)\|_{\xHone(\mathcal{E})}^2\right),
			\end{aligned} 
		\end{multline*}
		for arbitrary $\ell > 0$, and similarly one can treat
		\begin{equation*}
			\int_{t_1}^t \int_{\partial \mathcal{E}} \left(\nu_2 \xvec{M}_2 \xvec{u} \cdot \partial_t \xvec{B} + \nu_1 \xvec{L}_1 \xvec{B} \cdot \partial_t \xvec{u}\right) \, \xdx{S} \xdx{s} = \nu_1 	\int_{t_1}^t \partial_t \int_{\partial \mathcal{E}} \xvec{L}_1 \xvec{B} \cdot \xvec{u} \, \xdx{S} \xdx{s}.
		\end{equation*}
		Therefore, the inequality \eqref{equation:prfstrgslest2} can be inferred from \eqref{equation:sKem}, the transformations provided in \Cref{subsubsection:changeofvariables}, and the basic energy estimate \eqref{equation:beepr}.
	\end{proof}
	\begin{rmrk}
		When $\xvec{M}_1, \xvec{M}_2, \xvec{L}_1, \xvec{L}_2 \in \xCinfty(\overline{\mathcal{E}};\mathbb{R}^{N \times N})$ are arbitrary, one can resort to parallel energy estimates as used in \cite[Lemma 9]{CoronMarbachSueur2020} for the Navier--Stokes equations. Hereto, one additionally multiplies in \eqref{equation:MHD_ElsaesserExtGenLem} with $\partial_t\xvec{z}^{\pm}$, which, in combination with the estimates that arise from multiplying \eqref{equation:MHD_ElsaesserExtGenLem} with $\mathbb{P} \Delta \xvec{z}^{\pm}$, allows to absorb norms of the form $\|\xvec{z}^{\pm}\|_{\xY}$, where
		\[
		\xY \coloneqq \xHone((t_1,t);\xLtwo(\mathcal{E})) \cap \xLtwo((t_1,t);\xHn{{2}}(\mathcal{E})).
		\]
		Thus, similarly to \cite[Page 995]{Guerrero2006}, one can use interpolation arguments to draw a conclusion as in \Cref{lemma:aux1212}. 
	\end{rmrk}
	
	It remains to further bound the integrals $J^{\pm}$ in \eqref{equation:prfstrgslest2}. Applying embedding and interpolation inequalities, one obtains for any small constant $\ell > 0$ the bound
	\begin{equation}\label{equation:regi14es2}
		\begin{aligned}
			J^{\pm} & \leq \|(\xvec{z}^{\mp} \cdot \xnab) \xvec{z}^{\pm}\|_{\xLtwo(\mathcal{E})} \|\mathbb{P} \Delta\xvec{z}^{\pm}\|_{\xLtwo(\mathcal{E})} \\
			& \leq C\|\xvec{z}^{\mp}\|_{\xHone(\mathcal{E})}^{\frac{1}{2}}\|\xvec{z}^{\mp}\|_{\xHtwo(\mathcal{E})}^{\frac{1}{2}} \| \xnab \xvec{z}^{\pm}\|_{\xLtwo(\mathcal{E})} \|\mathbb{P} \Delta\xvec{z}^{\pm}\|_{\xLtwo(\mathcal{E})} \\
			& \leq C(\ell)\|\xvec{z}^{\mp}\|_{\xHone(\mathcal{E})} \|\xvec{z}^{\mp}\|_{\xHtwo(\mathcal{E})} \|\xvec{z}^{\pm}\|_{\xHone(\mathcal{E})}^2 + \ell \|\mathbb{P} \Delta\xvec{z}^{\pm}\|_{\xLtwo(\mathcal{E})}^2 \\
			& \leq C(\ell)\|\xvec{z}^{\mp}\|_{\xHone(\mathcal{E})}^2 \|\xvec{z}^{\pm}\|_{\xHone(\mathcal{E})}^4 + \ell \left(\|\xvec{z}^{\mp}\|_{\xHtwo(\mathcal{E})}^2 + \|\mathbb{P} \Delta\xvec{z}^{\pm}\|_{\xLtwo(\mathcal{E})}^2 \right).
		\end{aligned}
	\end{equation}
	Moreover, the estimate \eqref{equation:sKem} and Young's inequality with $1/3+4/6=1$ provide
	\begin{equation}\label{equation:regi14es3}
		\begin{aligned}
			\|\xvec{z}^{\mp}\|_{\xHone(\mathcal{E})}^2 \|\xvec{z}^{\pm}\|_{\xHone(\mathcal{E})}^4 \leq C\sum\limits_{\square\in\{+,-\}}\left(\|\xvec{z}^{\square}\|_{\xLtwo(\mathcal{E})} + \|\xcurl{\xvec{z}^{\square}}\|_{\xLtwo(\mathcal{E})}\right)^6,
		\end{aligned}
	\end{equation}
	while \eqref{equation:estlerayproj} allows to infer
	\begin{equation}\label{equation:regi14es4}
		\begin{aligned}
			\|\xvec{z}^{\mp}\|_{\xHtwo(\mathcal{E})}^2 \leq  C \sum\limits_{\square\in\{+,-\}} \left(\|\mathbb{P} \Delta \xvec{z}^{\square}\|_{\xLtwo(\mathcal{E})}^2 + \| \xcurl{\xvec{z}^{\square}}\|_{\xLtwo(\mathcal{E})}^2 + \| \xvec{z}^{\square}\|_{\xLtwo(\mathcal{E})}^2\right).
		\end{aligned}
	\end{equation}
	Thus, by combining \cref{equation:beepr,equation:prfstrgslest2,equation:regi14es2,equation:regi14es3,equation:regi14es4}, one obtains
	\begin{equation*}\label{equation:regi14es5}
		\begin{aligned}
			\int_{t_1}^{t} G^{1/2}(s) \, \xdx{s} & \leq  C(\ell)\sum\limits_{\square\in\{+,-\}} \int_{t_1}^{t} \left(\|\xvec{z}^{\square}(\cdot,s)\|_{\xLtwo(\mathcal{E})} + \|\xcurl{\xvec{z}^{\square}}(\cdot,s)\|_{\xLtwo(\mathcal{E})}\right)^6 \, \xdx{s} \\
			& \quad + \sum\limits_{\square\in\{+,-\}} \int_{t_1}^{t} \left( \ell \|\mathbb{P} \Delta \xvec{z}^{\square}(\cdot,s)\|_{\xLtwo(\mathcal{E})}^2 + C \| \xcurl{\xvec{z}}^{\square}(\cdot,s)\|_{\xLtwo(\mathcal{E})}^2\right) \, \xdx{s} \\
			& \quad + \sum\limits_{\square\in\{+,-\}} \left(\ell \| \xcurl{\xvec{z}}^{\square}(\cdot,t)\|_{\xLtwo(\mathcal{E})}^2 +  C(C_1,\ell)\|\xvec{z}^{\square}_0\|_{\xLtwo(\mathcal{E})}^2\right).
		\end{aligned}
	\end{equation*}
	Therefore, for sufficiently small parameters $\delta_1, \delta_2 \in (0,1)$, one arrives at
	\begin{equation}\label{equation:prfstrgslest3}
		\begin{aligned}
			\int_{t_1}^{t} F^{\delta_1}(s) + G^{\delta_2}(s) \, \xdx{s} & \leq C\sum\limits_{\square\in\{+,-\}} \int_{t_1}^{t} \left(\|\xvec{z}^{\square}(\cdot,s)\|_{\xLtwo(\mathcal{E})} + \|\xcurl{\xvec{z}^{\square}}(\cdot,s)\|_{\xLtwo(\mathcal{E})}\right)^6 \, \xdx{s} \\ 
			& \quad +  C  \sum\limits_{\square\in\{+,-\}}  \int_{t_1}^{t} \| \xcurl{\xvec{z}}^{\square}(\cdot,s)\|_{\xLtwo(\mathcal{E})}^2 \, \xdx{s} \\
			& \quad + C(C_1,\delta_1,\delta_2) \sum\limits_{\square\in\{+,-\}}  \|\xvec{z}^{\square}_0\|_{\xLtwo(\mathcal{E})}^2.
		\end{aligned}
	\end{equation}
	
	In order to apply Gr\"onwall's lemma in \eqref{equation:prfstrgslest3}, first one utilizes the elementary inequality
	\begin{multline*}
		\sum\limits_{\square\in\{+,-\}}\left(\|\xvec{z}^{\square}(\cdot,t)\|_{\xLtwo(\mathcal{E})} + \|\xcurl{\xvec{z}^{\square}}(\cdot,t)\|_{\xLtwo(\mathcal{E})}\right)^2 \\ \leq 2\sum\limits_{\square\in\{+,-\}}\left(\|\xvec{z}^{\square}(\cdot,t)\|_{\xLtwo(\mathcal{E})}^2 + \|\xcurl{\xvec{z}^{\square}}(\cdot,t)\|_{\xLtwo(\mathcal{E})}^2\right),
	\end{multline*}
	such that for $t > t_1$ and sufficiently small $c_1 > 0$ one has the estimate
	\begin{multline*}
		\sum\limits_{\square\in\{+,-\}}\left(\|\xvec{z}^{\square}(\cdot,t)\|_{\xLtwo(\mathcal{E})} + \|\xcurl{\xvec{z}^{\square}}(\cdot,t)\|_{\xLtwo(\mathcal{E})}\right)^2 + c_1 \sum\limits_{\square\in\{+,-\}} \int_{t_1}^t \|\xvec{z}^{\square}(\cdot, s)\|_{\xHtwo(\mathcal{E})}^2 \, \xdx{s}\\
		\begin{aligned}
			&  \leq \sum\limits_{\square\in\{+,-\}}\left( C(C_1,\ell) \|\xvec{z}^{\square}_0\|_{\xLtwo(\mathcal{E})}^2 + C\int_{t_1}^t\left(\|\xvec{z}^{\square}(\cdot,s)\|_{\xLtwo(\mathcal{E})} + \|\xcurl{\xvec{z}^{\square}}(\cdot,s)\|_{\xLtwo(\mathcal{E})}\right)^6 \, \xdx{s}\right).
		\end{aligned}
	\end{multline*}
	Thus, for a generic constant $C = C(C_1,\ell) > 0$, the function
	\begin{equation*}
		\begin{aligned}
			\Phi(t) & \coloneqq C \sum\limits_{\square\in\{+,-\}}\left( \|\xvec{z}^{\square}_0\|_{\xLtwo(\mathcal{E})}^2 + \int_{t_1}^t\left(\|\xvec{z}^{\square}(\cdot,s)\|_{\xLtwo(\mathcal{E})} + \|\xcurl{\xvec{z}^{\square}}(\cdot,s)\|_{\xLtwo(\mathcal{E})}\right)^6 \, \xdx{s}\right),
		\end{aligned}
	\end{equation*}
	obeys $\dot{\Phi}/\Phi^3 \leq C$, where $\dot{\Phi} = \xdrv{\Phi}{t}$. Taking $s_1 > 0$ small enough and integrating the latter differential inequality leads for $t \in [t_1, t_1 + s_1]$ to
	\[
	\Phi(t)^2 \leq \frac{C^2\left( \sum\limits_{\square\in\{+,-\}} \|\xvec{z}^{\square}_0\|_{\xLtwo(\mathcal{E})}^2 \right)^2}{1 - 2(t-t_1)C^3\left( \sum\limits_{\square\in\{+,-\}} \|\xvec{z}^{\square}_0\|_{\xLtwo(\mathcal{E})}^2 \right)^2}.
	\]
	Consequently, for some constant $C_2 > 0$ and all $t \in [t_1, t_1 + s_1]$ one has the estimate
	\begin{equation*}
		\begin{gathered}
			\sum\limits_{\square\in\{+,-\}} \left(\|\xvec{z}^{\square}(\cdot,t)\|_{\xHone(\mathcal{E})}^2 + c_1 \int_{t_1}^t \|\xvec{z}^{\square}(\cdot, s)\|_{\xHtwo(\mathcal{E})}^2  \, \xdx{s}\right)
			\leq C_2 \sum\limits_{\square\in\{+,-\}} \|\xvec{z}^{\square}_0\|_{\xLtwo(\mathcal{E})}^2.
		\end{gathered}
	\end{equation*}
	Therefore, there exists $t_2 \in (t_1,2T^*/3)$ and $C_3 > 0$ such that
	\begin{equation*}
		\begin{gathered}
			\|\xvec{z}^{\pm}(\cdot, t_2)\|_{\xHtwo(\mathcal{E})} \leq \sqrt{\frac{C_3}{s_1}}\left(\|\xvec{z}^{+}_0\|_{\xLtwo(\mathcal{E})}^2+\|\xvec{z}^{-}_0\|_{\xLtwo(\mathcal{E})}^2\right)^{\frac{1}{2}}.
		\end{gathered}
	\end{equation*}
	
	\paragraph{Step 3. Additional estimates for $\partial_t \xvec{z}^{\pm}$ when $t \in (t_2,T^*)$.}
	By taking~$\partial_t$ in \eqref{equation:MHD_ElsaesserExtGenLem}, multiplying the resulting equations with $\partial_t \xvec{z}^{\pm}$ respectively, and integrating over $\mathcal{E}$, one obtains for
	\[
	H(t) \coloneqq \frac{1}{2} \sum\limits_{\square\in\{+,-\}} \left(\frac{d}{dt} \| \partial_t \xvec{z}^{\pm}(\cdot, t) \|_{\xLtwo(\mathcal{E})}^2 + (\lambda^+ \square \, \lambda^-) \| \xcurl{(\partial_t\xvec{z}^{+}\square \, \partial_t\xvec{z}^{-})}(\cdot, t) \|_{\xLtwo(\mathcal{E})}^2\right)
	\]
	the estimate
	\begin{equation*}
		\begin{aligned}
			H & \leq - \int_{\mathcal{E}} (\partial_t \xvec{z}^{-} \cdot \xnab) \xvec{z}^{+}  \cdot \partial_t\xvec{z}^{+} \, d \mathbf{x} - \int_{\mathcal{E}} (\partial_t \xvec{z}^{+} \cdot \xnab) \xvec{z}^{-}  \cdot \partial_t\xvec{z}^{-} \, d \mathbf{x} \\
			& \quad	+ \sum\limits_{\substack{(\triangle,\circ)\in \\ \{(+,-), (-,+)\}}} \int_{\partial\mathcal{E}} \left(\lambda^+\xvec{\rho}^{\circ}(\partial_t\xvec{z}^+,\partial_t\xvec{z}^-) + \lambda^- \xvec{\rho}^{\triangle}(\partial_t\xvec{z}^+,\partial_t\xvec{z}^-) \right)\cdot \partial_t\xvec{z}^{\circ} \, \xdx{S}.
		\end{aligned}
	\end{equation*}
	Therefore, considerations similar to the previous steps lead for some constants $C_4 > 0$ and~$c_2 \in (0,1)$, a possibly small time $s_2 \in (0,T^*/3)$ and all $t \in [t_2, t_2 + s_2]$, to the bound
	\begin{equation*}
		\begin{aligned}
			\|\partial_t\xvec{z}^{\pm}(\cdot, t)\|_{\xLtwo(\mathcal{E})}^2 + c_2 \int_{t_2}^{t} \|\partial_t \xvec{z}^{\pm}(\cdot, s)\|_{\xHone(\Omega)} \, \xdx{s} \leq C_4 \sum\limits_{\square\in\{+,-\}} \|\partial_t\xvec{z}^{\square}(\cdot, t_2)\|_{\xLtwo(\mathcal{E})}^2,
		\end{aligned}
	\end{equation*}
	noting that
	\begin{equation*}
		\begin{aligned}
			\|\partial_t\xvec{z}^{\pm}(\cdot, t_2)\|_{\xLtwo(\mathcal{E})}^2 & \leq \|(\xvec{z}^{\mp}\cdot \xnab)\xvec{z}^{\pm}(\cdot, t_2)\|_{\xLtwo(\mathcal{E})}^2 + \|\xnab p^{\pm}(\cdot, t_2)\|_{\xLtwo(\mathcal{E})}^2 \\
			& \quad + \|\Delta(\lambda^{\pm}\xvec{z}^{+}  +\lambda^{\mp}\xvec{z}^{-})(\cdot, t_2)\|_{\xLtwo(\mathcal{E})}^2 \\
			&\leq C\sum\limits_{\square\in\{+,-\}} \|\xvec{z}^{\square}(\cdot, t_2)\|_{\xHtwo(\mathcal{E})}^2.
		\end{aligned}
	\end{equation*}
	Hence, there exists a time $t_3 \in [t_2,t_2+s_2]$ and a constant $C_5 > 0$ such that
	\begin{equation*}
		\begin{aligned}
			\|\partial_t\xvec{z}^{\pm}(\cdot, t_3)\|_{\xHone(\mathcal{E})}  \leq \sqrt{\frac{C_5}{s_2} \sum\limits_{\square\in\{+,-\}} \|\partial_t\xvec{z}^{\square}(\cdot, t_2)\|_{\xLtwo(\mathcal{E})}^2}.
		\end{aligned}
	\end{equation*}
	
	\paragraph{Step 4. Conclusion.}
	The proof of \Cref{lemma:reg} is concluded by shifting the time derivative and the nonlinear terms in \eqref{equation:MHD_ElsaesserExtGenLem} to the right-hand side, followed by multiple applications of \Cref{lemma:stsreg} and \Cref{remark:Wkp}. This first provides $\xLinfty([t_2,t_2+s_2];\xHtwo(\mathcal{E}))$ bounds for $\xvec{z}^{\pm}$ and subsequently $\xHn{3}(\mathcal{E})$ bounds for $\xvec{z}^{\pm}(\cdot,t_3)$.

	\section*{Acknowledgments}
	{\bf Funding:} This research was partially supported by National Key R\&D Program of China under Grant No. 2020YFA0712000, National Natural Science Foundation of China under Grant No. 12171317, Strategic Priority Research Program of Chinese Academy of Sciences under Grant No. XDA25010402, and Shanghai Municipal Education Commission under Grant No. 2021-01-07-00-02-E00087.

	\bibliographystyle{abbrv}
	\bibliography{mhd_visc_control}
\end{document}